%% file: Riemannian_geometry_to_higher_order.tex
\documentclass[
10pt,
letterpaper,
oneside,
headinclude,footinclude, 
BCOR=5mm, 
]{article}

\usepackage{mathptmx}
\usepackage{helvet}
\usepackage{courier}
\usepackage{type1cm}         

\usepackage{makeidx}         
\usepackage{graphicx}        
\usepackage[bottom]{footmisc}

\usepackage{dcolumn}
\newcolumntype{d}[1]{D{.}{.}{#1}}
\usepackage{amsthm}
\usepackage{mathrsfs}
\usepackage{amsmath}
\usepackage{amssymb}
\usepackage{newtxtext}       %
\usepackage{newtxmath}       
\usepackage{epigraph}
\usepackage[LGR,T1]{fontenc}
\usepackage{accents}
\usepackage{bbding}
\usepackage{lscape}

\usepackage{hyperref}

\title{\normalfont{Riemannian Geometry to Higher Order in the Infinitesimals}} 

\author{{William Edward Bies
	} \\ 
{E-mail: william.bies.phd@gmail.com
}}

\date{June 6, 2024} 

\begin{document}
	
	\numberwithin{equation}{section}
	\theoremstyle{plain}
	\newtheorem{theorem}{Theorem}[section]
	\newtheorem{proposition}{Proposition}[section]
	\newtheorem{lemma}{Lemma}[section]
	\newtheorem{corollary}{Corollary}[section]
	\newtheorem{postulate}{Postulate}[section]
	
	\theoremstyle{definition}
	\newtheorem{definition}{Definition}[section]
	
	\theoremstyle{remark}
	\newtheorem{remark}{Remark}[section]
	
	\renewcommand{\qedsymbol}{\small \RectangleBold}
	
	\maketitle
	
	\setcounter{tocdepth}{2} 
	
	\tableofcontents
	
	\newpage
	
	\section*{Abstract}
	
	Differential geometry may be generalized to allow infinitesimals to any order. The purpose of the present contribution is to show that the theory so developed expands received geometrical ideas in an interesting way, rich in potential for future exploration. The first order of business is to furnish the notion of a higher tangent vector, as defined abstractly by means of commutative algebra, with a workable interpretation in terms of spatial intuition. Then we introduce the differential calculus of the so-called jet connection, viz., an extension of the usual affine connection that takes higher tangent vectors as its arguments---thereby enabling us to give a sense to parallel transport in the direction of a higher tangent, what has (to our knowledge) never been entertained before. After generalizing the Riemannian metric tensor to include a dependence up to any order in the infinitesimals, we arrive at natural analogues of the Levi-Civita connection and the Riemannian curvature tensor which display novel phenomena rooted in interactions among infinitesimals differing in order. Finally on the integral side, an intrinsic theory of integration adapted to integrands possibly of higher than first order in the differentials is developed, with a view towards eventually defining an action functional that will be applicable in the general theory of relativity.
		
	\vspace{12 pt}
	
	2020 Mathematics Subject Classification. Primary: 53A45 Differential geometric aspects in vector and tensor analysis. Secondary: 58A05 Differentiable manifolds, foundations; 53A17 Differential geometric aspects in kinematics; 58C99 Calculus on manifolds, none of the above.
		
	\vspace{12 pt}
	
	Keywords: higher-order infinitesimals, differential calculus, higher-order differential geometry 
	
	\vspace{12 pt}
		
	Subject Classifications: Riemannian geometry, local differential geometry, analysis on manifolds
	
	\vspace{12 pt}
	
	
	\include{section_1}
	\include{section_2}
	\include{section_3}
	\include{section_4}
	\include{section_5}
	\include{discussion}

\include{references}
\end{document}

%% file: section_1.tex
\section{Introduction}\label{chapter_1}

At the root of the present work lies a twofold reconceptualization of received ideas, first the mathematical, as to the pure doctrine of motion and second, the dynamical, as to quantity of motion (momentum). The basic premise informing our research is that the conventional Riemannian geometry everyone employs these days may well be adequate as a guide to the phenomena at leading order, but that (as Riemann himself foresees) the presence of infinitesimals of higher than first order may reflect itself in subtler effects to be revealed in regimes not so far thoroughly investigated, whether in the domain of the infinitely small or across sufficiently large distances and times. As we shall see, the ordinary tensor calculus can certainly be generalized to include contributions from higher-order infinitesimals. The difficulty of the matter consists in the need to seek out an appropriate starting point so as to ensure that the theory so obtained becomes suitably ample and versatile. As we suggest in the following subsections, the fruitful development of an higher-order geometry has, to date, been hindered by inappropriate definitions it will be our concern to overcome. 

\subsection{Motivations from Pure Mathematics}

Wallis in 1656 \cite{wallis} seeks to devise a generally applicable technique of quadrature by replacing Cavalieri's geometrical indivisibles with arithmetical infinitesimals that can be summed according to a known rule. At this stage, the infinitesimals arrived at via a limiting process interpolate a discrete sequence and the answer may be obtained by induction, heuristically if not entirely rigorously. In view of his focus on carrying out quadratures, Wallis has little use for any algebraical properties his infinitesimals may obey. Newton's development of the calculus takes Wallis' results as starting point, but he too prefers to stress the analytical operations of differentiation and integration so that any algebraical role of infinitesimals falls by the wayside, see \cite{newton_fluxions}, \cite{newton_de_analysi} and \cite{newton_de_methodis}. Therefore, we turn to Fontenelle in 1727 \cite{fontenelle}. Inspired by Descartes and Leibniz, Fontenelle adopts the position that the infinite, both the infinitely large and the infinitely small, far from being purely indeterminate in fact may be treated as a known quantity with which one can perform arithmetical operations. Thus, we can confer a sense on expressions such as $2 \infty, 3 \infty, 1/\infty$, $\infty \cdot \infty = \infty^2$ and many others. Hence, there exists a differential calculus defined by regular conventions to be sought out. Fontenelle's key insight is to recognize the existence of orders of infinity and to retain entities of higher order rather than to discard them as always being negligible in the limit.

After Cauchy, who still admits a place, though minimal, for infinitesimal magnitudes, Riemann's groundbreaking definition of his integral for sufficiently well behaved functions leaves no room for infinitesimals to play any role in the calculus, as is confirmed by Weierstrass' completely rigorous epsilontic formulation of real analysis. The situation with respect to this question has remained unchanged ever since, notwithstanding the twin twentieth-century revolutions of Robinson's non-standard analysis and Grothendieck's scheme theory. For the former, while indeed formally valid and reintroducing infinitesimal magnitudes in a rigorous fashion, happens to be too unwieldy for anyone to reason effectively with them, and in consequence nobody has ever made significant use of non-standard analysis in theoretical physics, to this author's knowledge. As for the latter, the case seems to be one rife with unrealized promise. For while Grothendieck offers a more intuitively accessible notion of an infinitesimal than does Robinson and great though the employment to which algebraic geometers have put his scheme-theoretical program may be, they have continued to be disconnected from analysis and so too have yet to enter mathematical physics in an instrumental way.

The right framework in terms of which to handle infinitesimals is that of modern commutative algebra, where one applies scheme theory in the context of the algebra of smooth functions on a differentiable manifold. Thus, one arrives at non-reduced elements, or jets, which from their definition admit of intuitive visualization. But the dual objects, or higher tangent vectors, are thus far defined only algebraically, as linear functionals on spaces of jets. What is the geometrical meaning of higher tangency, and how can it be connected with the concept of osculating contact? The stimulus to the investigations herein detailed has been, precisely, the discovery of a workable and intuitive interpretation of what a higher-order tangent vector would have to be. Another equally vital question is that as to whether it be possible to integrate a higher tangent vector field in order to obtain a diffeomorphism flow? It is not immediately evident how to do so while avoiding overdetermination. As we shall see in {\S}\ref{chapter_3}, one can indeed arrive at a sense in which higher tangents generate a flow if one allow virtual displacements in the cotangent space. Thus, by conferring inertial properties on space itself, as it were, we undertake a far-reaching revision of the very concept of motion.

\subsection{The Present Work in the Context of Previous Literature}
 
The idea of a differential geometry to higher order has previously been broached in the literature, but held to be unpromising for insufficient reasons. Pohl \cite{pohl}, indeed, defines infinitesimals of higher order in terms equivalent to ours, but fails to go beyond a consideration of Chern classes to any notions that would qualify as differential geometric, properly speaking. Hence, he misses the jet connection and the possibility of defining a Riemannian metric to higher order (thus, no Levi-Civita jet connection and no generalized Riemannian curvature tensor). Needless to say, Pohl makes no application of his formalism to the general theory of relativity, at all. Meyer \cite{meyer} considers higher tangents and jets to second order only, having in mind applications to the stochastic calculus on manifolds. He employs an altogether different definition of affine connection (1-vector valued only and restricted to be the identity on 1-vectors themselves). Hence, he too misses the possibility of a generalized Riemannian metric, the Levi-Civita jet connection and the generalized Riemannian curvature tensor. Thus, his contention that any potential extension of his ideas to the case of infinitesimals of greater than second order would be irrelevant has to be taken with a grain of salt---his artificial constraint on the form of the jet connection means that he cannot apply his theory to relativistic physics, either.
 
While infinitesimals to any order function as common ingredients in many theories of pure mathematics, for instance, Grothendieck's scheme theory \cite{grothendieck}, the Russian school of differential calculus \cite{gamkrelidze}, \cite{nestruev}, \cite{vinogradov}, Kock's synthetic geometry \cite{kock} or Sardanashvily's work on jet spaces in the context of classical field theory \cite{sardanashvily}, the fault germane to all of these authors (from the physicist's point of view) consists in a neglect of any metrical structure to space such as we know very well plays a central role in the general theory of relativity. The same criticism could be leveled against Connes' non-commutative geometry \cite{connes}. His version of the standard model as emerging from an almost-commutative algebra, the conceptual bases of which are Dirac's electron-positron theory and Yang-Mills' gauge theory, remains subject to a large measure of arbitrariness as far as the choice of the non-commutative algebras governing the internal degrees of freedom goes. Moreover, the price one pays in sacrificing spatial intuition when going from commutative to non-commutative algebras seems to be too steep.
 
The mentioned works have higher infinitesimals but no geometry; the opposite problem appears in the work of Boucetta on a Riemannian geometry in the context of Lie algebroids \cite{boucetta}. The concept of a jet connection to be advanced below in {\S}\ref{chapter_3} does not appear to be equivalent to Boucetta's standard connection which can be defined on every Lie algebroid with the aid of an anchor function. Of course, the higher tangent vector fields can be viewed as a Lie algebroid where the Lie bracket $[X,Y]$ is obtained from the commutator of $X$ and $Y$ viewed as linear partial differential operators. But our definition of a higher-order connection differs essentially from Boucetta's in obeying a Leibniz product rule, which is the appropriate extension to higher-order differential operators of the derivation property satisfied by a first-order operator. In other words, there is no anchor function in our approach. As far one can tell, the use of an anchor function in the conventional approach to the standard Lie algebroid connection prevents different orders of infinitesimals from interacting in an effective way. In consequence, the sense in which one discusses parallel transport there is quite different  from what we mean here by a jet geodesic and so fails to embody any novel kinematical notions of the kind our work discloses.\footnote{What we are seeking is a covariant notion of partial differentiation with respect to an arbitrary multi-index. In particular, observe that one will want to posit a distinction between, say, $\nabla_{\partial^2/\partial x \partial y}$ and $\nabla_{\partial/\partial x} \nabla_{\partial/\partial y}$. If the latter were all, one would encounter the problem that, in general, $\nabla_{\partial/\partial x} \nabla_{\partial/\partial y} \ne \nabla_{\partial/\partial y} \nabla_{\partial/\partial x}$.} All along we keep always in mind the concrete case of higher tangent vector fields and generalized tensors built from them, whereas Boucetta stays at the highly abstract level of an arbitrary Lie algebroid and hence omits to avail himself of spatial intuition.

It should be clear that the idea of going to higher order will involve a certain amount of repetition of results from the first-order case, in that the necessary machinery will often (but not always) involve a fairly routine extension of what has been done before. That is why we include references to the corresponding results in the textbooks by J.M. Lee and W. Thirring. But we then go on to do things with the formalism that are non-trivial, in the sense that they make essential use of the coupling between infinitesimals of different order, which of course cannot enter when one limits oneself to first-order infinitesimals. Specific formal results never before anticipated include proposition \ref{jeteuclid}, theorem \ref{existence_theorem_for_jet_geodesics}, theorem \ref{exist_levi_civita}, theorem \ref{tensoriality_of_curvature}, {\S}\ref{riemannian_appendix} and proposition \ref{properties_of_integral}.

%% file: section_2.tex
\section{Foundations}\label{chapter_2}

If the introduction of higher infinitesimals into differential geometry is to be consequential, they need to be rendered accessible to spatial intuition. The problem is that, to date, the geometrical concept of osculating contact has yet to be connected with the analytical concept of partial differentiation. Therefore, we seek a sense in which a higher tangent vector field may be viewed microlocally in cotangent space and the flow it generates integrated. When this is done, the higher components will be seen to correspond to virtual displacements which combine to generate the collective motion described by the integrated flow.

\subsection{Formal Definition of Jets and Higher-Order Tangents}\label{formal_def}

Our guiding intuition is that space may have more structure than is captured by ordinary tensor analysis. Let $\phi_\lambda$ be a 1-parameter group of diffeomorphisms acting on a differentiable manifold $M$. Its orbits $\gamma: [a,b] \rightarrow M$ are 1-dimensional curves in $M$ along which the tangent vector to $\gamma$ is translated by the diffeomorphism flow. To gain a full picture of the local behavior of $\phi_\lambda$, one has to look in an open neighborhood enclosing $\gamma$. The nearby space can twist, turn, shear and converge or diverge in various directions under the action of $\phi_\lambda$. Jets, being higher-order constructs, could be an apt tool for describing such phenomena. A vector attached to $\gamma$ itself detects only the direction of flow and its bending and local reorientation with respect to the line of flow, according to the well-known Serret-Frenet formula and its generalization to spaces of dimension greater than three \cite{lanczos}. 

Let $M$ be an $n$-dimensional differentiable manifold, $p \in M$ a point in $M$ {\cite{warner}}. In the spirit of the Russian school, everything will be based on the commutative algebra of smooth functions on $M$, denoted $C^\infty(M)$. It is natural to characterize points in $M$ by their maximal ideals. We adopt a local perspective and restrict attention to the algebra $C^\infty(M)_p$ of germs of smooth functions at $p$. Let $\mathfrak{m}_p$ denote the ideal in $C^\infty(M)_p$ of germs of functions that vanish at $p$. The usual construction is to represent a vector $\varv$ at $p$ as linear derivation of $C^\infty(M)_p$, the vector space of all such derivations being denoted $T_pM$. As we shall see in a moment, however, the derivation property constitutes something of a distraction and will have to be relinquished. The first elementary result is to characterize vectors so defined in algebraic terms as the dual space $T_pM \cong (\mathfrak{m}_p/\mathfrak{m}_p^2)^*$ (cf. \cite{warner}, Lemma 1.16). An element $[f]$ of the space $\mathfrak{m}_p/\mathfrak{m}_p^2$ is naturally viewed as the differential $df$ of the function $f$ at $p$. These constructions can be generalized to higher-order tangents and differentials as follows (cf. \cite{warner}, \S 1.26):

\begin{definition}\label{jetdef}
	An $r$-jet, $r \ge 1$ \textup{(}or $r$-th order differential\textup{)}, at $p$ is an element of $J^{r*}_pM := \mathfrak{m}_p/\mathfrak{m}_p^{r+1}$. The $r$-th order differential of a germ of a function $f \in C^\infty(M)$ is given by $d^rf_p = [f-f(p)]$, where the square brackets indicate the operation of taking the equivalence class modulo $\mathfrak{m}_p^{r+1}$. A higher-order vector, or a tangent of order $r$, is an element of the dual space $J^r_pM := (\mathfrak{m}_p/\mathfrak{m}_p^{r+1})^*$.
\end{definition}

\begin{remark}
	We imitate the conventional notation of $T_pM$ and $T^*_pM$ for first-order vectors and differentials, respectively. Note that a differential $j \in J^{r*}_pM$ acts as a linear functional on vectors $\varv \in J^r_pM$ via the canonical identification $J^{r*}_pM \cong J^{r*}_pM^{**} = J^r_pM^*$.
\end{remark}

\begin{remark}
	On why the derivation property is inessential to being a vector: as suggested from its definition above, a vector should be a linear operator that looks locally at germs of functions as represented through their jets. If one were to insist upon retaining the property of being a derivation with respect to the multiplicative structure of the algebra of jets, it would be tantamount to imposing an unnatural constraint, for we know from the ordinary calculus that differentiation observes, in its behavior with respect to multiplication, not always the derivation property per se but the Leibniz rule, which is to be thought of as a generalized derivation property applicable to higher-order derivatives.
\end{remark}

We can give explicit expressions for these quantities if we adopt coordinates $x^1,\ldots,x^n$ in a neighborhood of $p$. When the coordinate functions are understood, let us adopt the notation
\begin{align}\label{eq1}
	d^\alpha|_p &= d^{\alpha_1}x^1 \cdots d^{\alpha_n}x^n = [(x^1-x^1(p))^{\alpha_1}\cdots(x^n-x^n(p))^{\alpha_n}] \\
	\partial_\alpha|_p &= \frac{1}{\alpha!} \frac{\partial^{|\alpha|}}{\partial (x^1)^{\alpha_1}\cdots \partial (x^n)^{\alpha_n}} \bigg|_p \label{def_vector_basis}
\end{align}
where $\alpha = (\alpha_1,\ldots,\alpha_n)$, $\alpha_i \in \vvmathbb{N}_0$, $i=1,\ldots n$, $\alpha! = \alpha_1!\cdots \alpha_n!$, $|\alpha|=\sum_{i=1}^n \alpha_i$ is a multi-index and $\partial_\alpha|_p$ is the dual vector defined by $\partial_\alpha|_p (d^\beta|_p) \ = \delta_{\alpha\beta}$ ($\delta_{\alpha\beta}$ designating the Kronecker delta symbol on multi-indices).

The following elementary statements are standard and therefore will be quoted without proof.

\begin{proposition}
The $d^\alpha|_p$ with $1 \le |\alpha| \le r$ form a basis of $J^{r*}_pM$ and the $\partial_\alpha|_p$ with $1 \le |\alpha| \le r$ form a basis of $J^r_pM$.
\end{proposition}

Let $J^{r*}M$ denote the disjoint union of the jet spaces at all points of $M$:
\begin{equation}\label{eq11}
	J^{r*}M = \coprod_{p \in M} J^{r*}_pM.
\end{equation}
Correspondingly, let $J^rM$ denote the disjoint union of the higher-order vector spaces at all points of $M$:
\begin{equation}\label{eq16}
	J^rM = \coprod_{p \in M} J^r_pM.
\end{equation}

\begin{proposition}\label{jet_bundle}
	The collections $J^{r*}M$ resp. $J^rM$ have the structure of a vector bundles over $M$.
\end{proposition}

We can now specify how the differential and vector basis elements transform under change of coordinates. Let $y^1,\ldots,y^n$ be another set of coordinates in the vicinity of $p$. Denote by $x: U \rightarrow \vvmathbb{R}^n$ and $y: V \rightarrow \vvmathbb{R}^n$ the coordinate-chart mappings. Then on the domains where they are defined, $y \circ x^{-1}$ and $x \circ y^{-1}$ are smooth mappings from $\vvmathbb{R}^n$ to $\vvmathbb{R}^n$ which are inverse to each other. We can perform a Taylor expansion around the point $p$ with respect to one or the other coordinates and equate like terms. We merely quote the final result to third order, which is included here because we shall have constant occasion to make reference to it in what follows. 

Let the jet at the point $p$ be defined in the two coordinate frames by
\begin{equation}
a_\alpha dx^\alpha|_p + a_{\beta\gamma} dx^\beta dx^\gamma|_p + a_{\delta\varepsilon\zeta} dx^\delta dx^\varepsilon dx^\zeta|_p = 
b_\mu dy^\mu|_p + b_{\nu\lambda} dy^\nu dy^\lambda|_p + b_{\rho\sigma\tau} dy^\rho dy^\sigma dy^\tau|p
\end{equation}
and the tangent vector there by
\begin{equation}
 \varv^\alpha \partial_\alpha|_p + \varv^{\beta\gamma} \partial_{\beta\gamma}|_p + \varv^{\delta\varepsilon\zeta} \partial_{\delta\varepsilon\zeta}|_p =
  \varw^\mu \partial^\prime_\mu|_p + \varw^{\nu\lambda} \partial^\prime_{\nu\lambda}|_p + \varw^{\rho\sigma\tau} \partial^\prime_{\rho\sigma\tau}|_p.
\end{equation}
Then the transformation law is given by
\begin{equation}\label{jet_transf_law}
	\begin{pmatrix}
		b_\mu \cr b_{\nu\lambda} \cr b_{\rho\sigma\tau}
	\end{pmatrix}
	= \begin{pmatrix}
		\frac{\partial x^\alpha}{\partial y^\mu}\big|_p & 0 & 0 \cr
		\frac{1}{2} \frac{\partial^2 x^\alpha}{\partial y^\nu \partial y^\lambda}\big|_p
		& \frac{\partial x^\beta}{\partial y^\nu}\big|_p \frac{\partial x^\gamma}{\partial y^\lambda}\big|_p & 0 \cr
		\frac{1}{6} \frac{\partial^3 x^\alpha}{\partial y^\rho \partial y^\sigma \partial y^\tau}\big|_p &
		\frac{1}{2} \left( \frac{\partial^2 x^\beta}{\partial y^\rho \partial y^\sigma}\big|_p \frac{\partial x^\gamma}{\partial y^\tau}\big|_p + \frac{\partial x^\beta}{\partial y^\rho}\big|_p \frac{\partial^2 x^\gamma}{\partial y^\sigma \partial y^\tau}\big|_p \right)
		& \frac{\partial x^\delta}{\partial y^\rho}\big|_p
		\frac{\partial x^\varepsilon}{\partial y^\sigma}\big|_p \frac{\partial x^\zeta}{\partial y^\tau}\big|_p
	\end{pmatrix}
	\begin{pmatrix}
		a_\alpha \cr a_{\beta\gamma} \cr a_{\delta\epsilon\zeta}
	\end{pmatrix}.
\end{equation}
resp.,
\begin{equation}\label{vector_transf_law}
	\begin{pmatrix} \varw^\mu \cr \varw^{\nu\lambda} \cr \varw^{\rho\sigma\tau} \end{pmatrix}
	= \begin{pmatrix}
		\frac{\partial y^\mu}{\partial x^\alpha}\big|_p & \frac{1}{2} \frac{\partial^2 y^\mu}{\partial x^\beta \partial x^\gamma}\big|_p & 
		\frac{1}{6} \frac{\partial^3 y^\mu}{\partial x^\delta \partial x^\varepsilon \partial x^\zeta}\big|_p \cr
		0 & \frac{\partial y^\nu}{\partial x^\beta}\big|_p \frac{\partial y^\lambda}{\partial x^\gamma}\big|_p &
		\frac{1}{2} \left( \frac{\partial^2 y^\nu}{\partial x^\delta \partial x^\varepsilon}\big|_p \frac{\partial y^\lambda}{\partial x^\zeta}\big|_p + \frac{\partial y^\nu}{\partial x^\delta}\big|_p \frac{\partial^2 y^\lambda}{\partial x^\varepsilon \partial x^\zeta}\big|_p \right) \cr
		0 & 0 & \frac{\partial y^\rho}{\partial x^\delta}\big|_p
		\frac{\partial y^\sigma}{\partial x^\varepsilon}\big|_p \frac{\partial y^\tau}{\partial x^\zeta}\big|_p
	\end{pmatrix} 
	\begin{pmatrix} \varv^\alpha \cr \varv^{\beta\gamma} \cr \varv^{\delta\varepsilon\zeta} \end{pmatrix}.
\end{equation}
The jet calculus is complicated compared to the usual vector calculus by the presence of higher-order derivatives of the coordinate functions in the transition functions between bases, which requires us to take more care than when deriving the corresponding transformation laws for ordinary tensors.

Let us denote the space of smooth sections of $J^{r*}M$ by $\mathscr{J}^{r*}(M)$ and that of $J^rM$ by $\mathscr{J}^r(M)$. A section $\omega \in \mathscr{J}^{r*}(M)$ will be called a jet field and a section $X \in \mathscr{J}^r(M)$ will be called a higher-order vector field. As for any vector bundle, $\mathscr{J}^{r*}(M)$ and $\mathscr{J}^r(M)$ are projective modules of finite type over $C^\infty(M)$ (\cite{nestruev}, Theorem 11.29 and 11.32). In any given chart on $U \subset M$ with coordinate functions $x^1,\ldots,x^n$, we have coordinate jet fields $d^\alpha: U \rightarrow J^{r*}M$ defined by $d^\alpha(p) = d^\alpha|_p$ and coordinate vector fields $\partial_\alpha: U \rightarrow J^rM$ defined by $\partial_\alpha(p)=\partial_\alpha|_p$. 

A higher-order vector field $X \in \mathscr{J}^r(M)$ can be viewed as a differential operator of order $r$ acting on functions $f \in C^\infty(M)$ if we define $(Xf)_p = X_p([f])$. Since 
\begin{equation}\label{eq17}
	[f] = \sum_{1\le|\alpha|\le r} \frac{1}{\alpha!} \frac{\partial^{|\alpha|} f}{\partial (x^1)^{\alpha_1}\cdots \partial (x^n)^{\alpha_n}} \bigg|_p d^\alpha 
\end{equation}
and $X = \sum_{1 \le |\beta| \le r} X^\beta \partial_\beta$, we have
\begin{equation}\label{eq18}
	Xf = \sum_{1 \le |\alpha| \le r} X^\alpha 
	\frac{1}{\alpha!} \frac{\partial^{|\alpha|} f}{\partial (x^1)^{\alpha_1}\cdots \partial (x^n)^{\alpha_n}}
\end{equation}
Since the definition of $Xf$ is intrinsic, the right-hand side of (\ref{eq18}) is in fact independent of the choice of coordinates.

Just as for vectors and covectors we have a concept of tensors of arbitrary mixed type formed by taking products of the tangent and cotangent bundles, let us define a generalized tensor bundle of type ${k \choose j}$ and order $r$ to be the vector bundle
\begin{equation}\label{eq19}
	K^{r,k}_jM = \overbrace{ J^{r*}M \otimes \cdots \otimes J^{r*}M}^{k~\mathrm{times}} \otimes \overbrace{ J^rM \otimes \cdots \otimes J^rM}^{j~\mathrm{times}}.
\end{equation}
Its space of smooth sections will be denoted $\mathscr{K}^{r,k}_j(M)$.

\begin{remark} The generalized tensors form a natural geometric structure which obeys a principle of general covariance \textup{(}cf. \textup{\cite{gamkrelidze}}, Chapter 6, {\S}1.5\textup{)}.
\end{remark}

\begin{definition}
	Let $\varphi: N \rightarrow M$ be a smooth mapping. If $\varphi$ takes $q \in N$ to $p \in M$, define the pullback of a jet $j \in J^{r*}_p(M)$ by $(\varphi^*j)_q := [ f \circ \varphi]_{\mathfrak{m}_q^{r+1}}$ where $j = [f]_{\mathfrak{m}_p^{r+1}}$ (it is independent of the choice of $f$). The jet field $\omega \in \mathscr{J}^{r*}(M)$ has a pullback $\varphi^*\omega \in \mathscr{J}^{r*}(N)$ given by $(\varphi^*\omega)_q = (\varphi^*\omega_{\varphi(q)})_q$. Define the pushforward of a higher-order vector $\varv \in J^r_qN$ at $q$ to be the vector at $p$ given by $\varphi_*\varv([f]_{\mathfrak{m}_p^{r+1}}) := \varv([f \circ \varphi]_{\mathfrak{m}_q^{r+1}})$ for $f \in C^\infty(M)$. If $\varphi$ is a diffeomorphism, we can define the pushforward of a higher-order vector field $X \in \mathscr{J}^r(N)$ to be $\varphi_*X \in \mathscr{J}^r(M)$ given by $(\varphi_*X)_p(f) := X_{\varphi^{-1}(p)}(f \circ \varphi)$.
\end{definition}

\subsection{Jets of Arbitrary Order}

To obviate any concerns about technical subtleties with infinite-dimensional manifolds, $\mathscr{J}^\infty(M)$ and $\mathscr{J}^{\infty *}(M)$ should be thought of as sheaves of modules over the manifold $M$. The sheaf $\mathscr{J}^\infty$ is to be obtained as the direct limit of the $\mathscr{J}^r$ with respect to the canonical injections $\mathscr{J}^{r_1} \rightarrow \mathscr{J}^{r_2}$, $r_2>r_1$, while the sheaf $\mathscr{J}^{\infty*}$ will correspondingly be the inverse limit of the $\mathscr{J}^{r*}$ with respect to the canonical projections $\mathscr{J}^{r_2*} \rightarrow \mathscr{J}^{r_1*}$, $r_2>r_1$. Roughly speaking, these objects behave somewhat like polynomials resp. formal power series; with respect to a given coordinate system, a tangent of indefinite order resp. a jet of indefinite order may contain non-zero components up to an arbitary but finite, resp. any order and is to be identified as equivalent to another if all components agree. Since every element of $\mathscr{J}^\infty$ is obtained by injection from some $\mathscr{J}^{r_0}$ with $r_0<\infty$, its operation upon a jet of indefinite order will always be well defined since once may canonically project the latter down to $\mathscr{J}^{r_0*}$ (and for a given tangent, $r_0$ remains unchanged under change of coordinate).

The following observation explains why it would be unnatural to stop at any finite order in the jets or tangents: the commutator of two linear differential operators of order $r$ is, in general, a differential operator of order $2r-1$. Hence, there cannot exist a Lie bracket that closes upon itself within $\mathscr{J}^r$ itself when $r>1$.

\subsection{Intuitive Picture of what a Higher-Order Tangent is}

Consider a trajectory leaving a given point in a certain direction. To linear order, an ordinary 1-vector as velocity tells us where it goes. In view of this, one might expect a higher-order tangent vector to specify not only the immediate direction taken by the trajectory upon departure, but its deviation from linearity which would be quadratic, or perhaps even higher-order still, in time. However appropriate this manner of approaching the problem may be, a moment's reflection will show it to be insufficient for our purposes. As a count of dimension as vector spaces over $\vvmathbb{R}$ readily shows, such deviations cannot be thought of as iterated tangents, or tangents of tangents etc., which would live in $TTM$, $TTTM$ and $T^kM$, in general.  In $n$ dimensions, there are $n$ possible directions in which to turn when deviating from a strictly linear trajectory to second order, and $n$ further such directions at each higher order. This would imply a vector space of dimension $rn$, if it were to describe an $r$-th order higher tangent. But the dimension count does not agree; hence, we must abandon this approach. From a theoretical standpoint, one could make the telling criticism that iterated tangent bundles do not lead to a qualitatively new concept of tangency, in that, whatever the level one works at, one still deals with linear segments when going to the infinitely small (as is implicit in definition of a vector at a point as a derivation on germs of functions there, which viewed as a differential operator can only be of first order).

What have we failed to take into account here? Clearly, something is faulty in the intuitive picture sketched above. In order to disentangle what this may be, let us focus on what an ordinary 1-vector field does in the small (cf. \cite{schutz}, {\S}{\S}2.12-2.20). Therefore, let $X$ be a 1-vector field on a differentiable manifold $M$ in the vicinity of the point $p$. If we zoom in close enough to $p$, any spatial variation of $X$ can be neglected. In other words, we can work in the cotangent space $T^*_pM$ spanned by the $\xi_{1,...,n}$. If $X$ has components $\alpha_{1,...,n}$ in the corresponding coordinates, then its action on an arbitrary smooth function $f \in C^\infty(M)$ near $p$ can be written as
\begin{equation}
	f+df=\mathrm{exp}(-\lambda X)f = f - \lambda(\alpha_1\xi_1+\cdots+\alpha_n\xi_n)+o(\lambda).
\end{equation}
What could the right geometrical interpretation of such an expression be? If we view the function $f$ as elevation in a topographical contour map and think of the 1-vector $X_p = (\alpha_1,...,\alpha_n)$ in usual terms as an arrow pointing away from the origin at $p$, then what the leading correction to $f$ in the immediately preceding formula corresponds to is just the vertical increment in elevation one experiences upon moving infinitesimally away from the origin in the direction of $X_p$. The graph of $f$ will be locally a hyperplane passing through the origin of cotangent space at $p$ and the change in elevation $df$ under infinitesimal displacement in the direction indicated by $X_p$ will then be given by sliding the hyperplane in that direction. Taking a slightly less local point of view, $f$ should be transported along the integral curves of $X$.

What can we learn from this scenario? First, we can single out the point $p$ as the common zero locus of $n$ appropriately chosen independent functions $f_{1,...,n}$; so $\{p\}=Z_{f_1}\cap\cdots\cap Z_{f_n}$ where $Z_f:=\{q \in M|f(q)=0\}$. It suffices, thereby, to concentrate on the level sets of a single given function $f$. Here, restricting attention to the situation in the small means that $Z_f$ reduces to a hyperplane in the cotangent space. All that the 1-vector $X_p$ does there is to translate the hyperplane infinitesimally in a certain spatial direction. Now, by rights going to higher order in the infinitesimals ought to mean that one should expect the hyperplane $Z_f$ to become deformed under the action of an $r$-vector, $r\ge 2$, as points near the origin in the cotangent space will no longer all move with a common, uniform speed and velocity away from the origin.

As the next step in our investigation into how to imagine the action of a higher-order vector field, we wish to formalize this guiding intuition. Therefore, instead of a hyperplane consider the zero locus of a higher-order differential, or $r$-jet $j \in J^r_p(M)$. Associate to it the $r$-vector $\varv \in J^{*r}_p(M)$. Then, in the indicated coordinate system, $j$ will be given by a polynomial in the $\xi_{1,...,n}$ of $r$-th degree (i.e., $j$ is defined only modulo monomials of degree $r+1$; we must revisit the issue later of whether it be valid to neglect these). In place of a 1-vector, $\varv$ will now be given by $\varv = \alpha^\mu \partial_\mu$ where the summation extends over multi-indices $\mu$ up to order $r$. Picture $\varv$ intuitively as a possibly non-constant vector field in cotangent space (we shall have justify this statement more rigorously below as well). If $\varv$ \textit{were} a constant vector field in cotangent space, then substituting arbitrary choices of $j$ would mean taking the zero locus of a hypersurface passing through the origin and translating it uniformly and rigidly under the action of $\varv$ in the direction defined by $(\alpha_1,...,\alpha_n)$. So far, we reproduce the scenario just discussed where the zero locus $Z_f$ need no longer be a hyperplane alone. But the interesting case to consider is now precisely when $\varv$ happens \textit{not} to be constant in the cotangent space at $p$. For infinitesimal $\lambda$, we may presume the action of $\varv$ on the zero locus $Z_f$ corresponding to an arbitrary jet $j$ to be given by movement along the streamlines of the vector field in cotangent space associated to the higher-order vector $\varv$. Under these conditions, the hypersurface $Z_f$ becomes deformed to another hypersurface $Z_{f+df}$ which need not simply be a uniform translate of $Z_f$ in any spatial direction.

To continue along the lines sketched thus far must become unwieldy. If we are to penetrate any further into the mystery of higher-order vectors, we shall have to avail ourselves of a more powerful technique. This, in a moment; but first, we ask what does tangency mean in the setting of higher-order vectors?

In the setting of ordinary Euclidean geometry, we would picture a 1-vector at a point $p$ as being tangent to an immersed submanifold $N$ passing through that point if it lies in the restriction of the tangent space $T_p\vvmathbb{R}^n$ to $T_pN$, viewed as a subbundle (cf. Lee, Problem 10-14 \cite{lee_smooth_manif}). Now, for the sake of simplicity, we limit ourselves to the case of a hypersurface $Z_f$ given locally as the zero locus of a smooth function $f$, where $p$ is a regular point of $f$. Since, in order for the 1-vector $\varv$ to be tangent to $Z_f$ it must point along the level sets of $f$, it is evident that a necessary and sufficient criterion for tangency (in the usual sense) would be that $X_\varv f = 0$ at $p$, where $X_\varv$ denotes a 1-vector field extending $\varv$ in the neighborhood of $p$ viewed as a differential operator, which exists by Lee, Proposition 8.7 \cite{lee_smooth_manif}. But we are entitled to view a higher-order vector, denoted $\varv$ again, as a higher-order differential operator $X_\varv$. Then the obvious way to generalize the notion of tangency would be to define it via the condition $X_\varv f = 0$ at $p$, which is equivalent to the statement that $\varv(j)=0$ where $j$ denotes the $r$-jet of $f$ localized at the point $p$.

Let us explore this condition for a moment. Another way of picturing tangency of the 1-vector $\varv$ to $N$ at $p$ would be to require that there exist a 1-vector $w \in T_pN$ parallel to $\varv$. Any $w$ orthogonal to $\mathrm{grad}~f$ will do. So the second statement comes down to the condition that $\varv$ be orthogonal to $\mathrm{grad}~f$. Therefore, we employ duality between jets and vectors to define an inner product by appeal to the (non-canonical) isomorphism in Euclidean space that identifies $\partial_\alpha$ with $d^\alpha$. One has to exercise a little care when going from 1-jets and 1-vectors to the general case of $r$-jets and $r$-vectors, $r>1$. The gradient as ordinarily defined will not continue to be suited to our purposes, since, as a first-order differential operator, it cannot be sensitive to any higher-order terms in $f$ (expanded about $p$), i.e., to how the zero locus $Z_f$ may curve away from being a hyperplane. The solution to this difficulty is to go to the cotangent space, the Euclidean coordinates of which we may denote as $\xi_{1,\ldots,n}$. From now on, we shall understand by the gradient the operator $\mathrm{grad} = \dfrac{\partial}{\partial\xi_1} + \cdots + \dfrac{\partial}{\partial\xi_n}$.

To any $r$-vector $\varv = a^\alpha \partial_\alpha$ associate the function $g = a^\alpha\xi_\alpha$, which gives the corresponding jet $j=g ~\mathrm{mod}~ \mathfrak{m}_p^{r+1}$. Here, we have invoked the non-canonical isomorphism in Euclidean space. From now on, refer to vector $\varv$ resp. jet $j$ and its zero locus $Z_g$ as if they were interchangeable. As a result, we have at our disposal an inner product between $r$-vectors $\varv$ and $w$, to be equated to $\langle \varv,w\rangle := w(j) = a^\alpha b_\alpha$ if $w=b_\alpha\partial_\alpha$. Correspondingly, we may define the angle between $\varv$ and $w$ as
\begin{equation}
	\theta := \arccos \frac{\langle \varv, w \rangle}{\sqrt{\langle \varv,\varv \rangle}\sqrt{\langle w,w \rangle}}.
\end{equation}
The following immediate proposition underwrites the reasonability from a geometrical point of view of the manner of speaking implied by these definitions. Let $\varv$ be associated with $g$ and $w$ with $f$, respectively. Then:
\begin{proposition}
	The $r$-vector $\varv$ is parallel to $w$ if and only if $Z_f$ and $Z_g$ have osculating contact up to $r$-th order.	
\end{proposition}
\begin{proof}
	It is a standard result in connection with the Cauchy-Schwarz inequality 
	$|\langle \varv, w \rangle| \le \sqrt{\langle \varv, \varv \rangle}\sqrt{\langle w,w \rangle}$ that equality occurs if and only if the two vectors are parallel, or at least one of the two is zero (cf. Royden, {\S}10.8 \cite{royden}). But $\varv$ and $w$ are parallel if and only if $f$ and $g$ are equal up to a multiplicative constant, in which case $Z_f$ and $Z_g$ have osculating contact up to $r$-th order. Note that the property of being in osculating contact to a given order $r$ is well defined because the projections from jets of order $r_1$ to those of order $r_0$, $r_1>r_0$, are canonical.
\end{proof}

\begin{corollary}
	If $\varv$ is parallel to $w$ at order $r$, the (non-canonical) projections will be parallel at orders $s$, $1 \le s \le r$ as well.
\end{corollary}

As an example consider the unit sphere $S^n \subset \vvmathbb{R}^{n+1}$. Now the hyperplane passing through the point $(1,0,\ldots,0)$, namely $\Sigma_1 = \{ 0,x_2,\ldots,x_{n+1}| x_{2,\ldots,n+1} \in \vvmathbb{R} \}$, would conventionally be treated as tangent to $S^n$, and indeed it is to first order. But clearly osculating contact fails at any higher order, since we have in a neighborhood of $(1,0,\ldots,0)$ the jet expansion of $S^n$,
\begin{align}
	\sqrt{1-\xi_2^2 - \cdots - \xi_{n_1}^2} - 1 =& 
	- \frac{1}{2} \left( \xi_2^2 + \cdots + \xi_{n+1}^2 \right) - 	
	\frac{1}{8} \left( \xi_2^2 + \cdots + \xi_{n+1}^2 \right)^2 - \nonumber \\ 
	&\frac{1}{16} \left( \xi_2^2 + \cdots + \xi_{n+1}^2 \right)^3 - 
	\frac{5}{128} \left( \xi_2^2 + \cdots + \xi_{n+1}^2 \right)^4 - \cdots
\end{align}
To second order, then, it is the paraboloid $\Sigma_2: \xi_1 = \frac{1}{2} \xi_2^2 + \cdots + \frac{1}{2} \xi_{n+1}^2$ that has osculating contact with the sphere etc. NB: in this case, we can no longer truly say that $\Sigma_1$ is strictly tangent to $S^n$ at $(1,0,\ldots,0)$---counter to received views! Rather, the higher tangents must be generalized vectors of the form, for instance,
\begin{equation}
	a_1 \partial_{x_2} + a_3 \partial_{x_2x_2x_2} + a_5 \partial_{x_2x_2x_2x_2x_2} + \cdots
\end{equation}
and also others of the form
\begin{equation}
	\partial_{x_2x_2} - \partial_{x_3x_3}
\end{equation}
etc. where there are numerous such linear combinations satisfying the condition of tangency, increasingly many as a power of the order of the highest allowed non-zero component.

Orthogonality in a sense generalized to higher order becomes a subtler concept. Extrapolating from the 1-vector case, where $\varv$ and $w$ will be orthogonal, i.e., $\langle \varv,w \rangle = 0$, whenever the corresponding hyperplanes in cotangent space meet at right angles, we would like to say that $\varv$ and $w$ will be orthogonal as $r$-vectors, if the level sets of $Z_f$ and $Z_g$ form an orthogonal family of hypersurfaces in cotangent space.

Illustrate with $S^n$; $f=\sqrt{\xi_1^2+\cdots+\xi_n^2}$, $g=\xi_1$, this condition fails to be true when $\xi_1 \ne 0$ for the level sets of $Z_f$ consist of concentric spheres while the level sets of $Z_g$ are parallel hyperplanes in cotangent space and only the hyperplane passing through the origin intersects the spheres orthogonally.

We are accustomed to think in Euclidean space $\vvmathbb{R}^n$ that, to a given 1-vector, there will be $n-1$ directions perpendicular to it, filling out a hyperplane in space so that the linear dimension of the set of orthogonal 1-vectors is equal to the spatial dimension decremented by a unit, but with the generalized definition there are now more ways to be orthogonal, viz., a total comprising a $(N_r-1)$-dimensional linear manifold of higher-order vectors all simultaneously orthogonal at $(1,0,\ldots,0)$, which is $> n-1$ whenever $r>1$. Roughly speaking, these will consist in anything tangent in the ordinary sense (having osculating contact to first order only) to the hyperplane orthogonal in ordinary sense to grad g, where $g$ is associated to $\varv$. Since only first-order osculating contact matters here, there will be myriad ways of curving away non-linearly in the vicinity of any given point all of which reduce to the same hyperplane after projection to first order.

We want an orthogonal family filling all of cotangent space. This stipulation gives rise to a consistency condition: $Z_f$ is momentaneously invariant under the action of $\varv$  $\iff$ the derived vector field $V := \mathrm{grad}~g$ is everywhere tangent to $Z_f$ as a submanifold in cotangent space $\iff$ $X_\varv f = 0$ at $p$. The following key result anchors the consistency of the geometrical constructions proposed in this section:

\begin{proposition}\label{orthogonal_family_hypersurfaces}
	$\varv(j) = 0$ iff the derived vector field $V$ of $\varv$ is tangent (in the ordinary sense in cotangent space) to the level sets of $f$ iff $\mathrm{grad}~f \cdot \mathrm{grad}~g = 0$ everywhere in cotangent space.
\end{proposition}
\begin{proof}
	The second equivalence is trivial in as much as $g$ was constructed precisely so that $V = \mathrm{grad}~g$.	For the equivalence of the first and third conditions, first observe the equivalence of the following notions: if 
	$f=a_1x_1 + \cdots + a_nx_n + a_{12}x_1x_2+\cdots = a_\alpha \sqrt{(\alpha_1+1)\cdots(\alpha_n+1)} x^\alpha$ and
	$g=b_1x_1 + \cdots + b_nx_n + b_{12}x_1x_2+\cdots = b_\beta \sqrt{(\beta_1+1)\cdots(\beta_n+1)} x^\beta$, where 
	$x_{1,\ldots,n}$ are taken to be the coordinates spanning the cotangent space, then $\varv(j)=0 \iff a \cdot b = 0 \iff \int_\Omega f\bar{g}dz_1d\bar{z}_1\cdots
	dz_nd\bar{z}_n=0$ where in the lastmost case $f$ and $g$ are to be regarded as having been analytically continued to elements in the Hilbert space $L^2(\Omega,dz_1d\bar{z}_1\cdots
	dz_nd\bar{z}_n)$ defined on the domain $\Omega=\vvmathbb{D}\times\cdots\times\vvmathbb{D}$, $\vvmathbb{D}=\{ z \in \vvmathbb{C} | |z| \le 1 \}$ being the unit disk in the complex plane. Along the same lines, $\mathrm{grad}~f \cdot \mathrm{grad}~g = 0$ identically would be equivalent to $\partial_{x_1}f\partial_{x_1}g + \cdots + \partial_{x_1}f\partial_{x_n}g = 0$ identically or $\partial_{z_1}f\bar{\partial}_{z_1}g + \cdots + \partial_{z_1}f\bar{\partial}_{z_n}g = 0$ in the real part of the domain $\Omega$ given by $y_{1,\ldots,n}=0$ with $z_k=x_k+iy_k, k=1,\cdots,n$. The following observation simplifies the matter further:
	\begin{equation}
		\int_\Omega \left( \partial_{z_1}f\bar{\partial}_{z_1}g + \cdots + \partial_{z_1}f\bar{\partial}_{z_n}g \right) dz_1d\bar{z}_1\cdots
		dz_nd\bar{z}_n = 0
	\end{equation}
	is equivalent after integration by parts to the statement that
	\begin{equation}
		\int_\Omega \Delta (f\bar{g}) dz_1d\bar{z}_1\cdots dz_nd\bar{z}_n = 0,
	\end{equation}
	where $\Delta = \frac{\partial^2}{\partial_{z_1}\bar{\partial}_{z_1}} + \cdots + \frac{\partial^2}{\partial_{z_n}\bar{\partial}_{z_n}}$ since $(\Delta f)\bar{g} = f\Delta\bar{g}=0$ identically, $f$ and $g$ being polynomial and hence holomorphic as functions of the $z_{1,\ldots,n}$.
	
	$(\Rightarrow)$ Suppose therefore that $\varv(j)=0$, or $\int_\Omega f\bar{g}=0$. It will be sufficient to show that $\Delta (f\bar{g})=0$ at $x=y=0$ because under displacement of the real part $x$ to $x^\prime=x-c$ the condition that
	\begin{equation}
		\varv(j) = 0 = \left( a_\alpha \partial_\alpha \right) \left( b_\beta d^\beta \right) =  \left( a_\alpha \partial_\alpha \right) \left( b_\beta x^\beta \right)
	\end{equation}
	goes over into
	\begin{equation}
		\left( a_\alpha \partial^\prime_\alpha \right) \left( b_\beta x^{\prime\beta} \right) =
		\left( a_\alpha \partial_\alpha \right) \left( b_\beta (x-c)^\beta \right) = 0.
	\end{equation}
	Hence, set $h(z):=f(z)g(\bar{z})$ so that $\overline{h(\bar{z})}=
	\overline{f(\bar{z})\overline{g(\bar{z})}} = f(z)g(\bar{z})=h(z)$ seeing that the coefficients $a_\alpha, b_\beta$ are purely real (and viewing $h$ as a function of the $z_{1,\ldots,n}$ one at a time). Thence, by hypothesis $h=F+iG$ with real part $F$ symmetric  and imaginary part $G$ antisymmetric in the $y_{1,\ldots,n}$ by virtue of the Schwarz reflection principle (Conway, Theorem IX.1.1 \cite{conway}). Hold the values of $z_{2,\ldots,n}$ fixed. If we perform the remaining integral over $z_1=z$, we have after integration by parts,
	\begin{align}
		\int \left( \partial_{z}f(z) \right) g(\bar{z}) &= 
		\int \left( (\partial_x + i \partial_y )f(z) \right) g(\bar{z}) \nonumber \\ 
		&= -\int f(z) \left( (\partial_x+i\partial_y) g(\bar{z}) \right) +
		f\bar{g}\bigg|^{x=\varepsilon}_{x=-\varepsilon, y ~\mathrm{fixed}}
		+f\bar{g}\bigg|^{y=\delta}_{y=-\delta, x ~\mathrm{fixed}}.
	\end{align}
	The first term on the right-hand side vanishes as $g(\bar{z})$ is anti-holomorphic. As for the boundary terms, due to the antisymmetry of $h(z)=f\bar{g}$ under reflection in $y$ together with the symmetry of the region of integration over $dzd\bar{z}=dxdy$ (namely, the interior of the unit circle, $\vvmathbb{D}$), these terms cancel under integration with respect to the other variable of the pair (i.e., with respect to $y$ for the penultimate term and with respect to $x$ for the final term). Hence, we may conclude that $\int_\Omega (\partial_{z_1}f)\bar{g}=0$ and similarly that
	$\int_\Omega f\bar{\partial}_{z_1}\bar{g}=0$. Then, repeating the argument putting these in place of the condition $\int_\Omega f\bar{g}=0$ yields moreover that $\int_\Omega (\partial_{z1}f)(\bar{\partial}_{z_1}\bar{g})=0$. But the same argument works just as well with $z_k$ taking the place of $z_1$, $k=2,\ldots,n$.
	Therefore, $\int_\Omega \mathrm{grad}~f \cdot \overline{\mathrm{grad}}~\bar{g} =
	0$. However, the radius of the polydisk $\Omega$ is a matter of indifference. If we substitute any radius $\eta>0$ we would find that
	\begin{equation}
		\frac{1}{\eta^{2n}} \int^\eta_0 \mathrm{grad}~f \cdot \overline{\mathrm{grad}}~\bar{g} = 0
	\end{equation}
	and taking the limit as $\eta$ tends to zero leads to the conclusion that
	\begin{equation}
		\mathrm{grad}~f \cdot \overline{\mathrm{grad}}~\bar{g} \bigg|_{x=y=0} = 0,
	\end{equation}
	as was to be shown.
	
	$(\Leftarrow)$ Suppose that $\Delta (f\bar{g}) = 0$ identically on the real domain given by $y_{1,\cdots,n}=0$. Now let $h=f\bar{g}=F+iG$. Since by hypothesis $\Delta F=0$ identically on the real domain, it is harmonic there and by the Schwarz reflection principle again we know that $F$ is the real part of a holomorphic function $h$ such that $\overline{h(\bar{z})}=h(z)$; in other words, $G$ is antisymmetric under reflection in the $y_{1,\ldots,n}$. Due to the symmetry of the polydisk $\Omega$, then, we must have that  $\int_\Omega G = 0$. But since $\Delta(\bar{f}g) = \overline{\Delta(f\bar{g})} = 0$, the same argument interchanging the roles of $f$ and $g$ shows that $\int_\Omega F = 0$ as well, whence $\int_\Omega f\bar{g} = \int_\Omega h = \int_\Omega (F+iG) = 0$, or, what is equivalent, $\varv(j)=0$ as was to be shown.
\end{proof}

The import of this proposition is that it justifies the visualizable [anschauliches] picture of an higher-order vector as a 1-vector field in cotangent space. What happens when we seek to integrate an higher-order vector field $\varv$ in order to yield a flow? In the vicinity of a given point, this flow would be approximated by that of the 1-vector field $V_p$ in cotangent space, but another consideration supervenes when we broaden our field of view to encompass nearby points as well. In the higher-order case ($r>1$), the derived vector field $V_p$ need not be constant in the $\xi_{1,\ldots,n}$, but rather varies from place to place in cotangent space as one departs from the origin. The right way to imagine what is going on would be in terms of virtual motion of infinitely near points in space (not anymore simply to be identified with a conventional differentiable manifold, since we must now include the infinitesimals of all orders); these virtual motions may be either mutually reinforcing or mutually opposed and the resultant flow emerges from their collective behavior. Thus, in the case of an ordinary 1-vector field, one obtains in the small a coherent component of flow, as everyone is familiar with. Clearly, all kinds of complicated possibilities enter as soon as one goes either to yet higher order or to multiple dimensions, or both, and it may not as yet be immediately evident to everyone that the intuitive picture of an higher-order vector field just described can be made entirely consistent. So far our remarks have been very impressionistic; in a moment, in {\S}\ref{phoronomy} below, they are about to be rendered precise and explicit.

The intrinsic nature of the geometrical constructions in this section, moreover, stands very much in question since they were performed with the aid of the non-canonical isomorphism in Euclidean space. In the next section we formalize the informal ideas sketched in this section.

\subsection{Towards a Phoronomy for Higher-Order Vector Fields}\label{phoronomy}

The present section concerns the question of integration of general higher-order vector fields so as to obtain a corresponding flow and not so much what a higher-order vector field may be eo ipso.

Let $\theta$ be the canonical 1-form on $T^*M$ and $\omega=-d\theta$ the canonical symplectic form. If $\pi_{r,1}:J^{*r}(M)\rightarrow J^{*1}(M)$ denotes the canonical projection from $r$-jets to 1-jets, we can lift $\theta$ to $\hat{\theta} := \pi^*_{r,1}\theta$ (where $T^*M$ is to be identified with $J^{*1}(M)$ as usual).

\begin{definition}
	The canonical $r$-jet field over a differentiable manifold $M$ will be defined as
	\begin{equation}
		\Theta := \hat{\theta} + \frac{1}{2}\hat{\theta}\hat{\theta}+\cdots+\frac{1}{r!}
		\overbrace{\hat{\theta}\cdots\hat{\theta}}^{r ~\mathrm{times}},
	\end{equation}
	where the powers of $\hat{\theta}$ are to be understood in the sense of the natural multiplication of jets.
\end{definition}

\begin{definition}
	Let $T^*M$ be the cotangent bundle of the differentiable manifold $M$ and $\varv \in \mathscr{J}^{r}(M)$ an $r$-vector field on $M$. Denoting an arbitrary element of the cotangent bundle by $\xi=\xi_1d^1+\cdots+\xi_nd^n$ with respect to suitable coordinates $x^{1,\ldots,n}$ around the point $p \in M$ (always implicitly understood) and the injection $\iota:M\rightarrow J^r(M)$, $p\mapsto (x,\Theta(\xi))$, define the function $P_\varv: T^*M \rightarrow \vvmathbb{R}$ via
	\begin{equation}
		P_\varv(\xi) := \iota_*\varv(\Theta(\xi)).
	\end{equation}
\end{definition}

\begin{remark}
	The function $P_\varv$ is covariantly defined since all of the operations entering into its construction are intrinsic. Under a linear change of coordinate, $\xi$ undergoes a linear transformation to $\eta=B\xi$ and, correspondingly its powers go over into $B\otimes\cdots\otimes B \xi\cdots\xi$, while at the same time the coefficients $a_\alpha$ of $\varv$ go into $Aa$ where $A$ involves a sum of repeated tensor products of the Jacobian $J$. When contracted against each other, the factors of $J$ combine with the factors of $B$ to yield the identity, as they must. For the general non-linear change of coordinate, off-diagonal terms connecting terms different in degree enter into the transformation law of $\varv$ as well as of its argument $\Theta(\xi)$, but these always combine again to yield the identity. 
	
	Any series in the $\hat{\theta}(\xi)$ would serve just as well, if all we wished was to obtain a covariantly defined function on the cotangent bundle, but the factorial coefficients in the above definition have been chosen for the following reason. In view of the multinomial expansion,
	\begin{equation}
		(\xi_1+\cdots+\xi_n)^m = \sum_{|\beta|=m} \frac{m!}{\beta!}\xi^\beta,
	\end{equation}
	taking $\varv=a^\alpha\partial_\alpha=\dfrac{a^\alpha}{\alpha!}
	\dfrac{\partial^{|\alpha|}}{\partial x_1^{\alpha_1}\cdots\partial x_n^{\alpha_n}}$ we have
	\begin{equation}
		P_\varv(\xi)=\frac{1}{\alpha!}a^\alpha\xi^\alpha.
	\end{equation}
	Morally speaking, $P_\varv$ is just the total symbol of $\varv$ viewed as a linear partial differential operator. The only difference from the conventional expression for the total symbol consists in the factorials in the denominator, which arise for us from the decision to include them when defining the basis elements $\partial_\alpha$ (cf. equation (\ref{def_vector_basis}) above). Our definition is the natural one if these are to become the dual basis to the $d^\alpha$. Confusion reigns in the literature as to the proper definition and significance of the total symbol of a linear partial differential operator, for two reasons: first, as a rule one is interested solely in the principal symbol (consisting of those terms with $|\alpha|=r$) and second, it is customary in the field of the analysis of partial differential equations to work all the time in a fixed Cartesian frame. The covariance of the total symbol as just defined follows from the more complicated transformation laws applying to both $\varv$ and its argument $\Theta(\xi)$, which must be understood as generalized tensors in accordance with {\S}\ref{formal_def}.	
\end{remark}

We are prepared now to introduce the natural and intrinsic definition of the flow associated to a higher-order vector field. Then, we arrive at an ordinary 1-vector field $V \in \mathscr{X}(T^*M)$ in cotangent space by means of the symplectic gradient: $V$ is to be taken as the unique 1-vector field such that
\begin{equation}
	i_V \omega = dP_\varv.
\end{equation}
With respect to coordinates $(x^{1,\ldots,n},\xi_{1,\ldots,n})$,
\begin{equation}\label{derived_vf_cotangent}
	V = P_{\varv,\xi_1}\partial_{x^1}+\cdots+P_{\varv,\xi_n}\partial_{x^n}-
	P_{\varv,x^1}\partial_{\xi_1}-\cdots-P_{\varv,x^n}\partial_{\xi_n}.
\end{equation}

\begin{proposition}[Local existence of the derived flow]\label{existence_local_flow}
	If $\varv \in \mathscr{J}^{r}(M)$ is an $r$-th order vector field, to every point in cotangent space $(x,\xi) \in T^*M$ there exists a neighborhood $(x,\xi) \in U \subset T^*M$ and a time $\lambda_0>0$ such that for every $0 \le \lambda < \lambda_0$ the smooth maximal flow $\varphi_\lambda:U\rightarrow T^*M$ by diffeomorphisms satisfying $\varphi_{\lambda} \circ \varphi_\mu = \varphi_{\lambda+\mu}$ where defined whose orbits are precisely the integral curves of $V$, exists and is unique.
\end{proposition}
\begin{proof}
	Lee, Theorem 9.12 \cite{lee_smooth_manif}.
\end{proof}

The following immediate result ensures that the flow $\varphi_\lambda$ behaves as expected when applied to a smooth function on the base space $M$; in other words, it can be viewed as acting on such functions as we presume in the examples discussed above:

\begin{proposition}
	The property of exactness is preserved under the flow $\varphi_\lambda$ derived from $\varv$.
\end{proposition}
\begin{proof}
	Let $f$ be a smooth function defined on $M$, where if necessary we may restrict its domain to $\pi_M[U]$, where $U \subset T^*M$ is the open set on which the flow $\varphi_\lambda$ is defined, whose existence is guaranteed by proposition \ref{existence_local_flow}. We want to show that the image of $df$ under $\varphi_\lambda$ will be again an exact 1-form. To this end, it suffices to show that the flow $\varphi_\lambda$ can be expressed as the pullback of some smooth map $F_\lambda^{-1}$, since by the naturality of the exterior derivative (Lee, Proposition 14.26 \cite{lee_smooth_manif}) we would have $\varphi_\lambda(df)= F_\lambda^{-1*}df = d (f \circ F_\lambda^{-1})$. Therefore, define a mapping $F_\lambda:M \rightarrow M$ by $F_\lambda := \pi_M \circ \varphi_\lambda \circ d$, i.e., $p \in M \mapsto \pi_M(\varphi_\lambda(df_p))$. As the product of three diffeomorphisms for sufficiently small $\lambda$, $F_\lambda$ is itself a diffeomorphism. What has to be shown, thus, is that $\varphi_\lambda \big|_{\mathrm{Im}~df \subset T^*M} = (F_\lambda^{-1})^*$. All this calls for is some juggling of the definitions. Let an arbitrary element $w \in T_{F_\lambda(p)}M$. Then,
	\begin{align}
		(F_{\lambda,F_\lambda(p)}^{-1} df)^*(w) &= df_p( F_{\lambda*}^{-1} w) \nonumber \\
		&= w_{F_\lambda(p)}(f \circ F_\lambda^{-1}) \nonumber \\
		&= (\varphi_\lambda(df_p))(w),
	\end{align}
	as was to be shown, where the last step follows by the definition of $F_\lambda$.
\end{proof}

A further remark will be helpful in interpreting the geometrical significance of our construction of the derived 1-vector field in cotangent space $V$ from an higher-order vector field $\varv$ on $M$. Above, we have represented the action of $X_\varv$ on smooth functions through a superposition of that on the unit Gaussians $\varphi_\varkappa$. To see what happens microlocally, we promote the test function $\varphi_\varkappa$ to cotangent space by applying the exterior derivative: $d\varphi_\varkappa \in \Gamma(T^*M)$. Now, we are interested in test functions of this form only in their limit as the parameter $\varkappa$ tends to zero. In such a limit, the dominant part of $d\varphi_\varkappa$ will be located near the origin where $\| x \|$ is sufficiently small in norm, due to the suppression coming from the exponential. For the same reason, the higher spatial derivatives of $\varphi_\varkappa$ in $X_\varv$ will be suppressed relative to the lower there; hence, when forming our geometrical intuition for what the action of $X_\varv$ on $\varphi_\varkappa$ looks like, we can limit ourselves to considering the region in cotangent space near the zero section where $\| \xi \|$ is sufficiently small in norm as well, so that the leading terms dominate.

Lastly, let us remark on what it means that $\varphi_\lambda$ is defined only locally in cotangent space. First, the flow $\varphi_\lambda$ may not exist for all future times $\lambda>0$, but only up to some limiting time $\lambda_0$. Second, it might not be possible to bound $\lambda_0>0$ away from zero uniformly across all of space, or uniformly away from zero across the fiber in cotangent space for all given functions at a point (unless one happens to be dealing with a 1-vector field). All we can say is that we know how to transport a function on $M$ under the flow derived from $\varv$ for a possibly bounded time near any point in the domain of its definition (by the uniqueness in proposition \ref{existence_local_flow}, the local flows thereby defined must be consistent among themselves wherever they overlap).

%% file: section_3.tex
\section{Revising the differential calculus}\label{chapter_3}

The ideas about the possible role of higher-order infinitesimals in geometry advanced above can become effective only if one has at one's disposal the technique with which to handle them facilely. The issue facing us would be to extend the differential calculus on manifolds so as to permit, as far as possible, an intrinsic characterization of local behavior of jets and higher-order tangents. Thus, we require a notion of parallel transport. After duly enlarging the concept of affine connection, our principal result will be an existence theorem for geodesics of the jet connection. The conceptual setting here is non-trivial, since it is not immediately evident what it could mean to integrate a vector field of higher than first order. We are led to a definition of a 1-parameter semigroup of diffeomorphisms that, in a certain sense, possesses inertia. A 1-vector field has no inertia since the flow it generates consists of mappings from one point to another at a later time, but the point itself has no inherent tendency to prolong its motion in the direction in which it is headed at the moment. This, we shall see, constitutes a new feature that supervenes once one goes to higher order.

\subsection{Extension of the Differential Calculus to Jets}\label{parallel_transport_jets}

Now that we have the concept of a jet bundle, the question raises itself as to how jets at different points in space are related to one another. To visualize a jet geometrically, think of the jet $j = [f]$ at $p \in M$ in terms of the graph $\Gamma$ of $f$ near $p$ with respect to some set of local coordinates. When transformed under a diffeomorphism $\varphi$ the graph changes in shape. Suppose first that $\varphi$ takes $p$ to itself. The transformed jet is $\varphi^{-1*}j$. The graph $\Gamma_\varphi$ of $f \circ \varphi^{-1}$ will distort in shape to reflect the behavior of $\varphi$ around $p$. To first order, the slopes of the graph $\Gamma$ of $f$ can change and the level sets of $f$ can be rotated. To second order, the ellipsoid, paraboloid or hyperboloid defined by the Hessian quadratic form could be rotated and expanded or compressed or dilated, or all of these possibilities at once along different principal axes. At third order, $\varphi$ can modify the degree of asymmetry of the graph $\Gamma$, etc. for higher orders.

To make meaningful statements of this kind when $\varphi$ takes $p$ to another point $q \ne p$, we have to have a standard by means of which the jet spaces $J^{r*}_pM$ and $J^{r*}_qM$ can be compared. The tool for this purpose is known as a connection in the case of ordinary tensors formed from the tangent and cotangent spaces. Thus, we are seeking to formulate a concept of a `jet connection'. The usual connection defined for arbitrary vector bundles is not suited to our purposes. To see why, let $\pi: E \rightarrow M$ be a vector bundle and $\nabla: \mathscr{T}(M) \times \mathscr{E}(M) \rightarrow \mathscr{E}(M)$ a connection, where $\mathscr{T}(M)$ is the space of vector fields and $\mathscr{E}(M)$ the space of sections of $E \rightarrow M$. The concept of parallel transport along a curve $\gamma$ in $M$ may be defined as usual. When parallel transported along $\gamma$, a section $\xi \in \mathscr{E}(M)$ knows only about the values of $X \in \mathscr{T}(M)$ restricted to $\gamma$, $X|_\gamma$. This statement can be made precise as follows: let $Y \in \mathscr{T}(M)$ be another vector field that agrees with $X$ when restricted to $\gamma$, $Y|_\gamma = X|_\gamma$. Then the parallel transport of $\xi$ by $Y$ agrees with that by $X$, $\nabla_Y \xi|_\gamma = \nabla_X \xi|_\gamma$ (\cite{lee_riemannian_manif}, Lemma 4.2). The source of the problem is apparent: the vector field $X$ depends only on differentials of the first order; that is, we need a higher-order vector field. To remedy this defect (from our point of view), we proceed as follows. Parallel transport in the wider sense we are seeking here is to be thought of not just as a translation along a 1-dimensional curve, but (returning to our initial picture) as motion in a flow inside a tubular neighborhood in the limit as the cross-section tends to zero (to give it a name, we will describe this situation as a `jet stream').

Thus, we have to investigate more closely the local behavior of 1-parameter groups of diffeomorphisms. Let $\phi_\lambda$ be a 1-parameter group of diffeomorphisms acting on $M$ (also known as flows). Every flow has an infinitesimal generator $X_1 \in \mathscr{T}(M)$, where $\mathscr{T}(M)$ denotes the space of smooth sections of the tangent bundle $TM$, for which each orbit is an integral curve of $X_1$ (\cite{lee_smooth_manif}, Proposition 17.3). The converse also holds: every vector field $X_1 \in \mathscr{T}(M)$ is the infinitesimal generator of a unique maximal smooth flow (\cite{lee_smooth_manif}, Theorem 17.8). We would like to ask whether a higher-order vector field $X_r \in \mathscr{J}^r(M)$, $r \ge 1$, can be associated to a flow that characterizes its infinitesimal behavior to higher order. To do so, look at how $\phi_\lambda$ acts on functions. For $f \in C^\infty(M)$, define $f_\lambda = f \circ \phi_\lambda$. Then we can set
\begin{equation}\label{eq20}
	(X f)(p) = \sum_{s=1}^r X_s(\lambda) f(p) = \sum_{s=1}^r \frac{1}{s!} \lambda^{s} \frac{d^s f_\lambda(p)}{d\lambda^s}\bigg|_0,
\end{equation}
which is to be thought of as a formal power series in $\lambda$.

\begin{lemma}\label{infinitesimal_lemma}
	The vector field $X_1$ so defined is proportional to the infinitesimal generator of $\phi_\lambda$.
\end{lemma}
\begin{proof}
	For each $p \in M$, let $\theta^{(p)}(t) = \phi_t(p)$ be the orbit starting at $p$ and let $Y \in \mathscr{T}(M)$ be the infinitesimal generator of the flow $\phi_\lambda$. Then for any $f \in C^\infty(M)$,
	\begin{align}\label{eq21}
		\frac{1}{\lambda}(X_1f)(p) &= \lim_{t \rightarrow 0} \frac{f(\theta^{(p)}(t))-f(p)}{t} = \frac{d}{dt}\bigg|_0 f(\theta^{(p)}(t)) \nonumber \\
		&= \theta^{(p)}_* \left( \frac{d}{dt}{\bigg|}_0 \right) f= (Yf)(p),
	\end{align}
	as asserted.\end{proof} 

\begin{definition}
	We will call the higher-order vector field $X_r \in \mathscr{J}^r(M)$ associated to a flow $\phi_\lambda$ the $r$-th order infinitesimal generator of $\phi_\lambda$.
\end{definition}

\begin{remark}
	The reason for the coefficients in \textup{(\ref{eq20})} is that we would like to write, formally at least, $\phi_\lambda = \exp X(\lambda)$ and obtain $X_r$ from the terms up to order $\lambda^r$ in the expansion of the exponential. Evaluated near $\lambda=0$, the terms after the first are infinitesimally small compared to the leading term $\lambda X_1$.
\end{remark}

Not every higher-order vector field $Y \in \mathscr{J}^r(M)$ is the infinitesimal generator of a flow, since the flow is uniquely determined by its infinitesimal generator at first order. 

\begin{definition}
	If $X \in \mathscr{T}(M)$, define its $r$-th prolongation $\textup{pr}_rX \in \mathscr{J}^r(M)$ by the equation
	\begin{equation}\label{eq22}
		\textup{pr}_rXf = \sum_{s=1}^r \frac{1}{s!} \lambda^{s} \overbrace{X \circ \cdots \circ X}^{s}f,
	\end{equation}
	for every $f \in C^\infty(M)$ \textup{(}again, thought of as a formal power series in $\lambda$\textup{)}.
\end{definition}

The following proposition characterizes the relation between the infinitesimal generator of order $r$ and that of order 1.

\begin{proposition} Let $\phi_\lambda$ be a flow with infinitesimal generator $X_1$. Then the $r$-th order infinitesimal generator $X_r$ of $\phi_\lambda$ satisfies $X_r = \textup{pr}_r X_1$.
\end{proposition}
\begin{proof}
	The relation holds for $r=1$ by lemma \ref{infinitesimal_lemma}. Suppose the relation is true for $1,\ldots,r$.
	Then $X_{r+1}f - X_rf = \dfrac{1}{(r+1)!} \lambda^{r+1} \dfrac{d^{r+1} f_\lambda}{d\lambda^{r+1}}\big|_0$ and $\text{pr}_{r+1}X_1f - \text{pr}_rX_1f = \dfrac{1}{(r+1)!} \lambda^{r+1} \overbrace{ X_1 \circ \cdots \circ X_1}^{r+1}f$. Now,
	\begin{align}\label{eq23}
		X_1 \circ X_1 f &= \lim_{t_1 \rightarrow 0} \frac{ (X_1f)(\theta^{(p)}(t_1)) - (X_1f)(p)}{t_1} \nonumber \\
		&= \lim_{t_1 \rightarrow 0} \frac{1}{t_1} \lim_{t_2 \rightarrow 0} \frac{1}{t_2} \left(
		f(\theta^{(\theta^{(p)}(t_1))}(t_2)) - f(\theta^{(p)}(t_1)) - f(\theta^{(p)}(t_2)) + f(p) \right).
	\end{align}
	But by the group property of the flow, $\theta^{(\theta^{(p)}(t_1))}(t_2) = \theta^{(p)}(t_1+t_2)$. Hence,
	\begin{align}\label{eq24}
		X_1 \circ X_1 f &= 
		\lim_{t_1 \rightarrow 0} \frac{1}{t_1} \lim_{t_2 \rightarrow 0} \frac{1}{t_2} \left(
		f(\theta^{(p)}(t_1+t_2)) - f(\theta^{(p)}(t_1)) - f(\theta^{(p)}(t_2)) + f(p) \right) \nonumber \\ 
		&= \lim_{t_1 \rightarrow 0} \frac{1}{t_1} \left( \frac{df_\lambda}{d\lambda}\bigg|_{t_1}
		- \frac{df_\lambda}{d\lambda}\bigg|_0 \right) \nonumber \\
		&= \frac{d^2 f_\lambda}{d\lambda^2}\bigg|_0.
	\end{align}
	The extension of the argument to $\overbrace{X_1 \circ \cdots \circ X_1}^{r+1}$ is immediate.
\end{proof}

\begin{definition}
	We will say that an $r$-vector field is integrable if it can be written as the $r$-prolongation of a 1-vector field.
\end{definition}

Integrable higher-order vector fields thus arise naturally. We propose that the concept of parallel transport by a vector field along a curve be replaced by a concept of parallel transport in an infinitesimally small tubular neighborhood of a curve, which is equivalent to parallel transport by an integrable higher-order vector field along the curve. In the conventional geometry on differentiable manifolds, an affine connection constitutes a rule for covariantly differentiating one vector field (call it $Y$) along another (call it $X$). The important properties characterizing the resulting covariant derivative, denoted $\nabla_X Y$ are its linearity over $C^\infty(M)$ in $X$, linearity over $\vvmathbb{R}$ in $Y$ and an analogue of the Leibniz product rule for scalar multiples of $Y$, i.e., vector fields of the form $fY$, where $f \in C^\infty(M)$. 

Before formulating a formal definition of a jet connection, we must first attend to a preliminary matter. An ordinary vector field $X$ can be viewed as a first-order differential operator. Since we wish to extend the formalism to the case of higher-order vector fields, we must settle upon a consistent manner of treating derivatives of higher than first order. To arrive at the right product rule in the higher-order case, reflect first upon the situation in flat space. There, we want to equate $\nabla_X$ with $X$ for any higher-order vector field $X \in \mathscr{J}^r(\vvmathbb{R}^n)$. Now, we can express the vector field in terms of coordinates as $X = X^\alpha \partial_\alpha$. It will suffice, then, to examine the partial derivative $\partial_\alpha$ with respect to an arbitrary multi-index $\alpha$. For the ordinary derivative, Leibniz' product rule \cite{leibniz_1684} assumes the well-known form,
\begin{equation}
	\frac{d}{dx} fg = \frac{df}{dx}g + f\frac{dg}{dx}.
\end{equation}
The simplest non-trivial case beyond what we know already would be the multi-index at level $|\alpha|=2$ given by $\alpha=\mu\mu$, where $\mu=1,\ldots,n$ labels one of the coordinate directions; in other words, $\partial_\alpha = \partial^2/\partial (x^\mu)^2$. But from elementary multi-variable calculus, we have that
\begin{equation}
	\frac{\partial^2}{\partial (x^\mu)^2} fg = \left( \frac{\partial^2 f}{\partial (x^\mu)^2} \right) g + 2 \frac{\partial f}{\partial x^\mu}\frac{\partial g}{\partial x^\mu} + f \left( \frac{\partial^2 g}{\partial (x^\mu)^2} \right),
\end{equation}
which quickly leads by induction to Leibniz' product rule in the form
\begin{align}
	\frac{\partial^n}{\partial (x^\mu)^n} fg =
	&\left( \frac{\partial^n f}{\partial (x^\mu)^n} \right) g + \binom{n}{1} \frac{\partial^{n-1} f}{\partial (x^\mu)^{n-1}}\frac{\partial g}{\partial x^\mu} + \binom{n}{2}\frac{\partial^{n-2} f}{\partial (x^\mu)^{n-2}}\frac{\partial^2 g}{\partial (x^\mu)^2} + \cdots + \nonumber \\
	&
	\binom{n}{n-1}\frac{\partial f}{\partial x^\mu}\frac{\partial^{n-1} g}{\partial (x^\mu)^{n-1}} + 
	f \left( \frac{\partial^n g}{\partial (x^\mu)^n} \right) 
\end{align}
(cf. Courant and John, \cite{courant_john}, $\S$3.1). A multi-index may, in general, involve partial derivatives with respect to more than one coordinate at the same time, however. This situation is easily embraced if we employ a compressed notation:
\begin{equation}
	\partial_\alpha (fg) = \frac{\alpha!}{\alpha_1!\alpha_2!} (\partial_{\alpha_1}f) (\partial_{\alpha_2}g),
\end{equation}
where, according to the most straightforward extrapolation of Einstein's convention, the summation in expressions of this form will always be understood to extend over all multi-indices $\alpha_{1,2}$ such that $\alpha_1+\alpha_2=\alpha$ (unless otherwise indicated). Recall that for the multi-index $\alpha=(\alpha_1,\ldots,\alpha_n)$ we have
$\alpha!=\alpha_1!\cdots\alpha_n!$.

There are three desiderata one could want from a jet connection $\nabla: \mathscr{J}^r(M) \times \mathscr{J}^r(M) \rightarrow \mathscr{J}^r(M)$, written $(X,Y) \mapsto \nabla_XY$:

\begin{enumerate}
	\item Linearity over $C^\infty(M)$ in $X$: $\nabla_{fX_1+gX_2}Y = f\nabla_{X_1} Y + g\nabla_{X_2}Y$ for $f, g \in C^\infty(M)$.
	
	\item Linearity over $\vvmathbb{R}$ in $Y$: $\nabla_X(aY_1+bY_2) = a\nabla_XY_1 + b\nabla_XY_2$ for $a, b \in \vvmathbb{R}$.
	
	\item Leibniz product rule: $\nabla_X(fY) = \dfrac{\alpha!}{\alpha_1!\alpha_2!} X^\alpha \partial_{\alpha_1}f \nabla_{\alpha_2} Y$ for $f \in C^\infty(M)$.
\end{enumerate}

\begin{remark}
	Although the integrable case was our motivation, we do not require $X$ to be integrable. The generators $X_r(\lambda)$ of the jet stream will thus, in general, contain terms corresponding to the prolongations of $X_1$ as well as to an innovation introduced at each successive order in $\lambda$ by its higher tangent components. For how this works in free space, see equation (\ref{jet_stream_expansion}) below.
\end{remark}

These properties are direct transcriptions of the properties that hold for the usual affine connection defined on the tangent bundle. The form the jet connection must have follows from these stipulations. Let its action on the coordinate basis vectors be given by
\begin{equation}\label{eq25}
	\nabla_{\partial_\alpha}\partial_\beta = \Gamma^\gamma_{\alpha\beta}\partial_\gamma
\end{equation}
where $\partial_\alpha$, $1 \le |\alpha|\le r$, is the coordinate frame in $U \subset M$ and the Christoffel symbols $\Gamma^\gamma_{\alpha\beta}$ are smooth functions in $C^\infty(U)$---from now on, the Einstein summation convention over multi-indices being in force. Then by the product rule the covariant derivative of a higher-order vector field $Y = Y^\alpha \partial_\alpha$ in the direction of the higher-order vector field $X = X^\alpha \partial_\alpha$ is given by
\begin{equation}\label{jet_conn_on_vf}
	\nabla_XY = \left( \frac{\alpha!}{\alpha_1!\alpha_2!} X^\alpha \partial_{\alpha_1} Y^\beta \Gamma^\gamma_{\alpha_2\beta}\right) \partial_\gamma.
\end{equation}

To start with, we observe that the jet covariant derivative, if it exists, is local in the sense that information about what direction to apply the parallel transport is already encoded in the values of the higher-order vector field $X$ at a point through the higher-order tangents.

\begin{lemma}\label{locality}
	If $\nabla$ is a jet connection on a manifold $M$, $X \in \mathscr{J}^r(M)$, $Y \in \mathscr{J}^r(M)$ and $p \in M$, then $\nabla_XY|_p$ depends only on the value of $X$ at $p$ and on the values of $Y$ in an arbitrarily small neighborhood of $p$. 
\end{lemma}
\begin{proof}
	(Cf. \cite{lee_riemannian_manif}, Lemmas 4.1 and 4.2.) We have to show that if $X = \tilde{X}$ at $p$ and $Y = \tilde{Y}$ in a neighborhood $U$ of $p$, then $\nabla_XY|_p = \nabla_{\tilde{X}}\tilde{Y}|_p$. Replacing $Y$ by $Y - \tilde{Y}$, it suffices to show that $\nabla_XY|_p = 0$ if $Y$ vanishes identically on $U$. Let $\chi \in C^\infty(M)$ be a bump function supported in $U$ such that $\chi(p)=1$. Clearly, $\chi Y$ vanishes identically in $U$, so by linearity over $\vvmathbb{R}$ we have
	$\nabla_X(\chi Y) = \nabla_X( 0 \cdot \chi Y) = 0 \cdot \nabla_X(\chi Y) = 0$. Then by the product rule,
	$\nabla_X(\chi Y) = X^\alpha \frac{\alpha!}{\alpha_1!\alpha_2!} \partial_{\alpha_!}\chi \nabla_{\alpha_2}Y$. All but the final term on the right-hand side vanish because $Y = 0$ identically on the support of $\chi$. Hence, evaluating at $p$, we have $\nabla_XY|_p = 0$.
	
	As for locality in $X$, it suffices by linearity to show that $\nabla_XY|_p = 0$ whenever $X_p = 0$. In some set of coordinates in a neighborhood of $p$, we can write $X = X^\alpha \partial_\alpha$ with $X^\alpha(p)=0$, $1 \le |\alpha| \le r$. For any $Y \in \mathscr{J}^r(M)$, we have
	\begin{equation}\label{eq27}
		\nabla_XY|_p = \nabla_{X^\alpha \partial_\alpha} Y|_p = X^\alpha(p) \nabla_{\partial_\alpha} Y|_p = 0.
	\end{equation}
	Therefore, the desired locality in $X$ holds.
\end{proof}

Euclidean space $\vvmathbb{R}^n$ equipped with the standard coordinates has the trivial jet connection $\nabla^0$ defined by $\nabla^0_{\partial_\alpha}\partial_\beta=0$. In fact, a large class of jet connections exists on any manifold, as we presently show.

\begin{lemma}\label{localjetconn}
	Let $U \subset M$ be a coordinate chart. The jet connections on $U$ are in one-to-one correspondence with sets of smooth functions $\Gamma^\gamma_{\alpha\beta} \in C^\infty(U)$, $1 \le |\alpha| \le r$, $1 \le |\beta| \le r$, $1 \le |\gamma| \le r$ via the rule of equation (\ref{jet_conn_on_vf}).
\end{lemma}
\begin{proof} (Cf. \cite{lee_riemannian_manif}, Lemma 4.4.)
	Given a connection $\nabla$, (\ref{jet_conn_on_vf}) follows from the definition (\ref{eq25}) and the product rule. In the other direction, given arbitrary $\Gamma^\gamma_{\alpha\beta} \in C^\infty(U)$, the expression (\ref{jet_conn_on_vf}) is smooth if $X$ and $Y$ are smooth, and linearity over $C^\infty(U)$ in $X$ and over $\vvmathbb{R}$ in $Y$ are immediate. It remains to check the product rule:
	\begin{align}\label{eq29}
		\nabla_X(fY) &= X^\alpha \frac{\alpha!}{\alpha_1!\alpha_2!} \partial_{\alpha_1} (fY^\beta) \Gamma^\gamma_{\alpha_2\beta} \partial_\gamma \nonumber \\
		&=  X^\alpha \frac{\alpha!}{\alpha_1!\alpha_2!} \frac{\alpha_1!}{\beta_1!\beta_2!} \partial_{\beta_1}f \partial_{\beta_2}Y^\beta \Gamma^\gamma_{\alpha_2\beta} \partial_\gamma \nonumber \\
		&= X^\alpha \frac{\alpha!}{\beta_1!\beta_2!\alpha_2!} \partial_{\beta_1}f \partial_{\beta_2}Y^\beta \Gamma^\gamma_{\alpha_2\beta} \partial_\gamma \nonumber \\
		&=  X^\alpha \frac{\alpha!}{\beta_1!\gamma_1!} \partial_{\beta_1}f \nabla_{\gamma_1}Y,
	\end{align}
	where in going from the second to the third line the summation becomes one over multi-indices $\beta_1,\beta_2,\alpha_2$ such that $\beta_1+\beta_2+\alpha_2=\alpha_1+\alpha_2=\alpha$ and likewise in going from the third to the fourth, we define $\gamma_1=\beta_2+\alpha_2$ so that $\beta_1+\gamma_1=\beta_1+\beta_2+\alpha_2=\alpha$.
\end{proof}

\begin{proposition}[Existence]
	Every manifold $M$ admits a jet connection $\nabla$ of any order $r$.
\end{proposition}
\begin{proof} (Cf. \cite{lee_riemannian_manif}, Proposition 4.5.)
	Let $M$ be covered by coordinate charts $(U_a,\varphi_a)$ and let $\chi_a$ be a partition of unity subordinate to $(U_a,\varphi_a)$. Choose functions $\Gamma^{\gamma a}_{\alpha\beta} \in C^\infty(U_a)$ and define jet connections $\nabla^a$ on $U_a$ by lemma \ref{localjetconn}. A jet connection can be defined on all of $M$ via the formula
	\begin{equation}\label{eq30}
		\nabla_XY = \sum_a \chi_a \nabla^a_XY.
	\end{equation}
	Linearity over $C^\infty(M)$ in $X$ and over $\vvmathbb{R}$ in $Y$ follow from the corresponding properties for $\nabla^a$. The product rule can be verified as follows:
	\begin{align}\label{eq31}
		\nabla_X(fY) &= \sum_a \chi_a \nabla^a_X (fY) \nonumber \\
		&= \sum_a \chi_a \left( X^\alpha \frac{\alpha!}{\alpha_1!\alpha_2!} \partial_{\alpha_1}f \nabla^a_{\alpha_2}Y \right) \nonumber \\
		&= X^\alpha  \frac{\alpha!}{\alpha_1!\alpha_2!} \partial_{\alpha_1}f \sum_a \chi_a \nabla^a_{\alpha_2}Y \nonumber \\
		&=  X^\alpha  \frac{\alpha!}{\alpha_1!\alpha_2!} \partial_{\alpha_1}f \nabla_{\alpha_2}Y.
	\end{align}
	Therefore, the Leibniz product rule holds for $\nabla$.
\end{proof}

There is no canonical way to restrict a tangent vector field of higher than first order to a submanifold. At first, this may seem like a disappointment; yet it is of the essence of the geometrical problem. For what it tells us is that one has to remember its behavior nearby in transverse directions. In the case of a curve, on which we focus henceforward, we have labeled this situation the jet stream. Suppose one has a curve $\gamma:[a,b] \rightarrow M$. In order to represent the jet stream, select a hypersurface $\Sigma$ transverse to $\gamma$ at $p \in M$ and suppose one may develop it into a family of transverse hypersurfaces $\Sigma_\lambda$ such that $\Sigma_0 = \Sigma$. If we choose coordinates $x_{2,\ldots,n}$ along $\Sigma$ then we have, locally around $p$, a coordinate system $\lambda,x_2,\ldots,x_n$. Given a higher tangent $j \in J^r_p[a,b]$, we can push it forward to become a higher tangent $\gamma_* j \in J^rM$. To go in the other direction, suppose $X$ is a higher vector field defined on $M$. For any function $f$ on the interval $[a,b]$ we may define $f_\lambda$ as a function on $M$ near $p$ by extending $f$ so as to remain constant on the leaves $\Sigma_\lambda$. One obtains thereby a vector field $X_\lambda \in \mathscr{J}^r[a,b]$ living on $\gamma$, such that $X_\lambda f = X f_\lambda|_\lambda$. 

In the theory of the affine connection on a manifold, one has the concept of the covariant derivative along a curve. The analogue of this construction in the theory of jets is the covariant derivative along a jet stream. By a higher-order vector field along the jet stream $\gamma$ we will mean a smooth map $V: I \rightarrow J^rM$ such that $V(\lambda) \in J^r_{\gamma(\lambda)}M$ for every $t \in I$. Let us denote the space of such higher-order vector fields by $\mathscr{J}^r(\gamma)$. The higher-order vector field $V$ along $\gamma$ will be said to be extendable if there exists a higher-order vector field  on a neighborhood $U$ of the image of $\gamma$ that is related to $V(t)$ via $V(t) = \tilde{V}_\lambda(\gamma(t))$. The following proposition makes the concept of a covariant derivative of a higher-order vector field along a jet stream well-defined.

\begin{proposition}Let $\nabla$ be a jet connection on $M$. For every jet stream $\gamma: I \rightarrow M$, $\nabla$ determines a unique operator $D: \mathscr{J}^r(\gamma) \times \mathscr{J}^r(\gamma) \rightarrow \mathscr{J}^r(\gamma)$ that satisfies the properties:
	
	\noindent \qquad \textup{(i)} Linearity over $\vvmathbb{R}$: $D_X(aV+bW) = aD_XV + bD_XW$ for $a,b \in \vvmathbb{R}$.
	
	\noindent \qquad \textup{(ii)} Product rule: $D_X(fV) = X^\alpha \dfrac{\alpha}{\alpha_1!(\alpha-\alpha_1)!} f^{(\alpha_1)} D_{(\alpha-\alpha_1)}V$ for $f \in C^\infty(I)$.
	
	\noindent \qquad \textup{(iii)} If $V$ is extendable, then for any extension $\tilde{V}$ of $V$ and any $X \in \mathscr{J}^rM$, $D_{X_\lambda} V(t)  = \nabla_X \tilde{V}|_{\gamma(\lambda)}$.
\end{proposition}
\begin{proof} (Cf. \cite{lee_riemannian_manif}, Lemma 4.9.)
	First, uniqueness. Suppose $D:\mathscr{J}^r(\gamma) \times \mathscr{J}^r(\gamma) \rightarrow \mathscr{J}^r(\gamma)$ is an operator that satisfies (i),(ii),(iii) and let $\lambda_0 \in I$. Arguing as in the proof of lemma \ref{locality}, we know that $D_XV|_{\lambda_0}$ depends only on the values of $V$ in an interval $(\lambda_0-\varepsilon,\lambda_0+\varepsilon)$, where we extend the interval $I$ on which $\gamma$ is defined beyond its endpoint if necessary. In coordinates $\lambda,x^2,\ldots,x^n$ on $U$ near $\gamma(\lambda_0)$ we can write $V(\lambda)=V^\alpha(\lambda)\partial_\alpha$. Then by (i),(ii),(iii) and the extendability of the basis vector fields $\partial_\alpha$ we have
	\begin{align}\label{eq32}
		D_XV(t_0) &= X^\alpha \frac{\alpha!}{\alpha_1!(\alpha-\alpha_!)!} V^\alpha_{,\alpha_1}(\lambda_0) \nabla_{(\alpha-\alpha_1)}\partial_\alpha \nonumber \\
		&= X^\alpha \frac{\alpha!}{\alpha_1!(\alpha-\alpha_!)!} \left( V^\beta_{,\alpha_1}(\lambda_0)\Gamma^\gamma_{\beta(\alpha-\alpha_1)}(\gamma(t_0))X^\alpha_r \right)\partial_\gamma
	\end{align}
	Hence, $D_t$, if it exists, is uniquely defined. The formula (\ref{eq32}) can be used to define $D_t$ in the chart $U$. If more than one chart is needed to cover $\gamma(I)$, this formula defines $D_t$ in each chart and by uniqueness the definitions must agree on the overlaps.  
\end{proof}

\subsection{Existence of Jet Geodesics, or Diffeomorphisms with Inertia}\label{existence_of_jet_geodesics}

The concept of covariant differentiation along a jet stream supplies us with an invariant notion of acceleration. If $X$ is the $r$-th order infinitesimal generator of a jet stream along the curve $\gamma$, we define the acceleration of $\gamma$ to be the higher-order vector field $D_{X_\lambda}X_\lambda \in \mathscr{J}^r(\gamma)$. The curve $\gamma$ is said to be a jet geodesic with respect to $\nabla$ if its acceleration vanishes, $D_{X_\lambda}X_\lambda = 0$ identically on $\gamma$. Before turning to the problem of existence, we solve for the jet geodesics of the trivial jet connection in Euclidean space.

\begin{proposition}\label{jeteuclid}
	Consider Euclidean space $\vvmathbb{R}^n$ equipped with the trivial jet connection $\nabla^0$. The jet geodesics of $\nabla^0$ correspond to motion along integral curves of the mapping
	\begin{equation}\label{jet_euclid_ansatz}
		\theta^{(\mathbf{x})}: t \mapsto \theta^{(\mathbf{x})}(t) = \mathbf{x} + t \mathbf{v}^{(1)} + \frac{1}{2} 
		t^2 \mathbf{v}^{(2)} + \cdots +
		\frac{1}{r!} t^r \mathbf{v}^{(r)},
	\end{equation}
	for constant vectors $\mathbf{v}^{(s)} \in \vvmathbb{R}^n$, $s=1,\ldots,r$.
\end{proposition}
\begin{proof}
	Let the flow be given by the Ansatz
	\begin{equation}\label{eq36}
		\theta^{(\mathbf{x})}(t) = \mathbf{x} + t \mathbf{v}^{(1)} + \frac{1}{2} 
		t^2 \mathbf{v}^{(2)} + \cdots +
		\frac{1}{r!} t^r \mathbf{v}^{(r)},
	\end{equation}
	for all points $\mathbf{x} \in \vvmathbb{R}^n$ and some vectors $\mathbf{v}^{(s)} \in \vvmathbb{R}^n$. The curve $\gamma$ itself will be generated as the flow emanating from a given point: $\gamma(t)=\theta^{(\mathbf{x})}(t)$. Given an arbitrary smooth function $f$ on $\vvmathbb{R}^n$ let $f_\lambda = f(\theta^{(\mathbf{x})}(\lambda))$. But from the Ansatz it is evident that
	\begin{align}\label{jet_stream_expansion}
	\gamma_* \frac{\partial}{\partial\lambda} &= \mathbf{v}^{(1)} \partial_\mathbf{x} =: X_1 \nonumber \\
	\gamma_* \frac{\partial^2}{\partial\lambda^2} &= \mathbf{v}^{(2)} \partial_\mathbf{x} + \mathbf{v}^{(1)} \otimes \mathbf{v}^{(1)} \partial^2_\mathbf{x} =: X_2 \nonumber \\
	\gamma_* \frac{\partial^3}{\partial \lambda^3} &= \mathbf{v}^{(3)} \partial_\mathbf{x} + 3 \mathbf{v}^{(2)} \otimes \mathbf{v}^{(1)} \partial^2_\mathbf{x} + \mathbf{v}^{(1)} \otimes \mathbf{v}^{(1)} \otimes \mathbf{v}^{(1)} \partial^3_\mathbf{x} =: X_3	
	\end{align}etc. In other words, if we are at $\lambda_0$ along the integral curve starting at $\mathbf{x}$ we have 
	\begin{equation}
	f(\theta^{(\mathbf{x})}(\lambda_0+\lambda)) = (1 + \lambda X_1 + \cdots + \frac{1}{r!} \lambda^r X_r) f \mid_{\theta^{(\mathbf{x})}(\lambda_0)} + o(\lambda^r).
	\end{equation}
	If we set
	\begin{equation}
		X_\lambda = \partial_\lambda + \frac{1}{2}\frac{\partial^2}{\partial \lambda^2} + \cdots + \frac{1}{r!}\frac{\partial^r}{\partial \lambda^r} = \partial_\lambda + \partial_{\lambda\lambda} + \cdots + \partial_{\lambda\cdots\lambda ~(r ~\mathrm{times})}.
	\end{equation}
	then the infinitesimal change in $f$ is given up to $r$-th order by
	\begin{equation}\label{euclidean_flow}
	[f(\theta^{(\mathbf{x})}(\lambda_0+\lambda))-f(\theta^{(\mathbf{x})}(\lambda_0))](\gamma_* X_\lambda) = ( \gamma^* d^r f)(X_\lambda).
	\end{equation}
Equation (\ref{euclidean_flow}) says that, at any location along the flow which putatively is to correspond to a solution of the jet geodesic equation, the action of $X_\lambda$ is given by the same differential operator without explicit $\lambda_0$ dependence. Clearly then $D^0_{X_\lambda} X_\lambda = 0$. Therefore, the orbits of the flow (\ref{eq36}) are the jet geodesics in Euclidean $\vvmathbb{R}^n$.

Conversely, suppose we have a solution to the jet geodesic equation for $\nabla^0$ starting from $\mathbf{x}$ at $\lambda=0$. Then we can select $\mathbf{v}^{(1)},\ldots,\mathbf{v}^{(r)}$ so that equation (\ref{jet_euclid_ansatz}) matches the integral curve up to $r$-th order at $\lambda=0$, but since the jet geodesic equation itself is an ordinary differential equation of $r$-th order, by uniqueness this form of the solution continues to hold for all future time.
\end{proof}

Reflect for a moment on what this means. Clearly, for the case $r=1$ one recovers straight-line motion, as one should. The implication of proposition \ref{jeteuclid}, however, is that freely moving bodies in empty space acquire inertia in a higher sense. The first term, $\mathbf{v}^{(1)}$, would be just the momentum in the ordinary sense of a body of unit mass. At short enough times, this lowest-order term dominates, but at longer times the effect of the $\mathbf{v}^{(2,\ldots,r)}$ makes itself felt. The higher components of momentum are somehow endowed with the virtue of being able to produce a constant acceleration, or yet higher derivatives with respect to time. Their constancy (independence from $\lambda$) reflects a conservation law for higher-order momentum in free space.

Evidently, one could envision a generalization of the concept of impressed force. Just as a force in the Newtonian sense gives rise to an impulse, i.e., an increment or decrement of momentum during the time over which the force operates, here too a higher force must produce a higher impulse, or change in the respective higher component of momentum. But we already have at our disposal Einstein's profound notion of motion along a geodesic in curved space as the appropriate means by which to represent impressed forces in a covariant manner. Therefore, we have every right to expect that the same applies in a post-Einsteinian general theory of relativity when the geodesic is replaced with the jet stream, only now the higher infinitesimals ought to manifest themselves as producing forces beyond those currently known in Einstein's theory. To justify this viewpoint will take considerable work, see our remarks in {\S}\ref{discussion} below. For the time being, we have first to prove the existence of jet geodesics for a general jet connection and to investigate the associated phenomena more closely. 

Now, the prolongations of an ordinary 1-vector field serve merely to implement the flow over a finite interval of time via the exponential function (which will be convergent in the analytic case). With higher tangency, we are faced with a novel phenomenon, to which we have given the name of a diffeomorphism with inertia, of which the simple-minded investigation in {\S}\ref{phoronomy} above permits some description at very short times. We have assembled by now a formalism sufficient for us to derive the diffeomorphism flow with inertia up to a finite interval of time.

\begin{theorem}[Existence and Uniqueness of Jet Geodesics]\label{existence_theorem_for_jet_geodesics}
	Let $M$ be a manifold with jet connection $\nabla$. For any $p \in M$ and any initial datum $\varv \in J^r_pM$ which is admissible in the sense that there exists an analytic Cauchy hypersurface $\Sigma_0$ passing through $p$ having the property of being non-characteristic for $\varv$ there, then there exists a jet geodesic $\gamma: I \rightarrow M$ defined on a half-open interval $I \subset \vvmathbb{R}$ and satisfying $\gamma(t_0)=p$, $X_r(t_0)=\varv$ for $t_0 \in I$. Any two such jet geodesics agree on their common domain.
\end{theorem}
\begin{proof} 
	1) By Lee \cite{lee_smooth_manif}, Theorem 5.8 (local slice criterion for embedded submanifolds), we can find coordinates $(\lambda,x_2,\ldots,x_n)$ near $p$ such that $\Sigma_0$ is given by the condition $\lambda=0$. The jet geodesic equation $\nabla_X X=0$ defines a non-linear partial differential equation for which the Cauchy problem is well defined by virtue of the hypothesis on the initial datum, namely,
	\begin{equation}\label{jet_geodesic_pde}
		\frac{\alpha!}{\alpha_1!\alpha_2!} X^\alpha X^\beta_{,\alpha_1} \Gamma^\lambda_{\alpha_2,\beta} \partial_\lambda = 0.
	\end{equation}
	Equation (\ref{jet_geodesic_pde}) forms a quasilinear system in the $N_1+\cdots+N_r$ variables $X^\alpha$, where $N_s = \binom{n+s-1}{s}$, $s=1,\ldots,r$. The condition for $\Sigma_0$ to be non-characteristic reduces in our chosen coordinate system to $X^{(r,0,\ldots,0)} \ne 0$ in a neighborhood of $p$ on $\Sigma_0$. Hence, by the Cauchy-Kovalevskaia theorem \cite{john_pde} there will exist a local solution on some neighborhood of $p$ that extends the components of $\varv$ analytically away from $\Sigma_0$ to a vector field $X$ defined on $\Sigma_\lambda$ for $0 \le \lambda < \varepsilon$. 
	
	2) An heuristic argument will show why, given this $X$, the solution ought to exist and how the jet stream is generated. Consider a family of Gaussian test functions $\varphi_\varkappa(0)$ defined on $\Sigma$ with centroid at $p$, converging to a delta function at $p$ in the limit as $\varkappa$ tends to infinity. We may then produce the development in time via the relation
	\begin{equation}
		\varphi_\varkappa(\lambda) = \exp X(\lambda) \varphi_\varkappa(0),
	\end{equation}
	which is to be regarded as a power series in $\lambda$ uniformly convergent in an interval containing the origin. The first thing to note is that by construction $\nabla_X \varphi_\varkappa(\lambda) = 0$ on its domain of definition. Define the curve $\gamma(\lambda)$ from the centroids of $\varphi_\varkappa(\lambda)$ for given $\lambda$ as $\varkappa \rightarrow \infty$. Then $\varphi_\kappa(\lambda)X$ tends to $X_{\gamma(\lambda)}$ and $\nabla_X X=$ to $D_\lambda X=0$ on $\gamma$. The argument fails to be rigorous, of course, because it might not be possible to bound the domain of existence of the local solution uniformly away from zero as $\varkappa$ diverges.
	
	Before proceeding to the derivation of the jet stream, it will be expedient first to put $X$ into a canonical form. Consider now the first non-trivial case $r=2$. The hypothesis on the initial datum implies that the second-order part of $X$ assumes the form
	\begin{equation}\label{pre_canonical_two}
		X_2 = \varv \odot \varv + \varv \odot w + u_i \odot u_j,
	\end{equation}
	where $\varv$ is a vector with non-zero component in the $\lambda$-direction and $w, u_i$ are orthogonal to $\varv$. Without loss of generality, by projecting along $\varv$ we may suppose instead that $w, u_i$ lie in $T\Sigma_0$ (which is transverse to $\varv$ by assumption). Then if we replace $\varv$ by $\tilde{\varv} = \varv + \dfrac{1}{2} w$ we find
	\begin{equation}\label{canonical_two}
		X_2 = \tilde{\varv}\odot\tilde{\varv} + u_i \odot u_j;
	\end{equation}
	i.e., we can eliminate the $\varv\otimes w$ term. For the inductive step, consider the case of arbitrary $r \ge 2$. Now we have, as in equation (\ref{pre_canonical_two}), the highest-order part of $X$ may be written in the form
	\begin{equation}
		X_r = \overbrace{\varv \odot \cdots \odot \varv}^{r ~\mathrm{times}} + \overbrace{\varv \odot \cdots \odot \varv}^{r-1 ~\mathrm{times}} \odot w_0 +
		\overbrace{\varv \odot \cdots \odot \varv}^{r-2 ~\mathrm{times}} \odot w_1 \otimes w_2 + \cdots + u_{i_1} \odot \cdots \odot u_{i_r},
	\end{equation}
	where as before $\varv$ has a non-zero component in the $\lambda$-direction and $w_{0,1,2,\dots}$ and the $u_{i_s}$ are initially orthogonal to $\varv$, but may without loss of generality be projected along $\varv$ to lie in $T\Sigma_0$. Set $\tilde{\varv} = \varv + \dfrac{1}{r} w_0$ so as to obtain
	\begin{align}
		X_r &= \overbrace{\tilde{\varv} \odot \cdots \odot \tilde{\varv}}^{r ~\mathrm{times}} + \overbrace{\tilde{\varv} \odot \cdots \odot \tilde{\varv}}^{r-2 ~\mathrm{times}} \odot w_1 \odot w_2 + \cdots +  u_{i_1} \odot \cdots \odot u_{i_r} \nonumber \\
		&= \tilde{\varv} \odot \left( \overbrace{\tilde{\varv} \odot \cdots \otimes \tilde{\varv}}^{r-1 ~\mathrm{times}} + \overbrace{\tilde{\varv} \odot \cdots \odot \tilde{\varv}}^{r-3 ~\mathrm{times}} \otimes w_1 \odot w_2 + \cdots \right) +  u_{i_1} \odot \cdots \odot u_{i_r}
	\end{align}
	Hence, we have reduced the problem to the $r-1$ case. For let us look closer at equation (\ref{canonical_two}). Let $u_0 = u_1 + \cdots + u_{n-1}$ and set $u_i = u_i^\prime + \dfrac{1}{n-1}u_0 + A \tilde{\varv}$, where the coefficient $A$ is to be determined in a moment. If we substitute into equation (\ref{canonical_two}), we get (using $u_1^\prime + \cdots u_n^\prime = - A \tilde{\varv}$)
	\begin{align}
		X_2 &= (1-A^2) \tilde{\varv} \odot \tilde{\varv} + \frac{2A}{n-1} \tilde{\varv} \odot u_0 + \frac{1}{(n-1)^2} u_0 \odot u_0 + u_i^\prime \odot u_j^\prime \nonumber \\
		&= \left( 1+A^2 \left( \frac{1}{(n-1)^2}-1 \right) \right) \tilde{\varv} \odot \tilde{\varv} + \frac{2A}{n-1} \tilde{\varv} \odot u_0 + \frac{1}{(n-1)^2} u_0 \odot u_0 + u_i^{\prime\prime} \odot u_j^{\prime\prime} \nonumber \\
		&= \sqrt{1+A^{\prime 2}} \tilde{\varv} \odot \sqrt{1+A^{\prime 2}} \tilde{\varv} + \frac{1}{(n-1)^2} u_0 \odot u_0 + \frac{2A}{(n-1)\sqrt{1+A^{\prime 2}}} \sqrt{1+A^{\prime 2}} \tilde{\varv} \odot u_0 + u_i^{\prime\prime} \odot u_j^{\prime\prime} \nonumber \\
		&= \left( \sqrt{1+A^{\prime 2}} \tilde{\varv} + \frac{1}{n-1} u_0 \right) \odot
		\left( \sqrt{1+A^{\prime 2}} \tilde{\varv} + \frac{1}{n-1} u_0 \right) \nonumber \\
		&=: \bar{\varv} \odot \bar{\varv}
	\end{align}
	if $u_i^{\prime\prime} = u_i^\prime + \dfrac{A}{n-1} \tilde{\varv}$ and $A$ is chosen such that $\dfrac{A}{\sqrt{1+A^{\prime 2}}} = 1$, a solution for which always exists with $1 \le A < \dfrac{1}{\sqrt{2}}$.
	
	For the inductive step, suppose that when $s=2,\ldots,r-1$ it is possible to accomplish the following reduction:
	\begin{align}
		X_2 &= \overbrace{\varv \odot \cdots \odot \varv}^{s ~\mathrm{times}} + \overbrace{\varv \odot \cdots \odot \varv}^{s-1 ~\mathrm{times}} \odot w_0 +
		\overbrace{\varv \odot \cdots \odot \varv}^{s-2 ~\mathrm{times}} \odot w_1 \otimes w_2 + \cdots + u_{i_1} \odot \cdots \odot u_{i_s} \nonumber \\
		&= \overbrace{\bar{\varv} \odot \cdots \bar{\varv}}^{s ~\mathrm{times}} + u^\prime_{i_1} \odot \cdots \odot u^\prime_{i_s}.
	\end{align}
	Suppose we are given an expression of this form with $s=r$. Then undertake the following series of transformations (for the sake of simplicity, we suppress subscripts on the $w$ vectors):
	\begin{align}
		X_2 &= \overbrace{\varv \odot \cdots \odot \varv}^{r ~\mathrm{times}} + \overbrace{\varv \odot \cdots \odot \varv}^{r-1 ~\mathrm{times}} \odot w_0 +
		\overbrace{\varv \odot \cdots \odot \varv}^{r-2 ~\mathrm{times}} \odot w_1 \otimes w_2 + \cdots + w_{i_1} \odot \cdots \odot w_{i_r} \nonumber \\
		&= \varv \odot \left( \overbrace{\varv \odot \cdots \odot \varv}^{r-1 ~\mathrm{times}} + \overbrace{\varv \odot \cdots \odot \varv}^{r-2 ~\mathrm{times}} \odot w_0 + \cdots +
		\overbrace{w \otimes \cdots \odot w}^{r-1 ~\mathrm{times}} \right) +  w_{i_1} \odot \cdots \odot w_{i_r} \nonumber \\
		&= \varv \odot \overbrace{\bar{\varv} \odot \cdots \odot \bar{\varv}}^{r-1 ~\mathrm{times}} + \varv \odot \overbrace{w \odot \cdots \odot w}^{r-1 ~\mathrm{times}} + \overbrace{w \odot \cdots \odot w}^{r ~\mathrm{times}} \nonumber \\
		&= \left( \bar{\varv}+\bar{w} \right) \odot \overbrace{\bar{\varv} \odot \cdots \odot \bar{\varv}}^{r-1 ~\mathrm{times}} + \left( \bar{\varv} + \bar{w} \right) \odot  \overbrace{w \odot \cdots \odot w}^{r-1 ~\mathrm{times}} + \overbrace{w \odot \cdots \odot w}^{r ~\mathrm{times}} \nonumber \\
		&= \overbrace{\bar{\varv} \odot \cdots \odot \bar{\varv}}^{r ~\mathrm{times}} +  \overbrace{\bar{\bar{\varv}} \odot \cdots \odot \bar{\bar{\varv}}}^{r-1 ~\mathrm{times}} \odot w + \overbrace{w \odot \cdots \odot w}^{r ~\mathrm{times}} \nonumber \\
		&= \left( \varv_0 + \varv_1 \right) \odot \left( \varv_2 + \Delta w \right)^{\odot (r-1)} + \left( \varv_0 - \varv_1 \right) \odot \left( \varv_2 - \Delta w \right)^{\odot (r-1)} + \overbrace{w \odot \cdots \odot w}^{r ~\mathrm{times}}.
	\end{align}
	By a further induction using the fact that the tangent space decomposes as $\mathrm{\hat{\varv}}\oplus T\Sigma_0$, it is not too hard to show that expressions of this form can be reduced to
	\begin{equation}\label{reduced_prod_form}
		X_2 = \tilde{\varv}_1 \odot {\tilde{\varv}_2}^{\odot (r-1)} = \left( \varv + w \right) \left( \varv - w \right)^{\odot(r-1)},
	\end{equation}
	dropping any terms of the form $\overbrace{w \odot \cdots \odot w}^{r ~\mathrm{times}}$ as irrelevant for the time being. Consider now a linear change of coordinate that sends these two vectors to the following:
	\begin{align}
		\varv &\mapsto \varv^\prime = \xi( \delta \varv + \alpha w ) \\
		w &\mapsto w^\prime = \gamma + \beta w.
	\end{align}
	Then equation (\ref{reduced_prod_form}) becomes
	\begin{equation}
		X_2 = \xi \left((\delta + \gamma) \varv + (\alpha+\beta) w \right) \left( (\delta - \gamma) \varv + (\alpha - \beta) w \right)^{\odot(r-1)}.
	\end{equation}
	We want $\xi(\delta+\gamma)=\delta-\gamma$ and $\xi(\alpha+\beta)=\alpha-\beta$, which may readily be arranged (for instance, take $\delta=\alpha=1$, $\beta=\gamma=\frac{1}{2}$, $\xi=\frac{2}{3}$). After a rescaling we end up with
	\begin{equation}
		X_2 = \tilde{\varv}^{\odot r} +  \overbrace{w \odot \cdots \odot w}^{r ~\mathrm{times}},
	\end{equation}
	thus completing the inductive step. We conclude that we may put $X$ into the overall form,
	\begin{equation}
		X = \sigma \partial_{\lambda \cdots \lambda ~(r ~\mathrm{times})} + X^{(r)}_\perp + X^{(r-1)} + \cdots X^{(1)}.
	\end{equation}
	Here, the coefficient $\sigma$ may of course have any spatial dependence. But for each value of the transverse coordinates $(0,x_2,\ldots,x_n)$ on $\Sigma_0$ we have a first-order ordinary differential equation to solve, namely,
	\begin{equation}
		\left( \frac{d\bar{\lambda}}{d\lambda} \right)^r = \sigma
	\end{equation}
	in order to set the leading coefficient identically equal to one with respect to the coordinates $(\bar{\lambda},x_2,\ldots,x_n)$.
	
	All that remains is to apply a similar procedure to the subleading terms in $X$. From the vector transformation law given by equation (\ref{vector_transf_law}), it is evident that upon another change of coordinate of the form
	$\bar{\lambda} = \lambda + c_{r-1} \lambda^2$ with $c_{r-1}>0$ sufficiently large, we may ensure that $X^{(r-1,0,\ldots,0)}>0$. Then, applying what has already been shown to the $X^{(r-1)}$ part we get $X^{(r-1)}=\sigma_{r-1} \partial_{\lambda \cdots \lambda ~(r-1 ~\mathrm{times})} + X^{(r-1)}_\perp$. Continuing in this fashion (with a change of time coordinate of the form $\bar{\lambda}=\lambda+c_{r-2}\lambda^3$ in the next step), we therefore have that $X$ may be represented in the format,
	\begin{equation}
		X = \partial_{\lambda \cdots \lambda ~(r ~\mathrm{times})} + X^{(r)}_\perp +
		\varv_{r-1}^{\odot (r-1)} + X^{(r-1)}_\perp + \cdots +
		\varv_1 + X^{(1)}_\perp,
	\end{equation}
	where 
	\begin{equation}
		\varv_s = \partial_\lambda + w_s, \qquad s=1,\ldots,r-1.
	\end{equation}
	The same trick as above yields,
	\begin{equation}
		X = \tilde{\varv}_r^{\odot r} + X^{(r)}_\perp +
		\tilde{\varv}_{r-1}^{\odot (r-1)} + X^{(r-1)}_\perp + \varv_{r-2}^{\odot (r-2)} + X^{(r-2))}_\perp \cdots +
		\varv_1 + X^{(1)}_\perp.
	\end{equation}
	Transforming the time coordinate once again gives us
	\begin{equation}
		X = \partial_{\lambda \cdots \lambda ~(r ~\mathrm{times})} + X^{(r)}_\perp + \sigma_{r-1} \partial_{\lambda \cdots \lambda ~(r-1 ~\mathrm{times})}  + X^{(r-1)}_\perp + \varv_{r-2}^{\odot (r-2)} + X^{(r-2))}_\perp \cdots +
		\varv_1 + X^{(1)}_\perp.
	\end{equation}
	A successive series of like transformations on the lower order terms leads eventually to our final canonical expression for $X$, viz.,
	\begin{align}	
		X &= \partial_{\lambda \cdots \lambda ~(r ~\mathrm{times})} + X^{(r)}_\perp + \sigma_{r-1} \partial_{\lambda \cdots \lambda ~(r-1 ~\mathrm{times})}  + X^{(r-1)}_\perp + \cdots + \sigma_1 \partial_\lambda + X^{(1)}_\perp \nonumber \\
		&= \partial_{\lambda \cdots \lambda ~(r ~\mathrm{times})} + \sigma_{r-1} \partial_{\lambda \cdots \lambda ~(r-1 ~\mathrm{times})} + \cdots + \sigma_1 \partial_\lambda + X_\perp.
	\end{align}
	We could exploit our remaining freedom in the choice of coordinates to impose the condition $X^{(1)}_\perp=0$.
	
	The linear change of coordinate inside each tangent space that puts $X$ into canonical form defines a Jacobian $\Phi^\mu_\nu$ with the property that $\Phi^\mu_\nu=\delta^\mu_\nu$ for $2 \le |\mu|,|\nu| \le n$. Let $d\bar{\lambda} = \Phi^1_\nu d^\nu$. By the Frobenius theorem (cf. \cite{warner}, Theorem 1.60 and Proposition 2.30), the distribution $\mathrm{ker}~d\bar{\lambda}$ integrates to a foliation of a neighborhood of the point $p$ by the level sets of $\bar{\lambda}$ since $d\bar{\lambda}$ is obviously closed. Relabel $\bar{\lambda}$ as $\lambda$. Therefore, we may extend the Cauchy surface $\Sigma_0$ to $\Sigma_\lambda$ and we end up with simply $X_\lambda = \partial_{\lambda \cdots \lambda ~(r  ~\mathrm{times})} + \sigma_{r-1} \partial_{\lambda \cdots \lambda ~(r-1 ~\mathrm{times})} + \cdots + \sigma_1 \partial_\lambda$. For each value of the tranvserse coordinates $(0,x_2,\ldots,x_n)$ on $\Sigma_0$ we may trace the curve $\gamma_{(x_2,\ldots,x_n)}(\lambda) = (\lambda,x_2,\ldots,x_n)$. The restriction of $X$ to the curve $\gamma(\lambda)$ should satisfy the jet geodesic equation in the form $D_{X_\lambda} X_\lambda = 0$, which is an ordinary differential equation of order $r$ with initial conditions specified by the value of $\varv$ at the point $p$. Therefore, the existence theorem for ordinary differential equations \cite{arnold_ode} delivers a unique solution to the jet geodesic equation in an interval of time surrounding the origin.
\end{proof}

\begin{remark}
	The second half of this result makes for a higher-order analogue corresponding to the standard flow box of the 1-vector field around a regular point, see Lee, \cite{lee_smooth_manif}, Theorem 9.22. At higher order, we cannot expect to be able to make all of the coefficients $\sigma_{1,\ldots,r-1}$ identically equal to unity as this would overdetermine the coordinate $\lambda$ and we have availed ourself of our freedom to reparametrize in order to fix the coefficient of the leading term $\partial_{\lambda \cdots \lambda ~(r ~\mathrm{times})}$. Nevertheless, via a further transformation of the form $\bar{\lambda}=\lambda+a_2 \lambda^2 + \cdots + a_r \lambda^r$ we may at least arrange that $\lim_{\lambda \rightarrow 0} \sigma_{1,\ldots,r} = 1$. In other words, we have factored the flow generated by $X$ into free inertial motion in the constructed $\lambda$ coordinate for sufficiently short times plus a transverse component locally parallel to the $\Sigma_\lambda$. The linear part which dominates at the shortest times goes as $\partial_\lambda$ and contributes no transverse component. It should also be noted that the jet stream leaving from the point $p$ by all means depends on the initial datum on $\Sigma_0$ and not just on the starting point, as in the 1-vector field case.
	
	Lastly, for this reason the role of the Cauchy surface $\Sigma_0$ is ineliminable as soon as one goes to higher than first order. What makes sense to consider transporting under the flow are only functions constant along the leaves of the foliation $\Sigma_\lambda$ constructed as above from a given Cauchy surface $\Sigma_0$ around the point $p$ (perhaps this indicates how adroitly to generalize Huygens' principle to any higher tangent vector field). To illustrate the idea, consider the Laplacian in Euclidean $\vvmathbb{R}^n$ given by $\Delta=\partial_{11}+\cdots+\partial_{nn}$. For the Cauchy hypersurface take $\Sigma_0 = \{ x_1 = 0 \}$. To put $\Delta$ into canonical form, set $\lambda=e^x-1$ so that $\partial_1=e^x\partial_\lambda=(1+\lambda)\partial_\lambda$. Then $\partial_{11} = (1+\lambda)\partial_\lambda(1+\lambda)\partial_\lambda=(1+\lambda)^2 \partial_{\lambda\lambda} + (1+\lambda)\partial_\lambda$. We have
	\begin{equation}
		\Delta = (1+\lambda)^2 \partial_{\lambda\lambda} + (1+\lambda)\partial_\lambda + \partial_{22} + \cdots + \partial_{nn}.
	\end{equation}
	As we see, the best we can do is to arrange that the coefficients depending on $\lambda$ tend to unity in the limit as $\lambda \rightarrow 0$. In this choice of coordinate, a wavefront will propagate from $\Sigma_0$ in the $\lambda$-direction with unit initial speed. Due to the isotropy of the Laplacian, the problem will look the same around any other hyperplane $\Sigma$ in $\vvmathbb{R}^n$. To avoid the degeneracy in general, if instead $X=\Delta+a\partial_1$, with $a \ne 0$, there would be no need to change coordinate and its initial speed would be just $a$. The situation will perhaps become clearer if the Laplacian is a singular perturbation: for $\varepsilon>0$, let $X=\varepsilon^2 \Delta + a \partial_1$. Then rescaling to a coordinate $\lambda=\varepsilon^{-1}x_1$ we see that the wavefront leaves the vicinity of $\Sigma_0$ very fast when $\varepsilon>0$ is sufficiently small, at the speed $a/\varepsilon$. Thus, the virtual diffusion away from the centroid (of order unity in this choice of coordinate) becomes negligible compared to the translatory motion unless one has the degenerate case $a=0$. The lesson for us is that if one wants the behavior of the jet geodesic approximately to resemble the familiar translatory motion produced by an ordinary 1-vector field, the higher-order terms need to be singular perturbations.
\end{remark}

\subsection{Elementary Operations in the Differential Calculus of the Jet Connection}

There will be no deep results announced in this section, just necessary warm-up. We kick off by formulating some simple lemmas that make it easy to check tensoriality in the generalized sense. First,

\begin{lemma}[cf. Lee, \cite{lee_riemannian_manif}, Exercise 2.6]\label{local_smoothness_criterion}
	Let $F \in \Gamma(\mathscr{J}^{r*}\times \cdots \mathscr{J}^{r*}\times \mathscr{J}^{r}\times \cdots \times \mathscr{J}^{r},M)$ be a rough section. Then $F$ is a smooth generalized tensor field if and only if whenever $X_{1\ldots,k}$ are smooth $r$-vector fields and $\omega^{1,\ldots,j}$ are smooth $r$-jets defined on an open set $U \subset M$, the function
	\begin{equation}
		F(X_1,\ldots,X_k,\omega^1,\ldots,\omega^j)(p) := F_p(X_1|_p,\ldots,X_k|_p,\omega^1|_p,\ldots,\omega^k|_p)
	\end{equation}	
	defined on $U$ is smooth.
\end{lemma}
\begin{proof}
	By proposition \ref{jet_bundle} above, $\mathscr{J}^{r*}\times \cdots \mathscr{J}^{r*}\times \mathscr{J}^{r}\times \cdots \mathscr{J}^{r}$ is a tensor-product vector bundle for which a smooth local frame exists by Lee, Proposition 10.15. By the local frame criterion for smoothness in Lee, Proposition 10.22, $F$ is smooth if and only if its component functions of $F$ with respect to any smooth local frame are smooth. The indicated statement holds because the right-hand side will be expressible as just a linear combination of its component functions, hence smooth if they are, while conversely, if there exists a choice of $X_{1,\ldots,k},\omega^{1,\ldots,j}$ making the right-hand side fail to be smooth, we can complete the section $X_1\times \cdots \times X_k\times \omega^1\times \cdots \times \omega^j$ (which does not vanish identically, else it would be smooth) to yield a local frame on a neighborhood $U_0 \subset U$ for which the first component function fails to be smooth, whence by Proposition 10.22 neither can $F$ be smooth there. But if $F$ were smooth on $U$, its restriction to any $U_0 \subset U$ would have to be smooth as well. Therefore, $F$ would not be a smooth generalized tensor field.
\end{proof}

Now, the next result will show that the concept of multilinearity over $C^\infty(M)$ of an ordinary tensor extends to generalized tensors. By lemma \ref{local_smoothness_criterion}, when we evaluate a generalized tensor $F$ on a given set of arguments $X_{1,\ldots,k},\omega^{1,\ldots,j}$ we obtain the smooth function 
$F(X_{1,\ldots,k},\omega^{1,\ldots,j})$, thus inducing in a natural way a map
\begin{equation}
	F: \mathscr{J}^{r*}(M) \times \cdots \mathscr{J}^{r*}(M) \times \mathscr{J}^{r}(M) \times \cdots \mathscr{J}^{r}(M) \rightarrow C^\infty(M).
\end{equation}Here, we say that the map $F: \mathscr{J}^{r*}(M) \times \cdots \mathscr{J}^{r*}(M) \times \mathscr{J}^{r}(M) \times \cdots \mathscr{J}^{r}(M) \rightarrow C^\infty(M)$ is multilinear over $C^\infty(M)$ if for any smooth vector or covector fields $\alpha,\beta$ and for any smooth functions $f,g \in C^\infty(M)$, we have that
\begin{equation}
	F(\ldots,f\alpha+g\beta,\ldots)=fF(\ldots,\alpha,\ldots)+gF(\ldots,\beta,\ldots).
\end{equation}
Formally, this is the same statement as what holds for an ordinary tensor field. That this property suffices to characterize tensoriality in the generalized setting is the content of the following two lemmas:

\begin{lemma}[Extension Lemma; cf. Lee, \cite{lee_smooth_manif}, Lemma 8.6 and Proposition 8.7]
	The following statements hold:
	\begin{itemize}
		\item[$(1)$] If $\varv \in J_p^r(M)$ for any point $p \in M$, there exists a globally defined $r$-vector field $V \in \mathscr{J}^r(M)$ such that $V(p)=\varv$. 
		
		\item[$(2)$] If $\omega \in J^{r*}_p(M)$ for any point $p \in M$, there exists a globally defined $r$-jet field $\Omega \in \mathscr{J}^{r*}(M)$ such that $\Omega(p)=\omega$.	
	\end{itemize}
\end{lemma}
\begin{proof}
	The proof for jets parallels that for vectors exactly; therefore, we may limit ourselves to the latter case. Let $\varv \in J_p^r(M)$ for $p \in M$. Consider coordinates $x^{1,\ldots,n}$ defined in a neighborhood $U \ni p$. Since, by hypothesis, $M$ is a differentiable manifold, it is presumed to be paracompact and Hausdorff, hence by Royden \cite{royden}, Problem 9.34b, normal. Let $V$ and $W$ be the images in $M$ of any nested balls around $p$ in a coordinate neighborhood and denote by $\lambda$ be the smooth function on $M$ supported in $W$ such that $\lambda=1$ identically on $\bar{V}$ whose existence is assured by the extension lemma, Lee 2.26. Set $E_\alpha := \lambda \partial_\alpha$. Now, in as much as the $\partial_\alpha$ form a basis, we can write $\varv = \varv^\alpha \partial_\alpha|_p$ and set $V := \varv^\alpha E_\alpha$. Then this will be the sought-for extension from $p$ to all of $M$.
\end{proof}

\begin{remark}
	The proof goes through just as well for $r$-vectors and $r$-jets as it does for 1-vectors and 1-jets, the reason being that a differential operator depends only on the germ of a function at any given point and, as far as the extension property is concerned, it is irrelevant whether one stops at first order or goes on to higher orders. In light of this, the following generalization of the corresponding result that holds for ordinary tensors will come as little surprise. No essentially new concept enters into its proof.
\end{remark}

\begin{lemma}[Generalized Tensor Characterization Lemma, cf. Lee, \cite{lee_riemannian_manif}, Lemma 2.4; \cite{lee_smooth_manif}, Lemma 12.24]\label{tensor_char}
	A map
	\begin{equation}
		F: \mathscr{J}^{r*}(M) \times \cdots \mathscr{J}^{r*}(M) \times \mathscr{J}^{r}(M) \times \cdots \mathscr{J}^{r}(M) \rightarrow C^\infty(M)
	\end{equation}
	is induced by a $\binom{k}{j}$-tensor field if and only if it is multilinear over $C^\infty(M)$.	
\end{lemma}
\begin{proof} ($\Rightarrow$) Suppose $F$ to be a smooth $\binom{k}{j}$-tensor field. Then by its definition the map $(X_1,\ldots,X_k,\omega^1,\ldots,\omega^j) \mapsto F(X_1,\ldots,X_k,\omega^1,\ldots,\omega^j)$ is trivially pointwise multilinear over $C^\infty(M)$, but the module structure of generalized mixed tensor fields ensures that the image will be a smooth global section; cf. Lee \cite{lee_smooth_manif}, Exercise 10.11.
	
	($\Leftarrow$) Suppose $F$ to have the requisite multilinearity. First, to show locality. Let $X_1$ resp. $\omega^1$ be a vector field resp. jet field that vanishes identically on a neighborhood $U$ of $p \in M$. Choose a bump function $\xi$ supported on $U$ such that $\xi(p)=0$. Then by hypothesis,
	\begin{equation}
		0=F(\xi X_1,\ldots)(p)=\xi(p)F(X_1,\ldots)(p)
	\end{equation}
	respectively,
	\begin{equation}
		0=F(X_1,\ldots,\xi\omega^1\ldots)(p)=\xi(p)F(X_1,\ldots,\omega^1,\ldots)(p)
	\end{equation}
	and similarly for $X_{2,\ldots,k}, \omega^{2,\ldots,j}$. The displayed equations imply linearity of $F$ over $C^\infty(M)$, as in the proof of the bundle homomorphism characterization in Lee, Lemma 10.29. In other words, we have shown that $F$ depends locally on its arguments; i.e., its value at any given point $p$ depends only on the values assumed by the $X_{1,\ldots,k},\omega^{1,\ldots,j}$ in an arbitrarily small neighborhood of $p$. We can, in fact, do better: actually, $F$ acts pointwise in the sense that $F(X_1,\ldots,\omega^1,\ldots)(p)$ depends only the values of its arguments at $p$, $X_{1\ldots,k}(p),\omega^{1,\ldots,j}(p)$. To see why, suppose first that $X_1(p)=0$. Now, apply the extension lemma for vector fields to obtain smooth vector fields $E_\alpha$ globally defined on $M$ such that $E_\alpha=\partial_\alpha$ on some neighborhood of $p$; respectively, . Since $X_1=f_1^\alpha E_\alpha$ in a neighborhood of $p$, we have by multilinearity,
	\begin{align}
		F(X_1,\ldots,\omega^1,\ldots)(p) &= F(f_1^\alpha E_\alpha,\ldots,\omega^1,\ldots)(p) \nonumber \\
		&= f_1^\alpha(p) F(X_1,\ldots,\omega^1,\ldots)(p) = 0.
	\end{align}
	By linearity, $F(X_1,\ldots,\omega^1,\ldots)$ depends only on the value of $X_1$ at $p$. The same argument goes through for $X_{2,\ldots,k}$ and $\omega^{1,\ldots,j}$, where in the latter case we must appeal to the extension lemma for jets in place of vector fields.
	
	Given $\varv_{1,\ldots,k} \in J_p^r(M)$ and $\omega^{1,\ldots,j} \in J_p^{r*}(M)$ and extensions $V_{1,\ldots,k}$ respectively $\Omega^{1,\ldots,j}$ to smooth, globally defined vector fields, respectively, jet fields on $M$, we can define a rough generalized tensor field by
	\begin{equation}
		F_p(\varv_1,\ldots,\varv_k,\omega^1,\ldots,\omega^j) := F(V_1,\ldots,V_k,\Omega^1,\ldots,\Omega^j)(p)
	\end{equation}
	for $p \in M$. The discussion above shows that the rough vector field is well defined, viz., independent of the choice of extensions, and lemma \ref{local_smoothness_criterion} then guarantees that the resulting generalized tensor field is smooth.	
\end{proof}	

So far, the jet connection has been defined only on higher-order vector fields. In the conventional theory, however, one extends the affine connection to act in a natural way on arbitrary tensors of mixed type, and we should like to do the same for the jet connection. The first item to resolve, then, is to specify how the covariant derivative act on 1-forms, or higher-order jets. One this has been done, a further rule governing how it behaves with respect to tensor products will suffice to define it in general, since any generalized mixed tensor reduces by definition to a tensor product of a certain number of factors of the form $\mathscr{J}^r$ together with factors of the form $\mathscr{J}^{r*}$.

Now, the scalar product of two ordinary functions could be regarded as a tensor product of 0-vector fields. If we view it this way, the natural property to stipulate for the behavior of the jet connection under tensor product would be the following:
\begin{equation}\label{Leibniz_prod_rule}
	\nabla_\alpha (F \otimes G) = \frac{\alpha!}{\alpha_1!\alpha_2!} (\nabla_{\alpha_1}F) \otimes (\nabla_{\alpha_2}G),
\end{equation}
where by $\nabla_\alpha$ we understand $\nabla_{\partial_\alpha}$.
Pause for a moment at the first order of business, to define a jet connection on higher-order jets. To this end, it will be convenient to adopt the following stratagem: if we contract the jet against a vector field, we would obtain a scalar function. Clearly, the jet connection should act on scalar functions simply as $\nabla_X f = X f$, where $X \in \mathscr{J}^r(M)$. If, then, we expand the product according to the postulated rule, it will determine the Christoffel symbols of the jet connection acting on higher-order jets in terms of the already known jet connection acting on higher-order vector fields. This is the content of the following proposition:

\begin{proposition}[cf. Lee, \cite{lee_riemannian_manif}, Exercise 4.4]\label{jet_conn_on_1_forms}
	Let $M$ be a differentiable manifold on which a jet connection acting on higher-order vector fields is defined via the Christoffel symbols $\Gamma$, as above, and let a jet connection be defined acting on higher-order jets according to Christoffel symbols $\tilde{\Gamma}$; i.e., $\nabla_\alpha d^\beta = \tilde{\Gamma}_{\gamma\alpha}^\beta d^\gamma$. Then, if the $\Gamma$ are uniformly bounded by a sufficiently small constant $\delta>0$, there exists a unique solution for the $\tilde{\Gamma}$ satisfying Leibniz' product rule.	
\end{proposition}
\begin{proof}
	Let $\omega$ be a jet field and $Y$ a vector field. By supposition, we would like to have the following relation obeyed for all choices of $\omega$ and $Y$: the covariant derivative with respect to a vector field $X$ should be given by
	\begin{equation}
		\nabla_X \omega Y = X (\omega Y) = X^\alpha \frac{\alpha!}{\alpha_1!\alpha_2!}
		\left( \nabla_{\alpha_1} \omega_\nu d^\nu \right)
		\left( \nabla_{\alpha_2} Y^\mu \partial_\mu \right) \delta_{\mu\nu}.
	\end{equation}	
	To orient ourselves, we first consider the case when $r=1$: if the mooted relation is to hold for all $X, \omega, Y$, a necessary condition would be that
	\begin{equation}\label{jet_consistency}
		\frac{\alpha!}{\alpha_1!\alpha_2!} \tilde{\Gamma}_{\gamma\alpha_1}^\nu \Gamma_{\alpha_2\mu}^\gamma = 0
	\end{equation}
	for all $\mu,\nu$. This becomes, for $r=1$,
	\begin{equation}\label{consistency_lowest_order}
		\tilde{\Gamma}_{\gamma\alpha}^\nu \delta_\mu^\gamma + \delta_\gamma^\nu \Gamma_{\alpha\mu}^\gamma = 0,
	\end{equation}
	which is solved by $\tilde{\Gamma}_{\mu\alpha}^\nu = - \Gamma_{\alpha\mu}^\nu$, which we may write for short as $\tilde{\Gamma} = - {}^t\Gamma$. Note, this statement agrees with Lee, \cite{lee_riemannian_manif}, Exercise 4.4. It is immediate how to proceed in the case of general $r>1$; simply allow the summation in equation (\ref{jet_consistency}) to run over multi-indices of all orders up to $r$. This will result in cross-terms. One readily finds the expected necessary condition to be that
	\begin{equation}\label{gamma_gamma}
		\tilde{\Gamma}_{\mu\alpha}^\nu + \frac{\alpha!}{\alpha_1!\alpha_2!}\big|_{\alpha_{1,2} \ne 0} \tilde{\Gamma}_{\mu\alpha_1}^\lambda \Gamma_{\lambda\alpha_2}^\nu + \Gamma_{\alpha\mu}^\nu = 0.
	\end{equation}
	As the term quadratic in the Christoffel symbols becomes vanishingly small compared to the leading terms in the nearly flat limit, there must exist a solution as a perturbation upon equation (\ref{consistency_lowest_order}). Said more formally, we may appeal to the implicit function theorem (Lee, \cite{lee_smooth_manif}, Theorem C.40). Let $d$ denote the dimension of the set of Christoffel symbols depending on multi-indices up to order $r$. In a neighborhood of the origin $U \subset \vvmathbb{R}^d \times \vvmathbb{R}^d$, define the function $\Phi: U \rightarrow \vvmathbb{R}^d$ by
	\begin{equation}
		\Phi(\Gamma,\tilde{\Gamma}) = \tilde{\Gamma}_{\mu\alpha}^\nu + \frac{\alpha!}{\alpha_1!\alpha_2!}\big|_{\alpha_{1,2} \ne 0} 
		\tilde{\Gamma}_{\mu\alpha_1}^\lambda \Gamma_{\lambda\alpha_2}^\nu + \Gamma_{\alpha\mu}^\nu.
	\end{equation}
	The implicit function theorem applies in a domain where the Jacobian
	\begin{equation}
		\frac{\partial\Phi}{\partial\tilde{\Gamma}_{\beta\rho}^\sigma} = \mathrm{id}~ + \frac{\alpha!}{\beta!\beta^\prime!} \Gamma_{\beta\beta^\prime}^\sigma
	\end{equation}
	is non-singular (where $\beta^\prime=\alpha-\beta$, no summation over repeated indices). But the hypothesis on uniform boundedness ensures just this. Therefore, we may state the conclusion of the implicit function theorem applied here as follows: there exist neighborhoods of the origin $V_0,W_0 \subset \vvmathbb{R}^d$ and a smooth map $F:V_0 \rightarrow W_0$ such that the level set of $\Phi(\Gamma,\tilde{\Gamma})=0$ is given by $\tilde{\Gamma}=F(\Gamma)$.
	
	Sufficiency is immediate: that with the solution $\tilde{\Gamma}$, read the equations in the other direction in order to conclude that Leibniz' product rule will be satisfied.
\end{proof}

\begin{proposition}[cf. Lee, \cite{lee_riemannian_manif}, Lemma 4.6]\label{extension_to_arb_tensors}
	Let $\nabla$ be a jet connection. Then there is a unique way of extending $\nabla$ to act on generalized tensors of mixed type, all the while satisfying 
	\begin{itemize}
		\item[$(1)$] On $J^r(M)$, $\nabla$ agrees with the given connection.
		
		\item[$(2)$] On $J^0(M) = C^\infty(M)$, $\nabla$ goes over into ordinary differentiation of functions, or $\nabla_X f = Xf$.
		
		\item[$(3)$] $\nabla$ obeys Leibniz' product rule as in Equation (\ref{Leibniz_prod_rule}):
		\begin{equation}
			\nabla_\alpha (F \otimes G) = \frac{\alpha!}{\alpha_1!\alpha_2!} (\nabla_{\alpha_1}F) \otimes (\nabla_{\alpha_2}G).
		\end{equation}
		
		\item[$(4)$] $\nabla$ commutes with any partial trace in the sense that $\nabla_X ~\mathrm{tr}~F = \mathrm{tr}~ \nabla_X F$.
	\end{itemize}
\end{proposition}
\begin{proof}	
	Linearity over $\vvmathbb{R}$ trivially implies $\nabla_X 0 = 0$. Since equation (\ref{Leibniz_prod_rule}) depends on $\nabla$ only, it defines the operation of the jet connection on product vector bundles uniquely, if it exists. But equation (\ref{Leibniz_prod_rule}) is clearly multilinear over $C^\infty(M)$ in $X$ and linear over $\vvmathbb{R}$ in $F$ and $G$. Now, any section in the tensor product $F \otimes G$  can be written as a sum over decomposable products of elements in $F$ and $G$. Therefore, $\nabla_X$ acting on $F \otimes G$ exists and is uniquely defined.
	
	As for the property of $\nabla$ with respect to partial trace, see the proof in lemma \ref{coord_expr_for_conn} below. The action of $\nabla$ on covectors has to be the same as in proposition \ref{jet_conn_on_1_forms} due to uniqueness and (3) and (4) applied to the contraction of a covector against a vector field.
\end{proof}

The concept of the total covariant derivative to be introduced in the following proposition will prove instrumental to streamlining the proof of a number of results, in this and in following sections:

\begin{proposition}[cf. Lee, \cite{lee_riemannian_manif}, Lemma 4.7]\label{total_deriv}
	If $\nabla$ is a jet connection on a differentiable manifold $M$ and $F \in \mathscr{K}^{r,k}_j$, the map $\nabla F: \mathscr{J}^{r*} \times \cdots \times 
	\mathscr{J}^{r*} \times \mathscr{J}^r \times \cdots \mathscr{J}^r \rightarrow
	C^\infty(M)$ given by
	\begin{equation}
		(\nabla F)(\omega^1,\ldots,\omega^k,Y_1,\ldots,Y_j,X) := 
		(\nabla_X F)(\omega^1,\ldots,\omega^k,Y_1,\ldots,Y_j)
	\end{equation}
	defines a generalized $\binom{k}{j+1}$-tensor field.
\end{proposition}
\begin{proof}
	As in Lee, the proof follows immediately from the (here) generalized tensor characterization lemma, in that $\nabla_X F$ is a generalized tensor field and hence multilinear over $C^\infty(M)$ in its $(k+j)$ arguments, while it is also linear over $C^\infty(M)$ in $X$ by virtue of the definition of a jet connection.	
\end{proof}	

The following coordinate expression for total covariant derivative will be useful to derive:

\begin{lemma}[cf. Lee, \cite{lee_riemannian_manif}, Lemma 4.8]\label{coord_expr_for_conn}
	Let $\nabla$ be a jet connection. The components of the total covariant derivative of a $\binom{p}{q}$-tensor field $F$ with respect to a coordinate system are given by
	\begin{equation}
		F^{\mu_1\cdots\mu_p}_{\mu_1\nu_2\cdots\nu_q;\alpha}= 
		\frac{\alpha!}{\lambda!\alpha_1!\cdots\alpha_p!\beta_1!\cdots\beta_q}
		\left( \partial_\lambda F^{\gamma_1\cdots\gamma_p}_{\delta_1\cdots\delta_q} \right)
		\Gamma^{\mu_1}_{\alpha_1\gamma_1}\cdots\Gamma^{\mu_p}_{\alpha_p\gamma_p}
		\tilde{\Gamma}^{\nu_1}_{\beta_1\delta_1}\cdots\tilde{\Gamma}^{\nu_p}_{\beta_p\delta_p}.
	\end{equation}
\end{lemma}
\begin{proof}
	Using the rule for how the covariant derivative distributes over tensor products, we perform the expansion
	\begin{align}
		\nabla_\alpha \left( F^{\mu_1\cdots\mu_p}_{\nu_1\cdots\nu_q} \partial_{\mu_1} \otimes
		\cdots \otimes \partial_{\mu_p} \otimes d^{\nu_1} \otimes \cdots d^{\nu_q} 
		\right) &=
		\frac{\alpha!}{\lambda!\alpha_1!\cdots\alpha_p!\beta_1!\cdots\beta_q}
		\left( \partial_\lambda F^{\mu_1\cdots\mu_p}_{\nu_1\cdots\nu_q} \right) 
		\times \nonumber \\
		&\nabla_{\alpha_1}\partial_{\mu_1}\otimes\cdots\otimes\nabla_{\alpha_p}
		\partial_{\mu_p} \otimes
		\nabla_{\beta_1}d^{\nu_1}\otimes\cdots \otimes \nabla_{\beta_q}d^{\nu_q},
	\end{align}
	from which, after applying the definitions of the jet connection acting on basis elements, we may read off the expression,
	\begin{equation}\label{coord_expression}
		F^{\mu_1\cdots\mu_p}_{\nu_1\cdots\nu_q;\alpha}=
		\frac{\alpha!}{\lambda!\alpha_1!\cdots\alpha_p!\beta_1!\cdots\beta_q}
		\left( \partial_\lambda F^{\gamma_1\cdots\gamma_p}_{\delta_1\cdots\delta_q} \right)
		\Gamma^{\mu_1}_{\alpha_1\gamma_1}\cdots\Gamma^{\mu_p}_{\alpha_p\gamma_p}
		\tilde{\Gamma}^{\nu_1}_{\beta_1\delta_1}\cdots\tilde{\Gamma}^{\nu_p}_{\beta_p\delta_p}.
	\end{equation}
	From the right-hand side of equation (\ref{coord_expression}), it is clear that by linearity over $\vvmathbb{R}$ the covariant derivative commutes with permutations over indices; that is, if $\sigma \in S_p$ and $\rho \in S_q$, we have $\left(F^{\sigma(\mu_1)\cdots\sigma(\mu_p)}_{\rho(\nu_1)\cdots\rho(\nu_q)}
	\right)_{;\alpha} =
	F^{\mu_1\cdots\mu_p}_{\nu_1\cdots\nu_q;\alpha}$. Hence, we need consider partial traces over the first upper and the first lower indices only. In particular, if we trace over the first upper and the first lower index, we get,
	\begin{equation}
		F^{\mu_1\cdots\mu_p}_{\mu_1\nu_2\cdots\nu_q;\alpha}= 
		\frac{\alpha!}{\lambda!\alpha_1!\cdots\alpha_p!\beta_1!\cdots\beta_q}
		\left( \partial_\lambda F^{\gamma_1\cdots\gamma_p}_{\delta_1\cdots\delta_q} \right)
		\Gamma^{\mu_1}_{\alpha_1\gamma_1}\cdots\Gamma^{\mu_p}_{\alpha_p\gamma_p}
		\tilde{\Gamma}^{\mu_1}_{\beta_1\delta_1}\tilde{\Gamma}^{\nu_2}_{\beta_2\delta_2}\cdots\tilde{\Gamma}^{\nu_p}_{\beta_p\delta_p};
	\end{equation}
	i.e., $\nabla~\mathrm{tr}~F = \mathrm{tr}~\nabla F$.
\end{proof}

\begin{remark}
	Recall in equation (\ref{coord_expression}) our convention that $\nabla_0 = \mathrm{id}$. In light of proposition \ref{jet_conn_on_1_forms}, this formula reduces to the usual one in the case of first-order infinitesimals only ($r=1$). 
\end{remark}

Let us round out this section with a couple simple applications. Following Lee, we could for the purpose of illustration define a covariant Hessian
of a smooth function $u$ on $M$ as the following generalized 2-tensor field: $\nabla^2 u := \nabla \left( \nabla u \right)$. In order to obtain a convenient expression for the covariant Hessian, we introduce two concepts. First, we may transfer the action of the jet connection $\nabla$ on jets to vectors by identifying $d^\alpha$ with $\partial_\alpha$, in a given coordinate system. Denote the resultant jet connection acting on vectors by $\tilde{\Gamma}$. The construction is not at all canonical, but we are free to define (locally) a jet connection coordinate-wise however we please as long as it serve as an intermediary on the way to a coordinate-free result (cf. proposition \ref{torsion_tensor_properties} below). 

Second, we advert to a subtle issue that arises as soon as one goes beyond first order in the infinitesimals. In the definition of an ordinary connection, the covariant derivative $\nabla_X Y$ of a vector field $X$ acting on the vector field $Y$, the differentiation associated with $X$ operates on the components of $Y$ but does not go past $Y$ to operate directly on the scalar function to which this covariant derivative is applied; i.e., $(\nabla_X Y)u = (XY^\mu)u_{,\mu}+\Gamma^\lambda_{\nu\mu}X^\nu Y^\mu u_{,\lambda}$. As long as vector fields are applied to scalar functions only no problem arises. A vector field regarded as a linear partial differential operator, however, may first be composed with another vector field before being applied to a scalar function. In this case, an inconsistency arises in the higher derivatives once one proceeds beyond flat space. Logically, all derivatives ought to be covariant in a curved space. In order to meet this stipulation, we introduce a new concept, that of 
covariant composition of vector fields (viewed as linear operators), to be denoted $X \circ Y$ in place of $XY$, the ordinary composition of operators. It will be defined inductively if we make the assignments
\begin{align}
	\nabla_\alpha \circ u &:= \partial_\alpha u = u_{,\alpha}; \\
	\nabla_\alpha \circ X &:= \frac{\alpha!}{\alpha_1!\alpha_2!} X^\nu_{,\alpha_1} \Gamma^\mu_{\alpha_2\nu} \nabla_\mu \circ 
\end{align}
and
\begin{equation}\label{cov_comp_def}
	Y \circ X \circ :=  Y^\mu \frac{\mu!}{\mu_1!\mu_2!} \left( \nabla_{\mu_1} X \right) \left( \nabla_{\mu_2} \circ ~~~ \right).
\end{equation}

\begin{proposition}[cf. Lee, \cite{lee_riemannian_manif}, Exercise 4.5]\label{cov_hess}The covariant Hessian satisfies the following two identities:
	\begin{itemize}
		\item[$(1)$] $\nabla^2 u (X,Y) = Y(Xu) + (\tilde{\nabla}_Y X) u = Y(Xu) - (\nabla_Y X) u - \delta(X,Y) u$;
		
		\item[$(2)$] $\nabla^2 u (X,Y) = Y(Xu) + X(Yu) - Y \circ X \circ u + \varepsilon(X,Y) u$.
	\end{itemize}
	where the defects $\delta(X,Y)$ and $\varepsilon(X,Y)$ will be defined in the course of proof.
\end{proposition}
\begin{proof}
	For any $u \in C^\infty(M)$ and $X, Y \in \mathscr{J}^r(M)$, we compute $(\nabla u)(X) = \nabla_X u = X u$ so $\nabla u$ is just the 1-form given by $du = u_{,\alpha} d^\alpha$. Proceeding onwards:
	\begin{align}\label{cov_hess_expansion}
		\nabla^2 u(X,Y) &= \nabla \left( \nabla u \right) (X,Y) = \left( \nabla_Y \nabla u \right)(X) = Y^\mu \frac{\mu!}{\mu_1!\mu_2!} X^\nu u_{,\mu_1+\alpha} \tilde{\Gamma}^\alpha_{\nu\mu_2} =
		Y^\mu X^\nu \left( \nabla_\mu du \right)_\nu \\
		Y \left( X u \right) &=  Y^\mu \frac{\mu!}{\mu_1!\mu_2!} X^\nu_{,\mu_1} u_{,\mu_2+\nu} = Y^\mu \partial_\mu \left( X u \right) = Y^\mu \nabla_\mu \left( X u \right) \nonumber \\
		&= Y^\mu \nabla_\mu \mathrm{tr}~ X \otimes du \nonumber \\
		&= Y^\mu \frac{\mu!}{\mu_1!\mu_2!} \mathrm{tr}~ \nabla_{\mu_1} X \otimes \nabla_{\mu_2} du \nonumber \\
		&= \left( Y^\mu \nabla_\mu X \right) u +  Y^\mu \frac{\mu!}{\mu_1!\mu_2!}\bigg|_{\mu_{1,2}\ne 0} \mathrm{tr}~ \nabla_{\mu_1} X \otimes \nabla_{\mu_2} du + Y^\mu X^\nu \left( \nabla_\mu du \right)_\nu. 
	\end{align}
	But the final term on the right-hand side in equation (\ref{cov_hess_expansion}) is just the sought-for $\nabla^2 u(X,Y)$. If we define a vector field by
	\begin{equation}
		\delta(X,Y) u := Y^\mu \frac{\mu!}{\mu_1!\mu_2!}\bigg|_{\mu_{1,2}\ne 0} \mathrm{tr}~ \nabla_{\mu_1} X \otimes \nabla_{\mu_2} du,
	\end{equation}
	we immediately obtain the second equation in (1). As for the other relation in (1), expand the right-hand side of equation (\ref{cov_hess_expansion}) in Christoffel symbols with respect to any coordinate system:
	\begin{equation}
		Y(Xu) = Y^\mu \frac{\mu!}{\mu_1!\mu_2!}\bigg|_{\mu_1\ne 0} \frac{\mu_1!}{\lambda_1!\lambda_2!} \frac{\mu_2!}{\rho_1!\rho_2!} X^\nu_{,\lambda_1} \Gamma^\beta_{\lambda_2\nu} \tilde{\Gamma}^\alpha_{\beta\rho_2} u_{,\rho_1+\alpha} + \nabla^2 u(X,Y)
	\end{equation}
	and recognize that then in light of equation (\ref{gamma_gamma}) the Christoffel symbols may be regrouped as simply $-{}^t \tilde{\Gamma}$, which we now interpret as a connection acting to its left on the vector $X$ rather than to its right on $du$. The result is just, as claimed:
	\begin{equation}
		Y(Xu) = - \left( \tilde{\nabla}_Y X \right) u + \nabla^2 u(X,Y).
	\end{equation}
	
	Now, the covariant composition of $Y$ with $X$ is given by
	\begin{equation}\label{cov_comp_expansion}
		Y \circ X \circ =  Y^\mu \frac{\mu!}{\mu_1!\mu_2!} \left( \nabla_{\mu_1} X \right) \left( \nabla_{\mu_2} \circ ~~~ \right);
	\end{equation}
	so when applied to a scalar function $u$ the last term on the right-hand side becomes just $X(Yu)$, in view of the fact that the covariant derivative of a scalar coincides with the ordinary derivative. Upon comparison of the remaining terms in equation (\ref{cov_comp_expansion}) with their corresponding members in equation (\ref{cov_hess_expansion}), we obtain, after re-expressing them in terms of a trace, the desired result (2) with the definition
	\begin{equation}
		\varepsilon(X,Y) u :=  Y^\mu \frac{\mu!}{\mu_1!\mu_2!}\bigg|_{\mu_{1,2}\ne 0} \mathrm{tr}~ \nabla_{\mu_1} X \otimes \left(
		d \nabla_{\mu_2} u - \nabla_{\mu_2} du \right).
	\end{equation}
	Thus, (1) and (2) hold with the indicated definitions of the defects $\delta(X,Y)$ and $\varepsilon(X,Y)$.
\end{proof}

\begin{remark}
	For infinitesimals of first order only, we have $\tilde{\Gamma} = - {}^t \Gamma$ (cross-terms do not enter). Hence, up to the defects the formula for the covariant Hessian at which we arrive looks formally the same as what appears in Lee, only we understand here that it involves the jet connection and higher-order vector fields.
\end{remark}

\begin{lemma}\label{cov_hess_symm}
	Let $M$ be a manifold equipped with a jet connection $\nabla$. Then the following two properties hold:
	\begin{itemize}
		\item[$(1)$] The covariant Hessian is symmetric if and only if the Christoffel symbols of the jet connection acting on jets with respect to any coordinate frame are symmetric in the sense that 
		$\tilde{\Gamma}^\gamma_{\alpha\beta}=\tilde{\Gamma}^\gamma_{\beta\alpha}$. 	
		
		\item[$(2)$] $\tilde{\Gamma}$ is symmetric if and only if $\Gamma$ is symmetric.
	\end{itemize}
\end{lemma}
\begin{proof}
	$(1)$ Observe that
	\begin{equation}
		\nabla^2 u(X,Y) = Y^\mu X^\nu \nabla_\mu \left( \nabla_\nu u \right) 
		= Y^\mu X^\nu \frac{\mu!}{\mu_1!\mu_2!} u_{,\mu_1+\nu} \tilde{\Gamma}^\nu_{\lambda\mu_2} d^\lambda.
	\end{equation}
	It will be convenient to employ index-free notation so that the right-hand side of the above equation comprises terms of the form
	\begin{equation}
		\mathrm{tr}~\mathrm{tr}~\mathrm{tr}~ Y \otimes \tilde{\Gamma} \otimes du \otimes X.
	\end{equation}
	Now, by the assumed symmetry of $\tilde{\Gamma}$, this becomes as follows:
	\begin{align}
		\mathrm{tr}~\mathrm{tr}~\mathrm{tr}~ Y \otimes {}^t\tilde{\Gamma} \otimes du 
		\otimes X &= \mathrm{tr}~\mathrm{tr}~\mathrm{tr}~ \left( \tilde{\Gamma} Y \right) \otimes du \otimes X \nonumber \\
		&= \mathrm{tr}~\mathrm{tr}~\mathrm{tr}~ X \otimes du \otimes \left( \tilde{\Gamma} Y \right) \nonumber \\
		&= \mathrm{tr}~\mathrm{tr}~\mathrm{tr}~ X \otimes \tilde{\Gamma} \otimes du \otimes Y.
	\end{align}
	Therefore, after collecting terms we end up with $\nabla^2 u(X,Y) = \nabla^2 u(Y,X)$. Heading in the other direction, one reads off ${}^t\tilde{\Gamma}=\tilde{\Gamma}$ immediately from $\nabla^2 u(Y,X) = 
	\nabla^2 u(X,Y)$.
	
	$(2)$ We can suppress the combinatorial factors by temporarily understanding the Christoffel symbols to be normalized as $(\gamma!/\alpha!\beta!)\Gamma^\gamma_{\alpha\beta}$ and similarly for $\tilde{\Gamma}$. Write the defining equation (\ref{gamma_gamma}) in an index-free manner as follows, where $\mathrm{tr}_{14}$ denotes a partial trace over the first and fourth indices:
	\begin{equation}\label{gamma_gamma_index_free}
		\tilde{\Gamma} + \mathrm{tr}_{14} \tilde{\Gamma} \otimes \Gamma + \Gamma = 0.
	\end{equation}
	Suppose now the jet connection to be symmetric: ${}^{t_{23}}\Gamma=\Gamma$. Then substitute into the transpose with respect to the second and third indices of equation (\ref{gamma_gamma_index_free}) to yield,
	\begin{equation}
		{}^{t_{23}}\tilde{\Gamma} + {}^{t_{23}}\mathrm{tr}_{14} \tilde{\Gamma} \otimes \Gamma + \Gamma = 0.
	\end{equation}
	But since the operations apply to distinct subgroups of indices, we must have that
	${}^{t_{23}}\mathrm{tr}_{14} \tilde{\Gamma} \otimes \Gamma = \mathrm{tr}_{14} \left( {}^{t_{23}}\tilde{\Gamma} \right) \otimes \Gamma$ and so
	\begin{equation}
		{}^{t_{23}}\tilde{\Gamma} + \mathrm{tr}_{14} \left( {}^{t_{23}}\tilde{\Gamma} \right) \otimes \Gamma + \Gamma = 0;
	\end{equation}
	by uniqueness in proposition \ref{jet_conn_on_1_forms} we find that ${}^{t_{23}}\tilde{\Gamma}=\tilde{\Gamma}$; i.e., $\tilde{\Gamma}$ is symmetric as well. The other direction is immediate. 
\end{proof}

\begin{definition}\label{torsion_def}
	Let $\nabla$ be a jet connection on $M$. We define a map $\tau: \mathscr{J}^\infty(M) \times \mathscr{J}^\infty(M) \rightarrow \mathscr{J}^\infty(M)$ by 
	\begin{equation}\label{torsion_def_equ}
		\tau(X,Y) := \nabla_X Y - \nabla_Y X - [X,Y] - \delta(X,Y) + \delta(Y,X).
	\end{equation}
\end{definition}

\begin{proposition}[Cf. Lee, \cite{lee_riemannian_manif}, Problem 4.2]\label{torsion_tensor_properties}
	Let $M$ be a manifold equipped with a jet connection $\nabla$ and the map $\tau$ defined by equation (\ref{torsion_def_equ}). Then 
	\begin{itemize}
		\item[$(1)$] $\tau$ is a $\binom{2}{1}$-tensor field, which we shall call the torsion tensor of $\nabla$; 
		
		\item[$(2)$] we shall call $\nabla$ symmetric if its torsion tensor vanishes identically, which holds if and only if its Christoffel symbols with respect to any coordinate frame satisfy $\Gamma^\gamma_{\alpha\beta}=\Gamma^\gamma_{\beta\alpha}$. 	
	\end{itemize}
\end{proposition}
\begin{proof}
	$(1)$ In view of lemma \ref{tensor_char} and the anti-symmetry of $\tau$ under interchange of $X$ and $Y$, it suffices to show that $\tau(fX,Y)=f\tau(X,Y)$ for all smooth functions $f \in C^\infty(M)$. Thus, from its definition and proposition \ref{cov_hess} we have for any $u \in C^\infty(M)$,
	\begin{align}\label{torsion_cov_hess}
		\tau(X,Y)u &= \nabla_X Y - \nabla_Y X - XY + YX - \delta(X,Y) + \delta(Y,X) \nonumber \\
		&= \nabla^2 u (X,Y) - \nabla^2 u (Y,X).
	\end{align}
	But we know by proposition \ref{total_deriv} that the covariant Hessian is tensorial and hence multilinear in both of its arguments. Therefore, the torsion must be multilinear in $X$ as well, for
	\begin{align}
		\tau(fX,Y)u &= \nabla^2 u (fX,Y) - \nabla^2 u (Y,fX) \nonumber \\
		&= f \nabla^2 u (X,Y) - f \nabla^2 u (Y,X) \nonumber \\
		&= f\tau(X,Y)u.
	\end{align}
	
	$(2)$ By lemma \ref{cov_hess_symm}, $\nabla$ is symmetric if and only if $\tilde{\nabla}$ is. Again by the lemma, the latter statement is equivalent to symmetry of the covariant Hessian. But as we see above in (1), symmetry of the covariant Hessian implies vanishing of torsion and vice versa.
\end{proof}

\begin{remark}
	Observe that in the conventional case of infinitesimals of first order only, the covariant composition law simplifies to
	\begin{equation}
		Y \circ X \circ = \nabla_Y X + XY,
	\end{equation}
	which when inserted into definition \ref{torsion_def} produces
	\begin{equation}
		\tau(X,Y) = X \circ Y \circ - Y \circ X \circ,
	\end{equation}
	what will serve as an alternate and perhaps more perspicuous expression for the torsion, now recognizable as a direct measure of the non-commutativity of covariant composition. 
\end{remark}

This statement can be generalized to the case of higher order infinitesimals if we read off from proposition \ref{cov_hess} that $Y \circ X \circ - \varepsilon(X,Y) = XY + \nabla_Y X + \delta(X,Y)$; then substituting into definition \ref{torsion_def} we obtain the following expression:
\begin{equation}\label{torsion_cov_comp}
	\tau(X,Y) = X \circ Y \circ - Y \circ X \circ + \varepsilon(X,Y) - \varepsilon(Y,X).
\end{equation}
We can say a little more if the jet connection be taken to be symmetric, as forms the content of the following proposition:

\begin{proposition}\label{defect_symm}
	When the jet connection $\nabla$ is symmetric, the defects $\varepsilon(X,Y)$ and $\delta(X,Y)$ are symmetric and the following two identities hold:
	\begin{itemize}
		\item[$(1)$] $\tau(X,Y) = \nabla_X Y - \nabla_Y X - [X,Y] = 0$;
		
		\item[$(2)$] $\tau(X,Y) = X \circ Y \circ - Y \circ X \circ = 0$.
	\end{itemize}
\end{proposition}
\begin{proof}
	From the assumed symmetry and Propositions \ref{cov_hess_symm}, \ref{torsion_tensor_properties}, 
	\begin{equation}
		\tilde{\tau}(X,Y) = \tilde{\nabla}^2 u(X,Y) - \tilde{\nabla}^2 u(Y,X) = 0;
	\end{equation}
	while by its definition, the vanishing of torsion implies that
	\begin{align}
		0 = \tau(X,Y) &= \nabla_X Y - \nabla_Y X - XY + YX + \delta(Y,X) - \delta(X,Y) \nonumber \\
		&= \tilde{\nabla}^2 u(Y,X) - \tilde{\nabla}^2 u(X,Y) + \delta(Y,X) - \delta(X,Y) \nonumber \\
		&= \delta(Y,X) - \delta(X,Y).
	\end{align}
	This shows symmetry of the defect $\delta(X,Y)$. 
	
	As for the second defect, we start by computing in the case $X=\partial_\alpha$, $Y=\partial_\beta$.  From equation (\ref{torsion_cov_comp}), we have
	\begin{equation}
		\tau(\partial_\alpha,\partial_\beta) = \partial_\alpha \circ \partial_\beta \circ - \partial_\beta \circ \partial_\alpha \circ + \varepsilon(\partial_\alpha,\partial_\beta) - \varepsilon(\partial_\beta,\partial_\alpha) = 0,
	\end{equation}
	whence symmetry of $\varepsilon(\partial_\alpha,\partial_\beta)$ is equivalent to that of covariant composition of coordinate basis vectors, i.e., of $\partial_\alpha \circ \partial_\beta \circ$. But evidently this will involve a sum of terms in $\nabla_\alpha \partial_\beta = \Gamma^\gamma_{\alpha\beta}\partial_\gamma = \Gamma^\gamma_{\beta\alpha}\partial_\gamma = \nabla_\beta \partial_\alpha$ whenever $\nabla$ is symmetric. Therefore, $\varepsilon(\partial_\alpha,\partial_\beta)=\varepsilon(\partial_\beta,\partial_\alpha)$.
	
	That the defect is linear over $C^\infty(M)$ in its second argument, or $\varepsilon(X,gY)=g\varepsilon(X,Y)$, follows trivially from its definition. Meanwhile, that $Y \circ fX - \nabla_Y fX = f \left( Y \circ X - \nabla_Y X \right)$ follows from the lemma immediately following this proof. But for a symmetric jet connection, $X \circ Y \circ - Y \circ X \circ = X \circ Y - Y \circ X$, as can be seen by writing out the terms in which $X$ acts past $Y$, respectively $Y$ acts past $X$ to the right, leaving us with terms involving
	\begin{equation}
		\partial_\alpha \circ \partial_\beta \circ - \partial_\beta \circ \partial_\alpha \circ = \tau(\partial_\alpha,\partial_\beta) - \varepsilon(\partial_\alpha,\partial_\beta) + \varepsilon(\partial_\beta,\partial_\alpha) = 0.
	\end{equation}
	Hence,
	the second defect is almost linear over $C^\infty(M)$ in its first argument as well, or, which is to say, $\varepsilon(fX,Y)=f\varepsilon(X,Y)$ modulo terms in $\delta$. But the latter appear only in the combination $\delta(X,Y)-\delta(Y,X)$, which as we have just seen above vanishes when $\nabla$ is symmetric. Therefore, we can write,
	\begin{equation}
		\varepsilon(fX,gY)-\varepsilon(gY,fX) = fg \left( \varepsilon(X,Y)-\varepsilon(Y,X) \right).
	\end{equation}
	Given this, then symmetry of the second defect follows right away from what we have already established:
	\begin{equation}
		\varepsilon(X,Y)-\varepsilon(Y,X) = X^\alpha Y^\beta \left( 
		\varepsilon(\partial_\alpha,\partial_\beta) - \varepsilon(\partial_\beta,\partial_\alpha) \right) = 0.
	\end{equation}
	The two identities are immediate consequences of previously established relations together with symmetry of the defects.
\end{proof}

\begin{lemma}\label{cov_comp_vs_cov_deriv}
	Let $\nabla$ be a symmetric jet connection.
	The difference between covariant composition and the covariant derivative is multilinear in both of its arguments, or in other words,
	\begin{equation}
		Y \circ fX \circ - \nabla_Y fX = f \left( Y \circ X - \nabla_Y X \right);
	\end{equation}
	(linearity over $C^\infty(M)$ in $Y$ is trivial).
\end{lemma}
\begin{proof}
	Isolate the terms involving derivatives of $f$ as follows:
	\begin{align}
		\nabla_Y fX &:= f \nabla_Y X + \left( \nabla_Y fX \right)^\prime \\
		Y \circ fX &:= f Y \circ X + \left( Y \circ fX \right)^\prime.
	\end{align}
	Now, from Proposition \ref{cov_hess}, we may first compute
	\begin{align}
		\left( \nabla_Y fX \right) u &= Y(fXu) - \tilde{\nabla}^2 u(fX,Y) \nonumber \\
		&= f Y(Xu) - f \tilde{\nabla}^2 u(X,Y) + \left( Y f X \right)^\prime \nonumber \\
		&= f \left( \nabla_Y X \right) u + \left( Y f X \right)^\prime. 
	\end{align}
	Therefore, $\left(\nabla_Y fX \right)^\prime = \left( Y f X \right)^\prime$. Again, by proposition \ref{cov_hess}, we have
	\begin{align}
		\left( Y \circ fX \right) u &= - \nabla^2 u(fX,Y) + Y(fXu) + fX(Yu) + \varepsilon(fX,Y) \nonumber \\
		&= - f \nabla^2 u(X,Y) + fY(Xu) + fX(Yu) + f \varepsilon(X,Y) + \left( YfXu \right)^\prime \nonumber \\
		&= f \left( Y \circ X \right) u + \left( YfXu \right)^\prime,
	\end{align}
	where we invoke symmetry of the defect (proposition \ref{defect_symm}) along with its manifest linearity over $C^\infty(M)$ in its second argument in order to write, $\varepsilon(fX,Y)=\varepsilon(Y,fX)=f\varepsilon(Y,X)=f\varepsilon(X,Y)$.
	Therefore, we have shown that $\left( Y \circ fX \right)^\prime = \left( YfX \right)^\prime = \left( \nabla_Y fX \right)^\prime$. The claimed result is immediate, recognizing that the covariant composition need be applied here only to scalar functions. \end{proof}

\begin{remark}
	The presence of the defects $\delta(X,Y)$, respectively $\varepsilon(X,Y)$, in the case of a non-symmetric jet connection comes as something of a surprise. The torsion, as defined (however inappositely), must correspond to a natural concept, in light of the twin advantages of its tensoriality and the condition that its vanishing is equivalent to symmetry of the jet connection. Thus, it would offer a potentially productive avenue for further research to ponder whether the defects can be supplied with an intrinsic geometrical interpretation. We have, so far, no good grounds upon which to base any intuition as to how the covariant derivative should behave at higher than first order in the infinitesimals, especially in the absence of symmetry.
\end{remark}

%% file: section_4.tex
\section{Prologomenon to a Future Riemannian Geometry in the Generalized Sense}\label{chapter_4}

In his inaugural lecture of 1854 \cite{riemann_habilitation_lecture}, Riemann invokes the assumption that the line element $ds$ of an $n$-dimensional manifold could be expressed as the square root of a positive definite quadratic form in the differentials $dx$ of the coordinate functions; i.e., that quantities of higher order in the differentials could be neglected. He feels that, while this assumption leads to a simplification in the analytical formulation of geometry, its applicability to the infinitely small in the space of our experience is an empirical question, which was to be answered in the affirmative for the Newtonian physics of his day but which might have to be revisited if a simpler explanation of appearances were to become available.

\subsection{Differential Geometry to Higher Order in the Infinitesimals}

Riemann does not merely found the field of differential geometry that has come to bear his name but outlines a blueprint for how, in principle, we are to think about mathematical space in general and its relation to the world of experience, as encompassed by the natural sciences.  Riemann's original motivation for introducing his metric is to render it possible to measure distances in a manner independent of any choice of coordinates. In the presence of infinitesimals, the line element $ds$, which is supposed to measure the increment in length when one sets out in an arbitrary direction, has now to be promoted to include a possible dependence on higher jets, for we are supplementing the concept of direction to embrace higher tangent vectors. Therefore our guiding theme will be to define a metric involving possible dependence on higher infinitesimals and to seek out a \textit{generalized} tensor corresponding to the known Riemannian curvature which now may take higher tangents among its arguments.

Many would suppose this impossible. Hence, after sketching the epistemological basis of our own theory, we shall take it upon ourselves to refute the arguments going back to Helmholtz and to Weyl purporting to show that the form of the metric tensor has to be restricted to be homogeneous and quadratic in the coordinate differentials. It proves necessary for us to go into some detail, in as much as these authors' contention would seem to have impeded the formulation of an infinitesimal geometry to higher order for well over a century.

\subsection{On the Metrization of Space}
\label{metrization_of_space}

In Riemann's natural philosophy, reconstruction of experience in thought requires metrization, for the metric concept is essential to the quantification of motion. But, once equipped with the notions of fluent and fluxion, it becomes possible in the ordinary calculus to measure the velocity as an instantaneous derivative and thereby effectively to quantify motion for a wide range of problems. Now, to follow along the lines of Riemann's precedent, we wish to develop the concept of a mathematical space with a view towards its application in physics.

\subsubsection{Preparatory Remarks towards a Notion of Measure for Higher-Order Infinitesimals}

Geometric considerations motivate the positing of a metric. If there were no metric by which to measure displacements in space, the formulae expressing physical laws would become unnecessarily complicated; the situation is analogous to what would hold if there were no general covariance, either. Thus, from a Machian instrumentalistic perspective, the introduction of the concept of a metric is very well justified, since the aim of theoretical science is to reduce the phenomena of nature to a simple description, indeed, as simple as possible (anyone who, further, adheres to a principle of reality will thus be inclined to regard the metric not as a mere device by which to aid our reasoning, but as existing in the world). But, seeing as these statements flow from quite general propositions about the nature of science, we have every reason to expect that metrical concepts will continue to have relevance to geometry when infinitesimals of higher order are allowed to play a role.

Accordingly, we seek an analogue of the well-known metric of first-order geometry. To start with, let us ask what is a line element $ds$? The fundamental concept is that of an assignable measure to every curve in space, its length $s$. Then the line element $ds$ is supposed to be an infinitesimal from which the length $s$ of every curve can be obtained by integration. Since we want to have a measure of the length of every possible curve, the line element has to be able to measure movement in an arbitrary direction away from any given point in space. The idea then, is that we measure the length of a given curve by setting out in a given direction corresponding to its initial tangent and keep changing direction as we trace out the full interval of the curve, integrating the line element $ds$ all along as we go. But, from our present perspective, the concept of direction can include, in principle, higher-order deviations. Hence, we are faced with the problem of figuring how consistently to incorporate them into the formalism.

Thus, instead of a 1d trajectory by itself, we ought to deal with the jet stream which occupies a tubular neighborhood around it. The metric, in this context, measures the volume of the jet stream in the limit as its cross-section shrinks and tends to zero (that is, the length of a 1d trajectory is to be thought of as the ratio of its volume to its cross-sectional area in the limit as the latter becomes vanishingly small). As we shall remark upon below in {\S}\ref{chapter_5}, the procedure here being sketched coincides with what is ordinarily done in as much as the very definition of the concept of a metric along the trajectory induced by pulling back the Riemannian metric in the ambient space under the submersion implicitly invokes just such a limiting ratio. But something new can be expected to arise in the situation presently under consideration, when differentials of higher than first order are supposed possibly to make a material difference. For at every stage before completion of the limiting procedure, the jet stream will fill a volume of finite, non-vanishing measure. Thus, we envision volume elements that somehow (in a sense to be rendered precise in the next section) register an arc segment rather than merely a line segment. Therefore, it is intuitively clear that the metric, whatever it be, must include differentials of higher than first order alone.

Now Riemann assumes, provisionally, that the line element should be quadratic in the coordinate differentials of first order \cite{riemann_habilitation_lecture}; he does later entertain the possibility of an expression (homogeneous of) fourth order in differentials, but dismisses it as unnecessarily complicated.\footnote{See also Lipschitz \cite{lipschitz}, who around the same time proposes a line element of the form $ds^p = f(dq_1,...,dq_n)$, where $f$ is homogeneous of degree $p$.} In any case, he presumes the expression for the line element has to be homogeneous. Since we want to dispute this contention, let us seek to motivate why it isn't really the case. Riemann's reasoning behind his contention---it is justified to suppose---must be the same as what everyone has concluded ever since the introduction of infinitesimal analysis by Cavalieri in the early seventeenth century, namely, that, in the limit as the differential $dx$ becomes infinitely small, any higher powers of $dx$ must become negligible in comparison with $dx$ itself. Against this reasoning, however, we wish to note that a (finite) differential ought to be construed as arising from the integration of a corresponding differential operator of the same order. Yet, it is by no means the case that the differential equations relevant to mathematical physics must always reduce to an expression that will be homogeneous in the derivatives. For instance, the Fokker-Planck equation of non-equilibrium statistical mechanics involves both a linear drift term and a quadratic diffusive term. Here, the operative assumption underlying the derivation of the Fokker-Planck equation is that the the distribution functions we work with in practice will have sufficient regularity that cutting off the Fokker-Planck operator at second-order in differentials will be justified. All the same, it is necessary to go to quadratic order if one wants to apply the equation to phenomena that take place over time scales long enough that the diffusive behavior matters. Speaking in general, then, the import of a statement such as this must be that we cannot, as a rule, always expect any possible terms of higher order in the metric to be negligible. For---what everyone up to now has missed---the derivative $\frac{d}{dx}$ diverges in the limit as $dx$ tends to zero at the same rate as $dx$ itself vanishes; thus, when the two are paired (or any like powers of the two), a cancellation results leading to the possibility that an appreciable contribution could very well persist after completion of the limiting procedure.

There remains a loose end to tie up. Both Helmholtz \cite{helmholtz}, in 1868, and Weyl \cite{weyl_einzigartigkeit}, in 1922, have published arguments to the effect that the line element defining a metric must be quadratic. Therefore, we respond in the following subsection with an explanation of why their reasoning is sophistical.

\subsubsection{How to Circumvent Helmholtz and Weyl}

The first author after Riemann to take up the question of the determination of the metrical properties of space is Helmholtz \cite{helmholtz}. His starting point is that all measurement of space rests ultimately upon the observation of congruences, but to speak of congruence at all one must have freely movable rigid bodies. Hence, Helmholtz posits four hypotheses: that i) spatial motions must be continuously differentiable, ii) there must exist movable rigid bodies satisfying a condition among the $2n$ coordinates of every pair of points, iii) the group of motions acts transitively on space and iv) any one-parameter subgroup of motions of a rigid body holding $n-1$ of its points fixed must close upon itself (which is not in fact necessarily true). Beginning from the general expression for an infinitesimal motion, imposing the conditions and integrating leads to the form of the line element $ds^2$ homogeneous and quadratic in the coordinate differentials. If Helmholtz' conclusion is to be contested, one must ask whether his initial expression for the differentials of the transformed coordinates in terms of the original ones is sufficiently general. Clearly, for us, it is not, since he ignores terms of second and higher order. Therefore, what his derivation really shows is that the leading dependence of the square of the line element is indeed homogenous and quadratic, but there will be corrections at higher order which we intend no longer to neglect.

Another line of argument to the same end is given by Weyl in 1922 \cite{weyl_einzigartigkeit}, who considers the essence of a metric to consist in the concept of infinitesimal congruences, which must be volume-preserving and form a group. The form of the metric will be known once one has decided which among the linear transformations of the tangent space at a point into itself or into that of an infinitesimally nearby point are to count as being congruences. Those that map from the tangent space at a point into itself will be termed rotations. Now, Weyl stipulates that the rotations should form a group, that any congruence to a given infinitesimally close point should be obtainable as a translation followed by a rotation (as would be the case with the ordinary Euclidean group) and, lastly, that composition of any two infinitesimal congruences should yield again an infinitesimal congruence.

A parallel translation from a given point to an infinitesimally nearby point would be a linear transformation realizable in some coordinate system under which every vector retains its coefficients unchanged. In an arbitrary coordinate system, it should be expressible in the usual way in terms of symmetric Christoffel symbols. Weyl regards everything that has been said so far as an analytical explication of our concepts of metric, connection and parallel translation. To complete the program set forth in this paper, one must invoke, synthetically, a second postulate demanding that among all possible systems of parallel translations (as determined by their Christoffel symbols) there should be a unique one in which the parallel translations are at the same time congruences. In the Pythagorean case, the fundamental form of the metric corresponds to a non-degenerate quadratic form in the components of vectors in the tangent space above any given point in space.

Weyl's principal result proves that the infinitesimal Euclidean group corresponding to the Pythagorean case satisfies all of the above requirements and that, conversely, any infinitesimal group that does so is equivalent to the Pythagorean case up to a choice of orientation. The derivation, which Weyl acknowledges not to be based on any geometrical insight into the problem but merely calculational, proceeds by writing the infinitesimal rotation group in the form $\mathrm{id} + A \varepsilon$ and neglecting all quantities that converge to zero faster than $\varepsilon$. The two main cases can be resolved in a natural way; then the general case can be reduced to these two in a manner that is `ziemlich unerfreulich'.

For our purposes, the crucial point is his restriction to transformations of the tangent space (either to itself or to an infinitesimally nearby tangent space) in which quantities of $o(\varepsilon)$ may be neglected, in other words, in which only the linear part matters. As soon as one is prepared to entertain tangents of higher than first order, however, it becomes no longer very natural to impose such a restriction. But a transformation with non-negligible non-linear part is capable of introducing a mixing among infinitesimals differing in order (see equations (\ref{jet_transf_law}) and (\ref{vector_transf_law}) for jets resp. tangents). Hence, there would be no reason why the fundamental form of the metric should always be quadratic only in the infinitesimals, and Weyl's argument ceases to apply.

\subsection{The Metric Tensor Incorporating Jets to Higher Order}

Consider the problem of rectification of a curve, informally. The usual procedure is implemented along the following lines: we can break the curve down into infinitely many infinitely small line segments, and compute its length by summing up the line elements of the infinitesimal segments as determined by the usual metric tensor. But now with higher-order terms figuring non-trivially in the generalized metric tensor, the na{\"i}ve procedure appears problematical. For we must expect the generalized metric tensor, to be defined in a moment, to be sensitive to the manner in which the true curve rounds the corners connecting one first-order line element to the next---that is to say, for the second-order infinitesimals there to be of consequence. The same goes by analogy to every succeeding order in the infinitesimals. While we have not as yet at our disposal a satisfactory and well-defined method of integration of differentials of order higher than the first, or differing in order (vide {\S}\ref{chapter_5}),
for the time being it will occupy our preliminary concern simply to define how a generalized metric tensor can exist in the first place and supply us with a consistent way of assigning sense to the concept of measure of a higher-order infinitesimal, or indeed of an expression of mixed degree in the infinitesimals. Once we have justified the assumption of a generalized metric tensor in the given form, we may then take up the question of performing quadrature with it, see below in {\S}\ref{chapter_5}.

At this point, it would appear to be straightforward to alight upon an appropriate Ansatz for a generalized metric tensor, if we view the ordinary metric tensor $g_p$ at the point $p \in M$ in its algebraical guise as a symmetrical element of $TM^*_p \otimes TM^*_p$, which yields a scalar upon contraction with two vectors, one on either side. What we would want is something that fulfills the same role for higher-order vectors. Thus, the following definition seems to be very natural:

\begin{definition}
	Let $M$ be a differentiable manifold. Then an $r$-th order generalized metric tensor on $M$ is a distinguished smooth section of the vector bundle $g \in \mathrm{Sym}(\mathscr{J}^{r*}(M) \otimes \mathscr{J}^{r*}(M))$, where the induced endomorphism $\mathscr{J}^r(M) \rightarrow \mathscr{J}^{r*}(M)$ is pointwise invertible. Given two $r$-vectors $X_p, Y_p$ at a point $p \in M$, their inner product with respect to the metric $g$ will be denoted as $\langle X_p, Y_p \rangle := \mathrm{tr}~\mathrm{tr}~g_p \otimes X_p \otimes Y_p$.
\end{definition}

The question that confronts us straightaway at this juncture is whether the metric so defined corresponds to what we should want based upon geometrical intuition. Since projection to 1-jets is canonical, we may always recover from the generalized metric tensor $g$ an ordinary metric tensor $\hat{g}$ that acts on 1-vectors only. In this sense, we extend the usual formalism. What does our extension say at higher order? Assume a coordinate basis exists in which the metric becomes diagonal (what is no longer immediate in the presence of higher infinitesimals). In this basis, vector components of different order would not interact with one another and within each order would behave as independent increments; but this property need not hold in the basis with which we started. Therefore, we must expect that generically we do have what is a novel geometrical property, the possibility of interaction among different orders of tangency, as far as the measure of size is concerned. Clearly, there must be extensive ramifications in the theory of geodesics and of integration. If the question is to be treated appropriately and covariantly, we ought to seek intrinsic geometrical invariants by which to get a handle on the phenomenon. This will be the task of the next and succeeding sections.

But first, let us ask what is the relation of the metric so defined to tangency in the generalized sense? The answer to this question comes down to what happens to proposition \ref{orthogonal_family_hypersurfaces} when the metric is no longer the standard Euclidean one and depends on whether it is possible to prove that local coordinates exist in which, near a given point, the metric always assumes something close to the standard Euclidean expression. If so, then proposition \ref{orthogonal_family_hypersurfaces} would remain true at least in the small. Our general picture of space as homogeneous and isotropic in the small requires, then, a result of this kind for its consistency; we shall return to the question below in {\S}\ref{chapter_5}.

\subsection{The Levi-Civita Jet Connection}

A key step in Riemannian geometry is to relate covariant differentiation to the metrical properties of space. To this end, one requires first a sense in which the two notions can be compatible:
\begin{definition}\label{def_compatibility}
	A jet connection $\nabla$ will be said to be compatible with the metric $g$ if, for all vector fields $X,Y,Z \in \mathscr{J}^r(M)$, we have $X \langle Y,Z \rangle = \nabla_X \langle Y,Z \rangle = X^\alpha \frac{\alpha!}{\alpha_1!\alpha_2!} \langle \nabla_{\alpha_1}Y,\nabla_{\alpha_2}Z \rangle$.
\end{definition}

\begin{lemma}
	Let $\nabla$ be a jet connection defined on a (pseudo-) Riemannian manifold $(M,g)$. Then a necessary and sufficient condition for $\nabla$ to be compatible with $g$ is that $\nabla g = 0$.
\end{lemma}
\begin{proof}
	$(\Rightarrow)$ Suppose that $\nabla$ is compatible with $g$ and work at the origin in a local chart with coordinates $x^{1,...,n}$. Now for $X=\partial_\alpha$, we have
	\begin{equation}\label{compatibility_condition}
		\nabla_\alpha  \langle Y,Z \rangle = \mathrm{tr} ~\mathrm{tr} \frac{\alpha!}{\alpha_1!\alpha_2!} \nabla_{\alpha_1} g \otimes \nabla_{\alpha_2} (Y \otimes Z).
	\end{equation}
	In order to derive a condition that will imply the vanishing of $\nabla g$ at the origin, let us take for $Y$ the form $Y=\varphi\partial_\lambda$ for an arbitrary smooth function $\varphi$ while for $Z=\partial_\rho$. But then,
	\begin{align}\label{compatibility_condition_2}
		\nabla_\alpha \varphi \partial_\lambda \otimes \partial_\rho &= \frac{\alpha!}{\beta_0!\beta_1!\beta_2!} \varphi_{,\beta_0} \left( \Gamma^\mu_{\beta_1\lambda}\partial_\mu \right) \otimes \left( \Gamma^\nu_{\beta_2\rho}\partial_\nu \right) \nonumber \\
		&= \varphi \left( \frac{\alpha!}{\beta_0!\beta_1!\beta_2!} \Gamma^\mu_{\beta_1\lambda}\partial_\mu \otimes  \Gamma^\nu_{\beta_2\rho}\partial_\nu \right) + \frac{\alpha!}{\beta_0!\beta_1!\beta_2!}\bigg|_{\beta_0 \ne 0} \varphi_{,\beta_0}
		\left( \Gamma^\mu_{\beta_1\lambda}\partial_\mu \otimes  \Gamma^\nu_{\beta_2\rho}\partial_\nu \right).
	\end{align}
	For a given multi-index $\alpha_0 \ne 0$, define $\varphi_{\alpha_0} = x^{\alpha_0}$ so that $\varphi_{\alpha_0,\delta}|_0 = 0$ unless $\delta=\alpha_0$. Setting aside a fixed multi-index $\alpha_1 \ne 0$, choose $\alpha=\alpha_0+\alpha_1$. For $\alpha_2 \ne 0,\alpha$, we have from equation (\ref{compatibility_condition_2}) that
	\begin{equation}
		\nabla_{\alpha_2} \varphi_{\alpha_0} \partial_\lambda \otimes \partial_\rho 
		= \varphi_{\alpha_0} \nabla_{\alpha_2} \partial_\lambda \otimes \partial_\rho
	\end{equation}
	because the second term cancels by construction. At the origin, $\varphi_{\alpha_0}=0$ and the right-hand side vanishes altogether. Substitute now into equation (\ref{compatibility_condition}) to get for $\alpha_2=\alpha_0 \ne 0, \alpha$:
	\begin{equation}
		\mathrm{tr} ~\mathrm{tr} \frac{(\alpha_0+\alpha_1)!}{\alpha_0!\alpha_1!} \left(
		\nabla_{\alpha_1} g \right) \varphi_{\alpha_0,\alpha_1} \partial_\lambda \otimes \partial_\rho = 0
	\end{equation}
	by hypothesis. Therefore, we must have $\nabla_{\alpha_1} g = 0$ at $x=0$. But coordinates can be chosen around any point on $M$ to put it at the origin and $\alpha_1$ is any arbitrary non-zero multi-index; so we may conclude by multilinearity that $\nabla g=0$ everywhere, in view of the fact that the vanishing of a tensor at a given point is a covariant property.
	
	$(\Leftarrow)$ Suppose that $\nabla g=0$. If we choose again $X=\partial_\alpha$, the terms $\nabla_{\alpha_1} g$ in equation (\ref{compatibility_condition}) vanish by hypothesis for $\alpha_1 \ne 0$, leaving us with
	\begin{align}
		\nabla_\alpha  \langle Y,Z \rangle &= \mathrm{tr}~ \mathrm{tr}~ g \otimes \nabla_\alpha (Y \otimes Z) \nonumber \\
		&= \mathrm{tr}~ \mathrm{tr}~ g \otimes \frac{\alpha!}{\alpha_1!\alpha_2!}
		\nabla_{\alpha_1} Y \otimes \nabla_{\alpha_2} Z \nonumber \\
		&= \frac{\alpha!}{\alpha_1!\alpha_2!}
		\langle \nabla_{\alpha_1} Y, \nabla_{\alpha_2} Z \rangle.
	\end{align}
	Since this holds for any non-zero multi-index $\alpha$, by multilinearity in $X$ this shows that $\nabla$ is compatible with $g$ by definition \ref{def_compatibility}.
\end{proof}

\begin{lemma}
	$\nabla g = 0$ iff $\nabla g^{-1} = 0$.
\end{lemma}
\begin{proof}
	In light of the fact that $g^{-1}$ is the inverse of $g$, we have
	\begin{equation}\label{g_ginv}
		\nabla_\alpha g g^{-1} = \nabla_\alpha d_\mu \otimes \partial^\mu =
		\frac{\alpha!}{\alpha_1!\alpha_2!} \tilde{\Gamma}^\mu_{\alpha_1\nu}
		\Gamma^\mu_{\alpha_2\lambda} = \tilde{\Gamma}^\lambda_{\alpha\nu}
		+ \frac{\alpha!}{\alpha_1!\alpha_2!}\bigg|_{\alpha_{1,2} \ne 0} \tilde{\Gamma}^\mu_{\alpha_1\nu}\Gamma^\mu_{\alpha_2\lambda} + \Gamma^\nu_{\alpha\lambda} = 0,
	\end{equation}
	where the $\tilde{\Gamma}$ were introduced above in lemma \ref{jet_conn_on_1_forms} in just such a way as to make the right-hand side of equation (\ref{g_ginv}) vanish. Thus, for any $X \in \mathscr{J}^r(M)$,
	\begin{equation}
		0 = \nabla_X g g^{-1} = \mathrm{tr} ~ \nabla_X \left( g \otimes g^{-1} \right)
		= \mathrm{tr}~ \left(
		X^\alpha \frac{\alpha!}{\alpha_1!\alpha_2!} \nabla_{\alpha_1} g \otimes \nabla_{\alpha_2} g^{-1} \right).
	\end{equation}
	If now $\nabla g=0$, then in particular $\nabla_{\alpha_1} g=0$ for all $\alpha_1 \ne 0$. Thus, only the last term on the right-hand side survives yielding $\mathrm{tr}~ g \otimes \nabla_X g^{-1} = 0$. But since $g$ is assumed to be non-degenerate, this can be true only if $\nabla_X g^{-1} = 0$ for all $X$, or in other words, $\nabla g^{-1}=0$. Conversely, if $\nabla g^{-1}=0$, only the first term on the right-hand side survives, leading by the same argument to $\nabla g=0$.
\end{proof}

\begin{theorem}[Existence of the Levi-Civita Jet Connection]\label{exist_levi_civita}
	Let $(M,g)$ be a (pseudo-) Riemannian manifold and suppose that there exist sufficiently small constants $\delta_{0,1}>0$ such that the metric and all of its derivatives are uniformly bounded; i.e., $|g_{\alpha\beta}|<\delta_0$ and $|g_{\alpha\beta,\gamma}|<\delta_1$ for all multi-indices $\alpha,\beta,\gamma$ with $|\alpha|,|\beta|,|\gamma|\le r$ (in some coordinate system around each point of $M$). Then there exists a unique jet connection $\nabla$ on $M$ that is symmetric and compatible with the metric $g$.
\end{theorem}
\begin{proof}
	In view of the lemmas, we can lift the proof from Lee \cite{lee_riemannian_manif}, Theorem 5.4, almost word-for-word, as long as it is understood that all indices go up to order $|\alpha| \le r$. For convenience, we reproduce the argument here. First note that, if compatibility with the metric is to hold, we must have now that for all vector fields $X,Y,Z \in \mathscr{J}^r(M)$ (including cyclic permutations):
	\begin{align}
		X \langle Y,Z \rangle &= X^\alpha \frac{\alpha!}{\alpha_1!\alpha_2!} \langle \nabla_{\alpha_1}Y, \nabla_{\alpha_2}Z \rangle \nonumber \\
		&= \langle \nabla_X Y,Z \rangle +
		\langle Y, \nabla_X Z \rangle + X^\alpha \frac{\alpha!}{\alpha_1!\alpha_2!}\bigg|_{\alpha_{1,2}\ne 0} \langle \nabla_{\alpha_1}Y, \nabla_{\alpha_2}Z \rangle \nonumber \\
		&= \langle \nabla_X Y,Z \rangle +
		\langle Y, \nabla_X Z \rangle + X \langle Y, Z \rangle_{\alpha_{1,2}\ne 0} \nonumber \\
		Y \langle Z,X \rangle &= \langle \nabla_Y Z,X \rangle +
		\langle Z, \nabla_Y X \rangle + Y \langle Z, X \rangle_{\beta_{1,2}\ne 0}
		\nonumber \\
		Z \langle X,Y \rangle &= \langle \nabla_Z X,Y \rangle +
		\langle X, \nabla_Z Y \rangle + Z \langle X, Y \rangle_{\gamma_{1,2}\ne 0}. 
	\end{align}
	By reason of the assumed symmetry, we can rewrite these equations as follows:
	\begin{align}
		X \langle Y,Z \rangle &= \langle \nabla_X Y,Z \rangle +
		\langle Y, \nabla_Z X \rangle + \langle Y, [X,Z] \rangle +
		X \langle Y, Z \rangle_{\alpha_{1,2}\ne 0} \nonumber \\
		Y \langle Z,X \rangle &= \langle \nabla_Y Z,X \rangle +
		\langle Z, \nabla_X Y \rangle + \langle Z, [Y,X] \rangle +
		Y \langle Z, X \rangle_{\beta_{1,2}\ne 0} \nonumber \\
		Z \langle X,Y \rangle &= \langle \nabla_Z X,Y \rangle +
		\langle X, \nabla_Y Z \rangle + \langle X, [Z,Y] \rangle +
		Z \langle X, Y \rangle_{\gamma_{1,2}\ne 0}
	\end{align}
	If we add the first two and subtract the third, we obtain
	\begin{align}
		X \langle Y,Z \rangle + Y \langle Z,X \rangle - Z \langle X,Y \rangle = &2 \langle \nabla_X Y,Z \rangle + \langle Y,[X,Z] \rangle + \langle Z, [Y,X] \rangle
		- \langle X,[Z,Y] \rangle + \nonumber \\
		& X \langle Y, Z \rangle_{\alpha_{1,2}\ne 0} +
		Y \langle Z, X \rangle_{\beta_{1,2}\ne 0} - Z \langle X, Y \rangle_{\gamma_{1,2}\ne 0}.
	\end{align}
	Then, solve for $\langle \nabla_X Y,Z \rangle$:
	\begin{align}\label{nablaxyz}
		\langle \nabla_X Y,Z \rangle = & \frac{1}{2} \bigg(
		X \langle Y,Z \rangle + Y \langle Z,X \rangle - Z \langle X,Y \rangle -
		\langle Y,[X,Z] \rangle - \langle Z, [Y,X] \rangle + \langle X,[Z,Y] \rangle - 
		\nonumber \\ &
		X \langle Y, Z \rangle_{\alpha_{1,2}\ne 0} -
		Y \langle Z, X \rangle_{\beta_{1,2}\ne 0} + Z \langle X, Y \rangle_{\gamma_{1,2}\ne 0}
		\bigg).
	\end{align}
	The expression is the same as the usual one, except that for $r \ge 2$ cross-terms quadratic in the Christoffel symbols arise in the second line. Consider existence first.  It will be convenient to work in a local chart with $X = \partial_\alpha$, $Y=\partial_\beta$ and $Z=\partial_\gamma$. If we evaluate equation (\ref{nablaxyz}) on coordinate vector fields, the Lie brackets vanish identically leaving us with the following relation:
	\begin{align}
		\Gamma^\lambda_{\alpha\beta}g_{\lambda\gamma} = \frac{1}{2} \bigg( &
		\partial_\alpha g_{\beta\gamma} - \frac{\alpha!}{\alpha_1!\alpha_2!}\bigg|_{\alpha_{1,2}\ne 0} g_{\mu\nu}
		\Gamma^\mu_{\alpha_1\beta}\Gamma^\nu_{\alpha_2\gamma} + \nonumber \\
		&\partial_\beta g_{\gamma\alpha} - \frac{\beta!}{\beta_1!\beta_2!}\bigg|_{\beta_{1,2}\ne 0} g_{\mu\nu}
		\Gamma^\mu_{\beta_1\gamma}\Gamma^\nu_{\beta_2\alpha} - \nonumber \\
		&\partial_\gamma g_{\alpha\beta} +  \frac{\gamma!}{\gamma_1!\gamma_2!}\bigg|_{\gamma_{1,2}\ne 0} g_{\mu\nu}
		\Gamma^\mu_{\gamma_1\alpha}\Gamma^\nu_{\gamma_2\beta}
		\bigg),
	\end{align}
	where we have used $g_{\alpha\beta}=\langle \partial_\alpha,\partial_\beta \rangle$ and $\nabla_\alpha \partial_\beta = \Gamma^\gamma_{\alpha\beta}\partial_\gamma$. Multiply from the left by $g^{-1}$ and regroup the terms to yield an expression satisfied by the Christoffel symbols:
	\begin{align}\label{Levi_Civita_formula}
		\Gamma^\gamma_{\alpha\beta} = & \frac{1}{2} g^{\gamma\delta} \bigg(
		\partial_\alpha g_{\beta\delta} + \partial_\beta g_{\delta\alpha} - \partial_\delta g_{\alpha\beta} - \nonumber \\ &\frac{\alpha!}{\alpha_1!\alpha_2!}\bigg|_{\alpha_{1,2}\ne 0} g_{\mu\nu}
		\Gamma^\mu_{\alpha_1\beta}\Gamma^\nu_{\alpha_2\delta} -
		\frac{\beta!}{\beta_1!\beta_2!}\bigg|_{\beta_{1,2}\ne 0} g_{\mu\nu}
		\Gamma^\mu_{\beta_1\alpha}\Gamma^\nu_{\beta_2\delta} +
		\frac{\delta!}{\delta_1!\delta_2!}\bigg|_{\delta_{1,2}\ne 0} g_{\mu\nu}
		\Gamma^\mu_{\delta_1\alpha}\Gamma^\nu_{\delta_2\beta}
		\bigg).
	\end{align}
	Now in the general case, we no longer have a formula in explicit form for the Christoffel symbols, but only a relation they satisfy involving the cross-terms. Nevertheless, under our assumptions it is easy to show existence of a solution by appeal to the implicit function theorem. Let $d_1$ denote the dimension of the space of metric coefficients and their derivatives up to order $r$ and $d_2$ the dimension of the space of Christoffel symbols, and let $U \subset \vvmathbb{R}^{d_1} \times \vvmathbb{R}^{d_2}$ be a neighborhood of the origin. Define now the function $\Phi:U\rightarrow \vvmathbb{R}^{d_2}$ by
	\begin{align}
		\Phi(g,\Gamma) = & \Gamma^\gamma_{\alpha\beta} - \frac{1}{2} g^{\gamma\delta} \bigg(
		\partial_\alpha g_{\beta\gamma} + \partial_\beta g_{\gamma\alpha} - \partial_\gamma g_{\alpha\beta} - \nonumber \\ &\frac{\alpha!}{\alpha_1!\alpha_2!}\bigg|_{\alpha_{1,2}\ne 0} g_{\mu\nu}
		\Gamma^\mu_{\alpha_1\beta}\Gamma^\nu_{\alpha_2\delta} -
		\frac{\beta!}{\beta_1!\beta_2!}\bigg|_{\beta_{1,2}\ne 0} g_{\mu\nu}
		\Gamma^\mu_{\beta_1\alpha}\Gamma^\nu_{\beta_2\delta} +
		\frac{\delta!}{\delta_1!\delta_2!}\bigg|_{\delta_{1,2}\ne 0} g_{\mu\nu}
		\Gamma^\mu_{\delta_1\alpha}\Gamma^\nu_{\delta_2\beta}
		\bigg). 
	\end{align}
	The Jacobian
	\begin{align}
		\frac{\partial\Phi^{\bar{\gamma}}_{\bar{\alpha}\bar{\beta}}}
		{\partial\Gamma^\gamma_{\alpha\beta}} = I + & \frac{1}{2} g^{\bar{\gamma}\delta} 
		\bigg(
		\frac{\bar{\alpha}!}{\alpha_1!\alpha_2!}\bigg|_{\alpha_{1,2}\ne 0} g_{\mu'\nu'}
		\big( \delta_{\mu'\mu}\delta_{\alpha_1\alpha}\delta_{\bar{\beta}\beta}
		\Gamma^{\nu'}_{\alpha_2\delta} +
		\delta_{\nu'\nu}\delta_{\alpha_2\alpha}\delta_{\delta\beta}
		\Gamma^{\mu'}_{\alpha_1\bar{\beta}} + \nonumber \\
		&
		\frac{\bar{\beta}!}{\beta_1!\beta_2!}\bigg|_{\beta_{1,2}\ne 0} g_{\mu'\nu'}
		\big( \delta_{\mu'\mu}\delta_{\beta_1\alpha}\delta_{\bar{\alpha}\beta}
		\Gamma^{\nu'}_{\beta_2\delta} +
		\delta_{\nu'\nu}\delta_{\beta_2\alpha}\delta_{\delta\beta}
		\Gamma^{\mu'}_{\beta_1\bar{\beta}} \big) - \nonumber \\
		&
		\frac{\delta!}{\delta_1!\delta_2!}\bigg|_{\delta_{1,2}\ne 0} g_{\mu'\nu'}
		\big( \delta_{\mu'\mu}\delta_{\delta_1\alpha}\delta_{\bar{\alpha}\beta}
		\Gamma^{\nu'}_{\delta_2\bar{\beta}} +
		\delta_{\nu'\nu}\delta_{\delta_2\alpha}\delta_{\bar{\beta}\beta}
		\Gamma^{\mu'}_{\delta_1\bar{\beta}} \big) \bigg)
	\end{align}
	will be non-singular everywhere on a sufficiently small neighborhood $U$ due to the uniform bounds on the metric and its derivatives. Therefore, by the implicit function theorem (Lee \cite{lee_smooth_manif}, Theorem C.40), there exist neighborhoods $V \subset \vvmathbb{R}^{d_1}$ and $W \subset \vvmathbb{R}^{d_2}$ of the origin and a smooth function $F:V\rightarrow W$ such that $\Phi^{-1}(0) \cup V \times W$ is the graph of $F$, that is, $\Phi(g,\Gamma)=0$ if and only if $\Gamma=F(g)$. Since the coefficients of the metric tensor depend smoothly on position, so will $p \mapsto F(g(p))$. This shows existence.
	
	It is evident from the defining formula, equation (\ref{Levi_Civita_formula}), that $\Gamma^\gamma_{\alpha\beta}=\Gamma^\gamma_{\beta\alpha}$, hence the jet connection $\nabla$ defined by the Christoffel symbols is symmetric (proposition \ref{torsion_tensor_properties}). As for compatibility with the metric, observe that $\nabla g = 0$ is equivalent to
	\begin{equation}
		X \langle \partial_\mu,\partial_\nu \rangle = X^\alpha \frac{\alpha!}{\alpha_1!\alpha_2!}
		\langle \nabla_{\alpha_1} \partial_\mu, \nabla_{\alpha_2} \partial_\nu \rangle = 0
	\end{equation}
	for all vector fields $X \in \mathscr{J}^r(M)$, which in turn is equivalent to the statement that 
	\begin{equation}
		g_{\mu\nu,\alpha} = \frac{\alpha!}{\alpha_1!\alpha_2!} g_{\lambda\sigma}
		\Gamma^\lambda_{\alpha_1\mu}\Gamma^\sigma_{\alpha_2\nu}
	\end{equation}
	for all multi-indices $\mu,\nu,\alpha$; $1 \le |\mu|,|\nu|,|\alpha| \le r$. But from the definition of the Christoffel symbols we have
	\begin{align}
		\Gamma^\lambda_{\delta\alpha}g_{\lambda\beta} + \Gamma^\lambda_{\delta\beta}g_{\alpha\lambda} =& \frac{1}{2} \bigg(
		g_{\alpha\beta,\delta} + g_{\delta\beta,\alpha} - g_{\delta\alpha,\beta} + g_{\beta\alpha,\delta} + g_{\alpha\delta,\beta} - g_{\delta\beta,\alpha} - 
		\nonumber \\
		&
		\frac{\delta!}{\delta_1!\delta_2!}\bigg|_{\delta_{1,2}\ne 0}
		\Gamma^\mu_{\delta_1\bar{\alpha}}\Gamma^\nu_{\delta_2\bar{\beta}} -
		\frac{\alpha!}{\alpha_1!\alpha_2!}\bigg|_{\alpha_{1,2}\ne 0} g_{\mu\nu}
		\Gamma^\mu_{\alpha_1\delta}\Gamma^\nu_{\alpha_2\beta} +
		\frac{\beta!}{\beta_1!\beta_2!}\bigg|_{\beta_{1,2}\ne 0} g_{\mu\nu}
		\Gamma^\mu_{\beta_1\delta}\Gamma^\nu_{\beta_2\alpha} - \nonumber \\
		&
		\frac{\delta!}{\delta_1!\delta_2!}\bigg|_{\delta_{1,2}\ne 0}
		\Gamma^\mu_{\delta_1\bar{\alpha}}\Gamma^\nu_{\delta_2\bar{\beta}} -
		\frac{\beta!}{\beta_1!\beta_2!}\bigg|_{\beta_{1,2}\ne 0} g_{\mu\nu}
		\Gamma^\mu_{\beta_1\alpha}\Gamma^\nu_{\beta_2\delta} +
		\frac{\alpha!}{\alpha_1!\alpha_2!}\bigg|_{\alpha_{1,2}\ne 0} g_{\mu\nu}
		\Gamma^\mu_{\alpha_1\delta}\Gamma^\nu_{\alpha_2\beta}
		\bigg). 
	\end{align}
	After canceling terms and transferring the sum over the multi-indices $\delta_{1,2} \ne 0$ to the left-hand side, we obtain just the condition that $\nabla g = 0$. Therefore, the jet connection thus defined is compatible with the metric.
	
	Last, uniqueness. If $\nabla^1$ and $\nabla^2$ are two jet connections that are symmetric and compatible with the metric, the first line in equation (\ref{Levi_Civita_formula}), being independent of the connection, will be identical for the two; this leaves only the cross-terms. Then evaluating the difference between the two we arrive at
	\begin{align}
		g_{\gamma\lambda}\Gamma^{1\lambda}_{\alpha\beta} - g_{\gamma\lambda}\Gamma^{2\lambda}_{\alpha\beta} = &
		\frac{1}{2} \frac{\alpha!}{\alpha_1!\alpha_2!}\bigg|_{\alpha_{1,2}\ne 0} g_{\mu\nu}\Gamma^{2\mu}_{\alpha_1\beta}\Gamma^{2\nu}_{\alpha_2\gamma} -
		\frac{1}{2} \frac{\alpha!}{\alpha_1!\alpha_2!}\bigg|_{\alpha_{1,2}\ne 0} g_{\mu\nu}\Gamma^{1\mu}_{\alpha_1\beta}\Gamma^{1\nu}_{\alpha_2\gamma} + 
		\nonumber \\
		&\frac{1}{2} \frac{\beta!}{\beta_1!\beta_2!}\bigg|_{\beta_{1,2}\ne 0}
		g_{\mu\nu}\Gamma^{2\mu}_{\beta_1\gamma}\Gamma^{2\nu}_{\beta_2\alpha} -
		\frac{1}{2} \frac{\beta!}{\beta_1!\beta_2!}\bigg|_{\beta_{1,2}\ne 0} g_{\mu\nu}\Gamma^{1\mu}_{\beta_1\gamma}\Gamma^{1\nu}_{\beta_2\alpha} - 
		\nonumber \\
		&\frac{1}{2} \frac{\gamma!}{\gamma_1!\gamma_2!}\bigg|_{\gamma_{1,2}\ne 0}
		g_{\mu\nu}\Gamma^{2\mu}_{\gamma_1\alpha}\Gamma^{2\nu}_{\gamma_2\beta} +
		\frac{1}{2} \frac{\gamma!}{\gamma_1!\gamma_2!}\bigg|_{\gamma_{1,2}\ne 0} g_{\mu\nu}\Gamma^{1\mu}_{\gamma_1\alpha}\Gamma^{1\nu}_{\gamma_2\beta}.
	\end{align}
	Multiply from the left by $g^{-1}$ and put this into the form,
	\begin{align}
		\Gamma^{1\lambda}_{\alpha\beta} - \Gamma^{2\lambda}_{\alpha\beta} = g^{\lambda\gamma} \bigg( &
		\frac{1}{2} \frac{\alpha!}{\alpha_1!\alpha_2!}\bigg|_{\alpha_{1,2}\ne 0} g_{\mu\nu}\Gamma^{2\mu}_{\alpha_1\beta}\Gamma^{2\nu}_{\alpha_2\gamma} -
		\frac{1}{2} \frac{\alpha!}{\alpha_1!\alpha_2!}\bigg|_{\alpha_{1,2}\ne 0} g_{\mu\nu}\Gamma^{1\mu}_{\alpha_1\beta}\Gamma^{1\nu}_{\alpha_2\gamma} + 
		\nonumber \\
		&\frac{1}{2} \frac{\beta!}{\beta_1!\beta_2!}\bigg|_{\beta_{1,2}\ne 0}
		g_{\mu\nu}\Gamma^{2\mu}_{\beta_1\gamma}\Gamma^{2\nu}_{\beta_2\alpha} -
		\frac{1}{2} \frac{\beta!}{\beta_1!\beta_2!}\bigg|_{\beta_{1,2}\ne 0} g_{\mu\nu}\Gamma^{1\mu}_{\beta_1\gamma}\Gamma^{1\nu}_{\beta_2\alpha} - 
		\nonumber \\
		&\frac{1}{2} \frac{\gamma!}{\gamma_1!\gamma_2!}\bigg|_{\gamma_{1,2}\ne 0}
		g_{\mu\nu}\Gamma^{2\mu}_{\gamma_1\alpha}\Gamma^{2\nu}_{\gamma_2\beta} +
		\frac{1}{2} \frac{\gamma!}{\gamma_1!\gamma_2!}\bigg|_{\gamma_{1,2}\ne 0} g_{\mu\nu}\Gamma^{1\mu}_{\gamma_1\alpha}\Gamma^{1\nu}_{\gamma_2\beta} \bigg).
	\end{align}
	We will now show uniquess as a result of the contraction mapping theorem. Define the following map on the space of connections, equipped with the $\ell_\infty$ norm in the finite sequence of coefficients of the Christoffel symbol:
	\begin{align}
		F(\Gamma) = \Gamma^{\lambda}_{\alpha\beta} - g^{\lambda\gamma} \bigg( &
		\frac{1}{2} \frac{\alpha!}{\alpha_1!\alpha_2!}\bigg|_{\alpha_{1,2}\ne 0} g_{\mu\nu}\Gamma^{2\mu}_{\alpha_1\beta}\Gamma^{2\nu}_{\alpha_2\gamma} -
		\frac{1}{2} \frac{\alpha!}{\alpha_1!\alpha_2!}\bigg|_{\alpha_{1,2}\ne 0} g_{\mu\nu}\Gamma^{\mu}_{\alpha_1\beta}\Gamma^{\nu}_{\alpha_2\gamma} + 
		\nonumber \\
		&\frac{1}{2} \frac{\beta!}{\beta_1!\beta_2!}\bigg|_{\beta_{1,2}\ne 0}
		g_{\mu\nu}\Gamma^{2\mu}_{\beta_1\gamma}\Gamma^{2\nu}_{\beta_2\alpha} -
		\frac{1}{2} \frac{\beta!}{\beta_1!\beta_2!}\bigg|_{\beta_{1,2}\ne 0} g_{\mu\nu}\Gamma^{\mu}_{\beta_1\gamma}\Gamma^{\nu}_{\beta_2\alpha} - 
		\nonumber \\
		&\frac{1}{2} \frac{\gamma!}{\gamma_1!\gamma_2!}\bigg|_{\gamma_{1,2}\ne 0}
		g_{\mu\nu}\Gamma^{2\mu}_{\gamma_1\alpha}\Gamma^{2\nu}_{\gamma_2\beta} +
		\frac{1}{2} \frac{\gamma!}{\gamma_1!\gamma_2!}\bigg|_{\gamma_{1,2}\ne 0} g_{\mu\nu}\Gamma^{\mu}_{\gamma_1\alpha}\Gamma^{\nu}_{\gamma_2\beta} \bigg).
	\end{align}
	Then, from the above argument, without loss of generality there is a constant $1>\delta>0$ such that the Christoffel symbols satisfy $\|\Gamma\|_\infty<\delta$ when they exist at all; so we have $\|F(\Gamma)-\Gamma\|_\infty = \|\Gamma^2\Gamma^2-\Gamma\Gamma)\|_\infty \le 2 C \delta^2$ for some constant $C>0$ depending on $\delta_{0,1}$ and $r$, where for simplicity we adopt an abbreviated notation. Let us suppose the bound $\delta>0$ be taken small enough to satisfy $\varepsilon = 2C\delta^2<1$. 
	But from its definition, clearly $F(\Gamma^2)=\Gamma^2$, whence our result: 
	\begin{align}
		F(\Gamma^1)-\Gamma^1 &= \Gamma^1\Gamma^1 - \Gamma^2\Gamma^2 \nonumber \\
		&= (\Gamma^2+F(\Gamma^1)-\Gamma^1)(\Gamma^2+F(\Gamma^1)-\Gamma^1) - \Gamma^2\Gamma^2 \nonumber \\
		&= (F(\Gamma^1)-\Gamma^1)^2 + \Gamma^2 (F(\Gamma^1) - \Gamma^2) + 
		(F(\Gamma^1) - \Gamma^1)\Gamma^2,
	\end{align}
	where in our condensed notation the expressions on the right-hand side are understood to involve possible numerical coefficients and factors of the metric tensor. Thus, the right-hand side involves $F(\Gamma^1)-\Gamma^1$ and $\Gamma^2$. But the former is bounded by $\varepsilon<1$ and the latter by $\delta<1$. Hence, when we take the $\ell_\infty$ norm of both sides, there will be a constant $M>0$ depending only on $r$ such that 
	\begin{equation}
		\|F(\Gamma^1)-\Gamma^1\|_\infty < \frac{\varepsilon}{1-M \| g^{-1} \|_\infty \| g \|_\infty \delta}\|F(\Gamma^1)-\Gamma^1\|_\infty,
	\end{equation}
	whence $F(\Gamma^1)=\Gamma^1$ if $\varepsilon>0$ and $\delta>0$ are taken to be small enough that the coefficient on the right-hand side is less than one. Putting our results together, we have $\Gamma^1-\Gamma^2 = F(\Gamma^1) - F(\Gamma^2) = (\Gamma^1+\Gamma^2-\Gamma^1) - (\Gamma^2) = 0$, which is to say, $\Gamma^1=\Gamma^2$. \end{proof}

\begin{remark}
	The condition of uniform boundedness in the statement of the theorem is not ideal in that it depends on the choice of coordinates; in order to remove it, one would need to introduce a covariant sense in which a manifold could be said to be nearly flat. This could be done in a pointwise fashion with normal coordinates, in which the deviation of the metric from the Euclidean is given in terms of the Riemannian curvature tensor, a geometrical invariant; then, the Levi-Civita jet connection would exist in a neighborhood of every point on $M$ in normal coordinates, but, being an intrinsic concept, its existence everywhere on $M$ would then be secured. The problem with this approach, however, is that we do not as yet have a theorem establishing the existence of suitable normal coordinates.
\end{remark}

\subsection{On the Curvature of Space in the Presence of Infinitesimals}\label{generalized_curvature_tensor}

Now that we have at our disposal the Levi-Civita jet connection as the natural means by which to formalize the concept of parallel transport with respect to the generalized Riemannian metric, let us take up the question as to how to arrive at an extended sense of curvature. Riemann's idea in the habilitation lecture is to propose that one study the curvature of space by singling out the $\frac{1}{2}n(n-1)$ Fl{\"a}chenrichtungen along which an infinitesimal surface can fan out from a given point. Later on, he would formalize his intuitive approach in terms of the curvature tensor that bears his name. For an excellent presentation of the right way to think about Riemann's intuitive idea in modern terms, see Lee, \cite{lee_riemannian_manif}, Chapter 8. There, Lee attaches a quantitative geometrical measure of curvature, viz. the sectional curvature, to every such Fl{\"a}chenrichtung by deriving it from the second fundamental form on the two-dimensional submanifold swept out by geodesics whose initial tangents lie in the corresponding two-dimensional subspace of the tangent space at the point. The requisite technique does not as yet exist to replicate the procedure at higher order. Nevertheless, it does orient us as to the direction in which we wish to head. For a geodesic can be thought of as a curve along which its tangent vector undergoes parallel transport. Therefore, since we already have a higher-order analogue, namely the jet connection of {\S}\ref{parallel_transport_jets}, we propose to introduce the natural generalization of curvature by resorting to it, to be applied now to what we have called the jet stream rather than just a 1-dimensional trajectory. Formally, curvature manifests itself through non-commutativity of parallel transport along the two independent initial tangents that define a Fl{\"a}chenrichtung (see Lee, \cite{lee_riemannian_manif}, Chapter 7). At higher order, of course, there will be more than $\frac{1}{2}n(n-1)$ possible pairs of such directions; roughly speaking, there should be one corresponding to any choice of two linearly independent $r$-vectors in $J^r_pM$. Then, if we wish to proceed along lines as close as possible to what is done in the case of differential geometry to first order, it immediately motivates the following definition. 

\begin{definition}\label{def_curv}
	Given a jet connection $\nabla$ on a manifold $M$, its curvature endomorphism will be defined by the formula
	\begin{equation}\label{def_curvature_endomorphism}
		R(X,Y)Z := \nabla_X \nabla_Y Z - \nabla_Y \nabla_X Z - \nabla_{[X,Y]} Z,
	\end{equation}
	for vector fields $X,Y,Z \in \mathscr{J}^\infty(M)$.
\end{definition}
The first order of business is to show that the curvature endomorphism so defined will, in fact, be a generalized tensor. Here, the generalized tensor characterization lemma makes the verification relatively easy.

\begin{lemma}\label{raise_cov_deriv}
	Covariant composition and the covariant derivative satisfy the relation
	\begin{equation}
		\nabla_Y \circ \nabla_X Z = \nabla_{\nabla_Y X} Z.
	\end{equation}
\end{lemma}
\begin{proof}
	By multilinearity it suffices to show the relation for $Y = \partial_\beta$. With $X = X^\alpha \partial_\alpha$, the left-hand side reads
	\begin{equation}
		\nabla_\beta \circ \nabla_X Z = \nabla_\beta \circ X^\alpha \nabla_\alpha Z
		= \frac{\beta!}{\beta_1!\beta_2!} X^\alpha_{,\beta_1} \nabla_{\beta_2} \circ \nabla_\alpha Z.
	\end{equation}
	On the right-hand side, we have
	\begin{align}
		\nabla_{\nabla_\beta X} Z &= \nabla_{\frac{\beta!}{\beta_1!\beta_2!} Xs^\alpha_{,\beta_1} \nabla_{\beta_2} \partial_\alpha} Z \nonumber \\
		&= \frac{\beta!}{\beta_1!\beta_2!} X^\alpha_{,\beta_1} \nabla_{\nabla_{\beta_2} \partial_\alpha} Z \nonumber \\
		&= \frac{\beta!}{\beta_1!\beta_2!}\frac{\gamma!}{\gamma_1!\gamma_2!} X^\alpha_{,\beta_1} \Gamma^\gamma_{\beta_2\alpha} Z^\nu_{,\gamma_1} \nabla_{\gamma_2} \partial_\nu \nonumber \\
		&= \frac{\beta!}{\beta_1!\beta_2!} X^\alpha_{,\beta_1} \nabla_{\beta_2} \circ \left( \nabla Z \right)_\alpha \nonumber \\
		&= \frac{\beta!}{\beta_1!\beta_2!} X^\alpha_{,\beta_1} \nabla_{\beta_2} \circ \nabla_\alpha Z. 
	\end{align}
	Therefore, the desired result holds.
\end{proof}

\begin{theorem}[Tensoriality of the Curvature]\label{tensoriality_of_curvature}
	For the Levi-Civita jet connection of Theorem \ref{exist_levi_civita}, the curvature endomorphism $R(X,Y)$ defined by equation (\ref{def_curvature_endomorphism}) is a generalized tensor field of type $\binom{3}{1}$.
\end{theorem}
\begin{proof}
	In view of lemma \ref{tensor_char}, all we have to show is that for all smooth functions $f \in C^\infty(M)$ and all vector fields $X,Y,Z \in \mathscr{J}(M)$, we have $R(fX,Y)Z=fR(X,Y)Z$ and $R(X,Y)fZ=fR(X,Y)Z$ (appealing to anti-symmetry in $X$ and $Y$). As to the first statement, compute as follows:
	\begin{align}
		R(fX,Y)Z &= \nabla_{fX} \nabla_Y Z - \nabla_Y \nabla_{fX} Z - \nabla_{[fX,Y]} Z \nonumber \\
		&= f \nabla_X \nabla_YZ - \nabla_Y f \nabla_X Z + \nabla_{- \nabla_{fX} Y + \nabla_Y fX} Z \nonumber \\
		&= f \nabla_X \nabla_Y Z - \nabla_Y f \nabla_X Z + f \nabla_{- \nabla_X Y} + \nabla_{\nabla_Y fX} Z.
	\end{align}
	Write the two terms in which $f$ has not been commuted to the left as follows, isolating terms involving derivatives of $f$:
	\begin{align}
		\nabla_Y f \nabla_X Z &:= f \nabla_Y \nabla_X Z + \left( \nabla_Y f \nabla_X Z \right)^\prime \\
		\nabla_{\nabla_Y fX} Z &:= f \nabla_{\nabla_Y X} Z + \left( \nabla_{\nabla_Y fX} Z \right)^\prime.   
	\end{align}
	Thus, we need to show that $\left( \nabla_Y f \nabla_X Z \right)^\prime = 
	\left( \nabla_{\nabla_Y fX} Z \right)^\prime$. From lemma \ref{raise_cov_deriv}, the latter quantity may be written as $\left( Y \circ f \nabla_X Z \right)^\prime$. By lemma \ref{cov_comp_vs_cov_deriv} therefore, also $\left( Y \circ fX \right)^\prime = \left( Y fX \right)^\prime = \left( \nabla_Y fX \right)^\prime$. Combine this statement with substitution of $\nabla_X Z$ for $X$ to yield for the curvature endomorphism,
	\begin{align}
		R(fX,Y)Z &= f \left( \nabla_X \nabla_Y Z - \nabla_Y \nabla_X Z + \nabla_{- \nabla_X Y + \nabla_Y X} Z \right) \nonumber \\
		&= f \left( \nabla_X \nabla_Y Z - \nabla_Y \nabla_X Z - \nabla_{[X,Y]} Z \right) \nonumber \\
		&= f R(X,Y)Z.
	\end{align} 
	
	As for linearity over $C^\infty(M)$ in $Z$, by multilinearity of the curvature endomorphism in its first two arguments it will suffice to show that
	\begin{equation}\label{curv_multilin_third}
		R(\partial_\alpha,\partial_\beta)fZ = \left( \nabla_\alpha \nabla_\beta - \nabla_\beta \nabla_\alpha \right) fZ = f \left( \nabla_\alpha \nabla_\beta - \nabla_\beta \nabla_\alpha \right) Z = fR(\partial_\alpha,\partial_\beta)Z
	\end{equation}
	for all multi-indices $\alpha,\beta$, since of course $[\partial_\alpha,\partial_\beta]=0$ always. The key point is that, by virtue of the symmetry of the Levi-Civita jet connection and lemma \ref{cov_hess_symm}, we may write
	\begin{equation}
		\nabla^2 Z(X,Y) = \left( \nabla_Y \nabla^2 u \right) (Z,X) = \left( \nabla_Y \nabla^2 u \right) (X,Z) = \nabla^2 X(Z,Y)
	\end{equation}
	and substitute into equation (\ref{curv_multilin_third}) to yield,
	\begin{align}
		R(\partial_\alpha,\partial_\beta)fZ &= \nabla_{\partial_\alpha} \nabla_{\partial_\beta} fZ - \nabla_{\partial_\beta} \nabla_{\partial_\alpha} fZ \nonumber \\ 
		&= \nabla^2 fZ(\partial_\beta,\partial_\alpha) - \nabla^2 fZ(\partial_\alpha,\partial_\beta) \nonumber \\
		&= \nabla^2 \partial_\beta (fZ,\partial_\alpha) - \nabla^2 \partial_\alpha(fZ,\partial_\beta) \nonumber \\
		&= f \nabla^2 \partial_\beta (Z,\partial_\alpha) - f \nabla^2 \partial_\alpha(Z,\partial_\beta) \nonumber \\
		&= f \nabla^2 Z(\partial_\beta,\partial_\alpha) - f \nabla^2 Z(\partial_\alpha,\partial_\beta) \nonumber \\
		&= f \nabla_{\partial_\alpha} \nabla_{\partial_\beta} Z - f \nabla_{\partial_\beta} \nabla_{\partial_\alpha} Z \nonumber \\ 		
		&= fR(\partial_\alpha,\partial_\beta)Z.
	\end{align}
	This step completes the proof.
\end{proof}

\begin{definition}
	In analogy to what is usual (cf. Lee, \cite{lee_riemannian_manif}, Equation 7.4), define the Riemann curvature tensor by lowering the last index; i.e.,
	\begin{equation}
		Rm(X,Y,Z,W) := \langle R(X,Y)Z, W \rangle.
	\end{equation}
\end{definition}

\begin{proposition}[Symmetries of the curvature tensor; cf. Lee, \cite{lee_riemannian_manif}, Proposition 7.4]\label{first_bianchi}
	For any vector fields $W, X, Y, Z \in \mathscr{J}^\infty(M)$, the generalized curvature tensor enjoys the following symmetries:
	\begin{itemize}
		\item[$(1)$] $Rm(W,X,Y,Z) = - Rm(X,W,Y,Z)$;
		
		\item[$(2)$] $Rm(W,X,Y,Z) = - Rm(W,X,Z,Y)$;
		
		\item[$(3)$] $Rm(W,X,Y,Z) = Rm(Y,Z,W,X)$;
		
		\item[$(4)$] $Rm(W,X,Y,Z) + Rm(X,Y,W,Z) + Rm(Y,W,X,Z) = 0$.
	\end{itemize}
\end{proposition}
\begin{proof}
	Just as in the ordinary first-order case, (1) follows immediately from the obvious antisymmetry of the curvature endomorphism: $R(W,X)Y=-R(X,W)Y$, upon taking the inner product with $Z$.
	
	If we we expand
	\begin{align}
		Rm(W,X,Y+Z,Y+Z) = &Rm(W,X,Y,Y) + Rm(W,X,Y,Z) + \nonumber \\
		&Rm(W,X,Z,Y) + Rm(W,X,Z,Z),
	\end{align}
	it will be enough to show that $Rm(W,X,V,V)=0$ for any $V$ in order to conclude to the identity (2). Due to compatibility of the Levi-Civita connection with the metric, we have
	\begin{align}
		0 &= \left( WX - XW - [W,X] \right) \langle Y , Y \rangle \nonumber \\
		& = \left( \nabla_W \nabla_X - \nabla_X \nabla_W - \nabla_{[W,X]} \right \langle Y , Y \rangle \nonumber \\
		&= \mathrm{tr}~\mathrm{tr}~ g \otimes \left( \nabla_W \nabla_X - \nabla_X \nabla_W - \nabla_{[W,X]} \right) Y \otimes Y.
	\end{align}
	But if we expand the iterated covariant derivatives,
	\begin{align}
		\nabla_W \left( \nabla_X \left( Y \otimes Y \right) \right) &=
		\nabla_W \left( X^\alpha \frac{\alpha!}{\alpha_1!\alpha_2!} \nabla_{\alpha_1} Y \otimes \nabla_{\alpha_2} Y \right) \nonumber \\
		&= W^\beta \frac{\beta!}{\beta_0!\beta_1!\beta_2!} X^\alpha \frac{\alpha!}{\alpha_1!\alpha_2!} \nabla_{\beta_1}\nabla_{\alpha_1} Y \otimes \nabla_{\beta_2}\nabla_{\alpha_2} Y.
	\end{align}
	Adding the term with factors in the reverse order we will end up on the right-hand side with
	\begin{align}
		\left( \nabla_W \nabla_X - \nabla_X \nabla_W \right) & Y \otimes Y =
		\left( \nabla_W \nabla_X Y - \nabla_X \nabla_W Y \right) \otimes Y +
		Y \otimes \left( \nabla_W \nabla_X Y - \nabla_X \nabla_W Y \right) + \nonumber \nonumber \\
		& \left( W \circ X \circ - X \circ W \circ \right)\bigg|_{\alpha_1+\beta_1 \ne 0 ~\mathrm{and}~ \alpha_2+\beta_2 \ne 0}  \left( Y \otimes Y \right) \nonumber \\
		=& \left( \nabla_W \nabla_X Y - \nabla_X \nabla_W Y - W \circ X \circ Y + X \circ W \circ Y \right) \otimes Y + \nonumber \nonumber \\
		& Y \otimes \left( \nabla_W \nabla_X Y - \nabla_X \nabla_W Y - W \circ X \circ Y + X \circ W \circ Y \right) + \nonumber \nonumber \\
		&\left( W \circ X \circ - X \circ W \circ \right) Y \otimes Y \nonumber \\
		=& \left( \nabla_W \nabla_X Y - \nabla_X \nabla_W Y \right) \otimes Y +
		Y \otimes \left( \nabla_W \nabla_X Y - \nabla_X \nabla_W Y \right),
	\end{align}
	where we have used symmetry of the Levi-Civita connection and lemma \ref{defect_symm}. Similarly,
	\begin{align}
		\nabla_{[W,X]} \left( Y \otimes Y \right) =& \left( \nabla_{[W,X]} Y \right) \otimes Y + Y \otimes 
		\nabla_{[W,X]} \left( Y \otimes Y \right) + \nonumber \\
		&[W,X]^\alpha \frac{\alpha!}{\alpha_1!\alpha_2!}\bigg|_{\alpha_{1,2}\ne 0} \nabla_{\alpha_1} Y \otimes \nabla_{\alpha_2} Y
	\end{align}
	and the last term on the right-hand side may be expressed as
	\begin{align}
		[W,X]^\alpha \frac{\alpha!}{\alpha_1!\alpha_2!}\bigg|_{\alpha_{1,2}\ne 0} \nabla_{\alpha_1} Y \otimes \nabla_{\alpha_2} Y &=
		\nabla_{\nabla_W X - \nabla_X W}\bigg|_{\alpha_{1,2}\ne 0} Y \otimes Y \nonumber \\
		&= \left( W \circ X \circ - X \circ W \circ \right)\bigg|_{\alpha_{1,2}\ne 0} Y \otimes Y \nonumber \\
		&= \left( W \circ X \circ - X \circ W \circ \right) Y \otimes Y \nonumber \\
		&= 0.
	\end{align}
	Therefore, putting our results together we may conclude that
	\begin{align}
		\mathrm{tr}~\mathrm{tr}~ g\otimes & \left( \left( \nabla_W \nabla_X Y - \nabla_X \nabla_W Y - \nabla_{[W,X]} Y \right) \otimes Y +
		Y \otimes \left( \nabla_W \nabla_X Y - \nabla_X \nabla_W Y - \nabla_{[W,X]} Y \right) \right) \nonumber \\ &= \langle R(W,X)Y,Y \rangle + \langle Y, R(W,X)Y \rangle = 0,
	\end{align}  
	as was to be shown.
	
	To prove (4), observe that by definition it is equivalent to
	\begin{equation}
		R(W,X)Y + R(X,Y)W + R(Y,W)X = 0.
	\end{equation}
	But from the original definition of the curvature endomorphism together with the symmetry of the Levi-Civita jet connection, we have
	\begin{align}
		\nabla_W \nabla_X Y - &\nabla_X \nabla_W Y - \nabla_{[W,X]} Y +
		\nabla_X \nabla_Y W - \nabla_Y \nabla_X W - \nabla_{[X,Y]} W + \nonumber \nonumber \\
		&\nabla_Y \nabla_W X - \nabla_W \nabla_Y X - \nabla_{[Y,W]} X
		\nonumber \\ 
		=& \nabla_W \left( \nabla_X Y - \nabla_Y X \right) + \nabla_X \left( \nabla_Y W - \nabla_W Y \right) + \nabla_Y \left( \nabla_W X - \nabla_X W \right) - \nonumber \nonumber \\
		&\nabla_{[W,X]} Y - \nabla_{[X,Y]} W - \nabla_{[Y,W]} X \nonumber \\
		=&
		\nabla_W [X,Y] + \nabla_X [Y,W] + \nabla_Y [W,X] - \nabla_{[W,X]} Y - \nabla_{[X,Y]} W - \nabla_{[Y,W]} X \nonumber \\
		=& [W,[X,Y]] + [X,[Y,W]] + [Y,[W,X]].
	\end{align}
	where appeal to proposition \ref{defect_symm} has been made in going from the second to the third lines and again from the third to the fourth lines. The last line vanishes identically in virtue of the Jacobi identity, which holds for any algebra of linear operators and thus is just as valid in the presence of higher-order terms as it is for 1-vector fields.
	
	Finally, to show (3), write out the cyclic permutations of (4):
	\begin{align}
		Rm(W,X,Y,Z) + Rm(X,Y,W,Z) + Rm(Y,W,X,Z) &= 0 \nonumber \\
		Rm(X,Y,Z,W) + Rm(Y,Z,X,W) + Rm(Z,X,Y,W) &= 0 \nonumber \\
		Rm(Y,Z,W,X) + Rm(Z,W,Y,X) + Rm(W,Y,Z,X) &= 0 \nonumber \\
		Rm(Z,W,X,Y) + Rm(W,X,Z,Y) + Rm(X,Z,W,Y) &= 0 
	\end{align}
	Sum the four lines. Apply (2) to the first two columns and (1) together with (2) to the third column. After cancellation, one is left with
	\begin{equation}
		2 Rm(Y,W,X,Z) - 2 Rm(X,Z,Y,W) = 0
	\end{equation}
	which is tantamount to (3).
\end{proof}

\begin{proposition}[Second Bianchi identity; cf. Lee \cite{lee_riemannian_manif}, Proposition 7.5]\label{second_bianchi}
	The total covariant derivative of the generalized curvature tensor satisfies the following identity:
	\begin{equation}
		\nabla Rm(X,Y,Z,V,W) + \nabla Rm(X,Y,V,W,Z) + \nabla Rm(X,Y,W,Z,V) = 0.
	\end{equation}
\end{proposition}
\begin{proof}
	As Lee observes, by the symmetries of proposition \ref{first_bianchi} the statement is equivalent to
	\begin{equation}\label{alternate_form_second_bianchi}
		\nabla Rm(Z,V,X,Y,W) + \nabla Rm(V,W,X,Y,Z) + \nabla Rm(W,Z,X,Y,V) = 0
	\end{equation}
	for which in turn it will be enough to to establish that
	\begin{equation}\label{curv_plus_cyclic}
		\langle R(W,Z)\nabla_V X + R(X,V)\nabla_W X + R(V,W)\nabla_Z X, Y \rangle = 0.
	\end{equation}
	For if we use this relation substituting $W^\alpha\nabla_{\alpha_1}$ for $\nabla_W$ and $\nabla_{\alpha_2}Y$ for $Y$ along with cylic permutations of $W, Z, V$, we find that in (invoking compatibility with the metric)
	\begin{align}
		\nabla_W Rm(Z,V,X,Y) &= \nabla_W \langle R(Z,V)X,Y \rangle \nonumber \\
		&= W^\alpha \frac{\alpha!}{\alpha_1!\alpha_2!} \langle \nabla_{\alpha_!} \nabla_Z \nabla_V X - \nabla_{\alpha_1} \nabla_V \nabla_Z X, \nabla_{\alpha_2} Y \rangle
	\end{align}
	along with its cyclic permutations, the respective terms cancel among themselves by reason of the claimed relation \ref{curv_plus_cyclic}. Then we would indeed have justified Equation (\ref{alternate_form_second_bianchi}).
	
	Now for the claimed relation, write it out as
	\begin{align}
		&\nabla_W \nabla_Z \nabla_V - \nabla_Z \nabla_W \nabla_V - \nabla_{[W,Z]} \nabla_V + \nonumber \\
		&\nabla_Z \nabla_V \nabla_W - \nabla_V \nabla_Z \nabla_W - \nabla_{[Z,V]} \nabla_W + \nonumber \\
		&\nabla_V \nabla_W \nabla_Z - \nabla_W \nabla_V \nabla_Z - \nabla_{[V,W]} \nabla_Z
	\end{align}
	But $\nabla_{[W,Z]} \nabla_V = \nabla_{\nabla_W Z - \nabla_Z W} \nabla_V =
	W \circ Z \circ V \circ - Z \circ W \circ V \circ = 0$ by proposition \ref{defect_symm}. Thus, the terms involving commutators drop out and we are left with just
	\begin{align}
		&\nabla_W \nabla_Z \nabla_V - \nabla_Z \nabla_W \nabla_V + \nonumber \\
		&\nabla_Z \nabla_V \nabla_W - \nabla_V \nabla_Z \nabla_W + \nonumber \\
		&\nabla_V \nabla_W \nabla_Z - \nabla_W \nabla_V \nabla_Z 
	\end{align}
	which can be regrouped as
	\begin{align}
		&\nabla_V \nabla_W \nabla_Z - \nabla_V \nabla_Z \nabla_W + \nonumber \\
		&\nabla_W \nabla_Z \nabla_V - \nabla_W \nabla_V \nabla_Z + \nonumber \\
		&\nabla_Z \nabla_V \nabla_W - \nabla_Z \nabla_W \nabla_V.
	\end{align}
	In other words,
	\begin{align}
		&\nabla Rm(Z,V,X,Y,W) + \nabla Rm(V,W,X,Y,Z) + \nabla Rm(W,Z,X,Y,V) = \nonumber \\
		&\nabla Rm(W,Z,X,Y,V) + \nabla Rm(Z,V,X,Y,W) + \nabla Rm(V,W,X,Y,Z) = \nonumber \\
		-&\nabla Rm(W,Z,Y,X,V) - \nabla Rm(Z,V,Y,X,W) - \nabla Rm(V,W,Y,X,Z), 
	\end{align}
	making use of the first Bianchi identity \ref{first_bianchi} to go to the last line. Hence, transferring the right-hand side to the left-hand side we have
	\begin{equation}
		\nabla Rm(Z,V,X+Y,X+Y,W) + \nabla Rm(V,W,X+Y,X+Y,Z) + \nabla Rm(W,Z,X+Y,X+Y,V) = 0
	\end{equation}
	since, as we saw in the proof of proposition \ref{first_bianchi}, all expressions of the form $Rm(A,B,C,C)$ vanish identically. This step completes the proof.
\end{proof}

\begin{remark}
	Lee prefers to prove the second Bianchi identity by appealing to normal coordinates, which are not available to us. Therefore, we employ a brute-force approach which, however, does not seem to end up being as tedious as Lee suggests it would (moreover, our calculation applies to any order!).
\end{remark}

\begin{remark}
	Sectional curvature cannot as yet be defined as we do not have an analogue of the exponential map; for same reason, neither are we ready as yet to prove the property that essentially characterizes flatness in ordinary Riemannian geometry, viz., Lee, \cite{lee_riemannian_manif}, Theorem 7.3, to the effect that a Riemannian space is flat if and only if its curvature tensor vanishes identically. Clearly, to check whether anything analogous to this result holds at higher order would be a most desirable objective of future research in order to arrive at a complete conceptual understanding of what the curvature endomorphism of definition \ref{def_curv} means geometrically. Until then, the use of the term curvature indicates a hope, not confirmed knowledge, that the endomorphism so defined will capture our intuitive notion of what curvature is.
\end{remark}

\begin{remark}
	As they stand, equations (\ref{Levi_Civita_formula}) for the Levi-Civita jet connection and (\ref{def_curvature_endomorphism}) for the Riemannian curvature endomorphism are formal expressions only in the case of infinitesimals of indefinitely high order. One needs to define the latter for vector fields of indefinitely high order (lying in the direct limit $\mathscr{J}^\infty(M)$) because otherwise, if one were to truncate at any finite $r \ge 2$, the proof of tensoriality would not go through. For the commutator of a vector field of order $r_1$ with one of order $r_2$ will be of order $r_1+r_2-1$ in general. Thus, we wish to return to this question below in {\S}\ref{riemannian_appendix} where we shall exhibit a sufficient condition for the convergence of all power series involving sums over multi-indices of all orders.
\end{remark}

\subsection{Calculations in Example Spaces}\label{example_spaces}

In this section we shall flesh out the theory of spaces equipped with a generalized Riemannian metric by way of a handful of the simplest examples. The procedure of incorporating higher jets is by no means uniquely determined. Nevertheless, one can be guided by theoretical desiderata. First, one wants the generalized metric to have full rank. Second, some care must be devoted to the choice of coordinates. Indeed, in {\S}\ref{chapter_5} we introduce what we term uniformizing coordinates, the reasons for which will be spelled out there. These enjoy the property that the metric tensor is diagonal when written in terms of them. But from a mathematical point of view, any given coordinate frame is just as good as any other in principle. Thus, without further ado we embark on a few calculations with them. For the sake of notational simplicity, let us start out in two-dimensional space and denote the primary coordinates by $x$ and $y$. Then the generalized Riemannian metric assumes the form,
\begin{equation}\label{Riemannian_metric_2d}
	\delta = (d^{x})^2 + (d^{y})^2 + (d^{xx})^2 + 2 (d^{xy})^2 + (d^{y})^2.
\end{equation}
Let us now show how to go into polar coordinates:
\begin{equation}
	x = r \cos \theta \qquad y = r \sin \theta.
\end{equation}
Referring to equation (\ref{jet_transf_law}), we may write
\begin{align}
	d^x &= x_{,r} d^r + x_{,\theta} d^\theta + \frac{1}{2} \left( x_{,rr} d^{rr} + 2 x_{,r\theta} d^{r\theta} + x_{,\theta\theta} d^{\theta\theta} \right) \nonumber \\
	&= \cos \theta d^r - r \sin \theta d^\theta - \sin \theta d^{r\theta} - \frac{1}{2} r \cos \theta d^{\theta\theta} \nonumber \\
	d^y &= y_{,r} d^r + y_{,\theta} d^\theta + \frac{1}{2} \left( y_{,rr} d^{rr} + 2 y_{,r\theta} d^{r\theta} + y_{,\theta\theta} d^{\theta\theta} \right) \nonumber \\
	&= \sin \theta d^r + r \cos \theta d^\theta + \cos \theta d^{r\theta} - \frac{1}{2} r \sin \theta d^{\theta\theta}.
\end{align} 
Substitute into equation (\ref{Riemannian_metric_2d}), expand and collect terms. Here, it is best to proceed by steps since the full expression is not very illuminating. Therefore, let
\begin{equation}
	\delta^{(1)} = (d^{x})^2 + (d^{y})^2 \qquad \delta^{(2)} = (\delta^{(1)})^2.
\end{equation}
The formula for $\delta^{(1)}$ by itself is manageable:
\begin{equation}\label{Riemannian_metric_2d_delta1}
	\delta^{(1)} = (d^r)^2 - r d^r d^{\theta\theta} + r^2 \left( (d^\theta)^2 + 
	\frac{1}{4} (d^{\theta\theta})^2 \right) + d^{r\theta} \left( d^{r \theta} + 2 r d^\theta \right).
\end{equation} 
The key to going further in order to reduce the generalized metric in polar coordinates into diagonal form is to find an orthonormal basis in the space of 2-jets. From equation (\ref{Riemannian_metric_2d_delta1}), one sees that we may take as the first two elements of our orthonormal basis the following:
\begin{align}\label{polar_jet_basis_1}
	e^r &= d^r - \frac{1}{2} r d^{\theta\theta} \nonumber \\
	e^\theta &= r d^\theta + d^{r\theta},
\end{align}
in terms of which
\begin{align}
	\delta^{(1)} &= (e^r)^2 + (e^\theta)^2 \nonumber \\
	\delta^{(2)} &= (e^r)^4 + 2 (e^r)^2 (e^\theta)^2 + (e^\theta)^4 = (e^{rr})^2 + 2 e^{rr} e^{\theta\theta} + (e^{\theta\theta})^2,
\end{align}
wherein we are entitled to neglect jets of third and higher order and write the remaining three orthonormal basis elements as follows:
\begin{align}\label{polar_jet_basis_2}
	e^{rr} &= d^{rr} \nonumber \\
	e^{r\theta} &= r d^{r\theta} \nonumber \\
	e^{\theta\theta} &= r^2 d^{\theta\theta}.
\end{align}
If now we make the radial coordinate constant, $r=1$, and eliminate all radial differentials, we end up with the second-order Riemannian metric on the circle $S^1$:
\begin{equation}
	\delta_{S^1} = (d^{\theta})^2 + (d^{\theta\theta})^2.
\end{equation}
With constant coefficients in front of the differentials, it is immediate that the generalized Riemannian curvature tensor must vanish identically on the circle.

Our next illustration will be a surface of revolution $P$ about the $z$-axis in three-dimensional space, given by a function $r=h(z)$. Working once more in the uniformizing frame it readily follows that the Riemannian metric in cylindrical coordinates assumes the form,
\begin{equation}\label{Riemannian_metric_3d_cyclindrical}
	\delta_{\vvmathbb{R}^3} = (e^r)^2 + (e^\theta)^2 + (e^z)^2 + (e^{rr})^2 + (e^{\theta\theta})^2 + (e^{zz})^2 + 2 (e^{r\theta})^2 + 2 (e^{rz})^2 + 2 (e^{\theta z})^2,
\end{equation}
with the orthonormal basis jet in the vertical direction just $e^z = d^z$. Impose the defining condition of the surface of revolution as $r=h(z)$ whence
\begin{align}
	d^r &= h^\prime d^z \nonumber \\
	d^{rr} &= h^{\prime\prime} d^z + h^\prime d^{zz}.
\end{align}
After substitution and elimination of 3-jets, equation (\ref{Riemannian_metric_3d_cyclindrical}) yields
\begin{align}
	\delta_P = & (1+h^{\prime 2}) (d^z)^2 + h^2 (d^\theta)^2 + (1+h^{\prime 2})^2 (d^{zz})^2 + h^2(1+h^2) (d^{\theta\theta})^2 + \left( h^2 (1 + 2 h^{\prime 2}) + 4 h^{\prime 2} \right) (d^{z\theta})^2 \nonumber \\
	& + 2 h d^{zz} d^{\theta\theta} + 2 h^\prime h d^z d^{\theta\theta}.
\end{align}
This expression may be diagonalized with the following orthonormal jet basis:
\begin{align}\label{cylindrical_jet_basis}
	e^z &= \left( 1 + h^{\prime 2} \right)^{1/2} d^z + \frac{h^\prime h}{\left( 1 + h^{\prime 2} \right)^{1/2}} d^{\theta\theta} \nonumber \\
	e^\theta &= h d^\theta \nonumber \\
	e^{zz} &= \left( 1 + h^{\prime 2} \right) d^{zz} \nonumber \\
		e^{z\theta} &= \left( h^2 (1+2 h^{\prime2}) + 4 h^{\prime^2} \right)^{1/2} d^{z\theta} \nonumber \\
	e^{\theta\theta} &= \left( h^2(1+h^2) - \frac{h^{\prime 2}h^2}{1+h^{\prime 2}} \right)^{1/2} d^{\theta\theta}.
\end{align}
On a surface, there can be only a single non-vanishing curvature 2-form in the 1-jets, namely $R^1_2$. Moreover, it is given simply by the exterior derivative $\text{\th} \omega^1_2$ since 
$\omega^1_1 \wedge \omega^1_2 = \omega^1_2 \wedge \omega^2_2 = 0$.
It is not too hard to calculate in the Cartan formalism that $\omega^z_\theta = - h^\prime d^\theta$ and hence
\begin{equation}
	R^z_\theta = - h^{\prime\prime} d^z \wedge d^\theta - h^{\prime\prime\prime} d^{zz} \wedge d^\theta.
\end{equation}
Thus, the curvature 2-form contains a 2-jet correction and one may read off that when $h$ is concave ($h^{\prime\prime}>0$) it yields, as usual, a surface of negative curvature (hyperboloidal) while when $h$ is convex ($h^{\prime\prime}<0$), a surface of positive curvature (ellipsoidal), to the extent that the correction in $h^{\prime\prime\prime}$ can be neglected. One could derive from equation (\ref{cylindrical_jet_basis}) formulae for the 2-jet curvature 2-forms, but it scarcely seems it would repay the effort without an ulterior motive in mind.

Lastly, we derive the form of the Riemannian metric to second order in Euclidean three-space, represented in radial coordinates. Here,
\begin{align}
	x &= r \cos \phi \sin \theta \nonumber \\
	y &= r \sin \phi \sin \theta \nonumber \\
	z &= r \cos \theta
\end{align}
and we have to compute, for instance,
\begin{align}
	d^x = & \, x_{,r} d^r + x_{,\theta} d^\theta + x_{,\phi} d^\phi + \frac{1}{2} \left( x_{,rr} d^{rr} + x_{,\theta\theta} d^{\theta\theta} + x_{,\phi\phi} d^{\phi\phi} + \right. \nonumber \\
	& + \left. 2 x_{,r\theta} d^{r\theta} + 2 x_{,r\phi} d^{r\phi} + 2 x_{,\theta\phi} d^{\theta\phi} \right),
\end{align}
with similar expressions for $d^y$ and $d^z$. Little purpose would be served by going through the tedious intermediate steps, which are analogous to what we have done above in polar coordinates in the plane, and we merely quote the final result.\footnote{Recommended to the reader as an instructive exercise to explore how to combine in an intelligently organized fashion the welter of terms yielded by brute force.} That is, the Riemannian metric in Euclidean three-space in radial coordinates is given to second order by the following orthonormal jet basis:
\begin{align}\label{radial_jet_basis}
	e^r &= d^r - \frac{1}{2} r d^{\theta\theta} - \frac{1}{2} r \sin^2 \theta d^{\phi\phi} \nonumber \\
	e^\theta &= r d^\theta + d^{r\theta} - \frac{1}{2} r^2 \sin \theta \cos \theta d^{\phi\phi} \nonumber \\
	e^\phi &= r \sin \theta d^\phi + \sin \theta d^{r\phi} + r \cos \theta d^{\theta\phi} \nonumber \\
	e^{rr} &= d^{rr} \nonumber \\
	e^{r\theta} &= r d^{r\theta} \nonumber \\
	e^{r\phi} &= r \sin \theta \, d^{r\phi} \nonumber \\
	e^{\theta\theta} &= r^2 d^{\theta\theta} \nonumber \\
	e^{\theta\phi} &= r^2 \sin \theta \, d^{\theta\phi} \nonumber \\
	e^{\phi\phi} &= r^2 \sin^2 \theta \, d^{\phi\phi}.
\end{align}
If, as above in the case of the circle $S^1$, we set $r=1$ and cancel all radial differentials, we get the Riemannian metric on the sphere $S^2$ to second order:
\begin{equation}
	\delta_{S^2} = (d^{\theta}-\frac{1}{2}\sin\theta\cos\theta d^{\phi\phi})^2 + (\sin \theta d^{\phi}+\cos\theta d^{\theta\phi})^2 + (d^{\theta\theta })^2 + 
	2 \sin^2 \theta (d^{\theta\phi})^2 + \frac{5}{4} \sin^4 \theta (d^{\phi\phi})^2. 
\end{equation}
An easy problem, once given $\delta_{S^2}$, is to compute the 1-jet curvature 2-form,
\begin{align}
	R^\theta_\phi &= \left( \sin \theta d^\theta + \cos \theta d^{\theta\theta} + \sin \theta \sin 2 \theta d^{\phi\phi} \right) \wedge \left( d^\phi - d^{\theta\phi} \right) \nonumber \\
	&= e^\theta \wedge e^\phi + ~\mathrm{h.o.},
\end{align}
in which one recognizes the constant positive curvature in the first term accompanied now by a second-order correction, as expected.

\subsection{Exterior Differential Calculus at Higher Order}

Recall that in the ordinary calculus on differentiable manifolds there exists a concept of a smooth 1-form, which is simply a device for organizing differentials in the coordinate functions across various positions in space in an intrinsic manner. A given smooth 1-form $\omega$ could possibly be written as a complete differential of a smooth function $f$, or $\omega = df$, but not always. The necessary condition for this to be the case is that in every coordinate chart $x^{1,\ldots,n}$ its components satisfy the relation
\begin{equation}\label{consistency_cond_first_order}
	\frac{\partial\omega_\mu}{\partial x^\nu} - \frac{\partial\omega_\nu}{\partial x^\mu} = 0
\end{equation}
identically. Since the property of being closed can without much difficulty be shown to be coordinate-invariant (cf. Lee \cite{lee_riemannian_manif}, Proposition 11.45), one naturally expects to be able to ascribe an intrinsic meaning to antisymmetrical expressions of such form, and indeed the theory of differential forms of arbitrary degree does just this. Observe, however, that there is no reason to stop at differentials of the first order, as is conventionally done. Clearly, in terms of concepts already herein introduced, we could regard an $r$-jet field $\omega \in \mathscr{J}^{r*}$ as again a 1-form and understand it to be a complete differential of a smooth function $f \in C^\infty(M)$, $\omega=\omega_\alpha d^\alpha=df$, if its components assume the form $\omega_\alpha=\partial_\alpha f$ for all multi-indices $\alpha$. But equality of mixed partial derivatives of smooth functions holds for all pairs of multi-indices, $\partial_\beta \partial_\alpha f = \partial_\alpha \partial f$ and therefore we arrive at the analogous necessary condition that
\begin{equation}\label{consistency_cond_any_order}
	\partial_\beta \omega_\alpha = \partial_\alpha \omega_\beta,
\end{equation}
where the indicated relation must hold for all pairs of multi-indices $\alpha, \beta, 1 \le |\alpha| \le r, 1 \le |\beta| \le r$. Equation (\ref{consistency_cond_any_order}) merely reproduces equation (\ref{consistency_cond_first_order}) in greater generality. We thus have every right to anticipate that the full calculus of differential forms on manifolds should admit a natural extension to include infinitesimals of higher than first order.

Towards the stated end, some elementary definitions. On a differentiable manifold $M$, let us consider the alternating generalized covariant tensors, to be denoted
\begin{equation}
	\bigwedge^k J^{r*}M := \coprod_{p \in M} \bigwedge^k \left( J_p^{r*}M \right).
\end{equation}
After arguments similar to those advanced in {\S}\ref{chapter_2}, $\bigwedge^k J^{r*}M$ forms a smooth subbundle of  the vector bundle of covariant $\binom{k}{0}$-tensors $\overbrace{J^{r*}M \otimes \cdots \otimes J^{r*}M}^{k ~\mathrm{times}}$ on $M$. A section of $\bigwedge^k J^{r*}M$ will be referred to as a generalized differential $k$-form of degree $k$ and the vector space of smooth $k$-forms denoted by
\begin{equation}
	\Omega^k(M) := \Gamma\left( \bigwedge^k J^{r*}M \right).
\end{equation}
As would be conventional, define the wedge product of two differential forms pointwise: $(\omega \wedge \eta)_p := \omega_p \wedge \eta_p$. Hence, the wedge product renders the total space of differential forms
\begin{equation}\label{total_diff_form}
	\Omega^*(M) := \bigoplus_{k=0}^\infty \Omega^k(M),
\end{equation}
where $\Omega^0(M)=C^\infty(M)$ and a wedge product involving a zero-form is to be interpreted as $f \wedge \eta = f \eta$, into an associative, anti-commutative graded algebra. Note the new feature here that there is, as yet, no evident manner of speaking of a top-degree form since $J_p^{r*}$ has dimension greater than $n = \dim M$. Thus, we have extended the summation in equation (\ref{total_diff_form}) to infinity, where of course above a certain limit (depending on the dimension $n$ of space and on the maximal order $r$ of the differentials allowed) all of the wedge products must vanish.

Some further notation. Once $k$-forms with $k \ge 2$ have been introduced, we shall have to entertain collections of multi-indices and the following formalism seems to be expedient: let us call a grand multi-index $A$ the collection $A = (\alpha_1,\ldots,\alpha_k)$ where each $\alpha_{1,\ldots,k}$ is itself a multi-index $\alpha_i = (\alpha_{i1},\ldots,\alpha_{in})$ and the set of such multi-indices is ordered lexicographically. Then a differential $k$-form may be written locally in any smooth chart as
\begin{equation}
	\omega = \sum_A^\prime \omega_A d^A = \sum_{A:~ \alpha_1 < \cdots < \alpha_k} \omega_A d^A,
\end{equation}
where the coefficients $\omega_A$ are smooth functions defined on the chart and $d^A$ stands as an abbreviation for $d^{(\alpha_1,\ldots,\alpha_k)} = d^{\alpha_1} \wedge \cdots \wedge d^{\alpha_k}$.

What makes differential forms into a powerful tool in the context of the ordinary calculus on manifolds is that they admit a natural differential operator known as the exterior derivative. So as to keep the exposition clear and in view of the fact that we are about to introduce a significant new concept in the foundations of geometry, we wish to resort to a terminological innovation. As we saw above, the complete differential of a smooth function $f$, $df$, as previously introduced in {\S}\ref{chapter_2} corresponding to the operation of sending the function around a given point $p$ to its equivalence class $[f-f(p)]$ modulo $\mathfrak{m}_p^{r+1}$, could be framed under the aspect of being the first instance of a general procedure for operating upon differential forms of any degree. Cartan originally invented the concept of the exterior derivative in order to capture the notion of closedness in a neat fashion that is inherently intrinsic in the sense of being independent of all choices of local coordinate chart. We should like then to do the same here, permitting now the infinitesimals of higher order to enter the picture. If we are to avoid the confusion of having too many senses of the differential in concurrent use within a single formula, it seems best define the exterior derivative of a generalized 1-form $\omega$ as the following 2-form:
\begin{equation}\label{exterior_deriv_of_1_form}
	\text{\th} \omega := \frac{1}{2} \left( \partial_\alpha \omega_\beta - \partial_\beta \omega_\alpha \right) d^\alpha \wedge d^\beta,
\end{equation}
where the implicit summation extends over all multi-indices $\alpha, \beta$ and the factor of one-half takes care of the repetition of terms under interchange of the two.\footnote{The new symbol appearing in equation (\ref{exterior_deriv_of_1_form}), $\text{\th}$ [thorn], corresponds to a letter in Old English that happens still to be available for our use since it dropped out of circulation during the late medieval period. Another defunct Old English letter $\eth$ [eth] has been coopted by physicists to denote the Dirac operator.} Now we should like to define along the same lines the exterior derivative of a generalized $k$-form $\omega \in \Omega^k(M)$ as the following $(k+1)$-form:
\begin{align}\label{def_exterior_deriv}
	\text{\th} \omega &= \text{\th} \left( \sum^\prime \omega_A d^A \right) := \sum^\prime \left( d \omega_A \right) d^A \nonumber \\
	&= \sum^\prime_A \sum_\beta \partial_\beta \omega_A d^\beta \wedge d^{\alpha_1} \wedge \cdots \wedge d^{\alpha_k}.
\end{align}
Thus, acting on a 1-form we obtain
\begin{align}
	\text{\th} \left( \omega_\alpha d^\alpha \right) &= \sum_{\alpha, \beta} \partial_\beta \omega_\alpha d^\beta \wedge d^\alpha \nonumber \\
	&= \sum_{\beta < \alpha} \partial_\beta \omega_\alpha d^\beta \wedge d^\alpha +
	\sum_{\beta > \alpha} \partial_\beta \omega_\alpha d^\beta \wedge d^\alpha \nonumber \\
	&= \sum_{A=(\alpha,\beta):~ \alpha < \beta}^\prime \left( \partial_\alpha \omega_\beta - \partial_\beta \omega_\alpha \right) d^A = \text{\th} \omega
\end{align}
in agreement with equation (\ref{exterior_deriv_of_1_form}).

\begin{lemma}[Cf. Lee, Lemma 14.16]\label{pullback_on_forms}
	Let $F: M \rightarrow N$ be a smooth map from the differentiable manifold $M$ into the differentiable manifold $N$. Then
	\begin{itemize}
		\item[$(1)$] $F^*: \Omega^k(N) \rightarrow \Omega^k(M)$ is linear over $\vvmathbb{R}$;
		
		\item[$(2)$] $F^*(\omega \wedge \eta) = (F^*\omega) \wedge (F^*\eta)$;
		
		\item[$(3)$] In any smooth chart,
		\begin{equation}
			F^* \left( \sum_A^\prime \omega_A d^{y_A} \right) = \sum_A^\prime \left( \omega_A \circ F \right) d^{(y \circ F)_A},
		\end{equation}
		where in an evident notation by $d^{y_A}$ we mean $d^{y_{\alpha_1}}\cdots d^{y_{\alpha_n}}$ and likewise by $d^{(y \circ F)_A}$, $d^{y_{\alpha_1} \circ F} \cdots d^{y_{\alpha_n}\circ F}$ 	
	\end{itemize}
\end{lemma}
\begin{proof}
	For \textit{any} covariant tensor field evaluated on the vectors $\varv_{1,\ldots,k}$ we have
	\begin{align}\label{pull_push_defs}
		\left( F^* \omega \right)_p ( \varv_1,\ldots,\varv_k) &= \omega_{F(p)} \left( F_{*p} \varv_1,\ldots, F_{*p} \varv_k \right) \\
		\left( F_{*p} \varv \right) (g) &= \varv(g \circ F).
	\end{align}
	Clearly, restriction to alternating tensor fields presents no problem and moreover respects the pointwise wedge product, so (1) and (2) are immediate. But (3) follows right away as well. For it suffices by (1) and (2) to show it only for 1-forms $f d^\alpha$. Then the statement $F^*(f d^\alpha) = (f \circ F) d^{(y \circ F)_\alpha}$ is just a repackaging of what equation (\ref{pull_push_defs}) defines.
\end{proof}

Before we can proceed to define a higher-order exterior derivative (first in $\vvmathbb{R}^n$), there is a minor complication that requires attention. The ordinary exterior derivative $d$ satisfies a graded derivation property with respect to the wedge product of differential forms and we should like any higher-order exterior derivative $\text{\th}$ to extend this property in a suitable fashion. In particular, terms containing derivatives with respect to multi-indices $A, |A| \ge 2$ ought to implement something like the generalized Leibniz product rule. One runs into a difficulty at this point, however, since pushing ahead blindly one soon encounters a clash between the anti-commutativity of the wedge product and the commutativity of differentiation with respect to differing independent coordinate directions. To get around the problem, we shall introduce another notion of a product involving two differential forms in addition to the wedge product. To distinguish the two, let us denote the former by $\vee$ (vee) and the latter by $\wedge$ (wedge).

Since derivatives with respect to differing multi-indices act independently, let us break down our candidate exterior derivative as follows:
\begin{equation}\label{definition_of_thorn}
	\text{\th} \omega := \sum_{\alpha > 0} \text{\th}_\alpha \omega,
\end{equation}
where
\begin{equation}
	\text{\th}_\alpha := d^\alpha \wedge \partial_\alpha \cdot,
\end{equation}
it being understood that the partial derivative acts first to the right on coefficients of basis co-vectors followed by taking the wedge product of the resultant with the basis co-vector $d^\alpha$ from the left. Now, any multi-index $\alpha > 0$ can potentially be broken down into a sum $\alpha = \alpha_1 + \alpha_2$, where the decomposition makes sense only in the case when both $\alpha_1$ and $\alpha_2$ are non-zero. Acting by itself, $\alpha$ gives rise to a term in $d^\alpha$, but clearly we cannot rewrite this as $d^\alpha = d^{\alpha_1} \wedge d^{\alpha_2}$, as superficially one might expect we would need to in order to obtain the Leibniz product rule. For the coefficient $\partial_\alpha \omega_A = \partial_{\alpha_1+\alpha_2} \omega_A$ is symmetric under interchange of $\alpha_1$ and $\alpha_2$ while $d^{\alpha_1} \wedge d^{\alpha_2}$ could be symmetric or antisymmetric, depending on how many indices have to be commuted past each other. If antisymmetric though, the net result would cancel when contracted against $\partial_{\alpha_1+\alpha_2}\omega_A$ and such cancellation need not necessarily hold in the correct formula for the exterior derivative of the wedge product of two differential forms.

The problem is easily remedied, however. The exterior derivative is supposed to raise the form degree by one, but if we want to express it as a sum over terms in $\text{\th}_{\alpha_1}$ and $\text{\th}_{\alpha_2}$ we require an operation that will lower the form degree by one on the right-hand side so that the increased degree coming from $\text{\th}_{\alpha_1}$ and that coming from $\text{\th}_{\alpha_2}$ will combine to produce something with only one added degree. Therefore, given two differential forms $\omega_{1,2} \in \Omega^{k_{1,2}}(M)$, define their vee product on monomials
\begin{align}\label{define_vee_product}
	\bigg( f d^{a_1} \wedge \cdots \wedge d^{a_{k_1}} \bigg) \vee &
	\bigg( g d^{b_1} \wedge \cdots \wedge d^{b_{k_2}} \bigg) :=  \nonumber \\
	& \sum_{1 \le j_1 \le k_1, 1 \le j_2 \le k_2} (-1)^{j_1+j_2}
	fg d^{a_{j_1}+b_{j_2}} \wedge \nonumber \\
	& d^{a_1} \wedge \cdots \wedge 
	\widehat{d^{a_{j_1}}}\wedge \cdots 
	d^{a_{k_1}} \wedge
	d^{b_{2}} \wedge \cdots \wedge \widehat{d^{b_{j_2}}} \wedge 
	\cdots \wedge d^{b_{k_2}}
\end{align}
and extend by linearity. Here, we take grand multi-indices $A=(a_1,\ldots,a_{k_1})$ and $B=(b_1,\ldots,b_{k_2})$ of degree $k_{1,2}$ respectively. The indicated operation is equivalent to a tensor product (not a wedge product) between $\omega_1$ and $\omega_2$ followed by symmetrization over their respective first indices and antisymmetrization over the other indices. With this we are prepared at last to state,	 

\begin{proposition}[Properties of the exterior derivative in $\vvmathbb{R}^n$; cf. Lee, Proposition 14.23]\label{properties_exterior_deriv}
	The exterior derivative satisfies the following properties:
	
	\begin{itemize}
		\item[$(1)$] $\text{\th}$ is linear over $\vvmathbb{R}$;
		
		\item[$(2)$] If $\omega$ is a smooth $k$-form and $\eta$ is a smooth $l$-form on an open set $U \subset \vvmathbb{R}^n$, then
		\begin{equation}
			\text{\th} \left( \omega \wedge \eta \right) = \text{\th}\omega \wedge \eta
			+ (-1)^k \sum_{\alpha>0} \frac{\alpha!}{\alpha_1!\alpha_2!}\bigg|_{\alpha_{1,2} \ne 0} \text{\th}_{\alpha_1} \omega \vee \text{\th}_{\alpha_2} \eta +
			(-1)^k \omega \wedge \text{\th}\eta;
		\end{equation}
		
		\item[$(3)$] $\text{\th} \circ \text{\th} = 0$;
		
		\item[$(4)$] $\text{\th}$ commutes with pullbacks in the sense that, if $U, V \subset \vvmathbb{R}^n$, $F: U \rightarrow V$ is a smooth map and $\omega \in \Omega^k(V)$, then
		\begin{equation}
			F^* \left( \text{\th} \omega \right) = \text{\th} \left( F^* \omega \right).
		\end{equation}
		
	\end{itemize}
\end{proposition}
\begin{proof}
	(1) Trivial. For (2), it suffices by linearity to consider $\omega = u d^A$ and $\eta = \varv d^B$ for smooth functions $u, \varv$ and multi-indices $A, B$, so that
	\begin{align}
		\text{\th} \left( \omega \wedge \eta \right) &= \text{\th} \left( u d^A \wedge \varv d^B \right) \nonumber \\
		&= \left( \sum_{\alpha > 0} \text{\th}_\alpha \right)  \left( u \varv d^A \wedge d^B  \right) \nonumber \\
		&= \sum_{\alpha > 0} \frac{\alpha!}{\alpha_1!\alpha_2!} \left( \text{\th}_{\alpha_1} u \vee \text{\th}_{\alpha_2} \varv \right) d^A \wedge d^B \nonumber \\
		&= \sum_{\alpha > 0} \left( \text{\th}_\alpha u d^A \wedge \varv d^B + (-1)^k \frac{\alpha!}{\alpha_1!\alpha_2!}\bigg|_{\alpha_{1,2} \ne 0} \text{\th}_{\alpha_1} u d^A \vee \text{\th}_{\alpha_2} \varv d^B +
		(-1)^k u d^A \wedge \text{\th}_\alpha \varv d^A \right) \nonumber \\
		&= \text{\th} \omega \wedge \eta + (-1)^k \sum_{\alpha > 0} 
		\frac{\alpha!}{\alpha_1!\alpha_2!}\bigg|_{\alpha_{1,2} \ne 0} \text{\th}_{\alpha_1} \omega \vee \text{\th}_{\alpha_2} \eta
		+ (-1)^k \omega \wedge \text{\th} \eta.
	\end{align} 
	
	For (3), let $\omega \in \Omega^k(U)$. Then compute directly,
	\begin{align}
		\text{\th} \left( \text{\th} \omega \right) &= \text{\th} \left( \sum_A^\prime \sum_\beta \partial_\beta \omega_A d^\beta \wedge d^A \right) \nonumber \\
		&= \sum_{A\beta}^\prime \sum_\gamma \partial_\gamma \partial_\beta \omega_A d^\gamma \wedge d^\beta \wedge d^A \nonumber \\
		&= \sum_A^\prime \sum_{(\beta,\gamma)}^\prime \left(
		\partial_\gamma \partial_\beta \omega_A - \partial_\beta \partial_\gamma \omega_A \right) d^\gamma \wedge d^\beta \wedge d^A = 0
	\end{align}
	by reason of the equality of mixed partial derivatives.
	
	For (4), it is sufficient to show the result for $k$-forms expressible as $\omega = u d^A$, where $u$ is some smooth function on $V$. We compute:
	\begin{align}
		F^* \left( \text{\th} \left( u d^A \right) \right) &= F^* \left( \text{\th} u \wedge d^A \right) \nonumber \\
		&= \left( \text{\th} \left( u \circ F \right) \wedge d^{x^{\alpha_1}\circ F} \wedge \cdots \wedge d^{x^{\alpha_k}\circ F} \right)
	\end{align}
	whereas
	\begin{align}
		\text{\th} \left( F^* \left( u d^A \right) \right) &= \text{\th} \left( u \circ F  d^{x^{\alpha_1}\circ F} \wedge \cdots \wedge d^{x^{\alpha_k}\circ F} \right) \nonumber \\
		&= \left( \text{\th} \left( u \circ F \right) \right) \wedge d^{x^{\alpha_1}\circ F} \wedge \cdots \wedge d^{x^{\alpha_k}\circ F}.
	\end{align}
	Therefore, $F^* \left( \text{\th} \omega \right) = \text{\th} \left( F^* \omega \right)$ for $\omega = u d^A$ and extending by linearity for any $\omega \in \Omega^k(V)$.
\end{proof}

\begin{lemma}\label{commutativity_of_vee_euclidean}
	Let $\omega$ and $\eta$ be homogeneous of degree $p$ and $q$ on an open set $U \subset \vvmathbb{R}^n$, respectively. If $p$ and $q$ are both even, then $\omega \vee \eta = - \eta \vee \omega$. If either $p$ is odd and $q$ is even or $p$ and $q$ are both odd, then $\omega \vee \eta = \eta \vee \omega$.
\end{lemma}
\begin{proof}
	From equation (\ref{define_vee_product}), when $p$ and $q$ are both even, $p-1$ and $q-1$ are both odd, hence $\omega \vee \eta$ can be rendered into $\eta \vee \omega$ with an odd number of transpositions (sc., $(p-1)(q-1)$). Thus, $\omega \vee \eta = - \eta \vee \omega$. When, however, either $p$ is odd and $q$ is even or both are odd, $(p-1)(q-1)$ will be even so as to give $\omega \vee \eta = \eta \vee \omega$.		
\end{proof}

\begin{corollary}\label{vee_cancellation_lemma_euclidean}
	Let $\omega$ and $\eta$ be homogeneous of degree $p$ and $q$  on an open set $U \subset \vvmathbb{R}^n$, respectively. If either $p$ and $q$ are both odd or $p$ is odd and $q$ is even, then the cross terms in proposition \ref{properties_exterior_deriv} cancel.
\end{corollary}
\begin{proof}
	When either $p$ and $q$ are both odd or $p$ is odd and $q$ is even, then $p+1$ and $q+1$ are either both even or even and odd. Therefore, by lemma \ref{commutativity_of_vee_euclidean}, the vee products in the cross-terms are anticommutative. If we expand $\text{\th}(\omega \wedge \eta)$ and $\text{\th}(\eta \wedge \omega)$ and add the two expressions, we straightforwardly find that the cross-terms sum to zero, leaving us with
	\begin{equation}
		\text{\th}(\omega \wedge \eta) = \text{\th} \omega \wedge \eta - \omega \wedge \text{\th} \eta.
	\end{equation}	
	Observe that the cross-terms need not always thus cancel, either because one has a product of homogeneous forms of even degree or because the forms could be inhomogeneous.	
\end{proof}

\begin{theorem}[Cf. Lee, \cite{lee_smooth_manif}, Theorem 14.24]\label{exist_exterior_deriv}
	Suppose $M$ is a differentiable manifold. Then there are (unique) operators $\text{\th}: \Omega^k(M) \rightarrow \Omega^{k+1}(M)$, $k \ge 0$, to be known as exterior differentiation, satisfying the following properties:
	\begin{itemize}
		\item[$(1)$] $\text{\th}$ is linear over $\vvmathbb{R}$;
		
		\item[$(2)$] If $\omega \in \Omega^k(M)$ and $\eta \in \Omega^l(M)$, then
		\begin{equation}
			\text{\th} \left( \omega \wedge \eta \right) = \text{\th}\omega \wedge \eta
			+ (-1)^k \sum_{\alpha>0} \frac{\alpha!}{\alpha_1!\alpha_2!}\bigg|_{\alpha_{1,2} \ne 0} \text{\th}_{\alpha_1} \omega \vee \text{\th}_{\alpha_2} \eta + 
			(-1)^k \omega \wedge \text{\th}\eta;
		\end{equation}
		
		\item[$(3)$] $\text{\th} \circ \text{\th} = 0$;
		
		\item[$(4)$] For $f \in \Omega^0(M) = C^\infty(M)$, $\text{\th}f$ is the differential of $f$ given by $\text{\th}f(X)=Xf$, $X \in \mathscr{J}^r(M)$.
	\end{itemize}
\end{theorem}
\begin{proof}
	The argument is precisely parallel to that given by Lee, at the formal level where we understand the expressions as now possibly including higher-order infinitesimals. Existence can be had by defining $\text{\th} \omega$ by means of the coordinate formula \ref{def_exterior_deriv} in each chart $(U, \varphi)$ and transferring from the chart to $M$:
	\begin{equation}
		\text{\th} \omega|_U := \varphi^* \text{\th} \left( \varphi^{-1*} \omega \right).
	\end{equation}
	The procedure is well defined since we may note that in another chart $(V,\psi)$, the map $\varphi \circ \psi^{-1}$ is a diffeomorphism between open subsets in the respective copies of $\vvmathbb{R}^n$; hence by proposition \ref{properties_exterior_deriv},
	\begin{align}
		\left( \varphi \circ \psi^{-1} \right)^* \text{\th} \left( \varphi^{-1*} \omega \right) &= \text{\th} \left( \left( \varphi \circ \psi^{-1} \right)^* \varphi^{-1*} \omega \right) \nonumber \\
		&= \text{\th} \left( \left( \psi^{-1*} \circ \varphi^* \circ \varphi^{-1} \right)^* \omega \right) \nonumber \\
		&= \psi^{-1*} \varphi^* \text{\th} \left( \varphi^{-1*} \omega \right)
	\end{align}
	and so (applying $\psi^*$ to both sides from the left)
	\begin{equation}
		\varphi^* \text{\th} \left( \varphi^{-1*} \omega \right) = \psi^* \text{\th} \left( \psi^{-1*} \omega \right).
	\end{equation}
	The last line, however, asserts the independence of $\text{\th} \omega$ from the chart in which it is defined. Properties (1)-(4) now follow from the corresponding items in proposition \ref{properties_exterior_deriv}.
	
	It remains to show uniqueness. The first desideratum is to ensure that any operator $\text{\th}$ satisfying (1) to (4) is determined locally. To this end, suppose $\omega_{1,2}$ to be two smooth $k$-forms that agree on an open set $U \subset M$. Then, $\text{\th} \omega_1|_U =\text{\th} \omega_2|_U$. To see why, take $p \in U$ arbitrary and let $\eta := \omega_2 - \omega_1$. We may find a smooth bump function $\psi$ defined on all of $M$ that becomes equal to unity identically on some neighborhood of $p$ but whose support is contained in $U$. Evidently, $\psi \eta = 0$ identically by construction. Hence,
	\begin{equation}
		0 = \text{\th} \psi \eta = \text{\th} \psi \wedge \eta + (-1)^k \sum_{\alpha>0} \frac{\alpha!}{\alpha_1!\alpha_2!}\bigg|_{\alpha_{1,2} \ne 0} \text{\th}_{\alpha_1} \psi \vee \text{\th}_{\alpha_2} \eta +
		(-1)^k \psi \wedge \text{\th}\eta.
	\end{equation} 
	Evaluate at $p$ and use the facts that $\psi(p)=1$ and $\psi_{p,\beta} = 0$ for any non-zero multi-index $\beta$. One concludes that $\text{\th} \omega_2 |_p - \text{\th} \omega_1 |_p = \text{\th} \eta |_p = 0$ by commutativity of the vee product. 
	
	Let now $\omega \in \Omega^k(M)$ be arbitrary. With respect to any chart $(U,\varphi)$ around the point $p \in M$, we may write $\omega$ in terms of the corresponding coordinates. By means of an appropriate bump function $\psi$, construct global smooth functions for the components $\tilde{\omega}_A$ and new coordinates $\tilde{x}_{1,\ldots,n}$ that agree with $\omega_A$ and $x_{1,\ldots,n}$ in a neighborhood of $p$ contained in $U$. By virtue of (1) to (4) together with the argument of the preceding paragraph, we may conclude that the indicated operator $\text{\th}$ is uniquely defined at $p$, but $p$ was arbitrary, and therefore $\text{\th}$ is uniquely defined on all of $M$.
\end{proof}

\begin{lemma}\label{commutativity_of_vee}
	Let $\omega$ and $\eta$ be homogeneous of degree $p$ and $q$ on a differentiable manifold $M$, respectively. If $p$ and $q$ are both even, then $\omega \vee \eta = - \eta \vee \omega$. If either $p$ is odd and $q$ is even or $p$ and $q$ are both odd, then $\omega \vee \eta = \eta \vee \omega$.
\end{lemma}
\begin{proof}
	Immediate from lemma \ref{commutativity_of_vee_euclidean}.
\end{proof}

\begin{corollary}\label{vee_cancellation_lemma}
	Let $\omega$ and $\eta$ be homogeneous of degree $p$ and $q$ on a differentiable manifold $M$, respectively. If either $p$ and $q$ are both odd or $p$ is odd and $q$ is even, then the cross terms in theorem \ref{exist_exterior_deriv} cancel.
\end{corollary}
\begin{proof}
	Immediate from corollary \ref{vee_cancellation_lemma_euclidean}.		
\end{proof}

\subsection{Covariant Differentiation and Curvature in the Cartan Formalism}\label{cartan_formalism}

{\'E}. Cartan proposed his differential-form based approach to differential geometry, centered on the exterior algebra of differential forms and its calculus, originally as early as 1922 \cite{cartan}, cf. Thirring \cite{thirring_vol_2}, {\S}4.1. It has longed figured as a curiosity in the mathematical community, for it is known to be equivalent to the usual approaches in finitely many dimensions and therefore tends to be dismissed as redundant, despite the greater degree of clarity with which Cartan's version of the theory exhibits its inherent geometrical content. In what follows though, we shall want to have recourse to it both as a particularly elegant means of framing general relations and as a convenient technique by which to aid computations, which necessarily become more complicated once higher-order infinitesimals enter the picture. The variational principle in particular receives a quite nice formulation in these terms.

The first order of business is to express jet covariant differential in Cartan's terms, for which jet affine connections have to be introduced. Then, as expected, the Riemannian curvature tensor can readily be written as a differential 2-form along the same lines as in the ordinary 1-jet case. This form of the curvature is key to the expeditious formulation of the field equations, where other methods fail.

Let $E_\alpha, ~e^\alpha$, where the multi-index ranges over all $1 \le |\alpha| \le r$, denote a local orthonormal frame of $J^{r}$ and its dual coframe in some open subset $U \subset M$. Then, in parallel with Cartan's philosophy, we may define the jet affine connections as 1-forms as follows:
\begin{equation}
	\nabla_X E_\alpha =: \omega^\beta_\alpha(X)E_\beta,
\end{equation}
to hold for any $r$-vector field $X$. Here, $\nabla$ denotes the natural action of the Levi-Civita jet connection $\nabla$ on vectors. The existence and uniqueness of the matrix of affine connections as smooth 1-forms follow at once from the generalized tensor characterization lemma, I.3.6 above.

First of all, we will need analogue of Lee, Proposition 14.29. Then we may follow the derivation of Lee's Problems 4.5 and 7.2.

\begin{lemma}[Cf. Lee, \cite{lee_smooth_manif}, Proposition 14.29]\label{diff_of_1_form}
	For any smooth 1-form $\omega$ and smooth $r$-vector fields $X$ and $Y$, we have
	\begin{equation}
		\text{\th} \omega(X,Y) = X(\omega(Y)) - Y(\omega(X))-\omega([X,Y]).
	\end{equation}	
\end{lemma}
\begin{proof}
	We can follow Lee's derivation line by line, mutatis mutandis. By linearity, it suffices to consider the case $\omega = u \text{\th} \varv$ (as the statement is local). The left-hand side becomes
	\begin{align}
		\text{\th} (u \text{\th} \varv)(X,Y) &= \text{\th} u \wedge \text{\th} \varv (X,Y)
		+ \frac{\alpha!}{\alpha_1!\alpha_2!}\bigg|_{\alpha_{1,2} \ne 0} \text{\th}_{\alpha_1}u \wedge \text{\th}_{\alpha_2} \text{\th}\varv(X,Y) \nonumber \\
		&= \text{\th}u(X)\text{\th}\varv(Y) -  \text{\th}\varv(X)\text{\th}u(Y) 
		+ \nonumber \\
		&\frac{\alpha!}{\alpha_1!\alpha_2!}\bigg|_{\alpha_{1,2} \ne 0} \left( X_{\alpha_1} \partial_{\alpha_1}u Y_{\alpha_2+\beta}\partial_{\alpha_2+\beta}\varv -
		X_{\alpha_1} \partial_{\alpha_1}\varv Y_{\alpha_2+\beta}\partial_{\alpha_2+\beta}u \right)
		\nonumber \\
		&= X (uY \text{\th}\varv) - X(\varv Y \text{\th}u).
	\end{align}
	But the right-hand side is given by
	\begin{align}
		X(u\text{\th}\varv(Y)) &- Y(u\text{\th}\varv(X)) - u\text{\th}\varv(
		X \circ Y \circ - Y \circ X \circ) \nonumber \\
		&= X(uY\varv) - Y(uX\varv) - u[X,Y]\varv \nonumber \\
		&= (XuY\varv + \cdots + uXY\varv) - (YuX\varv+ \cdots +uYX\varv) - uXY\varv + uYX\varv) \nonumber \\
		&= (Xu)(Y\varv) + \cdots - (X\varv)(Yu) - \cdots \nonumber \\
		&=  X (uY \text{\th}\varv) - X(\varv Y \text{\th}u).
	\end{align}
	The last step completes the proof.		
\end{proof}

The 1-forms $\omega^\alpha_\beta$ are so named because they determine a differential relation among the elements of the dual coframe that expresses how curvature enters the local problem, known as Cartan's first structural equation. To find it out, start from lemma \ref{diff_of_1_form} and write,
\begin{align}
	\text{\th} e^\alpha(X,Y) &= X e^\alpha(Y) - Y e^\alpha(X) - e^\alpha([X,Y]) \nonumber \\
	&= X Y^\alpha - Y X^\alpha - e^\alpha(\nabla_XY - \nabla_YX) \nonumber \\
	&= X Y^\alpha - Y X^\alpha - X^\beta \frac{\beta!}{\beta_1!\beta_2!}Y^\gamma_{,\beta_1}\Gamma^\alpha_{\beta_2\gamma}
	+ Y^\beta
	\frac{\beta!}{\beta_1!\beta_2!}X^\gamma_{,\beta_1}\Gamma^\alpha_{\beta_2\gamma} \nonumber \\
	&= X^\beta \frac{\beta!}{\beta_1!\beta_2!}  Y^\gamma_{,\beta_1} \Gamma^\alpha_{\beta_2\gamma} +
	Y^\beta \frac{\beta!}{\beta_1!\beta_2!}  X^\gamma_{,\beta_1} \Gamma^\alpha_{\beta_2\gamma} \nonumber \\
	&= X (\omega^\alpha_\gamma(Y)) - Y(\omega^\alpha_\gamma(X))
\end{align}
but
\begin{equation}
	e^\alpha \wedge \omega^\beta_\alpha(X,Y) = X^\alpha \omega^\beta_\alpha(Y)
	- Y^\alpha \omega^\beta_\alpha(X).
\end{equation}
Therefore, 
\begin{equation}\label{def_connection_1-form}
	\text{\th} e^\alpha = e^\beta \wedge \omega^\alpha_\beta.
\end{equation}

Define the curvature 2-form by
\begin{equation}
	\Omega_\alpha^\beta := \frac{1}{2} R_{\gamma\delta\alpha}^\beta e^\gamma \wedge e^\delta.
\end{equation}
and compute,
\begin{align}
	R(E_\gamma,E_\delta) E_\alpha &= \nabla_\gamma \nabla_\delta E_\alpha - 
	\nabla_\delta \nabla_\gamma E_\alpha -
	\nabla_{[E_\gamma,E_\delta]} E_\alpha \nonumber \\
	&= \nabla_\gamma \omega^\beta_\alpha(E_\delta) E_\beta - 
	\nabla_\delta \omega^\beta_\alpha(E_\gamma) E_\beta -
	\nabla_{[E_\gamma,E_\delta]} E_\alpha \nonumber \\
	&= \omega^\beta_\alpha(E_\delta)\omega^\lambda_\beta(E_\gamma)E_\lambda + \cdots + (E_\gamma \omega^\beta_\alpha(E_\delta))E_\beta
	- (\gamma\leftrightarrow\delta) - \omega^\beta_\alpha([E_\gamma,E_\delta]) E_\beta \nonumber \\
	&= 2 \text{\th} \omega^\beta_\alpha(E_\gamma,E_\delta) - 2 (\omega^\lambda_\alpha \wedge \omega^\beta_\lambda)(E_\gamma,E_\delta))
\end{align}
by lemma \ref{diff_of_1_form}. The inner terms reinforce one another due to symmetry (i.e., write out $R(E_\gamma+E_\delta,E_\gamma+E_\delta)=0$ and collect like terms with $\gamma_{1,2} \ne 0$ and $\delta_{1,2} \ne 0$). Hence, we obtain Cartan's second structural equation:
\begin{equation}
	\Omega^\mu_\nu = \text{\th} \omega^\mu_\nu - \omega^\lambda_\nu \wedge \omega^\mu_\lambda = \text{\th} \omega^\mu_\nu + \omega^\mu_\lambda \wedge \omega^\lambda_\nu.
\end{equation}

Now that we have a formula for the curvature in terms of the generalized exterior derivative of the jet affine connection, it is natural to ask what happens when we differentiate once again. Here and in what follows we shall bow to terminological convention in the physics community and denote the curvature 2-form with the letter $R^\mu_\nu$.

\begin{proposition}[Bianchi's differential identity; cf. Thirring, \cite{thirring_vol_2}, 4.1.25]\label{bianchi_differential_id}
	The curvature 2-form obeys the following identity upon differentiation:
	\begin{equation}
		\text{\th} R^\mu_\nu = \omega^\mu_\lambda \wedge R^\lambda_\nu - R^\mu_\lambda \wedge \omega^\lambda_\nu.
	\end{equation}
\end{proposition}
\begin{proof}
	In view of Cartan's second structural equation, the derivation proceeds directly as in Thirring:
	\begin{align}
		\text{\th} \left( \text{\th} \omega^\mu_\nu + \omega^\mu_\lambda \wedge \omega^\lambda_\nu \right) &= \text{\th} \omega^\mu_\lambda \wedge \omega^\lambda_\nu - \omega^\mu_\lambda \wedge \text{\th} \omega^\lambda_\nu + \cdots \nonumber \\
		&= R^\mu_\lambda \wedge \omega^\lambda_\nu - \omega^\mu_\lambda \wedge R^\lambda_\nu - 
		\omega^\mu_\sigma \wedge \omega^\sigma_\lambda \wedge \omega^\lambda_\nu +
		\omega^\mu_\lambda \wedge \omega^\lambda_\sigma \wedge \omega^\sigma_\nu + \cdots \nonumber \\
		&= R^\mu_\lambda \wedge \omega^\lambda_\nu - \omega^\mu_\lambda \wedge R^\lambda_\nu.
	\end{align}
	The cross-terms drop out as usual due to the relative minus sign in the vee expansions of $\text{\th} \left( \omega^\mu_\lambda \wedge \omega^\lambda_\nu \right)$ and $\text{\th} \left( \omega^\lambda_\nu \wedge \omega^\mu_\lambda \right)$.
\end{proof}

\begin{lemma}[Cf. Thirring, \cite{thirring_vol_2}, 4.1.23,b)]\label{noncovariant_curvature_vanishing_lemma}
	$\Omega^\alpha_\beta \wedge e^\beta = 0$.
\end{lemma}
\begin{proof}
	Expand
	\begin{align}
		0 &= \text{\th} \text{\th} e^\alpha = \text{\th} \left( e^\beta \wedge \omega^\alpha_\beta \right) \nonumber \\
		&= \text{\th} e^\beta \wedge \omega^\alpha_\beta - e^\beta \wedge \text{\th} \omega^\alpha_\beta - \frac{\gamma!}{\gamma_1!\gamma_2!}\bigg|_{\gamma_{1,2} \ne 0} \text{\th}_{\gamma_1} e^\beta \vee \text{\th}_{\gamma_2} \omega^\alpha_\beta.  
	\end{align}
	The leading two terms give
	\begin{equation}
		\text{\th} e^\beta \wedge \omega^\alpha_\beta - e^\beta \wedge \text{\th} \omega^\alpha_\beta = e^\gamma \wedge \omega^\beta_\gamma \wedge \omega^\alpha_\beta - e^\beta \wedge \text{\th} \omega^\alpha_\beta
		= e^\beta \wedge \left( \omega^\lambda_\beta \wedge \omega^\alpha_\lambda
		- \text{\th} \omega^\alpha_\beta \right) = - e^\beta \wedge \Omega^\alpha_\beta.
	\end{equation}
	As for the remainder, it cancels due to the combination of antisymmetry under exchange of the two factors but symmetry under exchange of $\gamma_1$ and $\gamma_2$. For we can write,
	\begin{align}
		\text{\th} \left( e^\beta \wedge \omega^\alpha_\beta \right) &= -\text{\th} \left( \omega^\alpha_\beta \wedge e^\beta \right) \nonumber \\ 
		&= \cdots + \frac{\gamma!}{\gamma_1!\gamma_2!}\bigg|_{\gamma_{1,2} \ne 0} \text{\th}_{\gamma_1} \omega^\alpha_\beta \vee \text{\th}_{\gamma_1} e^\beta \nonumber \\
		&= \cdots + \frac{\gamma!}{\gamma_1!\gamma_2!}\bigg|_{\gamma_{1,2} \ne 0} \text{\th}_{\gamma_2} e^\beta \vee \text{\th}_{\gamma_1} \omega^\alpha_\beta \nonumber \\
		&= \cdots + \frac{\gamma!}{\gamma_1!\gamma_2!}\bigg|_{\gamma_{1,2} \ne 0} \text{\th}_{\gamma_1} e^\beta \vee \text{\th}_{\gamma_2} \omega^\alpha_\beta 
	\end{align}
	whereas the leading two terms are of course left unchanged:
	\begin{align}
		-\text{\th} \left( \omega^\alpha_\beta \wedge e^\beta \right) &=
		- \text{\th} \omega^\alpha_\beta \wedge e^\beta + \omega^\alpha_\beta \wedge \text{\th} e^\beta + \cdots \nonumber \\
		&= - e^\beta \wedge \text{\th} \omega^\alpha_\beta + \omega^\alpha_\beta \wedge e^\gamma \wedge \omega^\beta_\gamma + \cdots \nonumber \\
		&= - e^\beta \wedge \Omega^\alpha_\beta + \cdots.
	\end{align}
	Therefore, the terms in $\gamma_{1,2} \ne 0$ drop out, leaving the required result.
\end{proof}

\subsection{Stepwise Definition of a Generalized Hodge-* Operator}\label{def_hodge_star}

The Hodge-* operator as conventionally defined makes sense only in the exterior algebra over a finite-dimensional space, for in a suitable orthonormal basis it consists in complementation with respect to a top-degree form. If we wish, however, to include infinitesimals of indefinitely high order, then there will be no top-degree form. Therefore, let us attend to what the Hodge-* really does in differential geometry before seeking to generalize it from $\Omega^*(M;d)$ to $\Omega^*(M;\text{\th})$. The problem is to understand where the higher infinitesimals fit in, in that the jets in each stalk do not form a module of finite type over $C^\infty(M)$ anymore.

Let $\pi_{r,s}, 1 \le s < r$ denote the canonical projection from $J^{*r}(M)$ onto $J^{*s}(M)$. Proceed inductively as follows: first define an orthonormal set in $J^{*1}$ as usual, with respect to $\hat{g}_1 := \pi_{\infty,1}^* g$. But in a given coordinate system now $J^{*1} \subset J^{*2}$ non-canonically by extending the 1-jets by zero (i.e., assigning a zero coefficient to any quadratic differential). Therefore, define a further orthonormal set in the orthogonal complement of $J^{*1}$ in $J^{*2}$ (relative to $\hat{g}_2 := \pi_{\infty,2}^* g$). Likewise, for each $s \ge 1$ we have the orthogonal complement of $J^{*s}$ in $J^{*(s+1)}$. Thus, we obtain a decomposition of the entire space of jets up to indefinite order into a direct sum of orthogonal subspaces $Q_{1,\ldots,r}$, each of which is finite-dimensional. Hence, $J^{*\infty} = \bigoplus_{s=1}^\infty Q_s$.
For any given jet in $j \in J^{*\infty}$ and any given $s \ge 1$, its projection onto $Q_s$, namely $\pi_{\infty,s}j$, is uniquely defined. Let $*_s:Q_s \rightarrow Q_s$ be the Hodge-* operator defined as usual within each orthogonal component. Then we can extend these by linear combination to a well-defined operation $* := *_1 \oplus *_2 \oplus \cdots$, which continues to make sense in the inductive limit of jets of indefinitely high order. Moreover, the generalized Hodge-* so defined extends by linearity in the obvious way to the algebra of alternating forms, $\Omega^*(M;\text{\th})$ even though the latter does not admit any top-degree form.

\begin{lemma}
	The generalized Hodge-*, viewed as an operator in $\Omega^*(M;\text{\th})$, is independent of the choices made in constructing it.
\end{lemma}
\begin{proof}
	Consider two direct-sum decompositions corresponding to different choices of orthonormal bases in $J^{*1}$ and in the sequence of complements $J^{*(s+1)}/J^{*s}$, $s \ge 1$:
	\begin{equation}
		J^{*\infty} = \lim_{\leftarrow}~ \bigoplus_{s=1}^n Q_s = \lim_{\leftarrow}~ \bigoplus_{s=1}^n Q^\prime_s.
	\end{equation}
	The first summand on either side is the same, namely, $J^{*1} = Q_1 = Q^\prime_1$. For $s>1$ and considering the bundles stalkwise over each point in space-time, each $Q_s$ resp. $Q^\prime_s$ is spanned by $m_s$ linearly independent jets of order $s$ and therefore $Q_s$ and $Q^\prime_s$ are isomorphic as Hilbert spaces with respect to the inner product $\langle \alpha, \beta \rangle := *_s \left( \alpha \wedge *_s \beta \right)$ in as much as they admit orthonormal bases with the same cardinality.
	
	Now, in any finite-dimensional inner-product space the Hodge-*, though initially defined by means of a choice of orthonormal frame, turns out to be independent of this choice. Therefore, under the isomorphism we have that $*_s = *^\prime_s$ for every $s \ge 1$; hence $* = *^\prime$ as well.
\end{proof}

In order to perform computations, it will be useful to have a formula by which to express the action of the Hodge-* operator on $p$-forms (which then extends, as usual, to all of $\Omega^*$ by linearity). The point is that relative to the orthonormal basis in terms of which the Hodge-* has been defined, the metric becomes block-diagonal:
$g = \hat{g}_1 \oplus \hat{g}_2 \oplus \cdots$, where $\hat{g}_s$ denotes the restriction of $g$ to $g|_{Q_s}$ (i.e., when viewed as an operator $\mathscr{J}^\infty \rightarrow \mathscr{J}^{*\infty}$). If we further adopt Thirring's compressed notation $e^{\alpha_1\cdots\alpha_p} := e^{\alpha_1} \wedge \cdots \wedge e^{\alpha_p}$, then his formula 1.2.16 becomes for us,
\begin{equation}
	*_s e^{\alpha_1\cdots\alpha_p} = g^{\alpha_1\beta_!}\cdots g^{\alpha_p\beta_p} e^{\beta_{p+1}\cdots \beta_{m_s}} \varepsilon_{\beta_1\cdots\beta_{m_s}} \frac{\sqrt{|\hat{g}_s|}}{(m_s-p)!},
\end{equation}
where $\varepsilon$ denotes the totally antisymmetric symbol and $m_s := \dim Q_s = \dim J^s - \dim J^{s-1}$ and $|\alpha_{1,\ldots,p}|=s$. An arbitrary $p$-form $\omega$ of compact support will be reducible in the sense of uniform approximation to a linear combination of products of factors
\begin{equation}
	\omega = \omega_1 \wedge \omega_2 \wedge \cdots,
\end{equation} 
the expression on the right-hand side being open-ended to the left and to the right. Here, $\omega_s = e^{\alpha_{s_1}} \wedge \cdots \wedge e^{\alpha_{s_q}}$, $q \le p$ and the sum over all such $q$'s = $p$. The case may arise that, on the left, there may be a finite sequence with $q_1=q_2=\cdots=0$, corresponding to the zero-form $1$, and the expression will be terminated on the right by an infinite sequence of $q_s=0$, $s \ge s_0$. To be precise about what we are doing, we are provisionally limiting consideration to jets of indefinitely high but not infinite order, denote these by $\Omega_0^*(M) \subset \Omega^*(M)$.

The action of the Hodge-* operator on decomposable pieces may now be written as,
\begin{equation}
	* \left( \omega_1 \wedge \omega_2 \wedge \cdots \right) = \left( *_1 \omega_1 \right) \wedge \left( *_2 \omega_2 \right) \wedge \cdots,
\end{equation}
from which we see that for a proper duality to hold, we need to adjoin to the $\Omega_0^*(M)$ we start out with another part, $\hat{\Omega}_0^*(M)$ consisting, in this basis, of linear combinations of infinite products of the basis elements $e^\alpha$ in lexicographical order, with only finitely many having $q_s<m_s$. Then
\begin{equation}
	*: \Omega_0^*(M) \rightarrow \hat{\Omega}_0^*(M)
\end{equation}
and vice versa
\begin{equation}
	*: \hat{\Omega}_0^*(M) \rightarrow \Omega_0^*(M).
\end{equation}

\begin{lemma}
	(1) $\hat{\Omega}_0^*(M) \cong \Omega_0^*(M)$ and (2) the action of the Hodge-* as described on $\Omega_0^*(M)\oplus\hat{\Omega}_0^*(M)$ is well defined; the closure of the latter subspace is all of $\Omega^*(M)$ and the Hodge-* extends uniquely to its closure.
\end{lemma}

As we shall see, the usual formulae involving the Hodge-* extended to the present context of jets of indefinitely high order go through unproblematically as long as the dependence is linear.

Owing to theorem I.4.3, differential forms in $\Omega^*(M)$ can be arranged into a cochain complex, viz., a resolution of the sheaf of smooth functions that is locally, or stalkwise, exact. It is natural, in this context, to ask about the global problem concerning the generalized de Rham cohomology groups, defined along the usual lines as the quotient of the space of closed forms by the subspace of exact forms: $H^p_{\mathrm{dR}}(M;\text{\th}):=Z^p(M;\text{\th})/B^p(M;\text{\th})$, $p \ge 0$. The vanishing of the generalized de Rham cohomology in the case of contractible manifolds forms the subject of the following result:

\begin{theorem}[Poincar{\'e} Lemma; cf. Lee, \cite{lee_smooth_manif}, Theorem 17.14]\label{poincare_lemma}
	If $O \subset \vvmathbb{R}^n$ denotes any star-shaped open subset in Euclidean space, then $H^p_\mathrm{dR}(O) = 0$ for $p \ge 1$.
\end{theorem}
\begin{proof}
	The strategy of proof is to unwind Lee's chain of lemmas and propositions, replacing the exterior derivative $d$ everywhere with its generalized version $\text{\th}$ and making the necessary alterations to the argument. The key step lies in construction of a homotopy operator. For each $t \in [0,1]$ define the mapping $i_t: M \rightarrow M \times [0,1]$ as follows:
	\begin{equation}
		i_t(x) := (x,t).
	\end{equation}
	Then, for each $p\ge 1$ we need to define a linear map $h: \Omega^p(M) \rightarrow \Omega^{p-1}(M)$ satisfying the condition,
	\begin{equation}
		h(\text{\th}\omega) + \text{\th} (h\omega) = i_1^* \omega - i_0^* \omega
	\end{equation}
	for all $\omega \in \Omega^p(M \times [0,1])$. If $s$ denotes the standard Cartesian coordinate on $[0,1] \subset \vvmathbb{R}$, let $S$ be the vector field on $M \times [0,1]$ given by $S := 0 \oplus \left( \partial_s + \partial_{ss} + \cdots \right)$ under the usual identification $J^r(M \times [0,1]) \cong J^r(M) \times J^r([0,1])$. We may now obtain a form of one lower degree by the higher-order integration,
	\begin{equation}
		h \omega := \int_{[0,1]} i_t^* (S \lrcorner \omega) \theta^t,
	\end{equation}
	where the integration is to be interpreted componentwise. The other term $\text{\th}(h\omega)$ may be computed by differentiating under the integral signs,
	\begin{equation}
		\text{\th}(h\omega) = \int_{[0,1]} \text{\th} \left( i_t^* (S \lrcorner \omega) \right) \theta^t =  \int_{[0,1]} i_t^* \text{\th} (S \lrcorner \omega)  \theta^t.
	\end{equation}
	Thus,
	\begin{equation}\label{homotopy_integral}
		h \left( \text{\th} \omega \right) + \text{\th}(h\omega) = \int_{[0,1]} \left[ i_t^* (S \lrcorner \text{\th}\omega) + i_t^*  \text{\th} (S \lrcorner \omega) \right] \theta^t.
	\end{equation}
	At this point, Lee appeals to Cartan's magic formula for the Lie derivative, $\mathscr{L}_X = i_X d + d i_X$, but we cannot do likewise as in the present state of development of the theory we have no corresponding concept of Lie derivative for higher-order vector fields. Fortunately, we have only this single expression to evaluate and its evaluation can be performed directly due to the simple nature of the pertinent quantities defined on the Cartesian product manifold $M \times [0,1]$. For we want the integrand to be given by
	\begin{equation}
		\text{\th}_t i^*_t \omega,
	\end{equation}
	in which case we may apply the fundamental theorem of calculus to its first-order components while all second- and higher-order components give no contribution, if we understand that $\omega$ should be a $p$-form defined on $M \times [0,1]$ by restriction of one of compact support on $M \times \vvmathbb{R}$. Equality with the integrand in equation (\ref{homotopy_integral}) follows if we consider that the total derivatives in the time direction correspond to partial derivatives with respect to time of the spatial components of $\omega$ together with the implicit dependence on time via components of $\omega$ having a temporal part, when contracted against a spatial displacement; that is, we must have
	\begin{equation}
		\text{\th}_t i^*_t \omega = \left[ i_t^* (S \lrcorner \text{\th}\omega)+ i_t^* \text{\th} (S \lrcorner \omega) \right] \theta^t.
	\end{equation}
	Making use of an appropriate bump function in the limit, we can arrange for any fluxes coming from higher-order integrations to reduce identically to zero and to give no net contribution to the entire integral $\int_{[0,1]} \text{\th}_t i^*_t \omega$. Therefore, we are left with
	\begin{equation}
		h \left( \text{\th} \omega \right) + \text{\th}(h\omega) = i^*_1 \omega - i^*_0 \omega.
	\end{equation}
	Existence of a homotopy operator as just shown implies equality of induced cohomology maps for homotopic smooth maps as in Lee, \cite{lee_smooth_manif}, Proposition 17.10. Homotopy invariance of generalized de Rham cohomology follows by the same argument as in Lee, \cite{lee_smooth_manif}, Theorem 17.11. Vanishing in positive degree of the generalized de Rham cohomology for contractible smooth manifolds and the Poincar{\'e} lemma may be derived as an immediate consequence, for one has the straight-line homotopy
	\begin{equation}
		H(x,t) = c + t(x-c)
	\end{equation} 
	around the point $c \in \vvmathbb{R}^n$ with respect to which $O$ is star-shaped (as in Lee, \cite{lee_smooth_manif}, Theorems 17.13 and 17.14).
\end{proof}

\begin{corollary}[Cf. Lee, \cite{lee_smooth_manif}, Corollary 17.15]\label{poincare_lemma_corollary}
	Let $M$ be a differentiable manifold. Then every point in $M$ has a neighborhood on which every closed form is exact. In particular, the de Rham cohomology of Euclidean space vanishes identically.
\end{corollary}

\begin{remark}
	In view of corollary \ref{poincare_lemma_corollary}, the de Rham theorem  \cite{warner} tells us that cohomology with respect to $\text{\th}$ is canonically isomorphic to that with respect to the ordinary exterior derivative $d$, which goes to show the former may be non-trivial under appropriate circumstances despite the difference in definition between chain operators, the former going up infinitely high in derivatives.
\end{remark}

\begin{lemma}[Rules for the Hodge-*; cf. Thirring, \cite{thirring_vol_2}, 1.2.18,a)b)c)]\label{hodge_rules}
	Let $\omega$ and $\eta$ be smooth $p$-forms. Then,
	\begin{itemize}
		\item[$(1)$] $\omega \wedge *\eta = \eta \wedge *\omega$ \\
		
		\item[$(2)$] $e_\alpha \wedge * e^{\beta_1\cdots\beta_p} = \sum_{r=1}^p (-1)^{r+p} \delta^{\beta_r}_\alpha * e^{\beta_1\cdots\beta_r\beta_{r+1}\cdots\beta_p} = (-1)^{p+1} * i_{e_\alpha} e^{\beta_1\cdots\beta_p}$ \\
		
		\item[$(3)$] $i_{e_\alpha} * e^{\beta_1\cdots\beta_p} = *e^{\beta_1\cdots\beta_p\alpha}$.  
	\end{itemize}
\end{lemma}
\begin{proof}
	Since only multilinear algebra on basis elements is involved, Thirring's solutions apply unchanged, understanding that the indices are now to be regarded as multi-indices. For the convenience of the reader, we reproduce his derivations here. Bear in mind, of course, that we are working with forms on a direct sum of pieces $s=1,2,\ldots$. Thus, the formulae below apply within each piece and are to be extended by linearity to the full space of forms as described above in {\S}\ref{def_hodge_star}.
	
	For (1),
	\begin{align}
		\omega \wedge *\eta &= \omega_{\alpha_1\cdots\alpha_p}\eta^{\beta_1\cdots\beta_p} \varepsilon_{\beta_1\cdots\beta_p\gamma_{p+1}\cdots\gamma_{m_s}} \frac{\sqrt{|\hat{g}_s|}}{(m_s-p)!} e^{\alpha_1\cdots\alpha_p\gamma_{p+1}\cdots\gamma_{m_s}} \nonumber \\
		&=\omega_{\alpha_1\cdots\alpha_p}\eta^{\beta_1\cdots\beta_p} \varepsilon_{\beta_1\cdots\beta_p\gamma_{p+1}\cdots\gamma_{m_s}} \frac{\sqrt{|\hat{g}_s|}}{(m_s-p)!}\varepsilon_{\alpha_1\cdots\alpha_p\gamma_{p+1}\cdots\gamma_{m_s}}e^{12\cdots m_s} \nonumber \\
		&= \eta \wedge *\omega.
	\end{align}
	
	For (2),
	\begin{align}
		e_\alpha \wedge * e^{\beta_1\cdots\beta_p} &= g_{\alpha\gamma}g^{\beta_1\gamma_1}\cdots g^{\beta_p\gamma_p} e^{\alpha\gamma_{p+1}\cdots\gamma_{m_s}} \frac{\sqrt{|\hat{g}_s|}}{(m_s-p)!} \nonumber \\
		&= \sum_{r=1}^p (-1)^{r+p} g^{\beta_1\gamma_1} \cdots g^{\beta_{r-1}\gamma_{r-1}}
		g^{\beta_{r+1}\gamma_{r+1}} \cdots g^{\beta_p\gamma_p} \varepsilon_{\gamma_1\cdots\gamma_{r-1}\gamma\gamma_{r+1}\cdots\gamma_{m_s}} \times \nonumber \\
		&\qquad \delta^{\beta_r}_\alpha e^{\alpha\gamma_{p+1}\cdots\gamma_{m_s}} \frac{\sqrt{|\hat{g}_s|}}{(m_s-p+1)!} \nonumber \\
		&= \sum_{r=1}^p (-1)^{r+p} \delta^{\beta_r}_\alpha * e^{\beta_1\cdots\beta_r\beta_{r+1}\cdots\beta_p}.
	\end{align}
	
	For (3),
	\begin{align}
		i_{e_\alpha} g^{\alpha_1\beta_1} \cdots g^{\alpha_1\beta_p} \varepsilon_{\beta_1\cdots\beta_{m_s}} \frac{\sqrt{|\hat{g}_s|}}{(m_s-p)!} &=
		g^{\alpha_1\beta_1} \cdots g^{\alpha_1\beta_p} g^{\alpha\gamma} \delta^{beta_{p+1}}_\gamma \times \nonumber \\
		& \qquad (m_s-p) \varepsilon_{\beta_1\cdots\beta_{m_s}} e^{\beta_{p+2}\cdots\beta_{m_s}} \frac{\sqrt{|\hat{g}_s|}}{(m_s-p)!} \nonumber \\
		&= * e^{\beta_1\cdots\beta_p\alpha}.
	\end{align}
	This completes the proof, subject to the above understanding.
\end{proof}

\subsection{Appendix}\label{riemannian_appendix}

Here, we wish to suggest a reasonable sufficient condition by which to ensure that the series defining the Levi-Civita jet connection and the Riemannian curvature endomorphism will always be convergent. Let us work in the coordinate basis and identify the differential $d^\alpha$ with the tangent $\partial_\alpha$. In these terms, the metric may be viewed as an operator $g: \ell^2(\vvmathbb{N}_0^n \setminus \{ 0 \}) \rightarrow \ell^2(\vvmathbb{N}_0^n \setminus \{ 0 \})$. For the sake of notational simplicity, henceforward we denote the Hilbert space as just $\ell^2$. A possible starting point would be to require the metric to assume the form
\begin{equation}
	g = \eta + K,
\end{equation}
where $K$ is a compact operator and in the Riemannian case $\eta$ designates the identity operator while in the pseudo-Riemannian case it should designate a bounded diagonal operator in uniformizing coordinates with eigenvalues $\pm 1$.

\begin{proposition}
	A metric tensor of the indicated form retains this form upon smooth changes of coordinate, in the Riemannian case, and upon generalized Lorentz transformations in the pseudo-Riemannian case.
\end{proposition}	
\begin{proof}
	An immediate consequence of the covariance of the metric tensor. Let $\Phi: \ell^2 \rightarrow \ell^2$ denote the change of coordinate. Then from
	\begin{equation}
		\langle \Phi \varv, \Phi \varv \rangle = \langle \varv, \varv \rangle
	\end{equation}
	we know that $\| \Phi \varv \| = \| \varv \|$, or $\| \Phi \| = 1$ so $\Phi$ is continuous in the strong topology on $\ell^2$. Now, let $\varv_n$ be a bounded sequence in the new coordinates. By covariance, $\Phi^{-1} \varv_n$ will be bounded in the old coordinates. By hypothesis, $K$ is compact which implies that there exists a convergent subsequence $K \Phi^{-1} \varv_{n_k}$. Hence, by continuity of $\Phi$, $\Phi K \Phi^{-1} \varv_{n_k}$ will be convergent as well, meaning that $\Phi K \Phi^{-1}$ is a compact operator. In other words, the metric tensor with respect to the new coordinates $g = \eta + \Phi K \Phi^{-1}$ has the required form. Here, we invoke that $\Phi \Phi^{-1}=1$ in the Riemannian case while in the pseudo-Riemannian case the condition that $\Phi \eta \Phi^{-1} = \eta$ defines what is meant by a generalized Lorentz transform.
\end{proof}

Let us investigate the inverse metric tensor $g^{-1}$. A little algebra shows that we may write
\begin{equation}
	g^{-1} = \eta + H
\end{equation}
with
\begin{equation}
	H = - \eta K \eta - \eta K \eta K \eta - \cdots;
\end{equation}

\begin{proposition}
	Let $K$ be a compact operator with $\| K \| < 1$. Then
	\begin{itemize}
		\item[$(1)$] $\eta K$ resp. $K\eta$ is compact;
		
		\item[$(2)$] $K + K^2 + \cdots$ is compact.
	\end{itemize}
\end{proposition}
\begin{proof}
	Let $\varv_n$ be a bounded sequence in $\ell^2$ and $K \varv_{n_k}$ the convergent subsequence that exists by hypothesis. Now if $\varv_n$ is a bounded sequence so is $\eta \varv_n$ and vice versa because $\|\eta\|=1$ and $\eta^2=1$.
	So we have that $K \eta \varv_{n_k}$ converges as well. Therefore $K \eta$ is compact. Alternately, from the convergence of the subsequence $K \varv_{n_k}$ we may conclude to that of $\eta K \varv_{n_k}$ by continuity of $\eta$. Hence, $\eta K$ is compact too.
	
	Every compact operator is bounded; the product $BK$ of a bounded operator $B$ with a compact operator $K$ is compact and the sum $K_1+K_2$ of two compact operators $K_{1,2}$ is compact (see, for instance, MacCluer, Propositions 4.6 and 4.9, \cite{maccluer}). Therefore,
	\begin{equation}
		K_m := K + K^2 + \cdots + K^m
	\end{equation}
	represents a sequence of compact operators so by MacCluer, Proposition 4.10 if there exists a bounded operator $K_\infty$ such that $\| K_m - K_\infty \| \rightarrow 0$, then $K_\infty$ must be compact as well. But given $\varv \in \ell^2$, we may define $K_\infty \varv := K \varv + K^2 \varv + \cdots$, which series converges since equal to $\dfrac{1}{1-K}-1$ applied to $\varv$ and the inverse exists if $\| K \| < 1$. Thus,
	\begin{equation}
		\| K_\infty - K_m \| \le \| K \|^{m+1} + \| K \|^{m+2} + \cdots = \frac{\| K \|^m}{1 - \| K \|},
	\end{equation}
	which obviously tends to zero as $m \rightarrow \infty$.
\end{proof}

\begin{corollary}\label{inverse_compactness}
	If the metric assumes the form $g=\eta+K$, with $K$ compact and $\| K \| < 1$, the inverse metric may be written $g^{-1}=\eta+H$ with $H$ compact and $\| H \| < 1 $ if $\| K \|$ is sufficiently small; i.e., $\| K \|/(1-\| K \|) <1$.
\end{corollary}

We want, however, eventually to show that the Levi-Civita jet connection is well defined. If we consider the map $(\ell^2)^{\otimes 2} \rightarrow \vvmathbb{C}$ induced by
$\varv \otimes w \mapsto \langle \varv, K w \rangle$ then in an orthonormal basis $e_\alpha$, we have
\begin{equation}
	\sum_{\alpha\beta} \| K e_\alpha \otimes e_\beta \|^2 = \mathrm{tr}~ K^*K.
\end{equation}	
Hence a natural condition to impose would be that $K$ be not only compact but also Hilbert-Schmidt (cf. Reed-Simon, Problem VI.48, \cite{reed_simon}). In order to ensure this, it will be convenient to impose a condition of uniform boundedness on the components of the metric tensor and all of their derivatives; namely, we require that there exist constants $0 < \lambda < 1$ and $M>0$ such that for all multi-indices $\alpha, \beta, \gamma$ we have
\begin{equation}
	| \partial_\alpha g_{\beta\gamma} | \le \lambda^{|\alpha|+|\beta|+|\gamma|}M.
\end{equation}
Let us establish some properties of this condition for later use. First of all, a sum of the following form arises:
\begin{equation}
	C_\lambda = \sum_{r\ge 0} \binom{n+r-1}{r} \lambda^{2r}.
\end{equation}
If we bound the factorials with the help of Stirling's formula in the form given by Rudin \cite{rudin}, Problem 8.20, viz.,
\begin{equation}
	e^{7/8} < \dfrac{m!}{m^{m+\frac{1}{2}}e^{-m}} < e,
\end{equation}
it becomes evident that
\begin{align}
	\binom{n+r-1}{r} &< C_0 \dfrac{(n+r-1)^{n+r-\frac{1}{2}}e^{-n-r+1}}{r^{r+\frac{1}{2}}e^{-r}} \nonumber \\
	&< C_1 \dfrac{(n+r-1)^{n+r-\frac{1}{2}}}{r^{n+r-\frac{1}{2}}r^{1-n}} \nonumber \\
	&< C_2 \left( 1 + \frac{n-1}{r} \right)^{n+r-\frac{1}{2}} r^{n-1} \nonumber \\
	&< C_3 r^{n-1}		
\end{align}
for some positive constants $C_{0,1,2,3}$. Hence, $C_\lambda < \infty$ and it is clear that $C_\lambda$ tends to zero as $\lambda$ does.

\begin{proposition}\label{unif_bd_properties}
	Let $K_{1,2}$ be operators satisfying the condition of uniform boundedness for some $0<\lambda<1$ and $M_{1,2}>0$ and let $B$ be a bounded operator. Then
	\begin{itemize}
		\item[$(1)$] $K_{1,2}$ are Hilbert-Schmidt	and hence compact;
		\item[$(2)$] $K_1+K_2$ satisfies the condition with bound $M_1+M_2$;
		
		\item[$(3)$] $K_1 K_2$ satisfies the condition with bound $C_\lambda M_1 M_2$; 
		
		\item[$(4)$] $BK_1$ and $K_1 B$ satisfy the condition with bound $\| B \| M_1$.	
	\end{itemize}
\end{proposition}
\begin{proof}
	For (1), invoke the result in Thirring \cite{thirring_vol_3}, 2.3.19, iii) that if two bounded operators satisfy $A \le B$, then their trace norms satisfy $\mathrm{tr}~ A^* A \le \mathrm{tr}~ B^*B$. In this case, $|K|_{1,2} \le \Lambda$ where $\Lambda$ is the diagonal operator given by $e_\alpha \mapsto \lambda^{-|\alpha|} M e_\alpha$. Clearly $\Lambda$ itself is Hilbert-Schmidt since
	\begin{equation}
		\mathrm{tr}~ \Lambda^* \Lambda = C_\lambda M^2 < \infty.
	\end{equation}
	Hence, $K_{1,2}$ are Hilbert-Schmidt as well. For their compactness, see  MacCluer, Theorem 4.15 \cite{maccluer}. Items (2), (3) and (4) are obvious.
\end{proof}

\begin{proposition}
	If $K_n$ is a sequence of operators satisfying the condition of uniform boundedness with common constants $0<\lambda<1$ and $M>0$ and $\| K_n - K_\infty \| \rightarrow 0$ for some bounded operator $K_\infty$, then $K_\infty$ obeys the condition of uniform boundedness as well with the same constants.
\end{proposition}
\begin{proof}
	Since by the above each $K_n$ is compact, so is $K_\infty$. Strong convergence implies weak convergence, whence the matrix elements of $K_\infty$ are given by limits of the corresponding matrix elements of the $K_n$. We can interchange the operation of taking the derivative $\partial_\alpha$ with the limit as $n$ tends to infinity because the operators $K_n$ are uniformly bounded (cf. Courant and John, {\S}7.4, b) and e), \cite{courant_john}: the hypothesized uniform boundedness of the term-by-term derivatives entails, in the limit, uniform continuity on compact sets by Cauchy's test and hence the permissibility of term-by-term differentiation). Thus, we have
	\begin{equation}
		|\langle e_\beta, \partial_\alpha K_\infty e_\gamma \rangle| = 
		\lim_{n \rightarrow \infty} |\langle e_\beta, \partial_\alpha K_n e_\gamma \rangle| \le \lambda^{|\alpha|+|\beta|+|\gamma|}M,
	\end{equation}
	since by hypothesis the same $\lambda$ and $M$ apply to all of the $K_n$.
\end{proof}

\begin{corollary}\label{inverse_metric_unif_bd}
	Suppose the metric assumes the form $g=\eta+K$ where $K$ satisfies the condition of uniform boundedness with $\lambda>0$ sufficiently small that $C_\lambda<1$ and $M<C_\lambda^{1/2}$. Then $g^{-1}=\eta+H$ where $H$ obeys the condition of uniform boundedness as well.
\end{corollary}
\begin{proof}
	By the derivation of corollary \ref{inverse_compactness} we know that $H$ exists and is given by the strong limit of a sequence $H_m$ where each element in the sequence satisfies the uniform boundedness condition by above results. But with our restriction on $M$ it follows that $\eta K \eta K \eta$ is bounded by $C_\lambda M^2 < C_\lambda^2$ and the successive powers may be likewise be bounded by $C_\lambda^m$, so the series has a limit with bound $0<M_H<\infty$. 
\end{proof}

Furnished with this condition, it proves not difficult to establish the existence of the Levi-Civita jet connection with nice convergence properties. First off, let us neglect the cross-terms; i.e., write
\begin{equation}\label{leading_Christoffel_term}
	\Gamma_0 := \frac{1}{2} g^{-1} \circ \left( K_1 + K_2 - K_0 \right),
\end{equation}
where $K_{0,1,2}$ are the operators $(\ell^2)^{\otimes 2} \rightarrow \ell^2$ induced by
\begin{align}
	K_0:~ & \varv \otimes w \mapsto \left( \varv^\alpha w^\beta \partial_\gamma K_{\alpha\beta} \right)_{\gamma \in \vvmathbb{N}_0^n \setminus \{0 \} } \\
	K_1:~ & \varv \otimes w \mapsto \left( \varv^\alpha w^\beta \partial_\alpha K_{\beta\gamma} \right)_{\gamma \in \vvmathbb{N}_0^n \setminus \{ 0 \} } \\
	K_2:~ & \varv \otimes w \mapsto \left( \varv^\alpha w^\beta \partial_\beta K_{\gamma\alpha} \right)_{\gamma \in \vvmathbb{N}_0^n \setminus \{ 0 \} } 
\end{align}

In order to establish the condition of uniform boundedness on $K_{0,1,2}$ and their derivatives, fix a multi-index $\delta$ and consider applying $\partial_\delta$ to them. For $K_0$, it is easy to see that
\begin{equation}
	|\partial_\delta \partial_\gamma K_{\alpha\beta}| = |\partial_{\delta+\gamma} K_{\alpha\beta}| \le \lambda^{|\alpha|+|\beta|+|\delta+\gamma|}M=
	\lambda^{|\alpha|+|\beta|+|\gamma|+|\delta|}M,
\end{equation}
so that its uniform boundedness follows from what is hypothesized about $K$ itself. Analogously for $K_1$,
\begin{equation}
	|\partial_\delta (K_1)_{\alpha\beta}^\gamma | = | \partial_{\delta+\alpha} K_{\beta\gamma}| \le \lambda^{|\delta+\alpha|+|\beta|+|\gamma|}M =
	\lambda^{|\alpha|+|\beta|+|\gamma|+|\delta|}M.
\end{equation}
The operator $K_2$ agrees with $K_1$ up to the symmetry that interchanges $\varv$ and $w$; hence a similar argument to that for $K_1$ yields the desired result for $K_2$ as well.

\begin{proposition}
	As an operator $(\ell^2)^{\otimes 2} \rightarrow \ell^2$, the Christoffel symbol $\Gamma_0$ exists and satisfies a condition of uniform boundedness (hence is Hilbert-Schmidt).
\end{proposition}
\begin{proof}
	In view of proposition \ref{unif_bd_properties}, the result follows immediately from the existence of $K_{0,1,2}$ with the requisite properties along with the inverse metric from corollary \ref{inverse_metric_unif_bd}. The condition of uniform boundedness on derivatives $\partial_\alpha \Gamma_0$ can be seen by distributing the derivative across $g^{-1}$ and the $K_{0,1,2}$. Since every summand in the resulting expression satisfies the condition of uniform boundedness, so will the sum where the degrees of the derivatives must add up to $|\alpha|$. The combinatorial factor is bounded by a constant to the power of $|\alpha|$; thus, if $\lambda$ is the constant controlling $K_{0,1,2}$ and $g^{-1}$, $\Gamma_0$ itself will have the constant $n\lambda$ in its place.
\end{proof}

We see that, if $\Gamma_0$ is to have good convergence properties, we must stipulate that $0<\lambda<\frac{1}{n}$ so that $0<n\lambda<1$. Next, we have to show existence of the full Levi-Civita jet connection from that of $\Gamma_0$. Schematically we may write equation (\ref{Levi_Civita_formula}) as
\begin{equation}
	\Gamma = \Gamma_0 + \Gamma \Gamma.
\end{equation}
Let then
\begin{align}
	\Gamma_1 &= \Gamma_0 + \Gamma_0 \Gamma_0 \nonumber \\
	\Gamma_2 &= \Gamma_0 + \Gamma_1 \Gamma_1 \nonumber \\
	\Gamma_3 &= \Gamma_0 + \Gamma_2 \Gamma_2,
\end{align}
etc., defining inductively a sequence $\Gamma_n$ which we want to converge to a limit $\Gamma_\infty$, which should agree with the Levi-Civita jet connection whose existence up to any finite order is guaranteed by theorem \ref{exist_levi_civita}. The point is that now, with the aid of the condition of uniform boundedness on the components of the metric tensor and their derivatives, we have good control over the convergence of the formal series and thus are in a position to justify the existence of $\Gamma_\infty$, here to be defined as acting on vector fields in the direct limit, i.e., having non-zero components up to indefinitely high order.

It is evident that
\begin{align}
	\Gamma_{n+1} - \Gamma_n &= \Gamma_n \Gamma_n - \Gamma_{n-1} \Gamma_{n-1}
	\nonumber \\
	&= (\Gamma_0 + \Gamma_{n-1}\Gamma_{n-1})(\Gamma_0+\Gamma_{n-1}\Gamma_{n-1}) - \Gamma_{n-1} \Gamma_{n-1}.
\end{align}
Thus, by descent we obtain every $\Gamma_n$ as a polynomial in $\Gamma_0$. Not only this, we can say that
\begin{equation}
	\Gamma_n - \Gamma_{n-1} = O(\Gamma_0^{n+1}),
\end{equation}
as follows by a straightforward induction. Since successive differences decline as an increasing power of $\Gamma_0$, we have reason to expect convergence of the telescoped series
\begin{equation}
	\Gamma_0 + (\Gamma_1-\Gamma_0) + (\Gamma_2 - \Gamma_1) + \cdots
\end{equation}
to some limiting $\Gamma_\infty$. The informal argument has now to be put on a firm basis.

\begin{proposition}
	Suppose the metric assumes the form $g=\eta+K$ where $K$ satisfies a condition of uniform boundedness with constants $0<\lambda<1$ and $M>0$ so that $\Gamma_0$ exists and itself satisfies uniform boundedness with constants $0<\lambda_0<1$ and $M_0>0$. Then the $\Gamma_n$ converge in the strong topology to some bounded operator $\Gamma_\infty: (\ell^2)^{\otimes 2} \rightarrow \ell^2$. Moreover, the $\Gamma_\infty$ so obtained itself satisfies a condition of uniform boundedness with constants $0<\lambda_\infty<1$ and $M_\infty>0$ if $\lambda$ be taken sufficiently small.
\end{proposition}
\begin{proof}
	First, as to existence of the limiting $\Gamma_\infty$. The bounded operators $(\ell^2)^{\otimes 2} \rightarrow \ell^2$, sc. $\mathscr{B}((\ell^2)^{\otimes 2},\ell^2)$, form a complete metric space (vide \cite{maccluer}, Proposition 2.3). Thus, we may hope to apply the contraction mapping theorem (\cite{reed_simon}, Theorem V.18). To prepare the ground for this step, it will be convenient to seek an operatorial expression for equation (\ref{Levi_Civita_formula}). We already have one for the leading part, namely in equation (\ref{leading_Christoffel_term}). For the remainder, we write
	\begin{equation}
		\Gamma = \Gamma_0 - T_1 - T_2 + T_0, 
	\end{equation}
	where $T_{0,1,2}$ are to be defined in a moment. First, we require notation for some modifications of the operators involved. If $\Gamma: (\ell^2)^{\otimes 2} \rightarrow \ell^2$ is represented in components by $\Gamma^\gamma_{\alpha\beta}$,
	let $\hat{\Gamma}: (\ell^2)^{\otimes 2} \rightarrow \ell^2$ be defined component-wise by $\hat{\Gamma}^\gamma_{\alpha\beta} := \Gamma^\beta_{\alpha\gamma}$. Since the three indices enter in the same way into the condition of uniform boundedness, if $\Gamma$ satisfies such a condition so will $\hat{\Gamma}$ with the same two constants $\lambda$ and $M$. The other modification we need is as follows. Suppose $A, B$ are functionals, i.e., map from $\ell^2 \rightarrow \vvmathbb{C}$. If they are given by components $A_\alpha$ and $B_\alpha$, define $A * B$ component-wise by $(A*B)_\alpha := \dfrac{\alpha!}{\alpha_1!\alpha_2!}\big|_{\alpha_{1,2} \ne 0} A_{\alpha_1} B_{\alpha_2} = P \circ (A \otimes B)$ where $P: (\ell^2)^{\otimes 2} \rightarrow \ell^2$ is the projection given by $e_\alpha \otimes e_\beta \mapsto \dfrac{(a+\beta)!}{\alpha!\beta!} e_{\alpha+\beta}$. To suppress the combinatorial factors, resort to a norm on $\ell^2$ obtained from the modified inner product $\langle e_\alpha, e_\beta \rangle = \frac{1}{\sqrt{\alpha!\beta!}}$; i.e., the orthonormal basis is $f_\alpha = \frac{1}{\sqrt{\alpha!}}e_\alpha$, $\alpha \in \vvmathbb{N}_0^n \setminus \{ 0\}$. In these terms, $P$ has norm $\| P \| = 1$.
	
	In terms of the notation just introduced, we can express $T_{0,1,2}$ as operators $(\ell^2)^{\otimes 2} \rightarrow \ell^2$ as follows:
	\begin{align}
		T_0 &:= \frac{1}{2} \mathrm{tr}_{23} g^{-1} \otimes \mathrm{tr}_{25} (g \circ \Gamma_1 ) * \Gamma_2 \\
		T_1 &:= \frac{1}{2} g^{-1} \circ \mathrm{tr}_{25} (g \circ \Gamma_1) * \hat{\Gamma}_2 \\
		T_2 &:= T_1 \circ S,
	\end{align}
	where $S$ is the symmetry induced by $\varv \otimes w \mapsto w \otimes \varv$ and we allow for the possibility that $\Gamma_1 \ne \Gamma_2$. Here, the subscripts indicate the slots over which the traces are to be taken.
	
	We wish to define a strong contraction $T$ in a suitable subspace of $\mathscr{B}((\ell^2)^{\otimes 2},\ell^2)$. Let therefore (in our compressed notation)
	\begin{equation}
		T(\Gamma) = \Gamma_0 + \Gamma \Gamma.
	\end{equation}
	If we can show $T$ to be a strong contraction, then its unique fixed point $\Gamma = T(\Gamma) = \Gamma_0 + \Gamma \Gamma$ will be the sought-for solution $\Gamma_\infty$. In fact, we are interested in the deviation from $\Gamma_0$ so set $\Gamma_{1,2} = \Gamma_0 + H_{1,2}$. In these terms, the putative contraction becomes 
	\begin{equation}
		T(H) = (\Gamma_0 + H)(\Gamma_0 + H).
	\end{equation}
	The subspace to which we restrict $T$ will be the ball $\| H \| \le \| \Gamma_0 \|$. If $T$ proves in fact to be a strong contraction, it will map from this ball into itself and we can employ $\Gamma_0$ to estimate $H$. For
	\begin{align}\label{contraction_estimate}
		\| T(H_1)-T(H_2) \| = &\| (\Gamma_0 + H_1)(\Gamma_0+H_1) - (\Gamma_0 + H_2)(\Gamma_0+H_2) \| \nonumber \\
		= &\| (H_1-H_2)\Gamma_0 + \Gamma_0(H_1-H_2) + H_1H_1 - H_2H_2 \| \nonumber \\
		= &\| (H_1-H_2)\Gamma_0 + \Gamma_0(H_1-H_2) + (H_1+H_2)(H_1-H_2)+ \nonumber \\
		& (H_1-H_2)H_1 + H_1(H_2-H_1) \| \nonumber \\
		\le &\| (H_1-H_2)\Gamma_0 \| + \| \Gamma_0(H_1-H_2) \| + \| (H_1+H_2)(H_1-H_2) \| + \nonumber \\
		&\| (H_1-H_2)H_1 \| + \| H_1(H_2-H_1) \|.
	\end{align}
	Recall at this juncture that these quadratic expressions schematically represent $T_0-T_1-T_2$ applied to their respective arguments. But $T_{0,1,2}$ are formed from given operators so we can estimate their norms. Certainly since $g=\eta+K$ with $\| K \| < 1$ sufficiently small that $g^{-1}=\eta+K_1$, $\| K_1 \| < 1$ (and $\| \eta \|=1$) we have $\| g \| \le 2$ and $\| g^{-1} \| \le 2$. Then
	\begin{equation}
		\| T_0 \| = \| \frac{1}{2} \mathrm{tr}_{23} g^{-1} \otimes \mathrm{tr}_{25} (g \circ \Gamma_1) * \Gamma_2 \| \le 2 C_{\lambda_0}^2 \| \Gamma_1 \| \| \Gamma_2 \|.
	\end{equation}
	Estimating $T_{1,2}$ is slightly more complicated since we have to put a bound on the star product. But we know in any case that the sum over the bionomial coefficients is given by $n^{|\alpha|}$ while $A$ and $B$ may be presumed to satisfy a condition of uniform boundedness with constants $\lambda_0$ and $M_0$. Hence, we can trace out one of them so as to find either $\| A*B \| \le C_{n\lambda_0} \| A \|$ or $\| A*B \| \le C_{n\lambda_0} \| B \|$. Applied to $\Gamma_{1,2}$, we conclude that
	\begin{equation}
		\| T_{1,2} \| \le 2 C_{n\lambda_0} \| \Gamma_{1,2}\|.
	\end{equation}
	(Since of course $\| S \| = 1$, we have that $T_2$ satisfies the same bounds as $T_1$.) Putting everything together, we have from equation (\ref{contraction_estimate}) that
	\begin{equation}
		\| T(H_1)-T(H_2) \| \le C_0 \| \Gamma_0 \| \| H_1 - H_2 \|,
	\end{equation}
	where $C_0>0$ is some positive constant. Now we may presume the bound $M$ on $K$ small enough that $C_0 \| \Gamma_0 \| < 1$ and $T$ will indeed be a strong contraction.
	
	We have proved therefore the existence of a unique $\Gamma_\infty$. As for the condition of uniform boundedness it ought to satisfy, we have to look at the structure of the $T_{0,1,2}$. Since they are composed of operators themselves satisfying uniform bounds, the overall expression must as well, according to the usual argument. A subtle point is raised by the star product. If we use the modified inner product with orthonormal basis $f_\alpha$, we can bound the star product by $\| A * B \| \le \| A \| \| B \|$. Hence by proposition \ref{unif_bd_properties}, $T=T_0-T_1-T_2$ obeys a condition of uniform boundedness, where the constant $\lambda_0$ has to be replaced by any $\lambda_1$ with $\lambda_0 < \lambda_1 < 1$. That is, we can restrict its domain to the part of the ball $\| \Gamma \| \le \| \Gamma_0 \|$ where uniform boundedness also holds (with respect to \textit{some} $\lambda_0<\lambda_1<1$ and $M_0<M_1<\infty$). If the fixed point were to lie on the boundary, where $M_1$ diverges, it would lead to a contradiction since we would then not be free to opt for another $\lambda_{11}$, $0<\lambda_0<\lambda_{11}<\lambda_1$. Thus the fixed point $\Gamma_\infty$ will lie inside this domain, or itself satisfy the uniform bound with $\lambda_1$ and $0<M_1<\infty$. Reverting now to the standard norm, the constant $\lambda_1$ must be further degraded to $\lambda_\infty$ with $\lambda_1 < \lambda_\infty < 1$.
\end{proof}

\begin{proposition}
	If the Christoffel symbols up to indefinitely high order are well defined in the sense that the sum over multi-indices converges, then the Levi-Civita jet connection is well defined as a mapping from $\mathscr{J}^\infty \times \mathscr{J}^\infty$ into $\mathscr{J}^\infty$.
\end{proposition}
\begin{proof}
	Let $X, Y \in \mathscr{J}^\infty$. This means that for some $r_{X,Y}$ the components $X_\alpha$, $Y_\beta$ vanish identically whenever $|\alpha|>r_X$ resp. $|\beta|>r_Y$. Then $\nabla_X Y$ reduces to a finite sum of type
	\begin{equation}
		X^\alpha \nabla_\alpha Y^\beta \partial_\beta = X^\alpha \frac{\alpha!}{\alpha_1!\alpha_2!} Y^\beta_{,\alpha_1} \Gamma^\lambda_{\alpha_2\beta} \partial_\lambda,
	\end{equation}
	where the summations extend over $|\alpha|\le r_X$, $|\beta|\le r_Y$. 
\end{proof}

In consequence,

\begin{proposition}
	If the Christoffel symbols up to indefinitely high order are well defined in the sense that the sum over multi-indices converges, then the Riemannian curvature endomorphism is well defined as a mapping from $\mathscr{J}^\infty \times \mathscr{J}^\infty \times \mathscr{J}^\infty$ into $\mathscr{J}^\infty$.
\end{proposition}
\begin{proof}
	Let $X, Y, Z \in \mathscr{J}^\infty$. Since by definition,
	\begin{equation}
		R(X,Y)Z = \nabla_X \nabla_Y Z - \nabla_Y \nabla_X Z - \nabla_{[X,Y]} Z,
	\end{equation}
	in view of the immediately preceding result all we have to show is that $[X,Y] \in \mathscr{J}^\infty$. But by hypothesis there exist $r_{X,Y}$ such that $X^\alpha$ resp. $Y^\beta$ vanish identically whenever $|\alpha|>r_X$ resp. $|\beta|>r_Y$. It follows at once that $[X,Y]^\gamma=0$ identically if $|\gamma| > r_X + r_Y -1$. This statement is tantamount to $[X,Y] \in \mathscr{J}^\infty$. 
\end{proof}

\begin{remark}
	The point is that the order of multi-index up to which the summations extend depends on the vector fields in $\mathscr{J}^\infty$ chosen as arguments. In general, one must expect that the components of the Levi-Civita jet connection and of the Riemannian curvature themselves may be non-zero up to indefinitely high order.
\end{remark}

To save on repetition, let us now make the standing assumption that the metric assumes the form $g=\eta+K$ where $K$ satisfies the condition of uniform boundedness with $0<\lambda<1$ and $M>0$ sufficiently small that the same applies to the Christoffel symbol $\Gamma$; i.e., $0 < \lambda_\infty < \frac{1}{n}$ as well.

\begin{proposition}\label{riemannian_unif_bd}
	The Riemannian curvature endomorphism satisfies a condition of uniform boundedness in the sense that its components obey
	\begin{equation}\label{riemann_unif_est}
		|\partial_\alpha R_{\beta\gamma\delta}^\mu| \le \lambda_1^{|\alpha|+|\beta|+|\gamma|+|\delta|+|\mu|}M_1
	\end{equation}
	for $\lambda_1=n\lambda_\infty$ and some $M_\infty<M_1<\infty$.
\end{proposition}
\begin{proof}
	If we work in the coordinate basis, the commutators $[\partial_\alpha,\partial_\beta]=0$ identically. Hence, the components of the Riemannian curvature are given by
	\begin{align}
		R_{\alpha\beta\gamma}^\delta &= \nabla_\alpha \nabla_\beta \partial_\gamma - \nabla_\beta \nabla_\alpha \partial_\gamma \nonumber \\
		&= \frac{\alpha!}{\alpha_1!\alpha_2!} \Gamma^\mu_{\beta\gamma,\alpha_1} \Gamma^\delta_{\alpha_2\mu} -
		\frac{\beta!}{\beta_1!\beta_2!} \Gamma^\mu_{\alpha\gamma,\beta_1} \Gamma^\delta_{\beta_2\mu}.
	\end{align}
	If we apply $\partial_\alpha$ to $R^\mu_{\beta\gamma\delta}$ from the above expression, we obtain
	\begin{equation}
		\partial_\alpha R^\mu_{\beta\gamma\delta} = \frac{\alpha!}{\alpha_1!\alpha_2!} \left( \frac{\beta!}{\beta_1!\beta_2!} \Gamma^\nu_{\gamma\delta,\alpha_1+\beta_1} \Gamma^\mu_{\beta_2\nu,\alpha_2} -
		\frac{\gamma!}{\gamma_1!\gamma_2!} \Gamma^\nu_{\beta\delta,\alpha_1+\gamma_1} \Gamma^\mu_{\gamma_2\nu,\alpha_2}.
		\right)
	\end{equation}
	Clearly, then the estimates on the Christoffel symbols and their derivatives suffice to show equation (\ref{riemann_unif_est}).
\end{proof}

\begin{proposition}
	The Ricci tensor is well defined and satisfies a condition of uniform boundedness in the sense that
	\begin{equation}\label{ricci_unif_bd}
		| \partial_\alpha R_{\mu\nu} | \le \lambda_1^{|\alpha|+|\mu|+|\nu|}M_2
	\end{equation}
	with $M_1<M_2<\infty$. Moreover, the scalar curvature is given by a convergent series.
\end{proposition}
\begin{proof}
	The components of the Ricci tensor are obtained by contraction:
	\begin{equation}
		R_{\mu\nu} := R_{\mu\lambda\nu}^\lambda.
	\end{equation}
	We may apply $\partial_\alpha$ term-wise to the right hand side. Then the claim follows from proposition \ref{riemannian_unif_bd}. The estimate (\ref{ricci_unif_bd}) with $\alpha=0$ along with the standing assumption on the metric  ensures convergence when we contract once again to yield the scalar curvature: $R = R_\mu^\mu = g^{\mu\nu}R_{\mu\nu}$.
\end{proof}

It follows that the Riemannian curvature and Ricci tensors correspond to well-behaved Hilbert-Schmidt and hence compact operators. This suggests that it is indeed sensible to speak of them as applying to vector fields lying in the closure $\overline{\mathscr{J}^\infty}$ of the direct limit $\mathscr{J}^\infty$ with respect to the $\ell^2$-norm. But the mathematical status of such a construct is not very clear at present. Another point deserving mention here concerns the scope of general covariance. From the point of view of spatial intuition, it seems that Riemannian geometry would have a natural interpretation only when the content of the metric is concentrated in low order in the infinitesimals; that is, while we may allow infinitesimals of arbitrarily high order, we want them to fall off in importance as their order tends to infinity. This means that we must confine ourselves to changes of coordinate away from the Cartesian frame with a suitable degree of regularity, since a proviso such as this would not be respected by an arbitrary smooth transformation. Thus, the principle of general covariance has somehow to be circumscribed, at least for models to be put forward as having relevance to physics---clearly a topic ripe for further investigation!

%% file: section_5.tex
\section{Revisiting the Foundations of the Integral Calculus}\label{chapter_5} 

Einstein, when formulating the general theory of relativity, could content himself with the differential and integral calculus in pretty much the form in which Newton and Leibniz left it, but, after the introduction of higher-order infinitesimals, the received concepts of the mathematical tradition ever since the seventeenth century are clearly no longer sufficient for our purposes. We have already indicated in {\S}\ref{chapter_3} how the role of differentiation is to be reconceived; we have now in this section to take up that of integration as well. The basic theory of integration of differentials of higher than first order on the real line is first surveyed. We proceed to Euclidean space in many dimensions and then to an arbitrary differentiable manifold, where one wants to formulate the theory in an intrinsic manner. The main difficulty to overcome is to frame a suitable generalization of the notion of orientation so as to allow passage between any two coordinate charts. Lastly, it is shown how a generalized Riemannian metric leads to a canonical choice of the macroscopic part of the volume form, thereby extending the usual theory.

\subsection{Modification of the Integral in the Presence of Non-Uniform Flow of Time}
\label{quadrature_on_real_line}

In real analysis as currently known, the procedure of integration (whether Riemann's or Lebesgue's) presumes a uniform standard of time, if we choose to designate the abscissa as a temporal dimension. By instituting a truncation to first order, Newton and Leibniz suppose that, as far as integration is concerned, the relevant part of every motion must be reducible to the case of uniform acceleration, locally in time.\footnote{Consider, that is, a body of mass $m$ in motion subject to a force $F$. The impulse it suffers over the interval of time from $t_0$ to $t_1$ may be written as $\int_{t_0}^{t_1} F dt$. Here, the acceleration as a function of time is $a = F/m$. For the purpose of performing the integral defining the impulse, the body's acceleration $a$ is regarded as being approximately constant over the infinitesimal increment of time from $t$ to $t+dt$.} For us, the key point will be that the difference quotient with which the time derivative is computed presupposes a uniform measure of time and therefore cannot be a covariant notion. Lipschitz \cite{lipschitz} in an 1872 paper on the variational calculus hints that the integrand (in his case, the action of classical mechanics) ought to be definable in an manner independent of the choice of coordinates. As soon as one considers integration of a differential of higher than first order, truncation as is customary in the calculus up to now clearly ceases to be admissible. Therefore, in {\S}\ref{quadrature_on_real_line}, we recur to the elementary problem of integration over a single real variable. In the remainder of this section, we consider first Euclidean spaces of any dimensions and then advance an intrinsic theory of integration adapted to the higher-order setting, on an arbitrary differentiable manifold.

Proceeding now to a formal definition of the integral, we have merely to recast Berkeley's antique language into modern parlance \cite{berkeley_analyst}. To render the sought-for integral intrinsic, replace the function forming the integrand with a jet; just as is done already in the calculus on differentiable manifolds where the integrand on the real line would be a 1-form, or equivalently, a 1-jet field, now it will be substituted with an $r$-jet field. The immediate and simplest non-trivial case would be to integrate $d^{xx}$, i.e., a differential expression of second order having a constant coefficient. In line with Berkeley's suggestions, we want its integral $F$ to be a function having quadratic rather than linear increase. Therefore, the natural proposal would be $F := \frac{1}{2} x^2$ so that $F^{\prime\prime}=1$ identically, viz., the coefficient of the constant second-order differential $d^{xx}$. Note in this context that just as in the ordinary calculus the antiderivative can be defined only up to an additive constant, here it will be defined up to an arbitrary linear expression. Clearly, the right way to integrate a constant differential of arbitrary order $k$, or
$d^{\overbrace{x\cdots x}^{k ~ \mathrm{times}}}$, would be to set $F_k := \dfrac{1}{k!} x^k$.

The problem that follows right away would be to integrate a differential expression having spatial dependence, say $f d^{xx}$. A moment's reflection leads to the following formula:
\begin{equation}
	F(x) := \int_0^x dx^\prime \int_0^{x^\prime} dx^{\prime\prime} f(x^{\prime\prime}),
\end{equation}
for then
\begin{equation}
	F^\prime(x) = \int_0^x dx^\prime f(x^\prime)
\end{equation}
whence
\begin{equation}
	F^{\prime\prime}(x) = f(x),
\end{equation}
as it should. The general expression for the antiderivative of a spatially dependent differential of arbitrary order $k$, $f d^{\overbrace{x\cdots x}^{k ~ \mathrm{times}}}$ would then assume the form,
\begin{equation}\label{antideriv_def}
	F_k(x) := \int_0^x dx^\prime \int_0^{x^\prime} dx^{\prime\prime} \cdots 
	\int_0^{x^{(k-1)}} dx^{(k)} f(x^{(k)}).
\end{equation}
As we demand the obvious property of linearity, the antiderivative $F = \int \omega$ of the general $r$-jet field $\omega = f_1 d^x + f_2 d^{xx} + \cdots + f_r d^{\overbrace{x\cdots x}^{r ~ \mathrm{times}}}$ can now simply be obtained as a superposition over $k=1,\ldots,r$ of each of its monomial parts $F_k$ as defined by equation (\ref{antideriv_def}); i.e., $F := F_1 + \cdots + F_r$.

\begin{remark}\label{indefinite_part_convention}
	Implicitly, the antiderivative will be unique only up to addition of an arbitrary polynomial of $(r-1)$-th degree. If, however, we wish it to have the desirable property that
	\begin{equation}
		\int_{x_0}^{x_1} \omega + \int_{x_1}^{x_2} \omega = \int_{x_0}^{x_2} \omega,
	\end{equation}
	the natural choice would be to equate the indefinite part to the extrapolation of whatever lies to the left; namely, to set
	\begin{equation}\label{antideriv_indefinite_part}
		F := F_1 + \cdots + F_r + a_0 + a_1 x + \frac{1}{2} a_2 x^2 + \cdots +\frac{1}{(r-1)!} a_{r-1} x^{r-1},
	\end{equation}
	where the coefficients are to be determined by $a_k := F^{(k)} |_{0^-}$ (the derivatives being evaluated from the left). The convention here adopted, needless to say, is essential to the validity of theorem \ref{fundamental_theorem_calculus}, which we are about to state in a moment. Proof of the coordinate-invariance of the antiderivative so defined will be deferred to {\S}\ref{integration_on_manif}.
\end{remark}

\begin{proposition}
	The antiderivative $F = \int \omega$ of an arbitrary generalized 1-form $\omega$ on $\vvmathbb{R}$, as defined by Equations \ref{antideriv_def} and \ref{antideriv_indefinite_part}, satisfies	
	\begin{equation}
		\int_{x_0}^{x_1} \omega + \int_{x_1}^{x_2} \omega = \int_{x_0}^{x_2} \omega.
	\end{equation}	
\end{proposition}
\begin{proof}
	After an immediate change of variables, equation (\ref{antideriv_def}) assumes the form,
	\begin{equation}\label{antideriv_def_translated}
		F_k(x) := \int_{x_0}^x dx^\prime \int_{x_0}^{x^\prime} dx^{\prime\prime} \cdots 
		\int_{x_0}^{x^{(k-1)}} dx^{(k)} f(x^{(k)}),
	\end{equation}
	which of course always amounts to $\int_{x_0}^{x_1}$ of a $(k-1)$-th order antiderivative. But then, for the definite part of equation (\ref{antideriv_indefinite_part}), each $F_k$ term ($k=1,\ldots,r$) will involve just an identity in the ordinary integral, namely, $\int_{x_0}^{x_1} + \int_{x_1}^{x_2} = \int_{x_0}^{x_2}$. It remains to check the indefinite part. Now, evidently our convention of remark \ref{indefinite_part_convention} is tantamount to the replacement of $\omega$ with
	\begin{equation}
		\tilde{\omega} := a_0 + (f_1+a_1) d^x + (f_2+a_2) d^{xx} + \cdots + (f_r+a_r) d^{\overbrace{x\cdots x}^{r ~ \mathrm{times}}}.
	\end{equation}
	Let us denote the integrand corresponding to $\int_{x_0}^{x_1}$ by $\tilde{\omega}$ while that corresponding to $\int_{x_1}^{x_2}$ by $\tilde{\eta}$, where the additional constant terms appearing in $\tilde{\eta}$ will be denoted with the letters $b_0, \ldots, b_r$. Therefore, all that has to be shown is that $b_k = f_k(x_1)+a_k$, $k=0, \ldots, r$. Now, after the above-mentioned change of variables, the antiderivative corresponding to a constant unit coefficient would be just $F_k = \frac{1}{k!} (x_1-x_0)^k$. By continuity, however, around $x=x_1+\varepsilon$ we have from Taylor's theorem (\cite{courant_john}, {\S}5.4b),
	\begin{equation}
		f_k(x)=f_k(x_1)+f_k^\prime(x_1)\varepsilon+\cdots+\frac{1}{r!}f_k^{(r)}\varepsilon^r+R,
	\end{equation}
	where it is a matter of indifference whether one employs Lagrange's or Cauchy's formula for the remainder $R$, for in any case it will be $o(\varepsilon^r)$. Hence, extracting the derivatives from the left as stipulated by our convention (remembering that we are using $\tilde{\omega}$ in place of $\omega$) yields $b_k=f_k(x_1)+a_k$, $k=0, \ldots, r$.
\end{proof}

\begin{theorem}[Fundamental Theorem of Calculus]\label{fundamental_theorem_calculus}
	If $F = \int_{x_0}^x \omega$, where $\omega$ is an exact generalized 1-form having compact support on $\vvmathbb{R}$, then $\text{\th}F = r \omega$.
\end{theorem}
\begin{proof}
	Immediate consequence of the definitions of the antiderivative $F$ and the generalized exterior derivative operator $\text{\th}$. For, by hypothesis, assuming $\omega =  f_1 d^x + f_2 d^{xx} + \cdots + f_r d^{\overbrace{x\cdots x}^{r ~ \mathrm{times}}}$ there exists a function $g$ such that $g^{(k)} = f_k$, $k=1,\ldots,r$. Then we may write,
	\begin{align}
		F_k(x) &= \int_{x_0}^x dx^\prime \int_{x_0}^{x^\prime} dx^{\prime\prime} \cdots 
		\int_{x_0}^{x^{(k-1)}} dx^{(k)} g^{(k)}(x^{(k)}) \\
		&= \int_{x_0}^x dx^\prime \int_{x_0}^{x^\prime} dx^{\prime\prime} \cdots 
		\int_{x_0}^{x^{(k-2)}} dx^{(k-1)} \left( g^{(k-1)}(x^{(k-1)})-g^{(k-1)}(x_0) \right) \\
		&= g(x_1) - g(x_0) + P_k(x_1-x_0),
	\end{align}
	where $P_k$ is a polynomial of degree $k-1$ whose coefficients depend on $g$ and its derivatives evaluated at $x_0$. Hence also
	\begin{equation}
		F_1 + \cdots F_r = r \left( g(x_1)-g(x_0) \right) + P_1(x_1-x_0) + \cdots P_r(x_1-x_0).
	\end{equation}
	The convention in force means, however, that the indefinite part of $F$ should vanish. The sum $P_1+\cdots+P_r$ equals zero as well, for we may displace $x_0$ sufficiently far to the left as to fall outside the support of $\omega$ and evaluate its coefficients there. Thus,
	\begin{align}
		\text{\th} F &= r g^\prime d^x + r g^{\prime\prime} d^{xx} + \cdots + r g^{(r)}  d^{\overbrace{x\cdots x}^{r ~ \mathrm{times}}} \\
		&= r f_1 d^x + r f_2 d^{xx} + \cdots + r f_r d^{\overbrace{x\cdots x}^{r ~ \mathrm{times}}} \\
		& = r \omega,
	\end{align}
	as was to be shown.
\end{proof}

\begin{remark}\label{integration_with_compact_sppt}
	We have proved theorem \ref{fundamental_theorem_calculus} only relative to the fundamental theorem of the ordinary calculus itself (as in Royden \cite{royden}, Theorem 5.10). Perhaps once the formalism has been better understood it will be possible to devise a more direct and elegant proof of the general result. The assumption of compactness of the support, moreover, is sufficient for the result in its present form but not perhaps necessary. Presumably, if it were to be omitted, one would want to impose on the components of $\omega$ a condition of rapid decrease towards infinity to the left. But, in this case, it is not anymore so evident what the proper form of $\text{\th}F$ will be, in view of the complication posed by the possibility that the antiderivatives could contain fluxes coming from negative infinity. A deeper study of the problem clearly is called for.
\end{remark}

Before, in the next section, launching into an investigation of the obviously very interesting problem of higher-order integration in many dimensions, we should like to dwell for a moment on a point that, subtle as it is, proves to be of the essence to the entire theory. The ordinary differential of first order with which we all have been so familiar for the span of four centuries, $d^x$, represents a degenerate special case in that it assigns zero measure to all fluxions of higher than first order. In general, we must expect to have to deal with jet fields whose components at every order, $f_k d^{\overbrace{x\cdots x}^{r ~ \mathrm{times}}}$, $k \ge 1$, could be non-vanishing. The statement that a component of a jet field at a given order at a given point is non-zero, however, depends on the coordinate system and therefore cannot correspond to an intrinsic property. What can we do, then? A good place to start would be to look for a normal form into which the jet field can be cast, at least in the vicinity of a given point in space. This consideration leads to the following proposition:

\begin{proposition}\label{jet_normal_form}
	Let $\omega \in \mathscr{J}^{*r}(\vvmathbb{R})$ be of the form $\omega =  f_1 d^x + f_2 d^{xx} + \cdots + f_r d^{\overbrace{x\cdots x}^{r ~ \mathrm{times}}}$ where $f_1 \ne 0$ at the point $p \in \vvmathbb{R}$. Then there exists a change of coordinate that puts it into the normal form $\omega_p = d^y |_p + d^{yy} |_p + \cdots + d^{\overbrace{y\cdots y}^{r ~ \mathrm{times}}} |_p$ at the point $p$.
\end{proposition}
\begin{proof}
	Without loss of generality we may take $p$ to be the origin. In one dimension, the transformation law (\ref{jet_transf_law}) becomes,
	\begin{equation}\label{jet_transf_law_1d}
		\begin{pmatrix}
			b_1 \cr b_2 \cr b_3
		\end{pmatrix}
		= \begin{pmatrix}
			x^\prime & 0 & 0 \cr
			x^{\prime\prime} & (x^\prime)^2 & 0 \cr
			x^{\prime\prime\prime} & x^\prime x^{\prime\prime} & (x^\prime)^3
		\end{pmatrix}
		\begin{pmatrix}
			a_1 \cr a_2 \cr a_3
		\end{pmatrix},
	\end{equation}
	where the continuation of the scheme to higher orders is evident. Setting $a_k := f_k(0)$, $k=1, \ldots, r$ we put the jet at the origin into the requisite form via a succession of polynomial transformations of increasing degree, which will all be invertible near the origin since $a_1 \ne 0$ by hypothesis. First the linear transformation $y=a_1 x$ leads to $x^\prime = 1/a_1$ and hence $b_1 = 1$. Relabel the $b_1, b_2, \ldots$  as $a_1, a_2, \ldots$ again and proceed to the second transformation, viz., $y=x+\frac{1}{2(1-a_2)}x^2$ so that $x^{\prime\prime}=1-a_2$ and $b_1=1,b_2=1$ where $b_3$ and on could be anything. It is clear that with a third step of the form $y=x+cx^3$ we could preserve $b_{1,2}=1$ while rendering $b_3=1$ with an appropriate choice of the coefficient $c$. All we need at the $k$-th stage would be a transformation of the form $y=x+cx^k$ such that $x^{(k)}=1-a_k$; hence, $c=1/k!(1-a_k)$. The sole non-degeneracy condition would be to ensure that one never encounters $a_k=1$. But clearly, if one were to have component coefficients of the form $(1,\ldots,1,a_k,a_{k+1},\ldots,a_r)$ with $a_k=1$, the $k$-th stage may simply be skipped since $a_k=1$ already, as we want it to. Therefore, by induction we have implicitly constructed a composite change of coordinates around the origin that leads to $\omega_0 = d^y |_0 + d^{yy} |_0 + \cdots + d^{\overbrace{y\cdots y}^{r ~ \mathrm{times}}} |_0$, as stipulated. 
\end{proof}

What is the significance of proposition \ref{jet_normal_form}? It tells us that as long as a jet over a given point satisfies the non-degeneracy condition $f_1 \ne 0$, when viewed in the right coordinates the jet could yield a non-zero result upon contraction with any higher tangent of arbitrary order. As such, the normal form at the given point of a non-degenerate jet enjoys a certain naturality. There need be no reason, of course, that there exist a coordinate in which the normal form holds identically on an open set. Nonetheless, due to continuity its components cannot immediately cross zero upon moving away from the point in question. 

We wish to suggest that the property just described is desirable for purposes of integration at higher order and merely extends what is customary in the received integral calculus on the Euclidean real line, where everybody uses the first-order differential $d^x$ so as to yield the standard Lebesgue measure which assigns to the interval $[a,b]$ the length $b-a$. The natural higher-order analogue would be
\begin{equation}\label{natural_frame_1d}
	\theta^x:=d^x+d^{xx}+\cdots+d^{\overbrace{x\cdots x}^{r ~ \mathrm{times}}}.
\end{equation}Manifestly, it will again be invariant under rigid translations. Let us remark on how to go back and forth between $d^x$ and $\theta^x$. Near the origin in the given coordinate $x$, we may write $\theta^x = x + x^2 + x^3 + \cdots$ where the series continues indefinitely, it being understood that the terms beyond the $r$-th order will be rendered immaterial upon transition to $\mathfrak{m}_0/\mathfrak{m}_0^{r+1}$. Now close enough to the origin we can find another coordinate $y$ such that $\theta^x = y ~\mathrm{mod}~ \mathfrak{m}_0^{r+1}$ via reversion of the series. Indeed, put
\begin{equation}
	1 + y = 1 + x + x^2 + \cdots = \frac{1}{1-x}
\end{equation}
so that $1 - x = \dfrac{1}{1+y}$ or (expanding)
\begin{equation}
	x = 1 - \frac{1}{1+y} = y - y^2 + y^3 - \cdots.
\end{equation}
Thus, evidently $x + x^2 + \cdots + x^r = y + o(y^r)$. But after modding out, this statement is tantamount to the sought-for relation $\theta^x = d^y$. The uniformizing coordinate $y$ is well defined of course only out to $|x|<1$. For the interval $[a,b]$ the postulated form of $\theta^x$ assigns the measure
\begin{equation}\label{interval_length_in_1d}
	\int_{[a,b]} \theta^x = b-a + \frac{1}{2}(b-a)^2 + \cdots \frac{1}{r!}(b-a)^r.
\end{equation}
In the limit of indefinitely high order, the natural measure tends to
\begin{equation}\label{interval_in_limit}
	\lim_{r \rightarrow \infty} \int_{[a,b]} \theta^x = e^{b-a}-1.
\end{equation}
As we see, countable (even finite) additivity for collections of disjoint sets fails.\footnote{Left as a homework exercise to show this. Hint: $e - 1 + e - 1 < e^2 - 1$; now extend the argument first to arbitrary intervals, then to Borel sets. By Carath{\'e}odory's extension theorem (\cite{royden}, Proposition 3.15 or Theorem 12.8), one can conclude to subadditivity but not additivity for Lebesgue-measurable sets.} Evidently measure theory at higher than first order demands a rethinking of many of our received concepts!

\begin{remark}
	For want of a better name, we shall refer to the above argument by which to establish the natural choice of a translation-invariant measure on the real line at higher order as a principle of plenitude.
\end{remark}

\subsection{The Problem of Quadrature, Stated in General}\label{quadrature_in_flat_space}

Taking our cue from the results of the previous section on the one-dimensional case, the problem of quadrature in the presence of higher-order infinitesimals has to be framed in arbitrary dimension (first of all in Euclidean space). We wish to forget about the kinds of integration we know very well (Riemann-Stieltjes, Lebesgue, Daniell etc.) and to return to the pregnant environment of the seventeenth century for inspiration. At that time, during the period leading up to the formal definition of the calculus by Leibniz \cite{leibniz_1684, leibniz_1686}, methods of quadrature were being discovered, both geometrical (Neil, Wren, van Heuraet, Huygens) and analytical (Wallis). But we intend here to focus on pictorial intuition, not on specialized techniques of limited applicability---the great virtue of Leibniz' integral being its generality. Thus, our goal in undertaking the exercise of the present section is to point the way to the proper extension of Leibniz' integral to take into account possible difformity; Newton's concept of fluents/fluxions helps us to imagine how to proceed, although Newton himself, unlike Leibniz, never proposed an analytical definition of the integral.

The first item on the agenda will be to integrate a constant-coefficient differential $d^\alpha$ where $\alpha=(\alpha_1,\ldots,\alpha_n)$ may be any multi-index with $|\alpha| \ge 1$ arbitrary. For simplicity, let us take the region of integration to be the rectangle $[a_1,b_1] \times \cdots \times [a_n,b_n] \subset \vvmathbb{R}^n$. Following suit from the previous section, it is evident that its indefinite integral ought to be given by an expression of the form
\begin{equation}
	F_\alpha := \frac{1}{\alpha!} (x_1-a_1)^{\alpha_1} \cdots (x_n-a_n)^{\alpha_n} + P,
\end{equation}
where the indefinite part $P$ will be a polynomial in the $x_{1,\ldots,n}$ of degree at most $\alpha_k-1$ in $x_k$, $k=1,\ldots,n$ (adopting the convention that a term of negative degree in one of its variables has only a multiplicative constant dependence on that variable). Thus, the definite integral becomes
\begin{equation}\label{simple_coord_differential}
	\int_{} d^\alpha = 
	\frac{1}{\alpha!} (b_1-a_1)^{\alpha_1} \cdots (b_n-a_n)^{\alpha_n} + P.
\end{equation}
Now for a differential with spatial dependence, $f d^\alpha$, with $f$ of compact support inside a rectangle $[a_1,b_1] \times \cdots \times [a_n,b_n] \subset \vvmathbb{R}^n$, say, in analogy to equation (\ref{antideriv_def}) we have for the antiderivative the following:
\begin{equation}
	F_\alpha(x_1,\ldots,x_n) := \left( \prod_{i=1,\ldots,n} \int_{a_i}^{x_{i1}} dx^{i1} \cdots \int_{a_i}^{x_{i\alpha_i}} dx^{i\alpha_i} \right) 
	f(x_{1\alpha_1},\ldots,x_{n\alpha_n}) + P,	
\end{equation}
where by convention a term in the product is to be omitted whenever $\alpha_i=0$ and where $x_{i0}=x_{i1}=x^i$, i.e., the $i$-th coordinate variable appearing of the left-hand side.

Following equation (\ref{natural_frame_1d}), in place of the simple coordinate differentials of equation (\ref{simple_coord_differential}) the natural frame in many dimensions specified by $\theta^{1,\ldots,n}$ would have the corresponding differential expression for the integrand: $\theta^{x_1}\cdots\theta^{x_n}$. The volume of rectangular prism in analogy to equation (\ref{interval_length_in_1d}) evidently becomes just the \textit{sum} resp. \textit{product},
\begin{align}\label{interval_length_in_many_d}
	\int_{[a_1,b_1] \times \cdots \times [a_n,b_n]} \theta^{x_1}\cdots\theta^{x_n} &= 
	\sum_{(1,\ldots,1) \le \beta \le (r,\ldots,r)}
	\frac{1}{\beta!} (b_1-a_1)^{\beta_1} \cdots (b_n-a_n)^{\beta_n} \\
	&= \prod_{i=1,\ldots,n} \left(
	b_i-a_i + \frac{1}{2}(b_i-a_i)^2 + \cdots \frac{1}{r!}(b_i-a_i)^r
	\right)
\end{align}
since the multidimensional integral considered in the usual sense of the multivariate calculus factorizes. Note: what everyone has hitherto regarded as the volume of a rectangular prism corresponds to the $(1,\ldots,1)$-term alone, which in the present context becomes merely the first term in a series. Evidently in the limit of indefinitely high order, we have
\begin{equation}
	\lim_{r \rightarrow \infty} \int_{[a_1,b_1] \times \cdots \times [a_n,b_n]} \theta^{x_1}\cdots\theta^{x_n} 
	= \prod_{i=1,\ldots,n} \left( e^{b_i-a_i} - 1 \right)
\end{equation}
which extends equation (\ref{interval_in_limit}) to define, once again, a subadditive but not additive translation-invariant measure on Borel sets in Euclidean $\vvmathbb{R}^n$.

To illustrate the complications that arise when measuring regions of other than rectangular shape, let us compute the area of a disk to higher order. If we proceed directly, we should find,
\begin{equation}
	\int_{\vvmathbb{D}} \theta^x \theta^y = \int_{-1}^1 \theta^x \left( \int_{-\sqrt{1-x^2}}^{\sqrt{1-x^2}} \theta^y \right) = \int_{-1}^1 \theta^x \left( e^{2\sqrt{1-x^2}} - 1 \right),
\end{equation}
which already cannot be evaluated in explicit form. Note that the area of the disk will no longer be twice the area of half of a disk, which would instead be given by
\begin{equation}
	2 \int_{-1}^1 \theta^x \left( e^{\sqrt{1-x^2}} - 1 \right).
\end{equation}
We may, however, with some ingenuity evaluate the area of the disk by going to polar coordinates, where we suppose the element of area to be given now by $r \theta^r \theta^\phi$. Then we have that the integral factorizes into,
\begin{equation}
	\int_0^1 r\theta^r \int_0^{2\pi} \theta^\phi =
	\left( e^{2\pi} - 1 \right) \left( \frac{1}{2}r^2 \bigg|_0^1 + \frac{1}{6} r^3 \bigg|_0^1 + \cdots \right) = \left( e^{2\pi} - 1 \right) \left( e - 2 \right).
\end{equation}
In order to recover the usual result, we must introduce a scale of length $L$ and register the radial coordinate in terms of the dimensionless quantity $r/L$. As for the angular coordinate, it is implicit that higher differentials should decline and become negligible as $L$ tends to zero. For convenience, set $L=1/k$. It is then not hard to derive the asymptotic formula,
\begin{equation}
	\int_0^1 r\theta^r \int_0^{2\pi} \theta^\phi = k^3 \left( e^{2\pi/k} - 1 \right) \left( e^{1/k} - 1 - \frac{1}{k} \right) = \pi + \frac{\pi^2+\pi/3}{k} + \frac{8 \pi^3 + 4\pi^2 + \pi}{12k^2} + \cdots
\end{equation}
The reader will appreciate that all of the familiar quadratures become more delicate once higher differentials have to be considered. 

\subsection{Intrinsic Integration on Manifolds in the Presence of Higher-Order Infinitesimals}\label{integration_on_manif}

Our basic idea can be simply put: in the ordinary theory of integration on manifolds, the integrand enters under the guise of a top-degree differential form. Why? Any alternating form assigns to a collection of linearly independent vectors the volume of the parallelepiped whose edges are spanned by the respective vectors (see Lee, \cite{lee_smooth_manif}, Problem 16.1). A differential form could be viewed as nothing but a convenient apparatus for keeping track of such a rule of assignment of volume in a coordinate-invariant manner as one moves from point to point in space. If the picture of higher-order infinitesimals we have been developing has any cogency, it stands to reason then that when one goes to the general case one ought to employ a collection of linearly independent higher tangents in order to play the role of the 1-vectors spanning the sides of the parallelepiped. Pictorially, we want, so to speak, a curvilinear parallelepiped. We must still expect an infinitesimal volume element to be determinantal so an alternating generalized tensor of some kind (i.e., a differential form) would seem to be the natural candidate for a volume form against which to integrate.

\subsubsection{Definition of a Generalized Volume Form and its Integration in many Dimensions}

The appropriate theory can now be obtained by following (mutatis mutandis) Lee's particularly lucid exposition in Chapter 16 of \cite{lee_smooth_manif}. We have already defined in {\S}\ref{quadrature_in_flat_space} how to integrate a differential form of compact support in Euclidean space. As in \cite{lee_smooth_manif}, we may use diffeomorphism invariance and partitions of unity in order to extend this definition to differential forms defined on orientable manifolds in a manner that will be invariant under orientation-preserving diffeomorphisms (thus, independent of any choice of local chart). The major subtlety that arises in connection with the presence of differentials of higher than first order is that we must introduce a distinction between macroscopic versus microscopic directions in space. For, integration in the usual sense makes sense only when the degree of the differential form agrees with the number of dimensions of space. But in the present context, the space of jets in the stalk over a given point in a differentiable $n$-manifold is spanned not only by the $d^{x_{1,\ldots,n}}$ but also includes differentials up to the $r$-th degree, where $r \ge 1$. Therefore, an arbitrary $n$-form could be degenerate in the sense that there may exist directions in the image of the tangent bundle under the canonical injection along which contraction against the given $n$-form yields zero---a phenomenon that cannot occur with a top-degree form in the ordinary sense on a manifold. Clearly, our intuitive idea of integration over space implies that the integrand ought to fill out all possible directions in space around every point. If we look more closely at the root of the problematical phenomenon just mentioned, it happens because implicitly we are operating with a more comprehensive notion of the term `direction' than what one would be confronted with, limiting oneself to infinitesimals of the first order. At first order, every possible direction corresponds to an infinitesimal translation of space. At higher order, though, when we speak of a many-fold direction corresponding to the differential $d^\alpha$ where $|\alpha| \ge 2$, we refer to a virtual motion in cotangent space that need not give rise to a collective infinitesimal spatial movement in the base space $M$ (in the case of constant coefficients). Therefore, we ought to distinguish between this case and the first-order case, and the terminology that seems fitting would be to say that the former represents a microscopic direction in space as opposed to the $n$ independent macroscopic directions with which we are familiar at first order.

How to render the concept precise? We want to say that a given $n$-form $\omega \in \Omega^n(M)$ is macrosopic if it is non-degenerate in all $n$ indepedent macroscopic directions; in other words, if $J^1_p(M) \subset J^r_p(M)$ denotes the image of the canonical injection, then we should have that $\ker \omega_p|_{J^1(M)} = 0$ for every $p \in M$.

\begin{lemma}\label{lead_macroscopic_vol_form}
	Let $\omega$ be a macroscopic $n$-form on the manifold $M$. Then around any point $p \in M$, it can be written uniquely in local coordinates in a chart $p \in U \subset M$ as $\omega = f \theta^{x_1} \wedge \cdots \wedge \theta^{x_n} + \eta$, where $f$ is a smooth function and $\eta \in \ker \pi_{r1}$ lies in the kernel of the canonical projection onto $\mathscr{J}^{1*}(U) \wedge \cdots \wedge \mathscr{J}^{1*}(U)$.
\end{lemma}
\begin{proof}
	The image of $\omega$ under $\pi_{r1}$ takes the form $f d^{x_1} \wedge \cdots d^{x_n}$, where owing to the macroscopic hypothesis $f$ must be everywhere non-zero. Now, modulo the kernel of $\pi_{r1}$ we are free to rewrite this in the desired form as $\omega =  f \theta^{x_1} \wedge \cdots \wedge \theta^{x_n} + \eta$ if we susbstitute $\eta := \omega - f \theta^{x_1} \wedge \cdots \wedge \theta^{x_n}$. As for uniqueness, suppose $\omega =  f_1 \theta^{x_1} \wedge \cdots \wedge \theta^{x_n} + \eta_1 = \omega =  f_2 \theta^{x_1} \wedge \cdots \wedge \theta^{x_n} + \eta_2$. Applying $\pi_{r1}$ to both sides we get $f_1 d^{x_1} \wedge \cdots d^{x_n} = f_2 d^{x_1} \wedge \cdots d^{x_n}$ from which we may conclude that $f_1 = f_2$ identically on $U$.
\end{proof}

\begin{remark}
	Here is where we make use of the condition that $\omega$ be macroscopic. The decision to project along $\theta^{x_1} \wedge \cdots \wedge \theta^{x_n}$ rather than along any other representative of $d^{x_1} \wedge \cdots d^{x_n}$ in the equivalence class modulo $\ker \pi_{r1}$ is not necessary of course, but motivated by the principle of plenitude.
\end{remark}
The above lemma will best be viewed as the first step in a finite inductive process, where we adopt the following notation:
\begin{equation}
	\theta^\alpha := \overbrace{\theta^{x_1} \cdots \theta^{x_1}}^{\alpha_1 ~\mathrm{times}} \cdots \overbrace{\theta^{x_n} \cdots \theta^{x_n}}^{\alpha_n ~\mathrm{times}}.
\end{equation}

\begin{lemma}\label{full_macroscopic_vol_form}
	Let $\omega$ be an arbitrary $n$-form on the manifold $M$. Then around any point $p \in M$, it can be written uniquely in local coordinates in a chart $p \in U \subset M$ as $\omega = f_A \theta^{\wedge A}$, where for the grand multi-index $A=(A_n,\ldots,A_n)$ we introduce the notation $\theta^{\wedge A} := \theta^{A_1} \wedge \cdots \wedge \theta^{A_n}$ (note that in this case we may presume the $\alpha_{1,\ldots,n}$ to be distinct).
\end{lemma}
\begin{proof}
	By lemma \ref{lead_macroscopic_vol_form}, $f_{(1,\ldots,1)}$ is uniquely fixed. Suppose now that the $f_A$ have been determined for all multi-indices with $\max(|A_1|,\ldots,|A_n|)\le s$. We wish to find the $f_A$ with $\max(|A_1|,\ldots,|A_n|)\le s+1$. To this end, denote the canonical projection from $\mathscr{J}^{r*} \wedge \cdots \wedge \mathscr{J}^{r*}$ to
	$\mathscr{J}^{s*} \wedge \cdots \wedge \mathscr{J}^{s*}$, $s \le r$, by $\pi_{rs}$. In other words, we have already selected a representative of $\omega$ modulo $\ker \pi_{rs}$. Clearly, $\eta_s := \omega - \sum_{A: \max(|A_1|,\ldots,|A_n|)\le s} f_A \theta^{\wedge A}$ consists in a linear combination of terms like $g_B d^{B_1} \wedge \cdots \wedge d^{B_n}$ with at least one of the $|B_{1,\ldots,n}| \ge s+1$. Simply set $f_B = g_B$ for all $B$ with $\max(|B_1|,\ldots,|B_n|)=s+1$. It remains to show their uniqueness just as before. For suppose
	\begin{equation}
		\omega = \sum_{A: \max(|A_1|,\ldots,|A_n|)\le s+1} f_A \theta^{\wedge A} + \eta_{s+1} = 
		\sum_{A: \max(|A_1|,\ldots,|A_n|)\le s+1} g_A \theta^{\wedge A} + \zeta_{s+1}
	\end{equation}
	with $\eta, \zeta \in \ker \pi_{r,s+1}$. Apply $\pi_{r,s+1}$ to both sides in order to arrive at a linear combination of terms in $d^{A_1} \wedge \cdots \wedge d^{A_n}$ with $\max(|A_1|,\ldots,|A_n|) \le s+1$. These being linearly independent, though, we may infer that $f_A = g_A$ for all multi-indices $A$ with $\max(|A_1|,\ldots,|A_n|) \le s+1$, as wanted. This completes the inductive step.
\end{proof}

\begin{proposition}
	Let $\omega$ be a macroscopic $n$-form on the manifold $M$. Then $\omega$ factorizes in the sense that there are a smooth function $f$ and $1$-forms $\eta^{1,\ldots,n}$ such that $\omega=f\eta^1 \wedge \cdots \wedge \eta^n$.	These quantities are unique up to non-zero factors $\lambda_{0,1,\ldots,n}$ such that $\prod_{k=0}^n \lambda_k = 1$.
\end{proposition}
\begin{proof}
	For the sake of simplicity, consider first the case in which $\omega$ assumes the form (without loss of generality taking the natural basis of 1-forms)
	\begin{equation}
		\omega = f \varepsilon^1 \wedge \cdots \wedge \varepsilon^n + g_{\mu_{1,\ldots,n}\nu_{1,\ldots,n}} \varepsilon^{\mu_1\nu_1} \wedge \cdots \wedge \varepsilon^{\mu_n\nu_n},
	\end{equation}
	where $f \ne 0$ by reason of the macroscopic assumption and without loss of generality, the $g_{\mu_{1,\ldots,n}\nu_{1,\ldots,n}}$ may be taken to be totally antisymmetric. Here, $\mu_i\nu_i=20,11,02$, $i=1,\ldots,n$. If we write
	\begin{equation}
		\eta^i = \varepsilon^i + h^i_{jk} \varepsilon^{jk},
	\end{equation}	
	then the condition that
	\begin{equation}
		\omega = f \eta^1 \wedge \cdots \wedge \eta^n
	\end{equation}
	becomes a system of equations for each multi-index $A, |A|=2$: 
	\begin{equation}
		f \prod_{k=1}^n h^A_{\mu_k\nu_k} = g_{\mu_{1,\ldots,n}\nu_{1,\ldots,n}}.
	\end{equation}
	Divide through by $f$ since $f$ is non-zero by hypothesis and take the logarithm of both sides:
	\begin{equation}
		\sum_{k=1,\ldots,n} \ln \pm h^A_{\mu_k\nu_k} = \ln g_{\mu_{1,\ldots,n}\nu_{1,\ldots,n}}/f.
	\end{equation}
	By insisting on lexicographical order of multi-indices, however, we may ensure that the resulting system of equations corresponds to an inhomogeneous linear equation with a lower-triangular matrix with constant coefficients. The matrix will be invertible as all of its diagonal elements equal $\pm 1$. Solve for the $\ln \pm h^A_{\mu_k\nu_k}$ and exponentiate. If it should happen that one of the $g_{\mu_{1,\ldots,n}\nu_{1,\ldots,n}}$ equal zero, set it to something non-zero and take the limit as its absolute value tends back to zero. In this limit, one or more of the $\ln \pm  h^A_{\mu_k\nu_k}$ will likewise diverge to negative infinity, but this just tells us that the corresponding $ h^A_{\mu_k\nu_k}$ should be set itself to zero. The coefficients $h^A_{\mu_k\nu_k}$, $k=1,\ldots,n$ are uniquely defined since inhomogenous linear equations have unique solutions. We are free, however, to scale $f$ by $\lambda_0$ and $\eta_k$ by $\lambda_k$ in such a way that $\prod_{k=0}^n \lambda_k = 1$ without affecting the result.
	
	The same procedure applies mutatis mutandis when we expand $\omega$ into a lexicographically ordered sum over multi-indices $A, |A|=1,\ldots,r$. This concludes the proof.
\end{proof}

\begin{proposition}\label{local_1_frame}
	Let $\omega$ be a macroscopic smooth $n$-form on the manifold $M$ that factorizes as $\omega = \eta^1 \wedge \cdots \wedge \eta^n$ (suppressing the factor of $f \ne 0$ by distributing it somehow among the $\eta^{1,\ldots,n}$) and further suppose that around a given point there exists a local coordinate system $x^{1,\ldots,n}$ in terms of which the coefficients of $\omega$ become analytic functions. Then there exists another local coordinate system $z^{1,\ldots,n}$ such that with respect to these coordinates $\omega$ simplifies to just $d^{z_1} \wedge \cdots \wedge d^{z_n}$ at the given point.
\end{proposition}
\begin{proof}
	It will be convenient to regard the $z^{1,\ldots,n}$ as the source and the $x^{1,\ldots,n}$ as the target. The indicated form will exist if we can choose the mapping from the source to the target such that the transition function of equation (\ref{jet_transf_law}) can yield any desired $\eta^{1,\ldots,n}$ on the left-hand side, where on the right-hand side we want these to reduce just to a Kronecker delta function in the first $n$ coordinates. But the lower block-diagonal structure of the transition function comes to our aid. For the required coefficients, say, of $\eta^1$ with respect to the target coordinates can readily be obtained from $d^{z_1}$ on the right-hand side merely with a corresponding column vector in the first position, the entries of which will be the successive derivatives of $z^1$ with respect to collections of the $x^{1,\ldots,n}$ coordinates for each multi-index. The same goes for $d^{z_k}$ for any $k=1,\ldots,n$. Now in the transformation at first order in the upper leftmost block will be invertible by virtue of the macroscopic hypothesis, which in turn guarantees invertibility of all the blocks on the diagonal. In this manner we arrive at a power series solution (i.e., polynomial to any finite degree $r$) defining the diffeomorphism from the source to the target coordinates. Evidently the power series must have a finite radius of convergence due to the implicit function theorem along with the hypothesis of analyticity.
\end{proof}

With this preliminary material out of the way, we follow Lee's exposition (in the first section of his Chapter 16) in order to supply a sense to the integral of a differential form on a manifold.

\begin{definition}\label{Euclidean_integration}
	Let $D \subset \vvmathbb{R}^n$ be a domain of integration; viz., a compact set whose boundary has measure zero. Suppose as guaranteed by lemma \ref{full_macroscopic_vol_form} the $n$-form $\omega$ can be put into the form $\omega = f_A  \theta^{A_1} \wedge \cdots \wedge \theta^{A_n}$ where the continuous functions $f_A$ have support inside $D$ and we fix the sign by requiring the the $A_{1,\ldots,n}$ to be in the obvious lexicographical order. Then we define the integral of $\omega$ over $D$ to be
	\begin{equation}
		\int_D \omega := \int_D f_A \theta^A,
	\end{equation}
	understanding by $\theta^A$ namely, the product $\theta^{A_1} \cdots \theta^{A_n}$. Speaking more generally, if $U \subset \vvmathbb{R}^n$ is an open set and $\omega$ is compactly supported in $U$, following Lee we define $\int_U \omega = \int_D \omega$ for any domain of integration $D \supset \mathrm{supp}~ \omega$ (extending $\omega$ by zero on the complement of its support).
\end{definition}

\begin{remark}
	By the uniqueness in lemma \ref{full_macroscopic_vol_form} the functions $f_A$ in definition \ref{Euclidean_integration} are well defined, but we shall nevertheless point out why we are justified in thus distinguishing them. The idea is that the $f_A$ with $\max(|A_1|,\ldots,|A_n|) \ge 2$ correspond to microscopic parts of the integrand (indeed, $f_B$ is microscopic relative to $f_A$ if $\max(|A_1|,\ldots,|A_n|) < \max(|B_1|,\dots,|B_n|)$). Anything lying in the kernel of $\pi_{r1}$ will be representable as a finite linear combination of terms of the form $g \theta^{A_1} \wedge \cdots \wedge \theta^{A_n}$ where at least one of the $A_{1,\ldots,n}$ satisfies $|A_{1,\ldots,n}| \ge 2$. But then each respective term in $\omega - f_{(1,\ldots,1)} \theta^{x_1} \wedge \cdots \wedge \theta^{x_n}$ will be of at least one higher order in the differentials than the main part of the integrand $f_{(1,\ldots,1)}$ and therefore should be negligible in comparison.
\end{remark}

\begin{remark}
	Manifolds with corners are not really needed for the theory of integration, but convenient to employ with many possible parametrizations (for instance, the torus $S^1 \times S^1$ where one would like to adopt as coordinates the angles in each respective copy of $S^1$). There is no call to formalize the concept, however, in as much as in any application one could just resort to a bump function whose value equals unity away from a neighborhood of the boundary, where the corners may reside, and equal to zero on the boundary itself. Multiply the integrand by this bump function in order to permit it to be integrated according to the usual theory without corners. Then pass to the limit where the size of the neighborhood of the boundary tends to zero.
\end{remark}

Having settled upon a reasonable convention as to what we should mean by higher-order integration in a compact domain in Euclidean space, the generalization to compactly supported macroscopic volume forms on arbitrary manifolds follows fairly straightforwardly along the lines of Lee \cite{lee_smooth_manif}, Chapter 16.

Clearly, we are going to need some generalization of the concept of a Jacobian and thus an organized fashion of writing out the change of coordinate formula (\ref{jet_transf_law}) when applied to the standard basis of jets. Therefore, we wish to find an expression for the generalized Jacobian transformation,
\begin{equation}\label{jacobian_def}
	\hat{J} := \begin{pmatrix}
		J^{(1,1)} & 0 & 0 \cr
		J^{(2,1)} & J^{(2,2)} & 0 \cr
		J^{(3,1)} & J^{(3,2)} & J^{(3,3)} 
	\end{pmatrix}.
\end{equation}
It is immediate that $J^{(m,m)} = \overbrace{J \otimes \cdots \otimes J}^{m ~\mathrm{times}} = J^{\otimes m}$ where
\begin{equation}
	J = \frac{\partial(x_1,\ldots,x_n)}{\partial(y_1,\ldots,y_n)}
\end{equation}
is the ordinary Jacobian. Equally evident from equation (\ref{jet_transf_law}) would be that
\begin{equation}
	J^{(m,1)} := \frac{\partial^m(x_1,\ldots,x_n)}{\overbrace{\partial(y_1,\ldots,y_n)\cdots\partial(y_1,\ldots,y_n)}^{m ~\mathrm{times}}}.
\end{equation}
A little thought leads to a formula for $J^{(b,a)}$:
\begin{equation}
	J^{(b,a)} := \sum_{a_{1,\ldots,a},b_{1.\ldots,b}: a_1+\cdots+a_a=a, b_1+\cdots+b_b=b} 
	\frac{a_1!\cdots a_a!b!}{b_1!\cdots b_b!a!}
	J^{(b_1,a_1)} \otimes \cdots \otimes J^{(b_b,a_a)},
\end{equation}
where implicitly one has to work one's way up in the respective orders of the jets, starting from $a,b=1$. Thus, for instance,
\begin{equation}
	J^{(3,2)} = \frac{1}{2} \left( J^{(2,1)} \otimes J^{(1,1)} + J^{(1,1)} \otimes J^{(2,1)} \right).
\end{equation}

Now, if we are to frame an intrinsic theory of integration on manifolds we stand in need of a pullback formula by which to express how the integrand behaves under change of coordinate. There are two senses in which the notion of integrand could be understood: either first, as a compound differential in all the Cartesian coordinates or second, as a top-degree volume differential form. The first sense would be of course the more basic in the order of discovery, but the second would be preferable from an advanced point of view in which all geometrical objects are to be described intrinsically; i.e., in a coordinate-free manner. The salient difference between the two is that differentials are commutative while volume forms are totally antisymmetric and the real significance of the pullback formula is that it mediates between the two notions of integrand. In the generalized setting in which we are working, the integrand is to be interpreted as a macroscopic volume $n$-form in jets of up to the $r$-th order. Due to linearity, we may allow the Jacobian to act on the individual components of the volume form and collect it in the result as a prefactor. As everyone knows, in the usual multivariate calculus this procedure gives rise to the determinant of the Jacobian. Clearly, if we proceed along similar lines we shall obtain a generalized formula that replaces the determinant, which will no longer be a scalar since the space of $n$-forms has greater than one in dimension (cf. Halmos, \cite{halmos}, {\S}53). Thus, with the aid of the following definition the requisite generalized pullback formula becomes immediate:

\begin{definition}
	Let $M$ be a differentiable manifold of dimensions $n$ on which $U, V$ are two coordinate charts at a given point $p \in M$ and denote the Jacobian between them by $\hat{J}$, as above in equation (\ref{jacobian_def}). Then by its determinant we understand,
	\begin{equation}
		\det \hat{J} := \bigwedge \overbrace{ \hat{J} \otimes \cdots \otimes \hat{J}}^{n ~\mathrm{times}}.
	\end{equation}
	
\end{definition}

\begin{proposition}[Pullback Formula; cf. Lee, Proposition 14.20]\label{pullback_formula}
	Let $F: M \rightarrow N$ be a smooth map between $n$-manifolds. If $x_{1,\ldots,n}$ and $y_{1,\ldots,n}$ are smooth coordinates on open subsets $U \subset M$ and $V \subset N$ respectively, and $\omega = f_A \theta^{\wedge A}$ is a macroscopic volume form on $N$, then the following holds on $U \cap F^{-1}[V]$:
	\begin{equation}
		F^* \omega = f_A \theta^{y\wedge A} = \omega_A \circ F \det \hat{J}^A_B ~\theta^{x\wedge B},
	\end{equation}
	where $\hat{J}$ represents the Jacobian transformation associated to $F$ in these coordinates and a summation over grand multi-indices $A, B$ is by convention understood.
\end{proposition}
\begin{proof}
	Evaluate both sides on $\partial_{x\beta}$. Equality will be justified by virtue of lemma \ref{pullback_formula} and proposition \ref{properties_exterior_deriv}.
	From
	\begin{equation}
		F^* \text{\th} \left( \frac{1}{\alpha_1!} y^{\alpha_1} \right) \wedge \cdots \text{\th} \left( \frac{1}{\alpha_n!} y^{\alpha_n} \right)=
		\text{\th} \left( \frac{1}{\alpha_1!} y^{\alpha_1} \circ F \right) \wedge \cdots \text{\th} \left( \frac{1}{\alpha_n!} y^{\alpha_n} \circ F \right) 
	\end{equation}
	we obtain after antisymmetrizing,
	\begin{equation}
		\text{\th} \left( \frac{1}{\alpha_1!} y^{\alpha_1} \circ F \right) \wedge \cdots \text{\th} \left( \frac{1}{\alpha_n!} y^{\alpha_n} \circ F \right) \left( \partial_{x\beta_1},\ldots,\partial_{x\beta_n} \right) =\det \hat{J}^A_B.
	\end{equation}
	If we allow the grand multi-index $B$ to be arbitrary, evidently one must sum on the right-hand side over the $n$-forms $\theta^{x\wedge B}$ and the indicated pullback formula results.
\end{proof}

By convention the differentials of whatever order in the integrand should have a common positive sign, but of course $\hat{J}$ can flip them so we need to set them right again by hand; the phenomenon arises already in the usual theory where a diffeomorphism can be either orientation-preserving or -reversing and the sign of the integral has to be adjusted accordingly (see Lee, \cite{lee_smooth_manif}, Propositions 16.1, 16.6); now there will be multiple orientation types of a mapping corresponding to the possibility of flipping the sign of each piece of the integrand separately. By inspection of the pullback formula, one infers that the orientation type is to be determined by the collection of mappings $J^{(b,a)}$ with $b \ge a$ and their signatures. In the case $a=b=1$, the signature must be given by $\mathrm{sgn}~ \hat{J}^{(1,1)} = \det J^{(1,1)}$. The signatures cannot all be independent as for instance $\det J^{(m,m)} = (\det J^{(1,1)})^{mn}$; in particular, it will always be positive for even $m$. This fact appears to be very material and we shall revisit it below, but there is no reason for $J^{(b_1,a)}$ to bear any relation to $J^{(b_2,a)}$ for $b_2 > b_1$ since nothing connects the partial derivatives of the transformation of different order.

Hence, we are seeking to generalize the concept of orientation of a manifold in a suitable sense. First, a pertinent observation. The space of jets carries a natural grading by order of the jet, from which one derives a natural grading of the $n$-fold alternating product according to $|A|$ where for the grand multi-index $A=(A_1,\ldots,A_n)$, we have $|A| := \max (|A_1|,\ldots,|A_n|)$. It will be convenient to employ the natural frame. Therefore, for $m=1,\ldots,n$ define
\begin{equation}
	\Lambda^{n,m} := \mathrm{span}~ \{ \theta^{\wedge A} : |A| = m \} 
\end{equation}
and equip it with the standard orientation corresponding to the basis elements $\theta^{\wedge A}$ arranged in lexicographical order.

\begin{lemma}\label{det_lemma}
	The sign of $\det \hat{J}^B_A$ depends only on whether $\hat{J}$ is orientation-preserving or orientation-reversing considered as a mapping from $\Lambda^{n,|A|}$ onto its image in $\Lambda^{n,|B|}$ in the induced orientation; in particular, it depends only on $|A|$ and $|B|$.
\end{lemma}
\begin{proof}
	Let $V \subset \Lambda^{n,|B|}$ denote the image of $\Lambda^{n,|A|}$ under $\hat{J}$. Since $\hat{J}$ derives from a change of coordinates, i.e., a diffeomorphism, its kernel vanishes and thus $V$ has the same dimension as $\Lambda^{n,|A|}$, call it $N_{|A|}$. According to Lee, \cite{lee_smooth_manif}, Proposition 15.3, each non-zero element $\eta \in \Lambda^{N_{|A|}}V^*$ determines an orientation of $V$, of which we have two:
	first, the restriction of the standard orientation $o_{|B|}$ in $\Lambda^{n,|B|}$ to $V$ and second, the pullback of the standard orientation $o_{|A|}$ in $\Lambda^{n,|A|}$ under $\hat{J}^{-1}|_V$. But both $o_{|B|}|_V$ and $\hat{J}^{-1}|_V^* o_{|A|}$ being top-degree $N_{|A|}$-covectors, they will be either a positive or a negative multiple of each other depending on whether $\hat{J}|_{\Lambda^{n,|A|}}$ is orientation-preserving or -reversing. Thus, 
	$\hat{J}^{-1}|_V^* o_{|A|}=\zeta o_{|B|}|_V$ with $\zeta=1$ resp. $\zeta=-1$.
	
	Now for a given grand multi-index $A=(A_1,\ldots,A_n)$, write $e_A := \theta^{\wedge A}$ and its image in $V$, $\left(\hat{J}^{-1}\right)^* e_A = J^{|B|,|A|} e_A = \det \hat{J}^B_A e_B$. By definition, however, $o_{|A|} = \bigwedge_A e_A$ and $o_{|B|} = \bigwedge_B e_B$, the wedge products being taken in lexicographical order. If we define a so-called superdeterminant $\mathrm{Det}~ J^{(|B|,|A|)}$ as the determinant of the non-singular mapping $\det \hat{J}: \Lambda^{n,|A|} \rightarrow V$, unwinding the definitions leads to the conclusion that it is equal to $\zeta$.
\end{proof}
Hence, we may define $\mathrm{sgn}~ \hat{J}^A_B := \mathrm{sgn}~ \det \hat{J}^A_B$. Thus, the orientation type of a mapping is entirely determined by $\frac{1}{2}r(r+1)$ plus-or-minus signs for $|A|,|B|=11,12,\ldots,1r,22,23,$ $\ldots,2r,\ldots,rr$. In the ordinary theory of integration on differentiable manifolds, $r=1$ so there is only one possible choice of sign to make, for $|A|=|B|=1$.

The sense in which the concept of orientation is to be generalized to macroscopic volume forms may best be described as introducing a relative orientation. To follow Lee in \cite{lee_smooth_manif}, Chapter 15, we begin with a concept of relative orientation to be defined on an arbitrary pair of vectors spaces $W \subset V$. As Lee observes, in default of any canonical ordering of bases in $W$ resp. $V$, we have to bring one in via the deus ex machina of a labeled basis, by which we intend an \textit{ordered} collection of vectors $e^{1,\ldots,\mathrm{dim}~V}$ forming a basis of $V$ such that the first $\mathrm{dim}~W$ of them form a basis of $W$, regarded as the coset space $V/W^\prime$ for some complementary subspace $W^\prime$; i.e., satisfying $V=W+W^\prime$.

We shall say that two labeled bases $e_{1,\ldots,\mathrm{dim}~V}$ and $f_{1,\ldots,\mathrm{dim}~V}$ are consistently ordered if the transition mapping $J:V \rightarrow V$ sending $e^\alpha$ to $f^\alpha$ has positive determinant when restricted to $W$ resp. $W^\prime$. More generally, suppose that $V=W_1+\cdots+W_r$ and group the labeled basis vectors into $r$ classes, $e_{1,\ldots,\mathrm{dim}~W_1}$, $e_{\mathrm{dim}~W_1+1,\ldots,\mathrm{dim}~W_1+\mathrm{dim}~W_2}$ etc., where
\begin{equation}
	W_1+\cdots+W_s = V/(W_{s+1}+\cdots+W_r), \qquad s=1,\ldots,r-1.				
\end{equation} and say that the two labeled bases are consistently if $J$ is block-diagonal with respect to the decompositon $V=W_1+\cdots+W_r$ with each non-zero block having positive determinant in the sense of lemma \ref{det_lemma} (which will readily be seen not to have any dependence on the manifold structure but only on the decomposition of the space of jets into the sum $\Lambda^{n,1}+\cdots+\Lambda^{n,r}$).

From the product rule for Jacobians it follows that the property of being consistently oriented is an equivalence relation on the set of all labeled bases.
Therefore, we have a well-defined notion of an orientation for $V$, namely, an equivalence class of its labeled bases. Let $M$ be a smooth manifold with or without boundary. A pointwise orientation on $M$ designates a choice of orientation of each higher tangent space $J^r_pM$, which decomposes into a sum of parts of increasing degree using the canonical injections, starting with $T_pM=J^1_pM$. The following two propositions encapsulate what we need to know for the purpose of integration:

\begin{proposition}[Cf. Lee, Proposition 15.3]\label{orientation_type}
	Let $V=W_1+\cdots+W_r$ be a vector space of dimension $N=N_1+\cdots+N_r$, where $N_s=\mathrm{dim}~W_s, s=1,\ldots,r$. An collection of elements $\omega^{(1,\ldots,r)} \in \bigwedge^* V^*$ will be said to be full if $\omega^{(s)} \in \bigwedge^{N_s} V^*, s=1,\ldots,r$. Then each full collection of orientation forms determines a unique orientation type $O_\omega$ for $V$ as follows: $O_\omega$ will be the set of labeled bases $e_{1,\ldots,N}$ such that $\omega^{1,\ldots,r)}(e^{N_1+\cdots+N_{s-1}+1},\ldots,e^{N_1+\cdots+N_s})>0$.
\end{proposition}
\begin{proof}
	The zero-dimensional case is trivial. Given $\omega^{(1,\ldots,r)}$ and $O_\omega$, we have to show every labeled basis $e_{1,\ldots,N} \in O_\omega$ belongs to the same orientation type as does every other labeled basis $f_{1,\ldots,N} \in O_\omega$. Let $B:V \rightarrow V$ be the map sending $e^\alpha$ to $f^\alpha$, $\alpha=1,\ldots,N$:
	\begin{align}
		\omega(f_{N_1+\cdots+N_{s-1}+1},\ldots,f_{N_1+\cdots+N_s}) &=
		\omega(Be_{N_1+\cdots+N_{s-1}+1},\ldots,Be_{N_1+\cdots+N_s}) \nonumber \\
		&=\det B^{(s,s)} \omega(e_{N_1+\cdots+N_{s-1}+1},\ldots,e_{N_1+\cdots+N_s}).
	\end{align}
	Hence, the two bases can belong to $O_\omega$ if and only if $\omega(e_{N_1+\cdots+N_{s-1}+1},\ldots,e_{N_1+\cdots+N_s})$ and $\omega(f_{N_1+\cdots+N_{s-1}+1},\ldots,f_{N_1+\cdots+N_s})$ have the same sign.
	The same goes for $\omega^{(s)}$ and the restriction of $\omega^{(s^\prime)}$ to  $W_1+\cdots+W_s$, $s^\prime>s$, where now the signature of $B^{(s^\prime,s)}$ is positive.
\end{proof}

\begin{proposition}[Cf. Lee, Proposition 15.5]
	Let $M$ be a smooth, connected $n$-manifold with or without boundary. Every full collection of non-vanishing orientation forms determines a unique orientation on $M$ for which the orientation type is positive at each point.
\end{proposition}
\begin{proof}
	In view of proposition \ref{orientation_type}, all that has to be checked is continuity. Choose a labeled basis of $J^r$ at each point that belongs to $O_\omega$ there. Then 
	\begin{equation}
		\omega^{(s)} = f_s e^{N_1+\cdots+N_{s-1}+1} \wedge \cdots \wedge e^{N_1+\cdots+N_s}
	\end{equation}
	for some continuous functions $f_{1,\ldots,r}$. Nonvanishing of the orientation forms implies that of the $f_{1,\ldots,r}$, which therefore cannot change sign within any connected neighborhood of a point. If negative, one may flip the sign of one of the respective basis vectors to make it positive. Therefore the pointwise orientation type is continuously defined on each connected component of $M$.
	
	Conversely, if $M$ be supposed oriented, let $\Lambda_+$ denote the set of all full collections of positive orientation forms. The intersection of $\Lambda_+$ with $\bigwedge^* J^*_pM$ at any given point is obviously convex. Then, by a partition-of-unity argument one can construct a full collection of global orientation forms.
\end{proof}

Collecting results so far assembled, we are at last in a position to define an intrinsic sense of integration on the general differentiable manifold along the same lines as in Lee \cite{lee_smooth_manif}. The main subtlety in the existing theory consists in setting up the integral in such a way that one can pass from one coordinate chart to another, where the transition function between the two could be either orientation-preserving or orientation-reversing. These two cases exhaust the possibilities; in the latter one wants to introduce an additional minus sign to compensate for the change in sign of the volume form originating from the pullback formula \ref{pullback_formula}. Here, we evidently stand in need of a generalized concept of orientation type to connect charts of possibly different type in a consistent manner.

With the above apparatus in hand, higher-order analogues of Lee's Propositions 16.4, 16.5 and 16.6 are immediate.

\begin{corollary}[Change of Variables; cf. Lee, Theorem C.26 and Rudin, Theorem 10.24]\label{change_of_var_formula}
	Suppose $D$ and $E$ are open domains of integration in $\vvmathbb{R}^n$ and $G: \bar{D} \rightarrow \bar{E}$ a smooth map that restricts to a diffeomorphism from $D$ to $E$. For every continuous function $f: \bar{E} \rightarrow \vvmathbb{R}$ and for any macroscopic grand multi-index $A=(A_1,\ldots,A_n)$ (meaning $A_{11}, \ldots, A_{nn} = 1,\dots,r$ and the $A_{1,\ldots,n}$ are distinct and in increasing lexicographical order),
	\begin{equation}
		\int_E f \theta^A = \mathrm{sgn} \hat{J}^A_B \int_D f \circ G ~|\det \hat{J}^A_B| ~\theta^B
	\end{equation}
\end{corollary}

\begin{proposition}[Cf. Lee, Proposition 16.1]\label{integral_under_diffeo}
	Suppose $D$ and $E$ to be open domains of integration in $\vvmathbb{R}^n$ and $G: \bar{D} \rightarrow \bar{E}$ to be a smooth map that restricts to an orientation-preserving or orientation-reversing diffeomorphism from $D$ to $E$. If $\omega = \omega_A \theta^{y \wedge A}$ is a macroscopic $n$-form on $\bar{E}$, then
	\begin{equation}
		\int_D G^* \omega = \int_E \omega.
	\end{equation}
\end{proposition}
\begin{proof}
	Let $y_{1,\ldots,n}$ be standard coordinates in $E$ and $x_{1,\ldots,n}$ in $D$. Then the change of variables formula together with proposition \ref{pullback_formula} yields
	\begin{equation}
		\int_E \omega = \int_E \omega_A \theta^A = \int_D \omega_A \circ G ~\mathrm{sgn}~ \hat{J}^A_B |\det \hat{J}^A_B| ~\theta^{xB} = \int_D \omega_A \circ G \det \hat{J}^A_B ~\theta^{x \wedge B} = \int_D G^* \omega,  
	\end{equation}
	as was to be shown.
\end{proof}

Given a macroscopic volume form $\varphi$ and an orientation type specified by $\omega^{(1,\ldots,r)}$, define $\left[ \varphi \right]^\pm$ as follows: select a labeled basis $e_{1,\ldots,N}$ corresponding to the orientation type and expand $\varphi$ in homogeneous components: $\varphi=\varphi_1+\cdots +\varphi_r$ where $\varphi_s \in \Lambda^{n,s}$, $s=1,\ldots,r$. Then assign a $\pm$ factor to $\varphi_s$ according to whether it agrees or disagrees with the orientation type $(s,1)$. If $\varphi_s=0$ it is irrelevant what sign we take. Write now
\begin{equation}
	\left[ \varphi \right]^\pm = \pm \varphi_1 \pm \cdots \pm \varphi_r,
\end{equation}
where each monomial term receives the respective $\pm$ sign. Therefore, if $M$ is an oriented smooth $n$-manifold and $\omega$ macroscopic $n$-form on $M$ compactly supported in the domain of a single chart $(U,\varphi)$, define the integral of $\omega$ over $M$ to be
\begin{equation}\label{def_compact_supp_integral}
	\int_M \omega := \int_{\varphi[U]} \left[ \left( \varphi^{-1} \right)^* \omega
	\right]^\pm
\end{equation}
(cf. Lee, \cite{lee_smooth_manif}, Equation 16.1). Proceeding along the lines already sketched by Lee, our next result is the following:

\begin{proposition}[Cf. Lee, \cite{lee_smooth_manif}, Proposition 16.4]\label{indep_of_chart}With $\omega$ as above, $\int_M \omega$ will be the same for every smooth chart whose domain contains $\mathrm{supp}~ \omega$.
\end{proposition}
\begin{proof}
	Suppose $(U,\varphi)$ and $(V,\psi)$ are two charts such that $\mathrm{supp}~ \omega \subset U \cap V$. For starters, consider the case where both charts carry the same generalized orientation. Then $\psi \circ \varphi^{-1}$ is an orientation-preserving diffeomorphism from $\varphi[ U \cap V ]$ to $\psi[ U \cap V]$ and by proposition \ref{integral_under_diffeo},
	\begin{align}
		\int_{\psi[V]} \left( \psi^{-1} \right)^* \omega &= \int_{\psi[ U \cap V ]} \left( \psi^{-1} \right)^* \omega = \int_{\varphi[ U \cap V ]} \left( \psi \circ \varphi^{-1} \right)^* \left( \psi^{-1} \right)^* \omega \\ 
		&= \int_{\varphi[ U \cap V ]} \left( \varphi^{-1} \right)^* \psi^* \left( \psi^{-1} \right)^* \omega = \int_{\varphi[U]} \left( \varphi^{-1} \right)^* \omega.
	\end{align}
	If the charts carry different generalized orientations, then equation (\ref{def_compact_supp_integral}) will involve a different sequence of $\pm$ signs on the right-hand side, depending on which chart is used, but the mapping between charts $\psi \circ \varphi^{-1}$ is no longer orientation-preserving, in consequence of which extra minus signs enter via proposition \ref{integral_under_diffeo} so as to compensate in the computation above. In every case, the respective definitions of $\int_M \omega$ agree.
\end{proof}

To integrate a compactly supported macroscopic $n$-form over an entire manifold, one may resort to a partition of unity and apply the definition \ref{def_compact_supp_integral}. Let $U_i$ be a finite open cover of $\mathrm{supp}~ \omega$ by domains of oriented charts and let $\psi_i$ be a partition of unity subordinate to this open cover. Then define the integral of $\omega$ over $M$ to be the sum,
\begin{equation}\label{def_integral}
	\int_M \omega := \sum_i \int_M \psi_i \omega,
\end{equation}
where each of the terms on the right-hand side is well defined of course because the support of $\psi_i \omega$ rests inside the domain of $U_i$ by hypothesis. It remains, thus, to show that their sum is well defined as well:

\begin{proposition}[Cf. Lee, \cite{lee_smooth_manif}, Proposition 16.5]
	The definition of $\int_M \omega$ just given does not depend either on the choice of open cover or on the partition of unity.
\end{proposition}
\begin{proof}
	Suppose $V_j$ another finite open cover of $\mathrm{supp}~ \omega$ by domains of oriented charts and $\chi_j$ a partition of unity subordinate to it. Starting with the $U_i$ we have,
	\begin{equation}
		\int_M \psi_i \omega = \int_M \sum_j \chi_j \psi_i \omega = \sum_j \int_M \chi_j \psi_i \omega.
	\end{equation}
	Sum over $i$ to obtain
	\begin{equation}
		\sum_i \int_M \psi_i \omega = \sum_{ij} \int_M \chi_j \psi_i \omega.
	\end{equation}
	In as much as the integrand in each term on the right-hand side is compactly supported in a single chart, its integral over $M$ will be well defined by proposition \ref{indep_of_chart}. The same argument starting from the $V_j$ yields
	\begin{equation}
		\sum_j \int_M \chi_i \omega = \sum_{ij} \int_M \chi_j \psi_i \omega.
	\end{equation}
	Either way we arrive at the same expression for $\int_M \omega$.
\end{proof}

\begin{proposition}[Cf. Lee, \cite{lee_smooth_manif}, Proposition 16.6]\label{properties_of_integral}
	Suppose $M$ and $N$ are non-empty oriented smooth $n$-manifolds and $\omega, \eta$ are two compactly supported macroscopic $n$-forms on $M$.
	\begin{itemize}
		\item[$(1)$] If $a,b \in \vvmathbb{R}$, then
		\begin{equation}
			\int_M \left( a \omega + b \eta \right) = a \int_M \omega + b \int_M \eta.
		\end{equation}
		
		\item[$(2)$] If $M^\pm$ denotes $M$ with another orientation, then
		\begin{equation}
			\int_{M^\pm} \omega = \int_M \omega^\pm
		\end{equation}
		where $\omega^\pm$ is to be defined as above under the identity map considered as an orientation-changing diffeomorphism.
		
		\item[$(3)$] If $\omega$ is positively oriented, then $\int_M \omega > 0$.
		
		\item[$(4)$] If $F:N \rightarrow M$ is a diffeomorphism, then
		\begin{equation}
			\int_M \omega = \int_N \left[ F^* \omega \right]^\pm.
		\end{equation}
	\end{itemize}
\end{proposition}
\begin{proof}
	The statements (1) and (2) are immediate. Item (3) is tantamount to what is meant by a positively oriented form, viz., one whose coefficients are always positive in a chart with the $(+,\ldots,+)$ orientation, or speaking more generally in case no such chart is available, one whose coefficients are always positive resp. negative in a chart with the $(\pm,\ldots,\pm)$ orientation. The terms in the defining sum in equation (\ref{def_integral}) will all be non-negative, with at least one strictly positive.
	
	To prove (4), it suffices to assume $\omega$ compactly supported in a single oriented chart, since any compactly supported macroscopic $n$-form can be written as a finite sum of such by means of a partition of unity. Thus, suppose $(U, \varphi)$ is an oriented chart whose domain contains $\mathrm{supp}~ \omega$. In this case, (4) is just proposition \ref{integral_under_diffeo} applied to the oriented chart $(F^{-1}[U], \varphi \circ F)$ on $N$ whose domain contains $\mathrm{supp}~ F^* \omega$.
\end{proof}

\begin{remark}
	Although sufficient to enable the definition of a coordinate-invariant integral, our approach is inelegant in that it appeals to a standard basis. Perhaps it would be possible to do better, if someone clever enough were to apply his mind to the problem.
\end{remark}

\subsubsection{Stokes' Theorem}

The final piece of machinery we need will be a higher-order analogue of Stokes' theorem; something like this has to hold if we are to have the flexibility in wielding the integral up to a total differential to which one is accustomed in the theory as currently fashioned. Now, the concept of a differentiable manifold with boundary does not present any new complication in the higher-order setting; we suppose the charts to be subsets of semi-infinite Euclidean half-spaces (usually denoted $\vvmathbb{H}^n$ )or infinite Euclidean spaces $\vvmathbb{R}^n$ as the case may be. The transition functions are merely supposed to be definable on an extension to an open set (relative to $\vvmathbb{R}^n$) containing the open set relative to $\vvmathbb{H}^n$. But since jets are defined by a local construction, they will not differ at the boundary from what we are used to; in other words, one can simply restrict to stalks over the open set in $\vvmathbb{H}^n$. 

First, we must introduce concept of induced orientation on the boundary. Viewing the boundary as an embedded hypersurface, by Lee, Theorem 5.11 and Problem 8.4, the always exists a smooth outward-pointing vector field along $\partial M$. The following analogue of a proposition in Lee is not very difficult:

\begin{proposition}[Cf. Lee, Proposition 15.24]
	Let $M$ be a smooth manifold with boundary. Then $\partial M$ is orientable and all outward-pointing vector fields along $\partial M$ determine the same macroscopic orientation type.	
\end{proposition}
\begin{proof}
	Appeal to the proof of the proposition in Lee to get a macroscopic orientation form from any outward-pointing vector field $N$, which fixes the macroscopic part of the orientation uniquely.
	
	A similar argument delivers the remaining parts of the orientation. For $s=2$ let $y^{2,\ldots,n}$ be local coordinates on $\partial M$ and $N \otimes N$, $N \otimes \partial_2, \ldots, N \otimes \partial_n$ be second-order tangents having an outward-pointing part. Define
	\begin{equation}
		\eta^{(2)} = \iota^*_{\partial M} N \otimes N \lrcorner N \otimes \partial_2 \lrcorner \cdots \lrcorner N \otimes \partial_n \lrcorner \omega^{(2)}.
	\end{equation}
	The fact that $N$ is nowhere tangent to $\partial M$ ensures that none of the 2-vectors $N \otimes N$, $N \otimes \partial_2, \ldots, N \otimes \partial_n$ ever vanishes on $\partial M$. Then the non-vanishing of $\omega^{(2)}$ as well ensures that $\eta^{(2)}$ can be part of a full complement of the orientation forms on $\partial M$. The argument is similar for $s=3,\ldots,r$. At each stage, one forms all possible $s$-vectors with outward-pointing parts and contracts $\omega^{(s)}$ against them to get an orientation form $\eta^{(s)}$ on $\partial M$. 
	
	To show independence on $N$, let $N^\prime$ be another outward-pointing vector field. By Lee, Proposition 5.41 their first components are always negative. Hence, the transition matrices between the dual bases of any order always go between components of like sign. In other words, the Jacobian signatures are positive in every case and therefore the two determine the same orientation type.
\end{proof}
The difference remains, however, that in the more general setting the integrand could have non-vanishing coefficients at higher order when expressed in arbitrary coordinates. At this point, the immediately preceding proposition proves adept. For it says that, locally, we can go to new coordinates in which all higher-order coefficients cancel in the integrand. Owing to this felicitous property, proof of an analogue to Stokes' theorem can be had by following the derivation in the first-order case and making the obvious adjustments. Therefore, we assert the following:

\begin{theorem}[Stokes' theorem; cf. Lee \cite{lee_smooth_manif}, Theorem 16.11]
	Let $M$ be an oriented smooth manifold with boundary and let $\omega$ be a compactly supported smooth $(n-1)$-form on $M$. Then
	\begin{equation}
		\int_M \text{\th} \omega = \int_{\partial M} \omega.
	\end{equation}
\end{theorem}
\begin{proof}
	Start with the simplest non-trivial case where $M = \vvmathbb{H}^n$ itself. If the support of $\omega$ is to be compact, it must be possible to obtain it from a bounded set in $\vvmathbb{R}^n$ in the relative topology. Therefore, there exists a $\rho>0$ such that $\mathrm{supp}~\omega \subset R := [-\rho,\rho] \times \cdots [\rho,\rho] \times [0,\rho]$. Write in standard coordinates there,
	\begin{equation}
		\omega = \sum_{A: \deg A = n-1} \omega_A d^{\wedge A}.
	\end{equation}
	Thus,
	\begin{align}
		\text{\th} \omega &= \sum_{A: \deg A = n-1} \text{\th} \omega_A d^{\wedge A} \\
		&= \sum_{A: \deg A = n-1} \sum_{b \notin A} \partial_b \omega_A d^b \wedge d^{\wedge A} \\
		&= \sum_{A: \deg A = n-1} \sum_{C: \deg C = n} (-1)^{\chi_{CA}} \partial_{C-A} \omega_A d^{\wedge C}.
	\end{align}
	Hence we compute
	\begin{align}
		\int_{\vvmathbb{H}^n} \text{\th} \omega &= \sum_{A: \deg A = n-1} \sum_{C: \deg C = n}  (-1)^{\chi_{CA}} \int_R \partial_{C-A} \omega_A d^{\wedge C} \\
		&= \sum_{A: \deg A = n-1} \sum_{C: \deg C = n}  (-1)^{\chi_{CA}} \int_{-\rho}^\rho dx_1 \cdots \int_{-\rho}^\rho dx_{n-1}  \int_0^\rho dx_n \partial_{C-A} \omega_A d^{c_1} \wedge \cdots \wedge d^{c_n}.
	\end{align}
	Now, each grand multi-index $C$ appearing in the sum differs from $A$ in only one position, which we are free to anti-commute to the first in order to perform the integration over it alone, leaving the rest of the integrand intact. When we do so, we recover a sum over a single multi-index $b \notin A$. Note that the sign $(-1)^{\chi_{CA}}$ simply goes away (in virtue of its definition in the first place). Suppose $b=(\beta_1,\ldots,\beta_n)$. The cases $\beta_n=0$ and $\beta_n>0$ have to be handled separately. The inner integrand in the former case will amount just to
	\begin{equation}
		\int_{-\rho}^\rho dx_1 \cdots \int_{-\rho}^\rho dx_{n-1} \partial_{(\beta_1,\ldots,\beta_n)} \omega_A d^{(\beta_1,\ldots,\beta_n)} 
		d^{(a_1,\ldots,a_{n-1})} = \int_{-\rho}^\rho dx_2 \cdots \int_{-\rho}^\rho dx_{n-1} \partial \omega_A \bigg|_{-\rho}^\rho = 0,
	\end{equation}
	where we over the $x_1$ variable and appeal to the fundamental theorem of calculus. Here, without loss of generality we may suppose $\beta_1>0$ (else take any one of the $\beta_k>0$). Since $\rho$ has been chosen large enough to lie outside the support of $\omega$, the integral vanishes term by term except possibly for those involving $\beta_n>0$. For them, an analogous procedure yields
	\begin{align}
		\int_{\vvmathbb{H}^n} \text{\th} \omega &= \sum_{A: \deg A = n-1} \pm \int_{-\rho}^\rho dx_1 \cdots \int_{-\rho}^\rho dx_{n-1}  \int_0^\rho dx_n 
		\partial_b \omega_A d^b \wedge d^{\wedge A} \\
		&= \sum_{A: \deg A = n-1} \pm \int_{-\rho}^\rho dx_1 \cdots \int_{-\rho}^\rho dx_{n-1} \partial_{b \setminus \beta_n} \omega_A \bigg|_0^\rho d^{A \cup b \setminus \beta_n} \\
		&= \sum_{A: \deg A = n-1} \pm \int_{-\rho}^\rho dx_1 \cdots \int_{-\rho}^\rho dx_{n-1} \partial_{b \setminus \beta_n} \omega_A(x_1,\ldots,x_{n-1},0) d^{A \cup b \setminus \beta_n}.
	\end{align} 
	In the last line, however, an argument invoking integration by parts as above in the case $\beta_n=0$ leads to the dropping out of all terms with any $\beta_{1,\ldots,n-1}>0$. Hence, the only contribution that survives is in fact
	\begin{equation}
		\sum_{A: \deg A = n-1} \pm \int_{-\rho}^\rho dx_1 \cdots \int_{-\rho}^\rho dx_{n-1} \omega_A(x_1,\ldots,x_{n-1},0) d^A = \int_{\partial \vvmathbb{H}^n} \omega.
	\end{equation}
	But we get just the right-hand side in the induced orientation. If we had $\vvmathbb{R}^n$ for $M$ in place of $\vvmathbb{H}^2$, the same argument as above simplifies since we could put the support of $\omega$ inside a box $[-\rho,\rho]^n$ and integration by leads leads to vanishing on the left-hand side while the empty boundary of $\vvmathbb{R}^n$ causes the right-hand side to vanish as well.
	
	Let now $M$ represent any smooth manifold with boundary. If we can still put the support of $\omega$ inside the domain of a single chart $(U,\varphi)$, only a slight modification of the argument is required. Without loss of generality, we can go to an orientation on the chart such that the transition function $\varphi$ becomes simply oriented. In this case, we have
	\begin{equation}
		\int_M \text{\th} \omega = \int_{\vvmathbb{H}^n} (\phi^{-1} )^* \text{\th} \omega = \int_{\vvmathbb{H}^n}  \text{\th} (\phi^{-1} )^* \omega = \int_{\partial\vvmathbb{H}^n}  (\phi^{-1} )^* \omega,
	\end{equation}
	and equality of the outermost terms forces equality of the inner two. Cleary, nothing would change in this computation if the chart were a region in $\vvmathbb{R}^n$.
	
	We are left with the case of an arbitrary compactly supported smooth $(n-1)$-form, which evidently we should be able to handle with a partition of unity. To wit, cover $\mathrm{supp}~ \omega$ with finitely many charts $U_{1,\ldots,m}$ of given orientation and find a partition of unity $\psi_{1,\ldots,m}$ subordinate to this finite covering. Then we may apply the preceding arguments to compute
	\begin{align}
		\int_{\partial M} \omega &= \sum_k \int_{\partial M} \psi_k \omega =
		\sum_k \int_M \text{\th} (\psi_k \omega) \\
		&= \sum_k \int_M \left( \text{\th} \psi_k \wedge \omega +
		\frac{\alpha!}{\alpha_1!\alpha_2!} \text{\th}_{\alpha_1}\bigg|_{\alpha_{1,2}\ne 0} \psi_k \vee \text{\th}_{\alpha_2} \omega +
		\psi_k \text{\th} \omega \right) \\
		&= \int_M \text{\th} \omega,
	\end{align}
	since the sum over the index $k$ can be taken inside the exterior derivatives $\text{\th}_{\alpha_{1,2} \ne 0}$ and these vanish identically when applied to $\sum_k \psi_k = 1$, leaving only the last term in the sum on the final line.
\end{proof}

\begin{corollary}[Integrals of exact forms; cf. Lee \cite{lee_smooth_manif}, Corollary 16.13]
	If $M$ is a compact oriented smooth manifold without boundary, then the integral of every exact form over $M$ vanishes:
	\begin{equation}
		\int_M \text{\th} \omega = 0 \qquad \mathrm{if}~ \partial M = \emptyset.
	\end{equation}	
\end{corollary}

\begin{corollary}[Integrals of closed forms over boundaries; cf. Lee \cite{lee_smooth_manif}, Corollary 16.14]
	Let $M$ be a closed compact oriented smooth manifold with boundary. If the form $\omega$ is closed on $M$, then its integral over $\partial M$ vanishes:
	\begin{equation}
		\int_{\partial M} \omega = 0 \qquad \mathrm{if}~ \text{\th}\omega = 0 ~\mathrm{identically}.
	\end{equation}	
\end{corollary}

\subsection{A Generalized Riemannian Volume Form}\label{generalized_Riemannian_volume_form}

We seek to determine what the postulated form of the flat metric in Euclidean space with respect to the natural frame $\theta^{x_{1,\ldots,n}}$ ought to be. In one dimension, clearly we want $\theta^x$ itself to be our Riemannian volume form. How to derive it from a metric, however? Furthermore, how to extend the construction to many spatial dimensions? In view of the principle of plenitude invoked in {\S}\ref{quadrature_on_real_line} when the idea is generalized to many dimensions, we wish to arrange that deviations to any order in any one variable can couple to deviations to any order in any other variable, so the following Ansatz seems to be the most immediate:
\begin{equation}\label{canonical_flat_metric}
	\delta := \sum_{\lambda=1}^\infty \sum_{\begin{smallmatrix}\alpha, \beta ~\mathrm{such~that}~|\alpha|\ge\lambda, |\beta|\ge\lambda, \\\mathrm{either}~ \alpha \le \beta ~\mathrm{or}~ \beta \le \alpha \end{smallmatrix}}  
	\frac{|\alpha|!}{\alpha!}\frac{|\beta|!}{\beta!}
	d^\alpha \otimes d^\beta,
\end{equation}
where for the multi-indices $\alpha=(\alpha_1,\ldots,\alpha_n)$ and $\beta=(\beta_1,\ldots,\beta_n)$ we understand $\alpha \le \beta$ if $\alpha_{1,\ldots,n} \le \beta_{1,\ldots,n}$. The reasoning behind this choice will become clearer in a moment. To sketch it in advance, we shall say that we want to sum over all orders in the infinitesimals and for an infinitesimal at a given order to interact with all others at an order greater than or equal to its own. First of all, we must verify that the symmetrical generalized covariant rank-$\binom{2}{0}$ tensor field so defined forms a good metric. 

\begin{lemma}
	The would-be canonical form of the metric in Euclidean space, $\delta$, defined by the formula \ref{canonical_flat_metric} is everywhere non-degenerate and positive definite.
\end{lemma}
\begin{proof}
	Since the expression is manifestly translation-invariant, it suffices to show this at a single point. Therefore, take the point to be the origin and go to uniformizing coordinates $y_{1,\dots,n}$. A subtlety occurs in dimension greater than one which we must attend to. The principle of plenitude suggests that given  the $y_{1,\ldots,n}$, when their germ is taken at the origin so as to yield a mixed differential, it is not enough merely to include $\theta^{x_k}$ as above, for we could expect not only the differentials of higher order in $x_k$ itself to enter but also differentials involving potential displacements in any of the additional directions in space. Therefore, we wish to replace $\theta^{x_k}$ with
	\begin{equation}\label{natural_diff_in_many_dim}
		\varepsilon^{x_k} := \sum_{\alpha ~\mathrm{such~that}~\alpha_k\ge 1} d^\alpha,
	\end{equation}
	for $k=1,\ldots,n$. If we are to define the sought-for uniformizing coordinates $y_{1,\ldots,n}$ as power series in the Cartesian coordinates $x_{1,\ldots,n}$, it is evident from equation (\ref{natural_diff_in_many_dim}) that we shall have $y_k = x_k + \mathrm{h.o.}$; i.e., to linear order the coordinates remain the same. But then it is obvious from the implicit function theorem that we can indeed define the $y_{1,\ldots,n}$ in terms of the $x_{1,\ldots,n}$ in a small-enough neighborhood of the origin, which is all we need. Thus, suppose we have gone to uniformizing coordinates such that 
	$\varepsilon^{x_{1,\ldots,n}}=d^{y_{1,\ldots,n}}$. Inspection of the defining formula \ref{canonical_flat_metric} for $\delta$ reveals that with respect to the uniformizing coordinates it assumes the simple form,
	\begin{equation}\label{reduced_canonical_form_of_metric}
		\delta = \varepsilon^{x_\mu} \otimes \varepsilon^{x_\mu} +
		\varepsilon^{x_\mu} \varepsilon^{x_\nu} \otimes \varepsilon^{x_\mu} \varepsilon^{x_\nu} + \cdots
		= d^{y_\mu} \otimes d^{y_\mu} +
		d^{y_\mu} d^{y_\nu} \otimes d^{y_\mu} d^{y_\nu} + \cdots,
	\end{equation}
	where a summation convention over repeated spatial indices is understood to be in force. Note, the combinatorial factors in the defining formula \ref{canonical_flat_metric} take care of all repetitions of terms which agree among themselves by virtue of the commutativity of jet multiplication; i.e., $d^{x_1}d^{x_2}=d^{x_2}d^{x_1}$, hence under summation one gets $2d^{(1,1,0,\ldots,0)}$ and so forth.  Temporarily, identify $d^{y_k}$ with $\partial_{y_k}$ so that $\delta$ may be regarded as an endomorphism of $J^r_0(\vvmathbb{R}^n)$. Now, equation (\ref{reduced_canonical_form_of_metric}) says that with respect to the uniformizing coordinates $\delta$ is diagonal with non-zero and positive entries on the diagonal. Therefore it is non-degenerate and positive definite, as claimed.
\end{proof}
Will the formula \ref{canonical_flat_metric} reproduce $\theta^{x_1}\wedge\cdots\wedge\theta^{x_n}$ as its generalized Riemannian volume form? Indeed, the answer ought to be yes, at least up to constant factors reflecting the number of distinct permutations of the spatial indices in any given monomial differential. First we need a definition of what we mean by a Riemannian volume form in the generalized sense and proof of its existence.

\begin{proposition}\label{metric_canonical_form}
	Let $g$ be a generalized metric tensor defined on a differentiable manifold $M$ possibly with spatial dependence. If $e^\alpha$, $\alpha=1,\ldots,N=N_1+\cdots+N_r$, is an orthonormal coframe basis defined on an open set $U \subset M$, it may be arranged so that the first $N_1=n$ elements are jets of first degree, the second $N_2$ are jets of second degree and so on, with the last $N_r$ being jets of $r$-th degree. By the degree of a jet, we mean the least $1 \le s \le r$ such that its projection to $J^{*s}$ is non-zero.
\end{proposition}
\begin{proof}	
	Consider the projections $\pi_{r,1} e^\alpha$. If fewer than $n$ of these are non-zero, the metric $g$ would be degenerate, contrary to hypothesis. If, however, more than $n$ of the $\pi_{r,1} e^\alpha$ were non-zero there would have to be a relation of linear dependence among them and we could construct a 1-vector $X_1$ such that $g(X_1,X_1)=0$ (where we implicitly identify $X_1$ with its image under the canonical injection into $J^r$). Thus, precisely $n$ of the $e^\alpha$ have non-zero projection onto 1-jets. Relabel them as $e^{1,\ldots,n}$. 
	
	At the second stage, consider all projections $\pi_{r,2} e^\alpha$ for $\alpha=n+1,\ldots,N$. For the same reason as in the previous paragraph, precisely $N_2$ of these will be non-zero and linearly independent; relabel these as $e^{n_1,\ldots,n+N_2}$. Here, one would use a possible relation of linear dependence among the non-zero $\pi_{r,2} e^\alpha$ to construct a 2-vector $X_2$ such that $g(X_2,X_2)=0$. Clearly, we may continue along the same lines until the collection $e^{1,\ldots,N}$ is labeled as a sequence of $N_1$ degree-1 jets followed by $N_2$ degree-2 jets, and so on up to the last $N_r$ degree-$r$ jets, as required.
\end{proof}

In general, it is too much to ask for there to exist a coordinate frame such that the coordinate differentials recover the orthonormal coframe basis, even just at a point. The 11-sector can be diagonalized by a coordinate transform as usual, but a glance at the jet transformation law (\ref{jet_transf_law}) reveals that there is too little room for play in the 22 and higher sectors for a corresponding diagonalization to be carried out there as well, by means of a coordinate transform. Thus, the uniformizing coordinates in flat space constitute rather the exception.

\begin{proposition}[Cf. Lee, \cite{lee_smooth_manif}, Proposition 15.29]\label{exist_Riemannian_vol_form}
	Suppose $(M,g)$ is an oriented generalized Riemannian manifold of dimension $n \ge 1$. Then there exists a smooth macroscopic orientation form $\omega_g \in \Omega^n(M)$ that satisfies
	\begin{equation}
		\omega_g(E_1,\ldots,E_n) = 1
	\end{equation}
	for every local oriented frame $E_{1,\ldots,n}$ on $M$. The $\omega_g$ so obtained is unique modulo $\mathrm{ker}~ \pi_{r1}$.
\end{proposition}
\begin{proof}
	Assume first existence and show uniqueness. If $E_{1,\ldots,n}$ is any local oriented orthonormal frame on an open subset $U \subset M$ and $e^{1,\ldots,n}$ the dual coframe, we may put $\omega_g = f e^1 \wedge \cdots \wedge e^n$. Then the condition reduces to $f=1$ identically, so that $\omega_g = e^1 \wedge \cdots \wedge e^n$ proving that it is uniquely determined up to jets of second or greater degree.
	
	To show existence, any $\omega_g$ so defined must be independent of the choice of oriented orthonormal frame. If $\tilde{E}_{1,\ldots,n}$ is another oriented orthonormal frame with dual coframe $\tilde{e}^{1,\ldots,n}$, let
	\begin{equation}
		\tilde{\omega}_g = \tilde{e}^1 \wedge \cdots \wedge \tilde{e}^n.
	\end{equation}
	Now, by hypothesis there exist smooth transition functions $A^j_i$ such that $\tilde{E}_i = A^j_i E_j$ and $A^j_i(p) \in O(n)$ for every $p \in U$. Hence, $\det A^j_i = \pm 1$, but the plus sign is demanded by the fact that the orientations of the two frames agree. Compute
	\begin{equation}
		\omega_g(\tilde{E}_1,\ldots,\tilde{E}_n) = \det e^j(\tilde{E}_i) = \det A^j_i = \tilde{\omega_g}(\tilde{E}_1,\ldots,\tilde{E}_n).
	\end{equation}
	Thus $\tilde{\omega}_g = \omega_g$ and the proposed pointwise definition yields a global $n$-form, which is clearly smooth and satisfies the unicity condition by construction.
\end{proof}

It will be convenient to have a coordinate-based expression for the generalized Riemannian volume form, as the following guarantees:

\begin{proposition}[Cf. Lee, \cite{lee_smooth_manif}, Proposition 15.31]\label{coord_Riemannian_vol_form}
	Let $(M,g)$ be an oriented Riemannian manifold of dimension $n \ge 1$ and $e^{1,\ldots,N_1+\cdots+N_r}$ an orthonormal coframe ordered as in proposition \ref{metric_canonical_form}. Let $\mathscr{E}^1 = \mathrm{span} (g e^1,\ldots,g e^n) \subset \mathscr{J}^r$. In any oriented coordinate chart, the Riemannian volume form has the local coordinate expression
	\begin{equation}\label{Riemannian_vol_formula}
		\omega_g = \sqrt{ \det g|_{\mathscr{E}^1 \otimes \mathscr{E}^1}} \theta^1 \wedge \cdots \wedge \theta^n.
	\end{equation}
\end{proposition}
\begin{proof}
	Let $U$ be an oriented chart with coordinates $x_{1,\ldots,n}$ and $p \in M$. In these coordinates, $\omega_g = f \theta^1 \wedge \cdots \wedge \theta^n + \eta$ for some everywhere positive coefficient function $f \in C^\infty(U)$ and $\eta \in \mathrm{ker}~ \pi_{r1}$. We may set $\eta=0$. To compute what $f$ is, let $E_{1,\ldots,n}$ be any oriented orthonormal frame defined on a neighborhood of $p$ (which we may denote as $U$ again) and $e^{1,\ldots,n}$ the corresponding dual coframe. Write the natural frame in terms of the given orthonormal frame as
	\begin{equation}
		\partial_i = A^j_i E_j
	\end{equation}
	and compute
	\begin{align}
		f &= \omega_g \left( \partial_{y_1},\ldots,\partial_{y_n} \right) = 
		e^1 \wedge \cdots \wedge e^n \left( \partial_{y_1},\ldots,\partial_{y_n} \right) \nonumber \\
		&= \det e^j(\partial_{y_i}) = \det A^j_i.
	\end{align}
	Observe however that
	\begin{equation}
		g_{ij} = \langle \partial_{y_i}, \partial_{y_j} \rangle_g = \langle A^k_l E_k, A^l_j E_l \rangle_g = A^k_lA^l_k \langle E_k,E_l \rangle_g = (A^t A)_{ij}.
	\end{equation}
	Thus,
	\begin{equation}
		\det g|_{\mathscr{E}^1 \otimes \mathscr{E}^1} = \det A^t A = \det A^t \det A = (\det A)^2,
	\end{equation}
	whence $f = \det A = \pm \sqrt{\det g|_{\mathscr{E}^1 \otimes \mathscr{E}^1}}$. But by the common orientation of the frames $\partial_{y_1,\dots,y_n}$ and $E_{1,\ldots,n}$ the sign must be positive. A last point. Take two orthonormal frames whose elements of first degree are denoted by $e^{1,\ldots,n}$ resp. $\tilde{e}^{1,\ldots,n}$. Evaluate the coefficients $g_{ij}$ resp. $\tilde{g}_{ij}$. Now since the determinant of a matrix remains unchanged under a similarity transformation, we have $\det g_{ij} = \det \tilde{g}_{ij}.$ Therefore, the determinantal factor in equation (\ref{Riemannian_vol_formula}) is, in fact, independent of the choice of orthonormal frame.
\end{proof}

\begin{remark}
	It happens that indeed $\varepsilon^{x_1}\wedge\cdots\wedge\varepsilon^{x_n} = \theta^{x_1}\wedge\cdots\wedge\theta^{x_n}$. For it is evident that any given multi-index can appear once only on the right hand side. Na{\"i}vely, it could appear multiple times on the left hand side, but repetitions are eliminated by the wedge product.
\end{remark}

\begin{corollary}
	The Ansatz \ref{canonical_flat_metric} yields $\varepsilon^{x_1}\wedge\cdots\wedge\varepsilon^{x_n}$ as its generalized Riemannian volume form in flat space.
\end{corollary}
\begin{proof}
	With respect to the natural frame, the leading part of $\delta|_{\mathscr{E}^1 \otimes \mathscr{E}^1} = \mathrm{diag}(1,\ldots,1)$ so $\det \delta|_{\mathscr{E}^1 \otimes \mathscr{E}^1} = 1$ identically. In view of proposition \ref{coord_Riemannian_vol_form} the required result follows immediately.	
\end{proof}

Nonlinear terms should transform in the tensor product representation, then preservation of form is more or less automatic. Here we can readily see why the off-diagonal terms in equation (\ref{canonical_flat_metric}) are necessary; if they were absent, the general orthogonal transformation would lead to a linear combination of differentials which, when tensored with itself, would generate non-zero off-diagonal terms and thereby contradict the assumed diagonal form if it be desired that the latter stay unchanged.

\begin{proposition}
	The canonical form $\delta$ of the metric in flat space is invariant under Euclidean motions of $\vvmathbb{R}^n$.
\end{proposition}
\begin{proof}
	The metric $\delta$ is manifestly invariant under rigid translations. Since an arbitrary Euclidean motion can be decomposed into a translation followed by a rotation, it remains to show invariance under rotations. Let $O \in SO(n)$. From equation (\ref{vector_transf_law}), the action of $O$ on higher tangent vectors may be written in the form,
	\begin{equation}
		\hat{O} := O \oplus O \otimes O \oplus \cdots \oplus \overbrace{O \otimes \cdots \otimes O}^{r~\mathrm{times}}.
	\end{equation}
	The invariance of $\delta$ follows from the form of equation (\ref{reduced_canonical_form_of_metric}) as a sum of rank-1 pieces:
	\begin{align}
		\hat{O}^t \delta \hat{O} &= \left( \hat{O} \varepsilon^{x_\mu} \right) \otimes \left( \hat{O} \varepsilon^{x_\mu} \right) +
		\left( \hat{O} \varepsilon^{x_\mu} \varepsilon^{x_\nu} \right) \otimes \left( \hat{O} \varepsilon^{x_\mu} \varepsilon^{x_\nu} \right) + \cdots \nonumber \\
		&= \varepsilon^{x_\mu} \hat{O}^t \otimes \hat{O} \varepsilon^{x_\mu} +
		\varepsilon^{x_\mu} \varepsilon^{x_\nu} \hat{O}^t \otimes \hat{O} \varepsilon^{x_\mu} \varepsilon^{x_\nu} + \cdots \nonumber \\
		&= \varepsilon^{x_\mu} \otimes \varepsilon^{x_\mu} +
		\varepsilon^{x_\mu} \varepsilon^{x_\nu} \otimes \varepsilon^{x_\mu} \varepsilon^{x_\nu} + \cdots = \delta,
	\end{align}
	since by virtue of orthogonality $\hat{O}^t \hat{O} = 1$. This step completes the proof.
\end{proof}

\begin{remark}
	Suppose one has put the generalized metric into the canonical form indicated by proposition \ref{metric_canonical_form}. It would prove of great interest, were it possible to write out the nearby dependence in terms of invariants in the generalized Riemannian curvature tensor, thus in analogy to the standard result that the metric can be expanded as (cf. \cite{lee_riemannian_manif}, Problem 10.3),
	\begin{equation}
		g_{\mu\nu} = \delta_{\mu\nu} - \frac{1}{3} R_{\mu\alpha\beta\nu} x^\alpha x^\beta + O(|x|^3),
	\end{equation}
	where indices $\mu, \nu, \alpha, \beta$ range over $1,\ldots,n$. A certain amount of technical apparatus needs first to be developed; one would want at least a theory of Jacobi fields to higher order, which is not necessarily going to be so straightforward (cf. Lee, Chapter 10 in \cite{lee_riemannian_manif}).	
\end{remark}

\subsection{Rectification of Curves}

In this section, we undertake a simple application of the theory for the sake of illustration, namely, to solve the problem of rectification of a curve, first for an arbitrary trajectory supposed to by given by its injection into its enveloping space (taken to be Euclidean so as to avoid inessential complications), $\iota: [a,b] \rightarrow \vvmathbb{R}^n$. We may pull back the Riemannian volume form $\delta_n$ to yield a volume form on the interval $[a,b]$ with coordinate $t$, $a \le t \le b$. By proposition \ref{jet_normal_form} the volume form so obtained, viz., $\iota^* \delta_n$, will be locally equivalent to $\theta^t$ except at places of degeneracy. Cover $[a,b]$ by neighborhoods vanishingly small in the limit and pass over to the uniformizing coordinate within each one and integrate against it as in {\S}\ref{quadrature_on_real_line}. 

The procedure will perhaps become clearer if exemplified in two of the most straightforward imaginable cases, the straight line and the parabola. For convenience, we work with the uniformizing frame in $\vvmathbb{R}^2$. Suppose the interval $[t_0,t_1]$ be mapped into the plane via $\iota: t \mapsto (At,Bt)$. By virtue of its being a straight line, $d^t$ has no second-order dependence itself. Hence ,
\begin{align}
	d^t &= A d^x + B d^y \nonumber \\
	d^{tt} &= A^2 d^{xx} + 2 AB d^{xy} + B^2 d^{yy}.
\end{align}
Remembering definition (\ref{def_vector_basis}), the pullback metric may be computed directly from
\begin{align}
	\iota_* \partial_t &= A \partial_x + B \partial_y \nonumber \\
	\iota_* \partial_{tt} &= A^2 \partial_{xx} + AB \partial_{xy} + B^2 \partial_{yy}
\end{align}
according to the formalism of {\S}\ref{example_spaces}:
\begin{equation}
	\hat{g} = \begin{pmatrix} A & B & A^2 & AB & B^2 \end{pmatrix}
	\begin{pmatrix}
		1 & 0 & 0 & 0 & 0 \\
		0 & 1 & 0 & 0 & 0 \\
		0 & 0 & 1 & 0 & 0 \\
		0 & 0 & 0 & 2 & 0 \\
		0 & 0 & 0 & 0 & 1 \\
	\end{pmatrix}
	\begin{pmatrix}
		A \\ B \\ A^2 \\ AB \\ B^2
	\end{pmatrix},
\end{equation}
from which one finds
\begin{equation}
	\det \hat{g} = A^2 + B^2 + A^4 + 2 A^2 B^2 + B^4.
\end{equation}
The length of the straight line becomes to second order,
\begin{align}\label{straight_line}
	\int_{t_0}^{t_1} & \sqrt{A^2 + B^2 + A^4 + 2 A^2 B^2 + B^4} \left( d^t + d^{tt} \right) \nonumber \\
	&= \sqrt{A^2 + B^2 + A^4 + 2 A^2 B^2 + B^4} \left( t_1 - t_0 + \frac{1}{2} (t_1 - t_0)^2 \right).
\end{align}
As it must, this formula reduces to the familiar result for low enough velocities and short enough times.

As for the parabola, we have instead the injection $\iota: t \mapsto (At,\frac{1}{2}Bt^2)$. Then
\begin{align}
	\iota_* \partial_t &= A \partial_x + Bt \partial_y \nonumber \\
	\iota_* \partial_{tt} &= A^2 \partial_{xx} + ABt \partial_{xy} + B^2t^2 \partial_{yy} + \frac{1}{2} B \partial_y.
\end{align}
The analogue of equation (\ref{straight_line}) for the arclength becomes now,
\begin{equation}\label{parabolic_arclength}
	\int_{t_0}^{t_1} \sqrt{|\hat{g}|} \left( d^t + d^{tt} \right) = \int_{t_0}^{t_1} \sqrt{A^2 + B^2 \left( t+\frac{1}{2} \right)^2 + A^4 + 2 A^2 B^2 t^2 + B^4 t^4} \left( d^t + d^{tt} \right), 
\end{equation}
what is not anymore evaluable in simple analytic form (the answer is in principle expressible in terms of elliptic functions, but too involved to be worth reproducing here). How is this result compatible with our expectation that the initial speed of a body starting along a parabolic trajectory should be $|A|$ and not $\sqrt{A^2+\frac{1}{4} B^2 + A^4}$? The answer is that one has now to distinguish between the speed as the magnitude of the 1-vector part of the velocity (after orthogonal projection onto the image of $J^1$ under the canonical injection) and the rate of increase of arclength, which incorporates the contribution from all higher components (cf. the convention in force behind equation (\ref{antideriv_indefinite_part})).

%% file: discussion.tex
\section{Discussion}\label{discussion}

It is natural to wonder whether the simple idea of incorporating infinitesimals to higher order into the foundations of geometry leads to any novel perspectives on traditional problems in the field. The question may be answered in the affirmative, on the basis of results we have found so far. The modern, intrinsic characterization of analysis on manifolds can indeed be extended immediately and consistently, once the proper definitions of a jet connection, the jet stream, the Levi-Civita jet connection and the generalized Riemannian curvature tensor have been devised. It seems we have arrived at a novel conception of kinematics and inertial motion, which reflects corrections to the metrical structure of space in the infinitely small originating in the coupling among infinitesimals differing in order, which, of course, could not have been studied before when one restricted oneself to the case of 1-jets and 1-vectors only.

The most exciting consequence of our new approach to differential geometry may be its implications for physics, however. Here we wish to conceptualize the transition from mathematical to physical space by reference to Patrizi, who reputedly is the first to have realized the need for a systematic treatment of space, see \cite{grant}. Patrizi, in a departure from the predominant Aristotelianism of the scholastic tradition, distinguishes physical space from mathematical space being one of the first to envision empty space as an arena in which physical processes take place. Therefore, empty space enjoys a certain priority when Patrizi differentiates between the concepts of locus versus void. See also Patrizi (\cite{patrizi}, pp. 229-231) on spacium versus locus in space, not the other way around. In other words, for him (physical) space is nothing but the manifold of possible loci (ibid.). Indeed, we might paraphrase his ideas by saying that a locus is to be obtained as an ideal limit of possible configurations in physical space. But take care: Patrizi also insists (op. cit., p. 240) that space must be thought of as a principle and not just as something to be abstracted from experience. We are, of course, free to make what use of Patrizi we please, and in light of modern views on the geometrical foundation of physics the last contention on Patrizi’s part may circumscribed by extracting its kernel of meaning to amount to a statement to the effect that space itself is potentially dynamical, as in modern times Einstein was the first to suppose. Hence, the idea of space serves to mediate between the purely mathematical construct of the continuum and the phenomena of physics (\cite{grant}, p. 206) and so, when space is supplemented with a contingent of infinitesimals as we have done in the present contribution, there must by rights be a sequel in physics, which we have now to find out.

In subsequent work, we intend to set forth the following far-reaching applications to fundamental physics of the Riemannian geometry to higher order we have just sketched. Free fall must now be described by geodesics in a higher sense and Einstein’s field equations acquire a hierarchy of higher sectors. As soon as one goes to second order, the profound revision of the concept of inertia manifests itself on the phenomenological level as a modified Newtonian dynamics obeyed by spacecraft in the solar system (thereby explaining the flyby anomaly). A second major implication is to open a path to field-theoretical unification on the classical level in the spirit of Einstein. The present theory leads immediately to another fundamental force of nature arising at each successive order in the jets. The 1-jet case reduces to gravity as known in the conventional general theory of relativity, of course. It is striking that, at the 2-jet level, one recovers the electroweak forces including spontaneous symmetry breaking from a single postulate, the proper generalization of the equivalence principle. At the 3-jet level, following the same procedure we obtain chromodynamics without any ad hoc modifications. See \cite{paper_ii} and \cite{paper_iii}.